%% file: rig-memoir.tex
\documentclass{amsbook}
\usepackage{amsthm,amssymb}
\usepackage[T1]{fontenc}
\usepackage{lmodern}
\usepackage{xcolor}
\usepackage{enumerate}
\usepackage[
  pdftitle={A rigidity framework for Roe-like algebras},
  pdfauthor={Diego Mart\'{i}nez, Federico Vigolo},
  pdfsubject={Mathematics},
  colorlinks=true, linkcolor=blue, linkbordercolor=blue, citecolor=red, citebordercolor=red, urlcolor=blue, linktocpage=true
]{hyperref}
\usepackage[capitalise, nameinlink, noabbrev, nosort]{cleveref} 
\usepackage[lite,initials]{amsrefs}

\usepackage{mathtools}          
\usepackage{microtype}          
\usepackage{bbm}          
\usepackage{bm}           
\usepackage[shortcuts]{extdash}       
\usepackage{subcaption}         
\usepackage{accents}
\usepackage{xspace}

\usepackage{xparse} 

\usepackage{tikz-cd}

\usepackage[inline]{enumitem}
\setlist[enumerate,1]{font=\normalfont}
\crefname{enumi}{}{} 
\crefname{enumi}{}{} 

\newcounter{proofpartcount}

\input{AuxCommandsBecauseWindows}

\newcommand*{\munit}[2]{\matrixunit{#1}{#2}}

\newcommand{\rlone}{rank-$\leq$1-preserving\xspace}
\newcommand*{\sssec}[1]{\smallskip\noindent{\bf #1}}

\newcommand*{\rfLemCloseLikeClassic}{\cite{roe-algs}*{Lemmas 3.22 and 3.31}\xspace}
\newcommand*{\rfLemRelCloseToFunciton}{\cite{roe-algs}*{Lemma 3.31}\xspace}
\newcommand*{\rfSsecCrsMap}{\cite{roe-algs}*{Subsection 3.2}\xspace}
\newcommand*{\rfDefCrsMap}{\cite{roe-algs}*{Definition 3.24}\xspace}
\newcommand*{\rfDefCrsComp}{\cite{roe-algs}*{Definition 3.27}\xspace}
\newcommand*{\rfLemCrsComp}{\cite{roe-algs}*{Lemma 3.28}\xspace}
\newcommand*{\rfLemSameSupportSameFunction}{\cite{roe-algs}*{Corollary 3.42}\xspace}
\newcommand*{\rfPropTransposeIsInverse}{\cite{roe-algs}*{Proposition 3.44}\xspace}

\newcommand*{\rfDefCrseModule}{\cite{roe-algs}*{Definition 4.5}\xspace}
\newcommand*{\rfDefAdmissibleModule}{\cite{roe-algs}*{Definition 4.15}\xspace}
\newcommand*{\rfDefDiscreteModule}{\cite{roe-algs}*{Definition 4.19}\xspace}
\newcommand*{\rfRmkExtensionToDiscrete}{\cite{roe-algs}*{Remark 4.20}\xspace}
\newcommand*{\rfRmkDiscreteIsDiscrete}{\cite{roe-algs}*{Remark 4.23}\xspace}
\newcommand*{\rfPropAdmissibleIsDiscrete}{\cite{roe-algs}*{Proposition 4.24}\xspace}
\newcommand*{\rfDefAmple}{\cite{roe-algs}*{Definition 4.11}\xspace}
\newcommand*{\rfCorExistsAmplePartition}{\cite{roe-algs}*{Corollary 4.28}\xspace}
\newcommand*{\rfLemQuasisupportOfVectors}{\cite{roe-algs}*{Lemma 4.3}\xspace}

\newcommand*{\rfUnifRoeAlg}{\cite{roe-algs}*{Examples 4.21 and 6.4}\xspace}

\newcommand*{\rfLemSupportComposition}{\cite{roe-algs}*{Lemma 5.3}\xspace}
\newcommand*{\rfDefCoarseSupport}{\cite{roe-algs}*{Definition 5.6}\xspace}
\newcommand*{\rfPropCoarseSupport}{\cite{roe-algs}*{Proposition 5.7}\xspace}
\newcommand*{\rfLemProperIffProper}{\cite{roe-algs}*{Lemma 5.11}\xspace}
\newcommand*{\rfLemOperationsOnQuasiLocal}{\cite{roe-algs}*{Lemma 5.18}\xspace}

\newcommand*{\rfLemQlIsAlgebra}{\cite{roe-algs}*{Lemma 6.6}\xspace}
\newcommand*{\rfLemLcIsAlgebra}{\cite{roe-algs}*{Lemma 6.6}\xspace}
\newcommand*{\rfMultRoeinCcp}{\cite{roe-algs}*{Corollary 6.9}\xspace}

\newcommand*{\rfRoelikeDisconnected}{\cite{roe-algs}*{Corollaries 6.16 and 6.18}\xspace}
\newcommand*{\rfThmApproxUnit}{\cite{roe-algs}*{Theorem 6.20}\xspace}
\newcommand*{\rfRoeAlgCap}{\cite{roe-algs}*{Theorem 6.20}\xspace}
\newcommand*{\rfThmCartanSubalg}{\cite{roe-algs}*{Theorem 6.32}\xspace}
\newcommand*{\rfApproxUnit}{\cite{roe-algs}*{Theorem 6.20}\xspace}

\newcommand*{\rfPropControlledIffAdControlled}{\cite{roe-algs}*{Proposition 7.1}\xspace}

\newcommand*{\rfCorExistCoveringIsom}{\cite{roe-algs}*{Corollary 7.12 and Remark 7.13 (ii)}\xspace}
\newcommand*{\rfCorAdProperLocalCpt}{\cite{roe-algs}*{Corollary 7.3}\xspace}

\newcommand*{\rfLemCoveringUniIsUnique}{\cite{roe-algs}*{Lemma 7.14}\xspace}
\newcommand*{\rfDefCEandCUni}{\cite{roe-algs}*{Definition 7.15}\xspace}
\newcommand*{\rfIsoCEtoCtrUni}{\cite{roe-algs}*{Theorem 7.16}\xspace}

\newcommand*{\rfEmbeddingCtrUniintoOut}{\cite{roe-algs}*{Theorem 7.18}\xspace}

\newcommand*{\rfDefRoeAlgsSpace}{\cite{roe-algs}*{Definition 8.1}\xspace}
\newcommand*{\rfLemExistsDiscreteIffCardinality}{\cite{roe-algs}*{Lemma 8.9}\xspace}

\begin{document}

\title[A rigidity framework for Roe-like algebras]{A rigidity framework for Roe-like algebras}
\date{\today}

\author[Diego Mart\'{i}nez]{Diego Mart\'{i}nez $^{1}$}
\address{Department of Mathematics, KU Leuven, Celestijnenlaan 200B, 3001 Leuven, Belgium.}
\email{diego.martinez@kuleuven.be}

\author[Federico Vigolo]{Federico Vigolo $^{2}$}
\address{Mathematisches Institut, Georg-August-Universit\"{a}t G\"{o}ttingen, Bunsenstr. 3-5, 37073 G\"{o}ttingen, Germany.}
\email{federico.vigolo@uni-goettingen.de}

\maketitle


\begin{abstract}
  In this memoir we develop a framework to study rigidity problems for Roe-like \cstar{}algebras of countably generated coarse spaces.
  The main goal is to give a complete and self-contained solution to the problem of \cstar{}rigidity for proper (extended) metric spaces. Namely, we show that (stable) isomorphisms among Roe algebras always give rise to coarse equivalences. 

  The material is organized as to provide a unified proof of \cstar{}rigidity for Roe-like \cstar{}algebras, algebras of operators of controlled propagation, and algebras of quasi-local operators.

  We also prove a more refined \cstar{}rigidity statement which has several additional applications. For instance, we can put the correspondence between coarse geometry and operator algebras in a categorical framework, and we prove that the outer automorphism groups of all of these algebras are isomorphic to the group of coarse equivalences of the starting coarse space.
\end{abstract}


\newcommand{\subjclass}[2][2020]{\bigskip\textbf{MSC #1:} #2}  
\subjclass[2020]{53C24, 48L89, 51F30, 52C25, 51K05}
%
%

\newcommand{\keywords}[1]{\bigskip\textbf{keywords:} #1}  
\keywords{Roe algebras; rigidity; coarse geometry}



\tableofcontents

\chapter{Introduction}

The main focus of this memoir is a \emph{\cstar{}rigidity} phenomenon that links coarse geometric properties of proper metric spaces on the one hand and algebraic/analytic properties of certain \cstar{}algebras on the other.

The general idea is that one can associate with certain geometric constructs some analytic counterparts. This can be seen as saying that the analytic setup is ``flexible enough'' to accommodate for the geometry of the spaces. It is a much deeper and rather surprising phenomenon that the analytic side is also ``rigid enough'' to imply that such an association often gives rise to a perfect correspondence between the two worlds.
This is the aspect that we intend to explain in this work. In the following, we will be particularly interested with two such results (a \emph{stable} and a \emph{refined rigidity} theorems) and their consequences.

\section{The Stable Rigidity Theorem: background and statement}\label{intro:background}
Let us start our explanation of the \cstar{}rigidity phenomenon from the coarse geometric side. A function $f\colon X\to Y$ between metric spaces is called \emph{controlled} if for every $r\geq 0$ there is some $R\geq 0$ such that if $d(x,x')\leq r$ then $d(f(x),f(x'))\leq R$.\footnote{\, 
For ease of notation, in this introduction $R$ denotes a positive real number, while in the rest of the memoir it denotes a binary relation. The role that $r>0$ and $R>0$ play here will be taken over by controlled entourages $E \in \CE$ and $F \in \CF$ respectively.}

A mapping $g\colon Y \to X$ is a \emph{coarse inverse} for $f$ if $g\circ f$ is within bounded distance from the identity (\emph{i.e.}\ $g(f(x))$ stays uniformly close to $x$ when $x$ ranges in $X$). The spaces $X$ and $Y$ are then \emph{coarsely equivalent} if there are controlled maps $f\colon X\to Y$ and $g\colon Y\to X$ that are coarse inverses to one another.
\emph{Coarse geometry} is concerned with the study of geometric properties that are invariant up to coarse equivalence.

One way of paraphrasing what coarse equivalences are is that the only piece of information about $X$ that is always retained up to coarse equivalence is which families of subsets of $X$ have uniformly bounded diameter. Considering what a weak form of equivalence this is, it is impressive how much information the coarse geometry of a space actually preserves.
Typical spaces of coarse geometric interest are Cayley graphs of finitely generated groups and coverings of compact manifolds: in this context there are deep connections between algebra, geometry/topology and analysis.
In this brief introduction to the subject, it is not possible to make justice to the ramifications of the coarse geometric approach. We shall therefore content ourselves to refer the reader to the books \cites{nowak-yu-book,dructu2018geometric,gromov-1993-invariants,roe_lectures_2003,roe_index_1996} for a glimpse to the depth of these ideas.

\smallskip

The introduction to the analytical side is perhaps most clear in the context of Riemannian manifolds. With such a manifold $M$ are associated the Hilbert space $L^2(M)$ and various algebras of operators on $L^2(M)$.
One such algebra of special interest is the \Star{}algebra of locally compact operators of finite propagation: its closure in $\CB(L^2(M))$ is known as the \emph{Roe algebra} $\roecstar M$ of $M$.
The original motivation to introduce this algebra stemmed from index-theoretical considerations.
Moreover, the K-theory of $\roecstar M$ is related with a coarse K-homology of $M$ via a certain \emph{coarse assembly map}. When this map is an isomorphism (\emph{i.e.}\ $M$ satisfies the ``coarse Baum--Connes Conjecture''), very strong consequences can be deduced, for instance with regard to the Novikov Conjecture  \cites{aparicio_baum-connes_2019,higson-roe-1995-coarse-bc,yu1995coarse,yu_coarse_2000,skandalis-tu-yu-coarse-gpds-02,yu-1998-novikov-groups}. More recently, the Roe algebra has also been proposed to model topological phases in mathematical physics \cite{ewert2019coarse}.

Such a Roe algebra can be similarly defined for any proper metric space $X$. To do so, the role of the Hilbert space $L^2(M)$ is played by an appropriate choice of \emph{geometric module} $\CH_X$, so that $\roecstar X$ can be defined as a \cstar{}subalgebra of $\CB(\CH_X)$. Specifically, $\CH_X$ must be an \emph{ample} module (the idea of using modules to abstract the properties of spaces of functions on manifolds goes back at least as far as \cite{atiyah1970global}).

Besides Roe algebras, we shall work with other related \cstar{}subalgebras of $\CB(\CH_X)$.
Specifically, the \cstar{}algebra $\cpcstar X$ of operators that are approximable by controlled propagation\footnote{\,
  In this content ``controlled propagation'' is a synonym for  ``finite propagation''. We use the former terminology because it extends more naturally to the setting of coarse spaces.} operators
(such operators are also referred to as being \emph{band dominated}~\cite{braga_gelfand_duality_2022}), and the \cstar{}algebra $\qlcstar X$ of \emph{quasi-local} operators. The former is the most closely related to the Roe algebras: the only difference is that the approximating operators need not be locally compact. The latter has other advantages, in that the quasi-locality condition is often simpler to verify. Without entering into details, let us here mention that $\cpcstar{X}$ is always a subalgebra of $\qlcstar X$ (approximable operators must be quasi-local), and $\cpcstar{X}=\qlcstar X$ if $X$ has Yu's property A~\cite{spakula-zhang-2020-quasi-loc-prop-a}. Examples where the inclusion $\cpcstar{X}\subseteq\qlcstar X$ is strict are given in \cite{ozawa-2023}.

If $X$ is a discrete metric space and $\CH_X$ is taken to be the---non ample---module $\ell^2(X)$, the \Star{}algebra of operators of controlled propagation is also known as translation algebra \cite{gromov-1993-invariants}*{page 262}. Its closure is the \emph{uniform Roe-algebra} $\uroecstar X$ (in other words, $\uroecstar X$ is the non-ample analog of either $\roecstar\variable$ and $\cpcstar\variable$).
In this memoir we shall collectively refer to all the \cstar{}algebras discussed above as \emph{classical Roe-like \cstar{}algebras}. We postpone a more detailed discussion of the Roe-like \cstar{}algebras to \cref{subsec: roe algs of modules}.

\smallskip

It has been known for a long time that Roe-like \cstar{}algebras are intimately connected with the coarse geometry of the underlying metric space. 
Specifically, it is simple to show that a coarse equivalence between two proper metric spaces $X$ and $Y$ gives rise to \Star{}isomorphisms at the level of the associated algebras $\roecstar\variable$, $\cpcstar\variable$, $\qlcstar\variable$, and a stable isomorphism (equivalently, a Morita equivalence) between $\uroecstar{X}$ and $\uroecstar{Y}$.\footnote{\,
  Two \cstar{}algebras $A$ and $B$ are \emph{stably \Star{}isomorphic} if $A \otimes \CK(\CH)\cong B \otimes \CK(\CH)$, where $\CK(\CH)$ is the algebra of compact operators on the separable Hilbert space $\CH$.
  Showing that coarse equivalences give rise to stable isomorphisms among uniform Roe algebras is more delicate that the construction of the isomorphisms among the other Roe-like \cstar{}algebras, and also requires the spaces to be \emph{uniformly} locally finite \cite{brodzki_property_2007}*{Theorem 4}.
  This is due to the lack of the ampleness condition on the geometric module.} 
The \emph{\cstar{}rigidity} question asks whether the converse is true.
Namely, it asks whether two metric spaces $X$ and $Y$ must be coarsely equivalent as soon as (one of) their associated Roe-like \cstar{}algebras are \Star{}isomorphic (resp.\ Morita equivalent).

The first foundational result on the problem of \cstar{}rigidity was obtained by \v{S}pakula and Willett in \cite{spakula_rigidity_2013}, where it is proven that \cstar{}rigidity holds for bounded geometry metric spaces with Yu's \emph{property A} \cite{yu_coarse_2000}. The latter is a rather mild but important regularity condition \cites{roe_ghostbusting_2014,sako_property_2014,brodzki_uniform_2013,sako_finite_2019}, which is very useful in the context of \cstar{}rigidity (see \cref{ssubsec:intro: approximations} below for more details about its use here).
After \cite{spakula_rigidity_2013}, a sequence of papers gradually improved the state of the art by proving \cstar{}rigidity in more and more general settings \cites{braga2021uniform,braga2020embeddings,braga_farah_vignati_2022,braga_gelfand_duality_2022,braga_farah_rig_2021,spakula_maximal_2013,braga2020coarse,li-spakula-zhang-2023-measured-asym-exp,jiang2023rigidity}.
Two final breakthroughs were obtained in \cites{braga_rigid_unif_roe_2022,rigidIHES}. In the former it is proved that uniformly locally finite metric spaces having stably isomorphic uniform Roe algebras must be coarsely equivalent. 
The latter manages to prove rigidity in the ample setting, showing that bounded geometry metric spaces with isomorphic $\roecstar\variable$ or $\cpcstar\variable$ must be coarsely equivalent.

\smallskip

At this point the picture is almost complete. The main grievances that remain are the following.
\begin{enumerate}[label=(\roman*)]
  \item In the ample case, \cite{rigidIHES} does not prove \emph{stable} rigidity;
  \item rigidity was not shown for $\qlcstar{\variable}$ (note however that the non-ample analogue of $\qlcstar\variable$ is also done in \cite{braga_rigid_unif_roe_2022});
  \item the rigidity results only apply to bounded geometry metric spaces;
  \item and the approaches mentioned before have parallels between them, but a unified approach is still missing.
\end{enumerate}
Of these issues, the first one would be easy to solve: one could improve the result of \cite{rigidIHES} to prove a stable rigidity the same way that it is done in \cite{braga_rigid_unif_roe_2022}. The other three are more serious, because \cite{rigidIHES} relies on deep works of Braga, Farah and Vignati that have only been shown in restricted settings \cites{braga_farah_rig_2021,braga_gelfand_duality_2022}.

In particular, one ``philosophical'' complaint is that many of the results surrounding the \cstar{}rigidity problem focus on one kind of Roe-like \cstar{}algebra at a time (quite often the uniform Roe algebra $\uroecstar{X}$, which is in many respects easier to analyse). Notable exception in this regard is the seminal paper \cite{spakula_rigidity_2013} (see also \cite{spakula2019relative}).
This is somewhat disappointing, as these individual rigidity results should be part of a broader ``rigidity theory''.

The first goal of this memoir is to develop the theory necessary to provide a full, completely general and unified solution to the stable \cstar{}rigidity problem. Namely, we give a self-contained proof of the following.

\begin{alphthm}[cf.\ \cref{thm: stable rigidity}]\label{thm:intro: rigidity}
  Let $X$ and $Y$ be proper (extended) metric spaces.
  If there is an isomorphism 
  \[
    \roeclikeone{X}\otimes \CK(\CH_1) \cong \roecliketwo{Y}\otimes\CK(\CH_2),
  \]
  then $X$ and $Y$ are coarsely equivalent.\footnote{\,
  In this context the maximal and minimal tensor products coincide, and we denote either/both by $\otimes$.}
\end{alphthm}

In the above, $\roeclikeone\variable$ and $\roecliketwo\variable$ denote any of the Roe-like \cstar{}algebras we consider in this work, \emph{i.e.}\ $\roecstar\variable$, $\cpcstar\variable$, $\qlcstar\variable$ or $\uroecstar\variable$. Likewise, $\CK(\CH_1)$, $\CK(\CH_2)$ are the compact operators on the Hilbert spaces $\CH_1$ and $\CH_2$, which may possibly be different and of arbitrary dimension (including both finite rank and non-separable spaces). 
For instance, \cref{thm:intro: rigidity} lets us compare $\uroecstar{X}\otimes\CK(\CH)$ with $\roecstar Y$, of $\cpcstar X$ with $\qlcstar Y$. 
Our choice of notation is meant to highlight the unifying nature of the proof we provide.

\begin{remark}
  As a matter of fact, if we are willing to exclude $\roecstar\variable$ from the statement, the rest of \cref{thm:intro: rigidity} applies even without the properness assumption.
\end{remark}

Combining \cref{thm:intro: rigidity} with classical results, we immediately obtain the following corollary
(the separability assumption below can be disposed: we only kept it to make the relation with \cref{cor:roe-rigidity} more immediate).

\begin{alphcor}[cf.\ \cref{cor:roe-rigidity}] \label{cor:intro: roe-rigidity}
  Let ${X}$ and ${Y}$ be proper separable extended metric spaces. Then the following are equivalent:
   \begin{enumerate} [label=(\roman*)]
     \item \label{cor:intro: roe-rigidity:1} $ X$ and $ Y$ are coarsely equivalent.
     \item \label{cor:intro: roe-rigidity:2} $\roecstar{{X}}$ and $\roecstar{{Y}}$ are \Star{}isomorphic.
     \item \label{cor:intro: roe-rigidity:3} $\cpcstar{{X}}$ and $\cpcstar{{Y}}$ are \Star{}isomorphic.
     \item \label{cor:intro: roe-rigidity:4} $\qlcstar{{X}}$ and $\qlcstar{{Y}}$ are \Star{}isomorphic.
   \end{enumerate}
   If $X$ and $Y$ are uniformly locally finite, then the above are also equivalent to the following:
   \begin{enumerate}[label=(\roman*),resume]
     \item \label{cor:intro: roe-rigidity:5} $\uroecstar{ X}$ and $\uroecstar{ Y}$ are stably \Star{}isomorphic.
     \item \label{cor:intro: roe-rigidity:6} $\uroecstar{ X}$ and $\uroecstar{ Y}$ are Morita equivalent.
   \end{enumerate}
 \end{alphcor}

 \begin{remark}
  \cref{cor:intro: roe-rigidity} shows that the (stable) isomorphism type of the Roe-like \cstar{}algebras are complete invariants for the coarse geometry of (bounded geometry) metric spaces.
  It also follows that every coarse geometric invariant of metric spaces gives rise to an invariant under stable \Star{}isomorphism for such Roe-like \cstar{}algebras. Since these are \cstar{}algebras of great interest within the operator-algebra community, it is an intriguing prospect to search for intrinsic analytic characterizations of such invariants.
 \end{remark}

\begin{remark}
  Recall that an \emph{extended metric space} is a set equipped with a metric function that is allowed to take the value $+\infty$.
  The setting of extended metric spaces is a natural one for our results, because these precisely correspond to coarse spaces with countably generated coarse structures (cf.\ \cref{prop:metrizable coarse structure iff ctably gen}). In fact, this countability assumption is key in various points of our arguments. 
  Namely, \cref{lem:spakula-willett easy}, \cref{thm: uniformization phenomenon,prop:equi-controlled rank-1} all make use of it.
  The extent to which countable generation is necessary for \cstar{}rigidity to apply is not yet entirely clear. However, \cref{ex:non-cnt-gen-coarse-space-with-roe-compact} below does show that rigidity fails for Roe algebras if both the countable generation and coarse local finiteness assumptions are dropped (see \cref{rem:rigidity-fails-without-cnt-gen}).

  One additional motivation to work with \emph{extended} metric spaces comes from considerations regarding reduced crossed products $\ell^\infty(S)\rtimes_{\rm red}S$ of (quasi-countable) inverse monoids. Indeed, such \cstar{}algebras are isomorphic to the uniform Roe algebra of extended metric spaces where the distance function \emph{does} take the value $+\infty$. We refer to \cref{cor:inv-sem-roe-rigid} for a more detailed discussion.
\end{remark}

\section{The route to rigidity}
Since the first work of \v{S}pakula and Willett \cite{spakula_rigidity_2013}, the route to rigidity has---more or less implicitly---always consisted of three fundamental milestones: \emph{spatial implementation of isomorphisms}; a \emph{uniformization phenomenon}; and the construction of appropriate \emph{approximations}.
What changed and progressed over time are the steps that are taken to reach these milestones, with fundamental innovations arising from the works of Braga, Farah and Vignati (various key proofs in our arguments are inspired by \cites{braga_farah_rig_2021,braga_gelfand_duality_2022}).
This work also follows this well trodden path. However, the language and various of the techniques that we develop in doing so are novel and represent one of the main contributions of this work.
Specifically, we found it helpful to recast the problem of \cstar{}rigidity using the language of coarse geometry developed by Roe in \cite{roe_lectures_2003}. This required a certain amount of ground work, which was done in \cite{roe-algs}: there Roe-like \cstar{}algebras for (coarse geometric modules of) arbitrary coarse spaces were introduced, and the foundations upon which we now build this memoir were laid.
This approach allowed us to distill the phenomena underlying the solution to the \cstar{}rigidity problem to their essence, which we believe to be an important added value of this work.
Introducing this non-standard language requires some time, and for this reason we do not use it in this introduction. This means that all the statements appearing here are simplified versions of the results that we actually prove, obtained by restricting our theorems to the more classical setting of proper (extended) metric spaces.

We shall now enter a more detailed discussion of the strategy we follow in order to prove \cref{thm:intro: rigidity}.

\subsection{Spatial implementation of isomorphisms}\label{ssubsec:intro: spatial implementation}
As already explained, Roe-like \cstar{}algebras of $X$ are always concretely represented as subalgebras of $\CB(\CH_X)$ for some choice of geometric module $\CH_X$.
In its simplest form, the spatial implementation phenomenon is the observation that if $\phi\colon\roecstar{X}\to\roecstar{Y}$ is an isomorphism, then there is a unitary operator $U\colon\CH_X\to\CH_Y$ at the level of defining modules such that $\phi$ coincides with the restriction of $\Ad(U)$ to $\roecstar X$, where $\Ad(U)$ denotes the action by conjugation $t\mapsto UtU^*$. This phenomenon follows from the classical result that an isomorphism between the compact operators $\CK(\CH)\to\CK(\CH)$ is always spatially implemented (spatially implemented homomorphisms are sometimes called \emph{inner}).

What we need in this work is a slight refinement of the above, because we need to show that also an isomorphism $\phi\colon \roeclikeone X\otimes\CK(\CH_1)\to\roecliketwo Y\otimes\CK(\CH_2)$ is spatially implemented by a unitary $U\colon\CH_X\otimes\CH_1\to\CH_Y\otimes\CH_2$. In our setting, proving this requires a few extra technicalities (cf.\ \cref{prop: isos are spatially implemented}) because we work with extended metric spaces and hence the Roe-like \cstar{}algebras need not contain all the compact operators. These steps are more or less standard, so we do not substantially improve the existing results, nor shall we comment on it further.

\subsection{The uniformization phenomenon}\label{ssubsec:intro: uniformization}
Suppose $\phi \colon \cpcstar{X} \to \cpcstar{Y}$ is a \Star{}homomorphism (which needs not be an isomorphism, and not even an embedding). 
Since $\cpcstar\variable$ consists of operators that are norm-approximable by operators of controlled propagation (see \cref{subsec: coarse supp and propagation,subsec: roe algs of modules} for actual definitions), $\phi$ clearly satisfies:

\medskip
\noindent\begin{minipage}{0.1\textwidth}
  \begin{center}
    (a)\label{intro:not-approx-controlled}  
  \end{center}
\end{minipage}\begin{minipage}{0.9\textwidth}
  Fix $\varepsilon > 0$ and $r \geq 0$. For every $t \in \cpcstar{X}$ of propagation $r$, there is some $R \geq 0$ and $s \in \cpcstar{Y}$ of propagation controlled by $R$ such that $\norm{\phi(t) - s} \leq \varepsilon$.
\end{minipage}
\smallskip

Note that, a priori, $R = R(\varepsilon, r, t)$ depends on $t$. The \emph{uniformization phenomenon} states that $R$, in fact, may be chosen independently of $t$. Namely, such a homomorphism $\phi$ automatically satisfies the stronger condition:

\medskip
\noindent\begin{minipage}{0.1\textwidth}
  \begin{center}
    (A)\label{intro:yes-approx-controlled} 
  \end{center}
\end{minipage}\begin{minipage}{0.9\textwidth}
  Fix $\varepsilon > 0$ and $r \geq 0$. There is some $R \geq 0$ such that for every $t \in \cpcstar{X}$ of propagation $r$, there is some $s \in \cpcstar{Y}$ of propagation controlled by $R$ such that $\norm{\phi(t) - s} \leq \varepsilon$.
\end{minipage}
\smallskip

If $\phi\colon\CB(\CH_X)\to\CB(\CH_Y)$ satisfies \hyperref[intro:yes-approx-controlled]{(A)} then we say that it is \emph{approximately controlled} (cf.\ \cref{def:quasi-controlled mapping}). 
We may then state (a special case) of the uniformization phenomenon as the following.
\begin{alphthm}[cf.\ \cref{thm: uniformization}] \label{thm:intro: uniformization}
  Let $X$ and $Y$ be proper extended metric spaces.
  Then any strongly continuous \Star{}homomorphism $\phi \colon \cpcstar{X} \rightarrow \cpcstar{Y}$ is approximately-controlled.
\end{alphthm}

As stated, \cref{thm:intro: uniformization} will not be surprising to an expert: an almost complete proof can already be found in \cite{braga_gelfand_duality_2022}*{Theorem 3.5}.
What \emph{is} surprising is the precise result we actually prove (see \cref{thm: uniformization}).
Namely, the uniformization phenomenon manifests itself under much weaker hypotheses than we had anticipated. For instance, $\phi$ needs not be defined on the whole of $\cpcstar{X}$, properness is not needed and neither is separability. We may even prove ``effective'' versions of \cref{thm: uniformization} (see \cref{rmk:effective uniformization}). We refer to \cref{rmk: on uniformization hypotheses} for more observations regarding the hypotheses of \cref{thm:intro: uniformization}.

Moreover, we also show that a \Star{}homomorphism $\phi \colon \cpcstar{X} \to \qlcstar{X}$ is automatically \emph{quasi-controlled} (cf.\ \cref{def:quasi-controlled mapping}). The definition of ``quasi-control'' is analogous to that of ``approximate control'', and is simply rephrasing \hyperref[intro:yes-approx-controlled]{(A)} with a \emph{quasi-local} mindset (as opposed to an \emph{approximate} one).
It follows easily from the definitions that an approximately controlled mapping $\phi\colon\CB(\CH_X)\to\CB(\CH_Y)$ is a fortiori quasi-controlled.
This last observation is important because, as we shall see, the weaker notion of quasi-control is already enough to construct the required coarse equivalences---and using it actually makes the construction clearer. 
It is this point of view that enables us---and \v{S}pakula--Willett before us---to provide a unified proof for \cstar{}rigidity that applies simultaneously to all Roe-like algebras. Specifically, this is the reason why in \cref{thm:intro: rigidity} we are able to compare Roe-like \cstar{}algebras of different kinds.

\begin{rmk}
   Throughout this memoir we will define several ``approximate'' and ``quasi'' notions. Just as indicated by the containment $\cpcstar\variable\subseteq\qlcstar\variable$, the approximate notions always imply their quasi counterparts.
\end{rmk}

The main points in our proof of \cref{thm:intro: uniformization} rely on some Baire-type ``compactness'' arguments that are clearly inspired by the insights of \cites{braga_farah_rig_2021,braga_gelfand_duality_2022}.

\subsection{Construction of approximations}\label{ssubsec:intro: approximations}
The next ingredient that is needed to prove \cstar{}rigidity is a procedure to pass from operators among modules $T\colon\CH_X\to\CH_Y$ to mappings among $X$ and $Y$. One classical approach is as follows.
Assume that $X$ and $Y$ are (uniformly) locally finite, $\CH_X=\ell^2(X)$ and $\CH_Y=\ell^2(Y)$. Fix some $\delta>0$ and construct a partially defined map $Y\to X$ sending a point $y\in Y$ to an arbitrarily chosen point $x$ such that $\scal{T(\chf{x})}{\chf{y}}>\delta$. One may then prove that, under appropriate conditions and small enough $\delta>0$, this gives an everywhere defined controlled mapping (see \emph{e.g.}\ \cites{braga_rigid_unif_roe_2022,spakula_rigidity_2013,rigidIHES}).

The approach we take here is related, but formally rather different. Since points and bounded sets cannot be distinguished within the coarse formalism, we see no reason to insist on constructing a \emph{function} from $X$ to $Y$. All that is needed is some subset of $Y\times X$ whose fibers have uniformly bounded diameter (this can be thought of as a ``coarsely well-defined'' function).
Taking this point of view lets us construct ``coarse functions'' approximating $T\colon \CH_X\to\CH_Y$ in a more natural way. Namely, we consider the \emph{relations} defined by
\[
  \appmap[T]{\delta}{R}{r} \coloneqq 
    \bigcup\left\{B\times A \; \mid \;
    \text{diam}\left(B\right) \leq R, \;\text{diam}\left(A\right) \leq r \;\; \text{and} \;\;
    \norm{\chf{B}T\chf{A}} > \delta \right\}. 
\]
On a first sight, this definition is more complicated than the previous, point-based, approach. However, it is much more flexible and easy to use. For one, it is a more canonical object that works equally well regardless of the spaces and modules under consideration,\footnote{\, Notice that before we had to specify that the spaces are discrete and the Hilbert spaces are $\ell^2$. Such steps are unnecessary within our formalism.} and whose properties are more streamlined to prove. 
Another major advantage of using relations instead of functions is that it gives a way to confound between mappings and supports of operators among modules (this will become clearer after a few definitions, see \cref{subsec: supports}).

A simple but crucial observation is that if $\Ad(T)$ is a quasi-controlled mapping as introduced in \cref{ssubsec:intro: uniformization}, then its approximations $\appmap[T]{\delta}{R}{r}$ are always \emph{controlled relations}.
Intuitively, this means that they give rise to controlled functions $D\to Y$ defined on the domain $D\coloneqq\pi_X(\appmap[T]{\delta}{R}{r}) \subseteq X$, and that these functions are uniquely defined up to closeness. In other words, controlled relations give rise to (partially defined) coarse maps from $X$ to $Y$. This is the tool needed to pass from isomorphisms to coarse equivalences. There are, however, two difficulties to overcome. The first one is that the approximation procedure is, alas, \emph{not} functorial (see \cref{rem:app-maps not functorial}): this is not a great issue, but it does mean that extra care has to be paid to verify coarse invertibility.
On the contrary, the second difficulty is a major one. The issue here is that, with no further assumptions, $\appmap[T]{\delta}{R}{r}$ may well be ``too small'', or even just the empty relation. That is, we need to show that the parameters can be chosen so that domain and image of $\appmap[T]{\delta}{R}{r}$ are large---coarsely dense, in fact.

Guaranteeing that for a quasi-controlled unitary $U\colon \CH_X\to\CH_Y$ there is an approximation $\appmap[U]{\delta}{R}{r}$ which is coarsely surjective---a necessary intermediate step in the proof that $\appmap[U]{\delta}{R}{r}$ is a coarse equivalence---is hard. 
In its essence, this issue is one of the core points that required hard work in each successive improvement on the problem of \cstar{}rigidity.
In the earliest results \cites{spakula_rigidity_2013,braga_gelfand_duality_2022} \emph{property A} (or weakenings thereof) was assumed precisely to overcome this problem.
The main contribution of \cite{braga_rigid_unif_roe_2022} was to give a general solution to the surjectivity problem in the setting of uniformly locally finite metric spaces equipped with their $\ell^2$ space (see \cite{braga_rigid_unif_roe_2022}*{Lemma 3.2}), while \cite{rigidIHES} provided an unconditional solution to the problem of coarse surjectivity proving a \emph{Concentration Inequality} (see \cite{rigidIHES}*{Proposition 3.2}).

The latter is the approach we follow in this memoir, as the Concentration Inequality always implies that $\appmap[U]{\delta}Rr$ can be made ``surjective enough'' (cf.\ \cref{prop: unconditional image estimate}).
With this at hand, the following statement is then easy to prove.

\begin{alphthm}[cf.\ \cref{cor: rigidity controlled unitaries}]\label{thm-intro: rigidity controlled unitaries}
  Let $U\colon\CH_X\to \CH_Y$ be a unitary (between faithful discrete modules) such that both $\Ad(U)$ and $\Ad(U^*)$ are quasi-controlled. Then, for every $0<\delta<1$ there are $r, R \geq 0$ large enough such that $\appmap[U]{\delta}{R}r\colon X\to Y$ is a coarse equivalence with coarse inverse $\appmap[U^*]{\delta}rR$.
\end{alphthm}

\subsection{Combining all the pieces}\label{ssubsec:intro: intro: combining things}
Using the ingredients we have illustrated thus far, we may already sketch a proof of the following result, which implies, among other things, the main theorem of \cite{braga_rigid_unif_roe_2022}*{Theorem 1.2}.

\begin{alphcor}\label{cor:intro: easy rigidity}
  Let $X$ and $Y$ be extended metric spaces. Suppose that there is an isomorphism 
  \(
    \roeclikeone{X} \cong \roecliketwo{Y},
  \)
  where $\roeclikeone\variable$ and $\roecliketwo\variable$ are any of $\cpcstar\variable$, $\qlcstar\variable$, $\uroecstar{\variable}$. Then $X$ and $Y$ are coarsely equivalent.
\end{alphcor}
\begin{proof}[Sketch of proof]
  As explained in \cref{ssubsec:intro: spatial implementation}, $\phi$ is induced by a unitary operator $U\colon\CH_X\to\CH_Y$. In particular, $\phi^{-1}$ is induced by the adjoint $U^*$.
  \cref{thm:intro: uniformization} implies that both $\Ad(U)$ and $\Ad(U^*)$ are quasi-controlled. Thus, by \cref{thm-intro: rigidity controlled unitaries}, $\appmap[U]{\delta}{R}{r}$ is a coarse equivalence between ${X}$ and ${Y}$.
\end{proof}

Note that \cref{cor:intro: easy rigidity} falls short of proving rigidity for $\roecstar\variable$ (whence \cite{rigidIHES}*{Theorem A}) and dealing with stable isomorphisms. The issue is that in both cases one cannot directly apply \cref{thm:intro: uniformization} to deduce that $\Ad(U)$ is quasi-controlled. This is because $\phi$ is only defined on some \cstar{}algebra which is  much smaller than $\cpcstar{X}$.
This is however not an issue for our framework. In fact, rather than directly proving \cref{thm-intro: rigidity controlled unitaries}, we prove a slightly stronger version of the Concentration Intequality (cf.\ \cref{prop: concentration-ineq}) and use it to prove a more general version of \cref{thm-intro: rigidity controlled unitaries} (cf. \cref{thm: rigidity quasi-proper operators}) which allows us to construct coarse equivalences by restricting to carefully selected submodules of $\CH_X$ and $\CH_Y$. This more general statement can be used together with \cref{thm:intro: uniformization} to directly prove \cref{thm:intro: rigidity} in its most general form (cf.\ \cref{thm: stable rigidity}).

\section{Refined rigidity and its consequences}
The main highlight of \cref{thm:intro: rigidity} is the level of generality under which it shows that coarse equivalences exist. Its downside is that it does not immediately provide much more information about the coarse equivalence and how it relates with the originating (stable) isomorphism.

At the price of renouncing to some generality, we can however prove more ``refined'' rigidity results that retain much more information on the relation between isomorphisms and coarse equivalences.
The precise statement we prove relies heavily on the notion of ``coarse support'' of operators between coarse spaces (and its ``approximate'' and ``quasified'' versions, cf.\ \cref{subsec: coarse supp and propagation,subsec: strong approx-quasi control}).
Staying true to the spirit of this introduction, we rephrase it as the following (slightly informal) statement.

\begin{alphthm}[cf.\ \cref{thm: strong rigidity,rmk: functoriality of qcsupp}]\label{thm:intro: refined rigidity}
  Let $X$ and $Y$ be proper extended metric spaces, and let $\phi\colon\roeclikeone X\to\roecliketwo Y$ be an isomorphism between ample Roe-like \cstar{}algebras. Then:
  \begin{enumerate}[label=(\roman*)]
    \item $\roeclikeone\variable=\roecliketwo\variable$ must be Roe-like \cstar{}algebras of the same kind;
    \item\label{item:thm:intro: refined rigidity : two} associating with $\phi$ its ``quasi-support'' defines a functorial mapping to the category of coarse spaces.
  \end{enumerate}
\end{alphthm}

The above applies to $\roecstar\variable$, $\cpcstar\variable$ and $\qlcstar\variable$ (the uniform Roe algebra $\uroecstar{\variable}$ is excluded by the ampleness condition).
The first part of the statement is perhaps the least interesting one: if one was not interested on quasi-locality it would be trivial to distinguish $\roecstar\variable$ from  $\cpcstar\variable$, as the latter is unital while the former is not.
On the other hand, it is an interesting piece of information that one is always able to distinguish between $\cpcstar\variable$ and $\qlcstar\variable$ (should these differ, of course).
One immediate application is the following.

\begin{alphcor}[cf.\ \cref{thm: strong rigidity}~\cref{thm: strong rigidity:one}]\label{cor:intro: cpcstar neq qlcstar}
  If $X$ is an extended metric space such that $\cpcstar X\neq\qlcstar X$, then $\cpcstar X$ is not isomorphic to $\qlcstar Y$ for any other extended metric space $Y$.
\end{alphcor}

\begin{remark}
  As already mentioned, examples of metric spaces where $\cpcstar X\subsetneq\qlcstar X$ are given in \cite{ozawa-2023}. \cref{cor:intro: cpcstar neq qlcstar} implies that there must be a purely coarse geometric and a purely \cstar{}algebraic condition characterizing whether $\qlcstar X$ is isomorphic to $\cpcstar X$. As pointed out by Ozawa in \cite{ozawa-2023}, it may well be the case that $\cpcstar X =\qlcstar X$ if and only if $X$ has property A. 
  Reading his proof, it becomes natural to ask whether $\cpcstar X\subsetneq\qlcstar X$ if and only if $\prod_{n}M_n$ embeds in $\qlcstar{X}$. If the answer to both questions was positive, this would give a new characterization of Property A (at least for uniformly locally finite metric spaces).
\end{remark}

The second part of \cref{thm:intro: refined rigidity} is very interesting.
Rephrasing it, it says that with any unitary $U\colon\CHx\to\CHy$ inducing an isomorphism $\roeclike X\to\roeclike Y$ we can canonically associate a coarse equivalence $X\to Y$ (unique up to closeness) and that this association is well-behaved under composition of unitaries $\CHx\to\CHy\to\CHz$.
This point of view resolves the lack of functoriality observed in \cref{ssubsec:intro: approximations} when discussing the construction of approximating relations $\appmap[T]\delta R r$.\footnote{\,
  Note however this only applies to \emph{unitaries} inducing \emph{isomorphisms}. Arbitrary homomorphisms need not give rise to well-defined coarse maps.
}

The mapping $X\to Y$ associated with the unitary $U$ is explicitly described as a sort of ``support'' for $U$ (see \cref{def: approximate and quasi support}).
To see why this is a reasonable thing to do, recall that with any coarse equivalence $f\colon X\to Y$ one can associate unitary operators $U\colon \CH_X\to \CH_Y$ whose support (defined as a subset of $Y\times X)$ is within finite Hausdorff distance from the graph of $f$---one says that such a $U$ \emph{covers $f$} (cf.\ \cref{def:controlled and proper operator,prop: existence of covering iso}).
\cref{thm:intro: refined rigidity}~\cref{item:thm:intro: refined rigidity : two} says that a ``quasi-fied'' version of the converse is always true: if a unitary induces an isomorphism of Roe-like algebras then it is quasi-supported on the graph of a coarse equivalence.
This extra information is very valuable and has a number of consequences.
For instance, the following corollaries are applications of a version of \cref{thm:intro: refined rigidity}~\cref{item:thm:intro: refined rigidity : two} specialized to $\roecstar\variable$ and $\cpcstar\variable$.

\begin{alphcor}[cf.\ \cref{cor:  aut is limit of controlled} and \cite{rigidIHES}*{Theorem 4.5}]
  Let $X$ and $Y$ be proper extend metric spaces. If $U\colon \CH_X\to\CH_Y$ is a unitary of ample modules inducing an isomorphism between $\roecstar\variable$ or $\cpcstar\variable$, then $U$ is the norm limit of unitaries that cover a fixed coarse equivalence $X\to Y$.
\end{alphcor}

\begin{alphcor}[cf.\ \cref{cor: cpc is multiplier of roec} and \cite{braga_gelfand_duality_2022}*{Theorem 4.1}]\label{cor: intro: multiplier}
  If $X$ is a proper extended metric space, then $\cpcstar{X}$ is the multiplier algebra of $\roecstar X$.
\end{alphcor}

\begin{alphcor}[cf.\ \cref{cor: aut of Roe}]\label{cor: intro: aut}
  If $X$ is a proper extended metric space, then $\aut(\cpcstar{X})=\aut(\roecstar{X})$.
\end{alphcor}

In \cref{cor: intro: multiplier,cor: intro: aut} it makes sense to compare those \cstar{}algebras because they are all naturally represented as subalgebras of $\CB(\CH_X)$.

\smallskip

Another interesting application of \cref{thm:intro: refined rigidity}~\cref{item:thm:intro: refined rigidity : two} is concerned with the groups of outer automorphisms of Roe-like \cstar{}algebras.
In the following, let $X$ be a fixed proper (extended) metric space, $\CH_X$ an ample module (the following discussion does not apply to $\uroecstar{X}$), and denote by $\coe X$ the group of coarse equivalences $X\to X$ considered up to closeness.

As we already mentioned, it is a classical observation that for any given coarse equivalence $f\colon X \to X$ one may choose a unitary operator $U_{f} \colon \CH_X \to \CH_X$ between ample geometric modules that covers $f$.
Conjugation by $U_f$ then defines isomorphisms of Roe-like \cstar{}algebras
\[
  \Phi_{f,\CR} \coloneqq \Ad(U_{f}) \colon \roeclike{X} \xrightarrow{\ \cong\ } \roeclike{X}.
\]

In the above construction, the choice of $U_f$ is highly non-canonical, and therefore assigning to $f$ the automorphism $\Phi_{f,\CR}\in \aut(\roeclike X)$ is not a natural operation.
However, one can show that different choices of $U_f$ are always conjugated via some unitary $u\in \U(\cpcstar X)$. This implies that the conjugacy class $[\Phi_{f,\CR}]\in \out(\roeclike X)$ \emph{is} uniquely determined.\footnote{\, 
  Here we are using that $\cpcstar X$ is always contained in the multiplier algebra of $\roeclike X$, so its unitaries are quotiented out in $\out(\roeclike X)$.
}
Moreover, close coarse equivalences give rise to the same outer automorphism. It follows that there are natural mappings
\[
  \varsigma_\CR\colon \coe X\to \out(\roeclike X), 
\]
which are also easily verified to be group homomorphisms.
It is relatively well-known that the homomorphisms $\varsigma_\CR$ are injective (see \cite{braga_gelfand_duality_2022}*{Section 2.2} or \rfEmbeddingCtrUniintoOut, and see also \cite{spakula_rigidity_2013}*{Theorem A.5} for an algebraic counterpart of this statement).

Using \cref{thm:intro: refined rigidity}, it is not hard to prove that each $\varsigma_\CR$ is also surjective (cf.\ \cref{cor: surjection to Out})---for spaces with property A this is one of the main results of \cite{braga_gelfand_duality_2022}. This proves the following.

\begin{alphcor}[cf.\ \cref{cor: all groups are iso}]\label{cor:intro: all groups are iso}
  Let $X$ be a proper extended metric space. The following groups are all canonically isomorphic.
  \begin{enumerate}[label=(\roman*)]
    \item \label{cor:intro: all groups are iso:1} The group of $\coe X$ of coarse equivalences of $X$ up to closeness.
    \item \label{cor:intro: all groups are iso:2} The group of outer automorphisms of $\roecstar{X}$.
    \item \label{cor:intro: all groups are iso:3} The group of outer automorphisms of $\cpcstar{X}$.
    \item \label{cor:intro: all groups are iso:4} The group of outer automorphisms of $\qlcstar{X}$.
  \end{enumerate}
\end{alphcor}

\begin{remark}
  Some remarks about \cref{cor:intro: all groups are iso} are in order.
  \begin{enumerate}[label=(\roman*)]
    \item On our way to prove \cref{cor:intro: all groups are iso}, we actually show that these groups are all isomorphic to the group of unitaries $U$ such that both $\Ad(U)$ and $\Ad(U^*)$ are \emph{controlled}, up to unitary equivalence in $\cpcstar{X}$. It is proved in \rfIsoCEtoCtrUni that this is isomorphic to $\coe X$.
    \item \cref{cor: intro: multiplier,cor: intro: aut} immediately imply that $\out(\roecstar{X})\cong\out(\cpcstar{X})$.
    \item If $X$ is a metric space with property A, it is proven in \cite{spakula-zhang-2020-quasi-loc-prop-a} that quasi-local operators are always approximable, \emph{i.e.}\ $\cpcstar X=\qlcstar X$ (see also \cite{ozawa-2023} for an alternative proof). In this case, the isomorphism $\out(\qlcstar{X})\cong\out(\cpcstar{X})$ is once again induced by the identification $\aut(\qlcstar{X})=\aut(\cpcstar{X})$.
    \item \cite{braga_gelfand_duality_2022}*{Theorem~B} states that if $X$ has property A then the map \(\varsigma_{\rm Roe}\) is a group isomorphism $\coe X\cong\out(\roecstar{X})$. The two points above explain why in this case the same holds for an arbitrary ample Roe-like \cstar{}algebra.
    On the other hand, without the property A assumption it is no longer true that quasi-local operators need to be approximable \cite{ozawa-2023}. When $\cpcstar X\neq \qlcstar X$, it follows from our refined rigidity theorem that there is a strict inclusion $\aut(\cpcstar X)\subsetneq\aut(\qlcstar X)$. In this case, the isomorphism $\out(\qlcstar{X})\cong\out(\cpcstar{X})$ is significantly more surprising.
  \end{enumerate}
\end{remark}

We conclude the introduction by noting that as an intermediate step in our proof of \cref{thm:intro: refined rigidity} we also prove the following result, which may be of independent interest.

\begin{alphthm}[cf.\ \cref{thm: quasi-proper}]\label{thm:intro: quasi-proper}
  Let $X$ and $Y$ be proper metric spaces equipped with (ample) modules $\CH_X$ and $\CH_Y$, and let  $T\colon\CH_X\to\CH_Y$ be an operator. Then the following are equivalent:
  \begin{itemize}
    \item $\Ad(T)$ preserves local compactness;
    \item the restriction of $\Ad(T)$ to $\roecstar{X}$ preserves local compactness;
    \item $T$ is \emph{quasi-proper}.
  \end{itemize}
\end{alphthm}

In the above, $T$ is said to be \emph{quasi-proper} if for every $\varepsilon>0$ and bounded set $B\subseteq Y$ there is some bounded $A\subseteq X$ such that $T^*\chf{B}\approx_{\varepsilon}\chf{A}T^*\chf{B}$. Without entering into details, this geometric characterization of the property of preserving local compactness is very helpful in the problem of \cstar{}rigidity. In fact, this gives us a tool to directly show that the uniformization phenomenon of \cref{ssubsec:intro: uniformization} also applies to mappings $\Ad(U)$ under the assumption that $\Ad(U)$ restricts to an isomorphism of Roe algebras (see \cref{cor: uniformization for roe algs}).

\section{Structure of the memoir}
As already mentioned, when writing this memoir we decided to part from the standard setting of (extended) metric spaces and developed a set of notations and conventions that are better suited for our coarse geometric investigations. This results in tidier and more conceptual proofs in the later chapters, but the price to pay for it is a certain overhead of definitions and elementary observations in the earlier ones. We believe this to be a worthwhile tradeoff. The memoir is structured as follows.

\cref{sec: general prelim} contains some general funtional analytic preliminaries.
In \cref{sec: coarse geometric setup} we introduce the coarse geometric language we will use throughout the memoir. This is based on the formalism of Roe \cite{roe_lectures_2003} with the notational conventions of \cites{coarse_groups,roe-algs}. We use certain non-standard definitions (especially regarding coarse maps) which are particularly well suited to highlight the interplay between coarse geometry and operator algebras.

In \cref{sec: coarse modules} we recall the formalism of coarse geometric modules. The definitions and conventions here explained are used extensively throughout. This language was developed in \cite{roe-algs}, which may be thought of as an introduciton to the present memoir. For a reader willing to accept a few black-boxes, this material can be understood without having previously read \cite{roe-algs}. We also introduce a few new definitions and results that are here needed but were out of place in \cite{roe-algs}, such as block entourages and the Baire property.

In \cref{sec: roe algebras} we explain the bridge from geometry to operator algebras by defining the coarse support of operators and the Roe-like \cstar{}algebras of coarse geometric modules. We then recall some structural properties of Roe-like \cstar{}algebras, and conclude with a discussion of submodules: this is a technical device that we employ to prove the stable rigidity theorem \cref{thm:intro: rigidity}.

\cref{sec: uniformization,sec: rigidity spatially implemented} contain the proofs of the phenomena explained in \cref{ssubsec:intro: uniformization,ssubsec:intro: approximations} respectively. These sections are completely independent from one another.
Their results are combined in \cref{sec: rigidity phenomena}, where \cref{thm:intro: rigidity} and  its immediate consequences are proved. Together, these three chapters form the core of the proof of the stable \cstar{}rigidity phenomenon.

We then start moving towards the refined \cstar{}rigidity theorem.
In \cref{sec: quasi-proper} we prove \cref{thm:intro: quasi-proper} (this proof is independent from the rest of the memoir), and we show how this result interacts with the uniformization phenomenon of \cref{sec: uniformization}.

\cref{sec: strong approx and out} starts by introducing the language needed to properly state the refined rigidity results. Namely, strong and effective notions of approximate and quasi-control for operators and their supports. It then proceeds with the proof of the main result \cref{thm:intro: refined rigidity}. Its consequences are collected in the final chapter, \cref{cor: intro: multiplier,cor:intro: all groups are iso}.

With the exception of the references to \cite{roe-algs}, this work is mostly self-contained.

\medskip\noindent\textbf{Assumptions and notation:}
by \emph{projection} we mean \emph{self-adjoint idempotent}. $\CH$ is assumed to be a complex Hilbert space which, unless otherwise specified, needs not be separable. Its inner product is $\scal{\cdot}{\cdot}$, and is assumed to be linear in the second coordinate. $\CB(\CH)$ denotes the algebra of bounded linear operators on $\CH$, and $\CK(\CH)\subseteq \CB(\CH)$ denotes the compact operators.

$\crse{X} = (X, \CE)$ and $\crse{Y} = (Y, \CF)$ denote coarse spaces, and $E \in \CE, F \in \CF$ are entourages. Bounded sets are usually denoted by $A \subseteq X$ and $B \subseteq Y$. $\CHx$ and $\CHy$ are coarse geometric modules associated to $\crse X$ and $\crse Y$ respectively.
We usually denote by $t, s, x, y, z \in \CB(\CH)$ general operators, whereas $p, q \in \CB(\CH)$ will be projections. 
For the most part, we will denote with capital letters $T,U,V,W$ operators between (different) coarse geometric modules, and with with lowercase letters $t,s,r,u,w$ operators within the same module, especially when considered as elements of Roe-like \cstar{}algebras.
Lastly, vectors are denoted by $v, w, u \in \CH$.

\medskip\noindent\textbf{Acknowledgements.}
We are grateful for several fruitful conversations about these topics with Florent Baudier, Tim de Laat, Ilijas Farah, Ralf Meyer, Alessandro Vignati and Wilhelm Winter. We are also grateful to Kostyantyn Krutoy for sharing a preliminary version of \cite{Krutoy2025categorical} with us.

\chapter{Functional analytic preliminaries} \label{sec: general prelim}
\section{Notation and elementary facts}\label{ssec:notation}
We start by briefly establishing notation and conventions that will be later used without further notice.

\sssec{Boolean Algebras.}
A \emph{unital Boolean algebra} of subsets of $X$ is a family $\fkA\subseteq \CP(X)$ closed under finite intersections, complements and containing $X$.

\sssec{Hilbert spaces.}
Given a Hilbert space $\CH$, we denote its unit \emph{ball} and unit \emph{sphere} respectively by
\begin{align*}
  \CH_{\leq 1}\coloneqq\{v\in\CH\mid \norm v\leq 1\}
  \quad\text{ and }\quad
  \CH_{1}\coloneqq\{v\in\CH\mid \norm v= 1\}.
\end{align*}
Likewise, $\CB(\CH)_{\leq 1}$ denotes the contractions in $\CB(\CH)$.
Given Hilbert spaces $(\CH_i)_{i\in I}$, we denote by $\bigoplus_{i\in I} \CH_i$ their Hilbert sum (or $\ell^2$-sum). 

\sssec{Operator Topologies.}
We denote by \emph{$\sot$} and \emph{$\wot$} the strong and weak operator topologies respectively.
In the following we will be especially concerned with the set of contractions $\CB(\CH)_{\leq 1}$. We will thus use the following notation.
\begin{notation}\label{notation:sot-open-basis}
  Given $\delta > 0$ and a finite $V \subseteq \CH_{\leq 1}$, let
  \[
    \sotnbhd{\delta}{V} \coloneqq \left\{t \in \CB(\CH)_{\leq 1} \; \middle| \; \max_{v \in V} \; \norm{tv} < \delta\right\} \subseteq \CB(\CH)_{\leq 1}.
  \]
\end{notation}

We will make use of the following characterization of \sot-neighborhoods.
\begin{lemma}\label{lemma:sot-basis-top}
  Let $\sotnbhd{\delta}{V}$ be as in \cref{notation:sot-open-basis}.
  \begin{enumerate}[label=(\roman*)]
    \item \label{lemma:sot-basis-top:ordered} If $\delta_2 \leq \delta_1$ and $V_1 \subseteq V_2$ then $\sotnbhd{\delta_2}{V_2} \subseteq \sotnbhd{\delta_1}{V_1}$.
  \end{enumerate}
  \begin{enumerate}[label=(\roman*), resume]
    \item \label{lemma:sot-basis-top:nghbd} the family $\paren{\sotnbhd{\delta}{V}}_{\delta, V}$, with $\delta>0$ and $V\subseteq \CH_{\leq 1}$ finite, forms a basis of open neighborhoods of $0$ for the $\sot$ in $\CB(\CH)_{\leq 1}$.
  \end{enumerate}
  Let also $\CV\subseteq \CH_{\leq 1}$ be a fixed Hilbert basis for $\CH$. Then
  \begin{enumerate}[label=(\roman*), resume]
    \item \label{lemma:sot-basis-top:nghbd-orth-sys} the family $\paren{\sotnbhd{\delta}{V}}_{\delta, V}$, with $\delta>0$ and $V\subseteq \CV$ finite, forms a basis of open neighborhoods of $0$ for the $\sot$ in $\CB(\CH)_{\leq 1}$.
  \end{enumerate}
\end{lemma}
\begin{proof}
  Both \cref{lemma:sot-basis-top:ordered,lemma:sot-basis-top:nghbd} are clear and follow from the definition of the $\sot$.
  Likewise, \cref{lemma:sot-basis-top:nghbd-orth-sys} follows easily from \cref{lemma:sot-basis-top:nghbd}. In fact, fix a finite $V\subseteq \CH_{\leq 1}$ and $\delta > 0$. Write each $v_k\in V$ as $\sum_{e_l\in\CV}v_{k,l}e_l$ and observe that there is a finite $V'\subseteq \CV$ such that
  \[
   \norm{\sum_{e_l\notin V'}v_{k,l}e_l}^2
   =\sum_{e_l\notin V'}\abs{v_{k,l}}^2\leq \frac{\delta^2}{4}
  \]
  for every $v_k\in V$. Write $v_k = v_k'+ v_k''$, where $v_k'$ is in the span of $V'$ and $v_k''$ in the orthogonal complement. Since we are only working with contractions, for every $t\in\CB(\CH)_{\leq 1}$ we have $\norm{tv_k}\leq \norm{tv_k'}+\norm{tv_k''}\leq \norm{tv_k'} +\delta/2$. Since $V'$ is finite, it is then clear that for some small enough $\delta'$ we obtain an inclusion $\sotnbhd{V'}{\delta'}\subseteq\sotnbhd{V}{\delta}$, as desired.
\end{proof}

\sssec{Orthogonal operators.}
Two operators $s,t\colon \CH\to\CH'$ are \emph{orthogonal} if $ts^*=s^*t=0$. If $\paren{t_i}_{i\in I}$ is a family of pairwise orthogonal operators of uniformly bounded norm, their \sot-sum $\sum_{i\in I}t_i$ always converges to some bounded operator $t$ of norm $\sup_{i\in I}\norm{t_i}$.

The space $\CP(I)$, formed by the subsets of $I$, can be topologized by identifying it with $\{0,1\}^I$ (with the product topology). Equivalently, this defines the topology of pointwise convergence.
\begin{lemma}\label{lem: sums of orthogonal is cts}
  If $\paren{t_i}_{i\in I}$ is a family of pairwise orthogonal operators of uniformly bounded norm, the mapping $\CP(I)\to\CB(\CH)$ sending $J\subseteq I$ to $t_J\coloneqq \sum_{i\in J}t_i$ is continuous with respect to the \sot.
\end{lemma}
\begin{proof}
  Rescaling if necessary, we may assume that the $t_i$ are contractions, so that the $\sotnbhd{\delta}{V}$ in \cref{notation:sot-open-basis} are a basis of open neighborhoods of 0. Fix one such neighborhood. For every $v\in V$, there exists a finite $I_v\subseteq I$ such that $\norm{t_J(v)}<\delta$ for every $J\subseteq I\smallsetminus I_v$, as otherwise it would be possible to find operators $t_J$ so that $\norm{t_J(v)}$ diverges.

  Let $\widetilde{I}$ be the finite union of the finite sets $I_v$ as $v$ ranges in $V$. If a net $(J_\lambda)_\lambda$ converges to some $J$, then $J_\lambda\cap \widetilde I = J\cap \widetilde I$ for every $\lambda$ large enough.
  It follows that $t_{J_\lambda}\in t_J +\sotnbhd{\delta}{V}$ for every $\lambda$ large enough.
\end{proof}

\sssec{Operators bounded from below.}
Giver $\eta>0$, we say that a bounded operator $T\colon\CH\to\CH'$ is \emph{$\eta$-bounded below} if $\norm{Tv}\geq\eta\norm{v}$ for every $v\in\CH$. We say that $T\colon\CH\to\CH'$ is a \emph{bi-Lipschitz isomorphism} if it is an invertible bounded operator with bounded inverse (equivalently, it is invertible and bounded below).
\begin{remark}
  If $T\colon\CH\to\CH'$ is a bi-Lipschitz isomorphism that is $\eta$-bounded from below, then so is its adjoint $T^*$. In fact, for every unit vector $w\in \CH'_1$, we may let $v \coloneqq T^{-1}(w)$. Then
  \[
    \norm{T^*(w)} = \sup_{u\in\CH_1}\scal{T^*T(v)}{u}=\scal{T(v)}{T(v)}/\norm{v}\geq \eta.
  \]
\end{remark}

\sssec{Adjoint action.}
Given a bounded $t\colon \CH\to\CH'$, we let $\Ad(t)\colon \CB(\CH)\to\CB(\CH')$ be the conjugation $s\mapsto tst^*$. Observe that $\Ad(t)$ is norm, \sot\ and \wot\ continuous. Moreover, it is a non-zero \Star{}homomorphism if and only if $t$ is an isometry (\emph{i.e.}\ $t^*t=1_\CH$), in which case it is even a \Star{}embedding. Indeed, suppose that \(tat^*tbt^* = tabt^*\) for all \(a, b \in \CB(\CH)\). Taking \(b = 1\) this implies that \(t^*t a t^*t = t^*t a (t^*t)^2\). If \(t^*t \neq 1\) then there is some \(v \in \CH\) of norm \(1\) such that \(\norm{t(v)} \neq \norm{v} = 1\). Letting \(a \coloneqq p_{t^*t(v)}\) be the projection onto $\angles{t^*t(v)}$ we get
\[
  \norm{(t^*t p_{t^*t(v)} t^*t) (v)} = \norm{(t^*t)^2 (v)} \neq \norm{(t^*t)^3 (v)} = \norm{t^*t p_{t^*t(v)} (t^*t)^2 (v)}.
\]
This yields that \(t^*t = 1\), as desired.

\sssec{Rank-one operators.}
Let $\CH$, $\CH'$ be Hilbert spaces and let $v\in \CH$, $v' \in \CH'$. We denote by $\matrixunit{v'}{v}$ the operator given by $\matrixunit{v'}{v} (h) \coloneqq \scal{v}{h} \, v'$.
Likewise, we denote by $p_v\in\CB(\CH)$ the orthogonal projection onto the span of $v$.

Observe that $\matrixunit{v'}{v}$ is the operator sending $v$ to $\norm{v}^2 \, v'$, and $v^\perp$ to $0$. Note we do \emph{not} assume that $v$ or $v'$ are unit vectors.
However, in the sequel we will only need to use $\matrixunit{v'}{v}$ with vectors $v,v'$ of norm at most one, in which case $\matrixunit{v'}{v}$ is a contraction.
When both $v$ and $v'$ are non-zero, $\munit{v'}{v}$ is a rank-one operator. Vice versa, every rank-one operator can be realized this way.

The following observations show that these operators behave like the usual elementary matrices $\{e_{ij}\}_{i,j = 1}^n \subseteq \MM_n$.
We spell out here this behavior, for it will be used extensively in the sequel.

\begin{lemma}\label{lem: matrixunit identities}
  For every $v\in\CH,\ v'\in\CH',\ v'' \in \CH''$ and projections $p\in\CB(\CH)$, $p' \in \CB(\CH')$, we have
  \begin{enumerate}[label=(\roman*)]
    \item\label{lem:e-v norm} $p_v=\munit{v}{v}/\norm{v}^2$.
    \item\label{lem:e-v adjoint} $\matrixunit{v'}{v}^* = \matrixunit{v}{v'}$.
    \item\label{lem:e-v composition} $\matrixunit{v''}{v'} \, \matrixunit{v'}{v} =\norm{v'}^2 \, \matrixunit{v''}{v}$.
    \item\label{lem:e-v norm projs} $\norm{p' \matrixunit{v'}{v} p} = \norm{p(v)} \, \norm{p'(v')}$.
  \end{enumerate}  
\end{lemma}
\begin{proof}
  The proof are rather straightforward computations. \cref{lem:e-v norm} is immediate. To prove \cref{lem:e-v adjoint}, notice that for every $w\in\CH$, $w'\in\CH'$ one has that
  \begin{align*}
    \scal{\matrixunit{v}{v'} \left(w'\right)}{w} = \scal{\scal{v'}{w'} v}{w} & = \scal{w'}{v'} \, \scal{v}{w} \\
    & = \scal{w'}{\scal{v}{w} v'} = \scal{w'}{\matrixunit{v'}{v} \left(w\right)}.
  \end{align*}
  Likewise, to show \cref{lem:e-v composition} it suffices to observe that
  \[
  \matrixunit{v''}{v'} \left(\matrixunit{v'}{v} \left(w\right)\right) = \scal{v'}{\scal{v}{w} v'} v'' = \scal{v'}{v'} \scal{v}{w} v'' = \norm{v'}^2 \, \matrixunit{v''}{v} \left(w\right),
  \]
  for every $w\in\CH$.

  It remains to verify \cref{lem:e-v norm projs}. By definition,
  \begin{align*}
    \norm{p' \matrixunit{v'}{v} p} 
    = \sup_{\substack{w \in p\left(\CH\right) \\ \norm{w} \leq 1}} \norm{p' \matrixunit{v'}{v} (w)} 
    & = \norm{p'(v')} \, \sup_{\substack{w \in p\left(\CH\right) \\ \norm{w} \leq 1}} \abs{\scal{v}{w}} \allowdisplaybreaks
  \end{align*}
  If $v\in p(\CH)^\perp$ this norm is zero and there is nothing to show. Otherwise, $p(v)/\norm{p(v)}$ realizes the supremum because $p(v)$ is, by construction, the closest vector in $p(\CH)$ to $v$. Hence 
  \[
  \norm{p' \matrixunit{v'}{v} p}
  = \norm{p'(v')} \, \scal{v}{\frac{p(v)}{\norm{p(v)}}}
  = \norm{p'(v')} \, \frac{\scal{p(v)}{p(v)}}{\norm{p(v)}} = \norm{p'(v')} \, \norm{p(v)}.\qedhere
  \]
\end{proof}

\section{Spatially implemented homomorphisms}
In this section we discuss the necessary ingredients for a \Star{}homomorphism of concretely represented \Star{}algebras to be \emph{spatially implemented}. Since the following facts are (for the most part) classical, we shall keep their proofs rather brief and only include them for the sake of completeness.

\begin{definition}
  Let $A\leq \CB(\CH)$ be a \Star{}algebra, and $\phi\colon A\to \CB(\CH')$ a \Star{}homomorphism. We say that $\phi$ is \emph{spatially implemented} if there exists an operator $W\colon \CH\to \CH'$ such that $\phi= \Ad(W)|_A$.
\end{definition}

Clearly, a necessary condition for a \Star{}homomorphism $\phi \colon A \to \CB(\CH')$ to be spatially implemented is that $\phi$ does not increase the rank of the operators, \emph{i.e.}\ $\rank(\phi(a))\leq\rank(a)$ for every $a\in A \subseteq \CB(\CH)$. In particular, $\phi$ sends operators of rank $\leq 1$ to operators of rank $\leq 1$. We call these \emph{\rlone}.
\begin{remark}
  If $A \leq \CB(\CH)$ contains no non-trivial operator of rank one then any $\phi \colon A \to \CB(\CH')$ is \rlone.
  However, the algebras we shall consider will generally contain a large amount of such operators (though not all, see \cref{prop: non-degenerate hom is spatially implemented}), so the \rlone condition will be far from void.
\end{remark}

Let $\CF(\CH) \subseteq \CB(\CH)$ be the \Star{}algebra of finite-rank operators. The following criterion can be useful to show that an operator is \rlone.
\begin{lemma}\label{lem: rank-1 to rank-1 if hereditary}
  Let $\phi \colon A \to \CB(\CH')$ be a \Star{}homomorphism, where $A\leq \CB(\CH)$ is a \Star{}algebra. Moreover, let $B\leq\CB(\CH')$ be a \cstar{}algebra containing $\CF(\CH')$ and $\phi(A)$.
  If $\phi(A)$ is a hereditary subalgebra of $B$,\footnote{\, Recall that a subalgebra \(A \subseteq B\) is \emph{hereditary} if \(b \in A\) for all non-negative \(b \in B\) such that \(b \leq a\) for some non-negative \(a \in A\).} then $\phi$ is \rlone.
\end{lemma}
\begin{proof}
  It is enough to show that $\phi$ sends rank-1 projections to rank-$\leq$1 projections, because $a^*a$ has rank-$\leq$1 for every rank-$1$ operator $a\in A$. Thus, $a = \lambda a p$, for some $\lambda \in \CCC$ and rank-$\leq$1 projection $p \in A$.

  Let $p\in A$ be a rank-1 projection and set $q\coloneqq\phi(p)$. Since $\phi$ is a homomorphism, $q$ is itself a projection.
  In particular, either $q=0$ or there is some $w\in \CH'$ such that $0 < q_w\leq q$, where $q_w\in\CF(\CH)\leq B$ denotes the orthogonal projection onto $\CCC \cdot w \subseteq \CH'$. Since $\phi(A)$ is hereditary, $q_w = \phi(a)$ for some $a\in A$.
  However, since $p$ is a rank-1 projection, $pap =\lambda p$ for some $\lambda \in \CCC$.
  On the other hand, $q_w = qq_wq = \phi(pap)=\lambda\phi(p) = \lambda q$ shows that $\lambda=1$ and $q=q_w$.
\end{proof}

The following shows that if $A \subseteq \CB(\CH)$ is large enough and $\phi$ is continuous enough, then the \rlone condition already implies spatial implementation. This is essentially a re-elaboration of the well-known fact that isomorphisms of algebras of compact operators are implemented by unitaries.
\begin{lemma}
  \label{lem: SOT-cts hom is spatially implemented}
  Any non-zero \rlone \Star{}homomorphism $\phi \colon \CF(\CH) \to \CB(\CH')$ is spatially implemented by an isometry.
\end{lemma}
\begin{proof}[Sketch of the proof]
  Since $\CF(\CH)$ is algebraically simple, $\phi(p_v)$ must be a rank-$1$ projection for every $v\in \CH_1$. In fact, if $\phi(p_v) = 0$ for some $v \in \CH_1$, any other rank-$1$ operator $t \in \CF(\CH)$ can be written as $t = xp_vy$ for some $x, y \in \CF(\CH)$ and therefore $\phi(t) = \phi(x) \phi(p_v) \phi(y) = 0$. Since rank-$1$ operators algebraically generate $\CF(\CH)$ it would follow that $\phi = 0$ on $\CF(\CH)$.

  Fix an orthonormal basis  $\paren{v_i}_{i\in I}$ for $\CH$. Then $q_i \coloneqq \phi(p_{v_i})$ defines a family of orthogonal rank-$1$ projections.
  We may then now argue as in \cite{murphy2014c}*{Theorem 2.4.8--paragraph 2 onwards} to find unit vectors $(w_i)_{i \in I} \subseteq \CH'$, with $w_i$ spanning the image of $q_i$, such that the assignment $v_i \mapsto w_i$ defines an isometry $W \colon \CH \to \CH'$ such that $\phi$ coincides with $\Ad(W)$ on $\CF(\CH)$.
\end{proof}

We will later need to deal with orthogonal sums of Hilbert spaces (this is necessary when working with \emph{disconnected} coarse spaces, see \cref{subsec: coarse spaces}).
Let $\CH=\bigoplus_{i\in I} \CH_i$ and let $p_i\in \CB(\CH)$ be the orthogonal projection onto $\CH_i \leq \CH$.
Observe that there is a natural containment $\prod_{i\in I}\CB(\CH_i) \subseteq \CB(\CH)$, seeing the former as block-diagonal matrices in the latter.
In the following, $\bigoplus_{i\in I}^{\rm alg}\CF(\CH_i)$ denotes the {\emph{algebraic} direct sum (as opposed to the $c_0$-sum)} of the \Star{}algebras of operators of finite rank. In other words, $\bigoplus_{i\in I}^{\rm alg}\CF(\CH_i) =\paren{\prod_{i\in I}\CB(\CH_i)}\cap \CF(\CH)$.

\begin{proposition}[cf.\ \cite{spakula_rigidity_2013}*{Lemma 3.1} and \cite{braga2020embeddings}*{Lemma 6.1}]\label{prop: non-degenerate hom is spatially implemented}
  Let $\CH=\bigoplus_{i\in I} \CH_i$ and $A\subseteq \prod_{i\in I}\CB(\CH_i)$ a \cstar{}algebra such that $\CF(\CH_i) \subseteq A$ for every $i\in I$.
  For a $\phi\colon A\to \CB(\CH')$ \rlone \Star{}homomorphism, the following conditions are equivalent:
  \begin{enumerate}[label=(\roman*)]
    \item \label{prop: non-degenerate hom is spatially implemented:sot-cont} $\phi$ is strongly continuous.
    \item \label{prop: non-degenerate hom is spatially implemented:spa-impl} $\phi$ is spatially implemented.
  \end{enumerate}
  Moreover, \cref{prop: non-degenerate hom is spatially implemented:sot-cont,prop: non-degenerate hom is spatially implemented:spa-impl} always hold if the restriction $\phi\colon \bigoplus_{i\in I}^{\rm alg}\CF(\CH_i)\to \CB(\CH')$ is a non-degenerate \Star{}representation.
\end{proposition}
\begin{proof}
  We already noted that \cref{prop: non-degenerate hom is spatially implemented:spa-impl} always implies \cref{prop: non-degenerate hom is spatially implemented:sot-cont},
  so we only need to prove the converse implication. For each $i\in I$, we define a partial isometry $W_i\colon \CH_i\to\CH'$ such that $\phi=\Ad(W_i)$ on $\CF(\CH_i)$. Namely, if $\phi$ is $0$ on $\CF(\CH_i)$ we put $W_i \coloneqq 0$. Otherwise, since $\phi$ is \rlone, we may define $W_i$ using \cref{lem: SOT-cts hom is spatially implemented}.

  Observe that for every $i\neq j\in I$ and $v_i\in \CH_i$, $v_j\in\CH_j$ the rank-1 projections $p_{v_i}$ and $p_{v_j}$ are orthogonal, and are thus sent to orthogonal projections via $\phi$. This implies that $W_i(\CH_i)$ and $W_j(\CH_j)$ are orthogonal subspaces of $\CH'$. Since the (partial) isometries $W_i$ are orthogonal, we deduce that the (strongly convergent) sum $W\coloneqq \sum_{i\in I}W_i$ is a well-defined partial isometry from $\CH$ into $\CH'$.
  By construction, $\phi$ and $\Ad(W)$ coincide on $\CF(\CH_i)$ for every $i\in I$, and therefore also on finite sums thereof. Namely, they coincide on $\bigoplus_{i\in I}^{\rm alg}\CF(\CH_i)$.
  Observe that the latter is strongly dense in $\prod_{i\in I}\CB(\CH_i)$, and hence the claim follows since $A\subseteq \prod_{i\in I}\CB(\CH_i)$ and both $\phi$ and $\Ad(W)$ are strongly continuous.

  \smallskip

  It remains to prove the `moreover' statement. We will show directly that the non-degeneracy assumption implies that $\phi$ coincides with $\Ad(W)$, where $W$ is constructed as above. 
  Fix $a\in A$ and $w\in\CH'$.
    By non-degeneracy of $\phi\colon\bigoplus_{i\in I}^{\rm alg}\CF(\CH_i)\to \CH'$, we may find for every $\varepsilon>0$ finitely many $t_k\in\bigoplus_{i\in I}^{\rm alg}\CF(\CH_i)$ and $w_k\in\CH'$ such that $\norm{w- \sum_{k}\phi(t_k)(w_k)}\leq \varepsilon$. Observe that for each such $t_k$ the product $at_k$ still belongs to $\bigoplus_{i\in I}^{\rm alg}\CF(\CH_i)$ (this uses that $A$ commutes with $1_i \in \CB(\CH_i) \subseteq \CB(\CH)$ for all $i \in I$).
  It follows that
  \begin{align*}
    \phi(a)(w) &\approx \phi(a)\bigparen{\sum_{k}\phi(t_k)(w_k)} \\
      &= \sum_k \phi(at_k)(w_k)\\
      &= \sum_k \paren{\Ad(W)(at_k)}(w_k)
      = \paren{\Ad(W)(a)}\bigparen{\sum_{k}(\Ad(W)(t_k))(w_k)},
  \end{align*}
  where $\approx$ means ``up to $\varepsilon$ multiplied by the norms of the relevant operators''.
  Since we already know that $\phi$ and $\Ad(W)$ coincide on $\bigoplus_{i\in I}^{\rm alg}\CF(\CH_i)$, the last expression is equal to 
  \[
    \paren{\Ad(W)(a)}\bigparen{\sum_{k}\phi(t_k)(w_k)} \approx \paren{\Ad(W)(a)}(w).
  \]
  The proof is completed letting $\varepsilon$ tend to zero.
\end{proof}

\begin{remark}\label{rmk: about spatial implementation}
  We end the section with a few remarks.
  \begin{enumerate}[label=(\roman*)]
    \item The `moreover' statement of \cref{prop: non-degenerate hom is spatially implemented} is a manifestation of the interplay between strict-continuity, unique extension property, and non-degeneracy (cf.\ \cite{lance_hilbert_modules_1997}*{Proposition 2.5} and \cite{blackadar2006operator}*{II.7.3.9}).
    \item In certain circumstances one can use the arguments of \cite{braga2020embeddings}*{Lemma 6.1} to slightly weaken the non-degeneracy assumption. For instance, if $A$ is unital and $\phi$ is a homomorphism such that $\phi(A)$ is a hereditary subalgebra of a large enough \cstar algebra of $\CB(\CH')$.
  \end{enumerate}
\end{remark}


\chapter{Coarse geometric setup}\label{sec: coarse geometric setup}
In this chapter we give a quick recap of the notions and conventions used in \cite{roe-algs}, which are in turn based on those of \cite{coarse_groups}. The proofs that we do not include here are all rather elementary (although they may require some care), and the reader may take them as exercises: proving them should illuminate the reasoning underlying the various definitions. For more details and comprehensive explanations we refer to the sources mentioned above.

\section{Coarse spaces}\label{subsec: coarse spaces}
A \emph{relation from $X$ to $Y$} is any subset $R\subseteq Y\times X$. For every $A\subseteq X$ or $\bar x\in X$ we let $R(A)\coloneqq \pi_Y(R\cap \pi_X^{-1}(A))=\{y\in Y\mid \exists x\in X,\ (y,x)\in R\}$ and $R(\bar x)\coloneqq R(\{\bar x\})$.
A relation \emph{on} $X$ is a relation from $X$ to $X$.
For every $A\subseteq X$, we let $\Delta_A\coloneqq \{(a,a)\mid a\in A\}\subseteq X\times X$ denote the \emph{diagonal over $A$}.

We denote by $\op{R}\coloneqq\braces{(x,y)\mid (y,x)\in R}\subseteq X\times Y$ the \emph{transposition} of $R$, and we say that $R$ is \emph{symmetric} if $R=\op R$.
The \emph{composition} of two relations $R_1\subseteq Z\times Y$ and $R_2\subseteq Y\times X$ is the relation
\[
  R_1\circ R_2\coloneqq\braces{(z,x)\mid\exists y\in Y \;\; \text{such that} \;\; (z,y)\in R_1 \; \text{and} \; (y,x)\in R_2}.
\]
Observe that $(R_1\circ R_2)(A)=R_1(R_2(A))$ for every $A\subseteq X$.

\begin{definition}
 A \emph{coarse structure on a set $X$} is a family $\CE$ of relations on $X$ closed under taking subsets and finite unions, such that
 \begin{itemize}
   \item $\CE$ contains the diagonal $\Delta_X$;
   \item $E\in\CE \Rightarrow \op{E}\in \CE$;
   \item $E,F\in\CE \Rightarrow E\circ F\in \CE$.
 \end{itemize}
 A \emph{coarse space} $\crse X = (X,\CE)$ is a set $X$ equipped with a coarse structure $\CE$.
\end{definition}

We call the elements $E \in \CE$ \emph{(controlled) entourages}.
The prototypical example of coarse space is constructed from an extended pseudo-metric space $(X,d)$: the coarse structure induced by the (extended) metric is defined as
\[
  \CE_d\coloneqq\braces{E\subseteq X \times X \, \mid \, \exists \, r<\infty \;\; \text{such that} \;\; d(x,y)<r\ \forall (x,y)\in E}.
\]

We collect below some pieces of nomenclature.
\begin{notation} \label{notation:bounded and controlled}
  Let $\crse X= (X,\CE)$ be a coarse space, $E\in \CE$, and $A,B \subseteq X$.
  \begin{enumerate}[label=(\roman*)]
    \item $A$ is \emph{$E$-bounded} if $A\times A\subseteq E$.
    \item a family $(A_i)_{i\in I}$ of subsets of $X$ is \emph{$E$\=/controlled} if $A_i$ is $E$-bounded for all $i \in I$.
    \item $B$ is an \emph{$E$-controlled thickening of $A$} if $A\subseteq B \subseteq E(A)$.
  \end{enumerate}
  Likewise, we say \emph{bounded} (resp.\ \emph{controlled}) to mean $E$-bounded (resp.\ $E$-controlled) for some $E\in \CE$.
\end{notation}

A coarse space $\crse X$ is \emph{(coarsely) connected} if every finite subset of $X$ is bounded. Given a coarse space $\crse X$, the base set $X$ can be uniquely decomposed as a disjoint union $X=\bigsqcup_{i\in I} X_i$ so that for every bounded $A\subseteq X$ there is a unique $i\in I$ so that $A\subseteq X_i$, and the restriction of $\CE$ to each $X_i$ makes it into a connected coarse space $\crse X_i$. We say that this is the \emph{decomposition of $\crse X$ into its coarsely connected components}, denoted $\crse X =\bigsqcup_{i\in I}\crse X_i$.
For instance, if $\crse X=(X,\CE_d)$ is the coarse space defined by an extended metric $d$, then $X=\bigsqcup_{i\in I} X_i$ is the partition in components consisting of points that are at finite distance from one another. In particular, metric spaces are always coarsely connected.

Observe that a coarse structure $\CE$ is a directed set with respect to inclusion. In particular, it makes sense to talk about \emph{cofinal} families, that is, $(E_i)_{i \in I} \subseteq \CE$ such that for all $E \in \CE$ there is some $i \in I$ such that $E \subseteq E_i$. 
The coarse structure \emph{generated} by a family of relations on $X$ is the smallest coarse structure containing them all.
The following is a classical fact, and not at all hard to prove (see, \emph{e.g.}\ \cite{roe_lectures_2003}*{Theorem 2.55}, \cite{protasov2003ball}*{Theorem 9.1}, or \cite{coarse_groups}*{Lemma 8.2.1}).

\begin{proposition}\label{prop:metrizable coarse structure iff ctably gen}
  Let $\CE$ be a coarse structure. The following are equivalent:
  \begin{enumerate} [label=(\roman*)]
   \item $\CE=\CE_d$ for some extended metric $d$;
   \item $\CE$ is countably generated;
   \item $\CE$ contains a countable cofinal sequence of entourages.
  \end{enumerate}
\end{proposition}

A family $(A_i)_{i\in I}$ of subsets of $\crse X$ is \emph{locally finite} if 
for all bounded $A \subseteq X$
\[
\#\braces{i\in I\mid A_i \cap A \neq \emptyset} < \infty.
\]
It is \emph{uniformly locally finite} if we also have 
\[
  \sup\braces{\#\braces{i\in I\mid A_i \cap A \neq \emptyset} \mid A \subseteq X \text{ is $E$-controlled}} < \infty
\]
for all $E \in \CE$.
A \emph{controlled partition} for a coarse space $\crse X$ is a partition  $X=\bigsqcup_{i\in I} A_i$ that is controlled in the sense of \cref{notation:bounded and controlled}.
\begin{definition} \label{def:loc-fin}
  A coarse space $\crse X$ is \emph{coarsely locally finite} if it admits a locally finite controlled partition.
  Likewise, $\crse X$ has \emph{bounded geometry} if there is a uniformly locally finite controlled partition.
\end{definition}

\begin{example}
  Every metric space of bounded diameter (trivially) has bounded geometry. More generally this is true for any coarse space $\crse X = (X,\CE)$ where the whole space $X$ is a bounded set. We call such a coarse space \emph{bounded coarse space}, and note that these are eminently uninteresting examples of coarse spaces, as they are ``coarsely trivial''.
\end{example}

\begin{example}
  Let $\CG$ be a connected graph. We can see it as a discrete metric space by equipping its vertex set with the path-metric, which assigns to any pair of points the length of the shortest path connecting them. Considering the family of closed balls of radius one shows that if every vertex in $\CG$ has finite degree then it is coarsely locally finite. The converse is not true without other assumptions, because $\CG$ may very well have many edges that do not contribute to its ``large scale geometry'', or even be a bounded coarse space!
  Similarly, if $\CG$ has degree uniformly bounded from above then it has bounded geometry.

  These examples are actually rather generic, as it can be shown that every ``coarsely geodesic'' coarse space is ``coarsely equivalent'' to a graph (see \emph{e.g.}\ \cite{coarse_groups}*{Proposition B.1.6.}).
  One class of graphs of special interest are \emph{Cayley graphs}. Namely, if $\Gamma$ is some group and $S\subseteq \Gamma$ is a generating set, the associated Cayley graph is the graph having one vertex for each $\gamma\in \Gamma$ and an edge between $\gamma$ and $\gamma s\in \Gamma$ for every $s\in S$. Of course, if $S$ is finite the Cayley graph will have bounded geometry. This construction can be extended to certain classes of semigroups as well, see \cite{chyuan_chung_inv_sem_2022} or \cref{subsec:inv-sem} below.

  If on the other hang the graph $\CG$ is disconnected, it is then natural to equip it with an extended metric by declaring that the distance between two vertices that are not joined by any path to be infinite. The same considerations above hold. 
\end{example}

\begin{example}
  It is simple to verify that every proper (extended) metric space is coarsely locally finite (recall that a metric space is \emph{proper} if all closed balls are compact). One important source of bounded geometry metric spaces is obtained by taking covers of manifolds: let $M$ be a compact Riemannian manifold. Then the Riemannian metric can be lifted to the universal cover $\widetilde M$, which makes it into a locally compact metric space. Exploiting the compactness of $M$, one can also show that $\widetilde M$ is a bounded geometry metric space.
  One very nice proof of this fact is by realising  that $\Gamma\coloneqq\pi_1(M)$ is a finitely generated group and applying the Milnor--Schwarz Lemma to deduce that its Cayley graph is coarsely equivalent to $\widetilde M$. 
\end{example}

\begin{remark}
  The above are examples of \emph{coarse geometric properties}, namely properties that are preserved under \emph{coarse equivalence} (this will be defined momentarily, after a few more pieces of notation).
\end{remark}

A \emph{gauge} for $\crse X$ is a symmetric controlled entourage containing the diagonal. We usually denote gauges by $\gauge$. 
These are useful to discuss coarse geometric properties, as such properties have often a definition of the form ``there is a gauge large enough such that\ldots''.
For instance, we say that $\gauge$ is a gauge witnessing coarse local finiteness of $\crse X$ if the latter admits a locally finite \gauge-controlled partition. Since gauges contain the diagonal, $A\subseteq \gauge(A)$ for every $A\subseteq X$.

\begin{remark}\label{rmk: uniform local finiteness}
  In the literature around (uniform) Roe algebras and, in particular, around rigidity questions (\emph{e.g.},~\cites{braga_gelfand_duality_2022,braga_rigid_unif_roe_2022,white_cartan_2018}), coarse spaces are often assumed to be \emph{uniformly locally finite} or, possibly, to have \emph{bounded geometry}.
  These notions are much stronger than coarse local finiteness and are rather convenient assumptions in various practical purposes (see~\cites{braga_rigid_unif_roe_2022,white_cartan_2018}).
  The techniques we develop in this memoir do not make use of them: the uniform local finiteness condition only appears in \cref{cor:roe-rigidity} in one of the implications that does not directly follow from our work.
\end{remark}

\section{Coarse subspaces, maps and equivalences} \label{subsec:coarse-maps}
We write $A\csub B$ if $A$ is contained in a controlled thickening of $B$. In such case we say $A$ is \emph{coarsely contained} in $B$. If $A\csub B$ and $B\csub A$ then $A$ and $B$ are \emph{asymptotic}, denoted $A \asymp B$ (this is a coarse geometric analog of ``being at finite Hausdorff distance''). Observe that $\asymp$ is an equivalence relation, which we use in the following.

\begin{definition} \label{def:coarse-subspace}
  A \emph{coarse subspace $\crse Y$ of $\crse X$} (denoted $\crse{Y\subseteq X}$) is the $\asymp$\=/equivalence class $[Y]$ of a subset $Y \subseteq X$. If $\crse Y=[Y]$ and $\crse Z=[Z]$ are coarse subspaces of $\crse X$, we say that $\crse Y$ is \emph{coarsely contained} in $\crse Z$ (denoted $\crse{Y\subseteq Z}$) if $Y\csub Z$.
  A subset $Y\subseteq X$ is \emph{coarsely dense} if $X\csub Y$ (or, equivalently, $\crse Y=\crse X$ as coarse subspaces of $\crse X$).
\end{definition}

\begin{remark}
  Given $\crse{Y\subseteq X}$, the coarse space $\crse Y=(Y,\CE|_Y)$ is uniquely defined (\emph{i.e.}\ does not depend on the choice of representative) up to canonical coarse equivalence (see \cref{def: coarse equivalence}). In categorical terms, coarse subspaces are subobjects in the coarse category \Cat{Coarse} \cite{coarse_groups}*{Appendix A.2}.
\end{remark}

Given $E\subseteq Y\times X$ and $E'\subseteq Y'\times X'$, we denote by
\[
 E\otimes E'\coloneqq \braces{\paren{(y,y'),(x,x')}\mid (y,x)\in E \; \text{and} \; (y',x')\in E'}
\]
the \emph{product relation from $X\times X'$ to $Y\times Y'$}. These are used below to define a coarse structure on the Cartesian product. 
Observe that if $D\subseteq X\times X'$ is a relation from $X'$ to $X$ then
\begin{equation}\label{eq:product relation image}
 (E\otimes E')(D)=E\circ D \circ \op{(E')}.
\end{equation}

Before continuing, this is a good point to start using the following.
\begin{convention}
  In the sequel, $\crse X$ (resp.\ $\crse Y$) will always denote a coarse space over the set $X$ (resp.\ $Y$) with coarse structure $\CE$ (resp.\ $\CF$).
\end{convention}

We may then return to products of coarse spaces. We use the following.
\begin{definition} \label{def:coarse-prod}
  \emph{$\crse{Y\times X}$} is the coarse space $(Y\times X, \CF\otimes\CE)$, where
  \[
   \CF\otimes\CE\coloneqq \braces{D \, \mid \, \exists \, F\in \CF \; \text{and} \; E\in \CE \;\; \text{such that} \; D\subseteq F\otimes E}.
  \]
\end{definition}

We need the above notion to define (partial) coarse maps in terms of relations instead of functions. This approach requires some additional formalism, but this effort is worth it, as the end results become much cleaner to state and prove. At various places, in fact, it will be much more natural to use relations than functions, see \cref{def:approximating crse map}.

We denote by $\pi_X \coloneqq Y \times X \to X$ and $\pi_Y \coloneqq Y \times X \to Y$ the usual projections onto the $X$ and $Y$-coordinates respectively.
\begin{definition}\label{def:controlled_function}
  A relation $R$ from $X$ to $Y$ is \emph{controlled} if
  \[
    \left(R\otimes R\right)\left(E\right) \in \CF
  \]
  for every $E\in \CE$. Moreover, $R$ is \emph{coarsely everywhere defined} if $\pi_X(R)$ is coarsely dense in $X$, and it is \emph{coarsely surjective} if $\pi_Y(R)=R(X)$ is coarsely dense in $Y$.
\end{definition}

Observe that \cref{eq:product relation image} implies that $R$ is controlled if and only if  $R\circ E\circ \op{R}\in \CF$ for all $E \in \CE$.

\begin{definition}\label{def:close_relations}
  Two controlled relations $R,R'$ from $\crse X$ to $\crse Y$ are \emph{close} if $R\asymp R'$ in $\crse{Y\times X}$, that is, if they define the same coarse subspace.
\end{definition}

It is not hard to show (see \rfSsecCrsMap) that if $R$ is a controlled relation and $R'\asymp R$, then $R'$ is also controlled. The property of being coarsely everywhere defined is preserved under closeness as well.
The following is therefore well posed.
\begin{definition}[cf.\ \rfDefCrsMap]\label{def:coarse-map}
  A \emph{partial coarse map $\crse R$ from $\crse X$ to $\crse Y$} is the $\asymp$-equivalence class $[R]$ of a controlled relation $R \subseteq Y \times X$.
  Likewise, a \emph{coarse map $\crse R$ from $\crse X$ to $\crse Y$} is a partial coarse map that is coarsely everywhere defined.
\end{definition}

One should think of controlled relations as controlled partial functions that are only coarsely well-defined.
The following lemma shows that \cref{def:coarse-map} is compatible with the usual notion of coarse map, as used \emph{e.g.}\ in \cite{coarse_groups}. The proof requires unraveling the definitions, but it is fairly straightforward. We refer to \cite{roe-algs} for details.
\begin{lemma}[cf.\ \rfLemCloseLikeClassic] \label{fact:relations vs functions}
  Let $f,f'\colon X\to Y$ be functions.
  \begin{enumerate}[label=(\roman*)]
    \item \label{fact:relations vs functions:controlled} $f$ is controlled in the sense of \cite{coarse_groups} (or \emph{bornologous} in~\cite{roe_lectures_2003}) if and only if its graph is a controlled relation from $X$ to $Y$.
    \item \label{fact:relations vs functions:close} $f$ and $f'$ are close in the sense of \cite{coarse_groups} if and only if their graphs are close.
    \item \label{fact:relations vs functions:graph} If $\crse R$ is a coarse map from $\crse X$ to $\crse Y$, there exists a function $f\colon X\to Y$ whose graph is close to $R$.
  \end{enumerate}
\end{lemma}

Note that \cref{fact:relations vs functions} is about \emph{functions}, \emph{i.e.}\ it only deals with coarsely everywhere defined relations. Dealing with partially defined coarse maps would require some extra finesse, as the domain of definition needs to be handled with care. One convenient feature of the approach via controlled relations is that these issues take care of themselves automatically, so one does not need to worry about them. More precisely, it follows from the definition of closeness in the product coarse structure that the following is well-posed.
\begin{definition}\label{def: crse domain and image}
  Let $\crse R$ be a partial coarse map from $\crse X$ to $\crse Y$. Then:
  \begin{enumerate}[label=(\roman*)]
    \item $[\pi_X(R)]\crse{\subseteq X}$ is the \emph{coarse domain of $\crse R$}, denoted $\cdom(\crse R)$;
    \item $[\pi_Y(R)]\crse{\subseteq Y}$ is the \emph{coarse image of $\crse R$}, denoted $\cim(\crse R) =\crse {R(X)}$.
  \end{enumerate}
\end{definition}

Extending \cref{fact:relations vs functions}, for every coarse partial map $\crse R$ from $\crse X$ to $\crse Y$, one may define a function $f\colon\pi_X(R)\to Y$ whose graph is close to $\crse R$ (see \rfLemRelCloseToFunciton). In view of this me may (and will) denote partial coarse maps from $\crse X$ to $\crse Y$ as functions $\crse{f\colon X\to Y}$ (or $\crse f\colon \cdom(\crse f)\to\crse Y$ if we want to make the coarse domain explicit).
Observe that $\crse f$ is coarsely everywhere defined (resp.\ coarsely surjective) precisely when $\cdom(\crse f)=\crse X$ (resp.\ $\cim(\crse f)=\crse Y$).
The following observation is sometimes convenient to prove that certain partial coarse maps coincide.

\begin{lemma}[cf.\ \rfLemSameSupportSameFunction]\label{lem: functions with same domain coincide} 
  If $\crse{f,g\colon X\to Y}$ are partial coarse maps with $\crse{f\subseteq g}$ and $\cdom(\crse g)\crse\subseteq \cdom(\crse f)$, then $\crse{f = g}$.
\end{lemma}

Composing partially defined coarse maps is a delicate matter. This is essentially because the coarse intersection of coarse subspaces is not always well defined. For instance, if we are given two lines $\ell_1$, $\ell_2$ in $\RR^2$ that ``oscillate'', getting sometimes close and sometimes far from one another in an irregular fashion, it is then not possible to find a sensible notion of coarse intersection for the coarse subspaces  $\crse \ell_1$, $\crse \ell_2$ that they generate.\footnote{\,
This is a known phenomenon, appearing for instance when trying to define coarse Mayer--Vietoris sequences \cite{higson1993coarse}. We refer to \cite{coarse_groups}*{Section 3.4} for a discussion of coarse intersections and related matters.
}
In turn, if $\ell_1$ is the image of some function $g$ and $\ell_2$ is the domain of a partially defined function $f$, there is no good way to define what $\crse{f\circ g}$ is. We refer the reader to \cite{coarse_groups}*{Example 3.4.7} for a more explicit example.

As it turns out, the definition we need is the following.
\begin{definition}[cf. \rfDefCrsComp]\label{def: coarse composition}
  Given coarse subsets $\crse{R\subseteq Y\times X}$ and $\crse{S\subseteq Z\times Y}$, we say that a coarse subset $\crse{T\subseteq Z\times X}$ is their \emph{coarse composition} (denoted $\crse{T = S\circ R}$) if it is the smallest coarse subspace of $\crse{Z\times X}$ with the property that for every choice of representatives $R\csub\crse R$ and $S\csub\crse S$ the composition $S\circ R$ is coarsely contained in $\crse T$.
  This definition specializes the obvious way to the case partial coarse maps $\crse{f\colon X\to Y}$ and $\crse{g\colon Y\to Z}$ to define $\crse{g\circ f\colon X\to Z}$.
\end{definition}

\begin{remark}
  For the sake of clarity, let us stress that in \cref{def: coarse composition} ``smallest'' means ``least element'' in the ordering given by the coarse containment. Explicitly, this means that if $\crse{f\circ g}$ is the coarse composition and $\crse{h\subseteq Z\times X}$ is a coarse subspace such that $R\circ S\csub \crse h$ whenever $R\csub \crse f$ and $S\csub \crse g$, then $\crse{f\circ g\subseteq h}$.
\end{remark}

As expected, the coarse composition of two partial coarse maps may not exist, because coarse containment need not have least elements. On the other hand, if it exists it is clearly unique.
The following simple lemma (proved in \cite{roe-algs}) provides us with a useful criterion to decide whether the composition exists, and also provides us with a representative for it.
\begin{lemma}[cf. \rfLemCrsComp] \label{lem: composition exists}
  If $\crse{f\colon X\to Y}$ and $\crse{g\colon Y\to Z}$ are partial coarse maps with $\cim(\crse f)\crse \subseteq \cdom(\crse g)$, then the coarse composition $\crse{g\circ f}$ exists. Moreover, if $\crse f=[R]$ and $\crse g=[S]$ are representatives such that $\pi_Y(R)\subseteq\pi_Y(S)$ then $\crse{g\circ f}=[S\circ R]$.
\end{lemma}

\begin{remark}
  Observe that \cref{lem: composition exists} always applies if $\crse g$ is a coarse map. Moreover, if we choose as representatives a function $g\colon Y\to X$ and a (partial) function $f\colon X\to Y$ then $\crse{g\circ f}$ is nothing but $[g\circ f]$. That is, coarse composition of coarse maps is a natural extension of the usual composition.
\end{remark}

The diagonal $\Delta_X\subseteq X\times X$ is the graph of the identity function.
True to our conventions, we denote:
\[
\cid_{\crse X} \coloneqq [\id_X] = [\Delta_X].
\]

\begin{remark}
  Note that a relation $E\subseteq X\times X$ is a controlled entourage if and only if $E\csub \Delta_X$. Namely, $E\in\CE$ if and only if $[E]\crse\subseteq\cid_{\crse X}$.
\end{remark}

We may finally define coarse equivalences as follows.
\begin{definition}\label{def: coarse equivalence}
  The coarse spaces $\crse X$ and $\crse Y$ are \emph{coarsely equivalent} if there are coarse maps $\crse{ f\colon X\to Y}$ and $ \crse{g\colon Y\to X}$ such that $\crse{f\circ g}=\cid_{\crse Y}$ and $\crse{g\circ f}=\cid_{\crse X}$. Such coarse maps are \emph{coarse equivalences} and are said to be \emph{coarse inverse} to one another.
\end{definition}

\begin{example}
  If $\crse X = (X,d_X)$ and $\crse Y = (Y,d_Y)$ are metric spaces, then they are coarsely equivalent as coarse spaces if and only if they are coarsely equivalent in the classical metric sense as we defined in \cref{intro:background}.
  Another classical characterization of coarse equivalence is as a coarsely surjective coarse embedding. Namely, it is not hard to verify that $(X,d_X)$ and $(Y,d_Y)$ are coarsely equivalent if and only if there exist unbounded increasing ``control functions'' $\rho_-,\rho_+\colon [0,\infty)\to [0,\infty)$ and a map
  \(f\colon X\to Y\) with coarsely dense image and such that
  \begin{equation}\label{eq:classic coarse embedding}
    \rho_-(d_X(x,x') )\leq d_Y(f(x),f(x')) \leq  \rho_+(d_X(x,x') )
  \end{equation}
  for every $x,x'\in X$.
\end{example}

\begin{example}
  A coarse space $\crse X$ is coarsely locally finite (resp.\ has bounded geometry) if and only if it is coarsely equivalent to a locally finite (resp.\ uniformly locally finite) coarse space.
  One direction is easily seen by fixing an appropriate partition of $\crse X$ and collapsing each of these regions to a point (\emph{i.e.}\ replace $X$ by the set indexing the partition). For the converse implication, it is enough to choose an everywhere-defined representative $f\colon X\to I$ for the coarse equivalence and consider the partitioning of $X$ into preimages of points of $I$.
\end{example}

Given a partial coarse map $\crse{f\colon X\to Y}$, we may always consider the symmetric coarse subspace $\op{\crse f}\crse{\subseteq X\times Y}$. Of course, $\op{\crse f}$ needs not be a partial coarse map, because the transpose of a controlled relation is not controlled in general.\footnote{\, Think of a map $f \colon X \to Y$ that collapses everything to a point.} If it is, we say that $\crse f$ is a \emph{partial coarse embedding} into $\crse Y$, or simply a \emph{coarse embedding} if $\cdom(\crse f)=\crse X$ (by definition, a partial coarse embedding $\crse f$ is a coarse embedding of $\cdom(\crse{f})$ into $\crse{Y}$).

\begin{example}
  A function $f\colon X\to Y$ between metric spaces is a coarse embedding if and only if it satisfies \cref{eq:classic coarse embedding} for some choice of unbounded increasing control functions $\rho_-,\rho_+\colon [0,\infty)\to [0,\infty)$.
\end{example}

It is a useful observation that $\op{\crse f}$ is a canonical ``tentative coarse inverse'' for $\crse f$. Namely, a patient scrutiny of the definitions given up to this point reveals the following.
\begin{proposition}[cf.\ \rfPropTransposeIsInverse]\label{prop:transpose is coarse inverse}
  Let $\crse f$ be a partial coarse embedding. Then $\op{\crse f}$ is a partial coarse embedding as well, and the compositions $\op{\crse{f}}\crse \circ \crse{f}$ and $\crse{f}\crse\circ\op{\crse{f}}$ are well-defined and are contained in $\cid_{\crse X}$ and $\cid_{\crse Y}$ respectively. 
  Moreover, $\crse{f}$ is coarsely everywhere defined if and only if $\op{\crse{f}}$ is coarsely surjective (and vice versa).
  Lastly, the following are equivalent:
  \begin{enumerate}[label=(\roman*)]
    \item\label{item:prop:transpose is coarse inverse-ce} $\crse{f}$ is a coarse equivalence;
    \item\label{item:prop:transpose is coarse inverse-*ce} $\op{\crse{f}}$ is a coarse equivalence;
    \item\label{item:prop:transpose is coarse inverse-cinv} $\crse{f}$ and $\op{\crse{f}}$ are coarse inverses of one another;
    \item\label{item:prop:transpose is coarse inverse-csur} $\crse{f}$ and $\op{\crse{f}}$ are coarsely surjective;
    \item\label{item:prop:transpose is coarse inverse-cdef} $\crse{f}$ and $\op{\crse{f}}$ are coarsely everywhere defined.
  \end{enumerate}  
\end{proposition}

To conclude this brief introduction to coarse geometry, later we will also need the following.
\begin{definition}\label{def: proper}
  A coarse subspace  $\crse{R\subseteq Y\times X}$ is \emph{proper} if $\op{R}(B)$ is bounded for every bounded $B\subseteq Y$.
\end{definition}

\begin{example}
  If $\crse R=\crse f$ is a coarse map, and the representative $f$ is a function, then $\op{f}(B)=f^{-1}(B)$ is just the preimage of $B$. That is, a (partial) coarse map is proper if preimages of bounded sets are bounded. It is clear that coarse embeddings are proper. The converse is, just as clearly, false.
\end{example}

\chapter{Coarse geometric modules}\label{sec: coarse modules}
The formalism of \emph{coarse geometric modules} developed in \cite{roe-algs} is designed to bridge between coarse geometry and operator algebras. This expands on ideas of Roe \cite{roe_lectures_2003}*{Section 4.4}, which in turn had precursors \emph{e.g.} in \cites{atiyah1970global,mooreschochet1988}. We refer to \cite{roe-algs} for a more detailed and motivated discussion.
In the following sections we shall briefly recall the main facts and definitions that will be used throughout this memoir. Once again, many results are stated without proof as their proof is fairly straightforward and including it would excessively increase the length of our exposition. References to the relevant statements in \cite{roe-algs} are provided.

\section{Coarse geometric modules and their properties}
Throughout the section, let $\crse X$ be a coarse space and let $\fkA\subseteq \CP(X)$ be a unital Boolean algebra of subsets of $X$. A \emph{unital representation of $\fkA$} is a projection-valued mapping $\chf{\bullet}\colon\fkA\to\CB(\CH)$, where $\CH$ is a Hilbert space, such that
\begin{itemize}
  \item $\chf{A' \cap A} = \chf{A'} \chf{A}$ (in particular, $\chf{A'}$ and $\chf{A}$ commute);
  \item if $A\cap A'=\emptyset$ then $\chf{A'\sqcup A} = \chf {A'} +\chf A$;
  \item $\chf{X} = 1_{\CH}$.
\end{itemize}
Furthermore, a unital representation $\chf{\bullet}\colon\fkA\to\CB(\CH)$ is \emph{non-degenerate} if there is a gauge $\gauge\in\CE$ such that the images $\chf{A}(\CH)$, as $A\in\fkA$ ranges among the $\gauge$-bounded sets, generate a dense subspace of $\CH$, \emph{i.e.}\
\[
  \CH = \overline{\angles{\chf{A}(\CH) \mid A\in\fkA ,\ \gauge\text{-bounded}}}^{\norm{\cdot}}.
\]
We call such an $\gauge$ a \emph{non-degeneracy gauge}.

\begin{definition}[cf. \rfDefCrseModule]\label{def:coarse-module}
  A \emph{(finitely additive) coarse geometric module for $\crse X$}, or an \emph{$\crse X$-module} for short, is a tuple $(\fkA,\CHx,\chf\bullet)$, where $\fkA\leq \CP(X)$ is a unital Boolean algebra, $\CHx$ is a Hilbert space and $\chf{\bullet}\colon \fkA \to \CB(\CHx)$ is a non-degenerate unital representation.
  We will denote the $\crse X$-module $(\fkA,\CHx,\chf\bullet)$ simply by $\CHx$.
\end{definition}

Given an $\crse X$-module $\CHx$ as in \cref{def:coarse-module}, we freely call the elements of $\fkA$ \emph{measurable subsets} of $X$. For every measurable $A\subseteq X$, we denote by $\CH_A$ the Hilbert space $\chf A(\CHx)$, which is a closed subspace of $\CHx$.
\begin{convention}
  From now on, $\CHx$ and $\CHy$ will always denote coarse geometric modules for the spaces $\crse X$ and $\crse Y$ respectively.
\end{convention}

The definition of coarse module given in \cref{def:coarse-module} is sufficient for a first rudimentary study of coarse spaces via operators on Hilbert spaces. However, in order to obtain more refined results extra assumptions are required. Specifically, we will often use the following conditions.
\begin{definition}[cf. \rfDefAdmissibleModule]\label{def:local admissibility}
  An $\crse X$-module $\CHx$ is \emph{admissible} (resp.\ \emph{locally admissible}) if there is a gauge $\gauge$ such that every subset (resp.\ every bounded subset) of $X$ has a measurable $\gauge$-controlled thickening.
\end{definition}
\begin{definition}[cf. \rfDefDiscreteModule]\label{def:discrete module}
  An $\crse X$-module $(\fkA, \CH)$ is \emph{discrete} if there is a \emph{discreteness gauge} $\gauge\in\CE$ and a partition $X=\bigsqcup_{i\in I}A_i$ such that
  \begin{enumerate}[label=(\roman*)]
    \item \label{def:discrete module:contr} $A_i$ is \gauge-controlled for every $i\in I$ (\emph{i.e.}\ it is a controlled partition);
    \item \label{def:discrete module:measu} $A_i\in \fkA$ for every $i\in I$;
    \item \label{def:discrete module:sum} $\sum_{i\in I}\chf{A_i} = 1 \in \CB(\CH)$ (the sum is in the strong operator topology);
    \item \label{def:discrete module:conv} if $C\subseteq X$ is such that $C\cap A_i\in\fkA$ for every $i\in I$, then $C\in\fkA$.
  \end{enumerate}
  We call such a partition $(A_i)_{i \in I}$ a \emph{discrete partition}.
\end{definition}

The following constitute the most important examples of modules.
\begin{example} \label{ex:module-uniform}
  Given a coarse space $\crse X= (X,\CE)$, the \emph{uniform module $\CH_{u,\crse X}$} is defined by the representation $\chf \bullet\colon\CP(X)\to\CB(\ell^2(X))$ given by multiplication by the indicator function.
  Given a cardinal $\kappa$, we define the \emph{uniform rank-$\kappa$ module $\CH_{u,\crse X}^{\kappa}$} as the tensor product $\CH_{u,\crse X}^{\kappa}\coloneqq \CH_{u,\crse X}\otimes\CH = \ell^2(X;\CH)$ where $\CH$ is a fixed Hilbert space of rank $\kappa$.
\end{example}
\begin{example} \label{ex:module-met-spc}
  Let $(X,d)$ be a separable metric space (\emph{e.g.}\ a proper one) and $\crse X =(X,\CE_d)$ be the metric coarse space. If $\mu$ is any measure on the Borel $\sigma$-algebra $\fkB(X)$, then multiplication by indicator functions defines a geometric module $\fkB(X)\to \CB(L^2(X, \mu))$.

  Typical examples of this could be locally compact $\sigma$-compact Hausdorff groups equipped with their Haar measures, or Riemannian manifolds with the measure induced by the volume form.
\end{example}

Observe that both \cref{ex:module-uniform,ex:module-met-spc} yield discrete modules. This is clear in \cref{ex:module-uniform}, and for \cref{ex:module-met-spc} it is enough to observe that the separability condition implies there is a countable controlled measurable partition $X=\bigsqcup_{n\in\NN}A_n$. It then follows that $L^2(X, \mu)=\bigoplus_{n\in\NN}L^2(A_n, \mu)$.

\begin{remark}\label{rmk: on discrete modules}
  We briefly collect below some facts and observations.
  \begin{enumerate}[label=(\roman*)]
    \item An $\crse X$-module satisfying \cref{def:discrete module:contr,def:discrete module:measu,def:discrete module:sum} of \cref{def:discrete module} always has a natural extension satisfying \cref{def:discrete module:conv} as well (cf. \rfRmkExtensionToDiscrete).
    \item Discrete modules are admissible. Conversely, \rfPropAdmissibleIsDiscrete shows that if $\crse X$ is coarsely locally finite then every locally admissible $\crse X$-module can naturally be extended to a discrete one.
    \item\label{rmk: on discrete modules:disc is disc} If $X=\bigsqcup_{i\in I}A_i$ is a discrete partition for a module $\CHx$, one may use it to make the index set $I$ into a coarse space $\crse I$ that is coarsely equivalent to $\crse X$. In this context, $\CHx$ gives rise to an $\crse I$-module of the form $(\CP(I),\bigoplus_{i\in I}\CH_{A_i},\delta_\bullet)$, where we let $\delta_i$ be the projection onto the $i$-th component of the sum (cf. \rfRmkDiscreteIsDiscrete).
    
    In other words, when working with discrete modules one does not lose much by thinking the space $\crse X$ to be a discrete collection of points $x_i$, each of which comes with its own Hilbert space $\CH_i$. The module $\CHx$ is then the space of square-integrable sections associating with each point $x_i\in X$ a vector in $\CH_i$.
    This point of view simplifies some operator-algebraic proofs, but the price to pay is a certain loss of immediacy if one is primarily interested in continuous spaces, \emph{e.g.}\ as in \cref{ex:module-met-spc}. In this work we mostly avoid using this approach, as we find it valuable to illustrate how the formalism of modules can be used in general. We realise that this is perhaps a questionable choice, and we invite the dissatisfied reader to assume that all the modules occurring in the sequel be of the form $\ell^2(X;\CH)$.
  \end{enumerate}
\end{remark}

Up to this point, coarse modules can still be rather trivial (\emph{e.g.}\ the trivial module $\{0\}$). The faithfulness condition below is necessary for coarse modules to ``fully represent'' the space $\crse X$. This is however not sufficient to also represent coarse maps with other spaces: this requires modules to be ``ample enough''.
\begin{definition}[cf. \rfDefAmple]\label{def: ample and faithful}
  Given a cardinal $\kappa$, we say that an $\crse X$-module $\CHx$ is \emph{$\kappa$-ample} if there is a gauge $\gauge\in\CE$ such that the family of measurable \gauge-controlled subsets $A\subseteq X$ so that $\chf{A}$ has rank at least $\kappa$ is coarsely dense in $\crse X$.

  If $\CHx$ is $1$-ample, we say that it is a \emph{faithful} coarse geometric module.
\end{definition}

In other words, $\CHx$ is $\kappa$-ample if the whole space $X$ is a controlled thickening of
\(
  \bigcup\braces{A\in\fkA\mid A\text{ is \gauge-controlled and } \rk(\chf{A})\geq\kappa}
\).
This condition is especially useful when the degree of ampleness coincides with the total rank of the module. Observe that if $\CHx$ is $\kappa$-ample and of rank $\kappa$ for some finite $\kappa$ then $\crse X$ must be bounded. Since this is a trivial case, we will only need to consider infinite cardinals.

\begin{example}
  If $(X,d)$ is a metric space, let $\CHx = L^2(X,\mu)$ where $\mu$ is some Borel measure as in \cref{ex:module-met-spc}. The measure $\mu$ is said to have \emph{full support} if $\mu(\Omega)>0$ for every non-empty open $\Omega\subseteq X$. Of course, if $\mu$ has full support then $\CHx$ is faithful.
  
  The counting measure on discrete metric spaces shows that even if $\mu$ has full support, the module $\CHx$ needs not be ample. However, if $\mu$ is such that $\mu(\{x\})=0$ for every $x\in X$, then $\CHx$ must indeed be ample. To see this, observe that every open $\Omega$ must contain more than one point (this rules out the existence of isolated points in $X$). Thus, \(\Omega\) can be written as a union of at least two sets with non-empty interior---and hence positive measure. Iterating such a decomposition shows that $L^2(\Omega,\mu)$ must have infinite rank.
\end{example}

\begin{remark}
  If $\CHx$ is a discrete module, it is not hard to show that it is $\kappa$-ample if and only if there is a discrete partition $X=\bigsqcup_{i\in I} A_i$ such that $\rk(\chf{A_i})\geq\kappa$ for every $i\in I$ (see \rfCorExistsAmplePartition).

  Similarly, if $\crse X$ is a coarsely locally finite coarse spaces and $\CHx$ is discrete, one may always construct a \emph{locally finite} discrete partition.
\end{remark}

\section{On supports of vectors and operators}\label{subsec: supports}
In the metric setting, the module $\CHx=L^2(X,\mu)$ consists of functions on $X$, and with any such function $f$ one can associate its \emph{support}, defined as the smallest closed subset of $X$ supporting $f$.
In the general setup this definition does not make sense because $X$ is not a topological space. However, it still makes sense to say that a function is supported on a certain measurable subset (the essential support). This is the point of view that we will take in the sequel.
\begin{definition}
  Given $v \in \CHx$ and a measurable $A \subseteq X$, we say \emph{$A$ contains the support of $v$} (denoted $\supp(v)\subseteq A$) if $\chf{A}(v) = v$.
\end{definition}

Observe that $\supp(v)\subseteq A$ if and only if $p_v = \chf A p_v = p_v \chf{A}$, where $p_v$ is---as usual---the orthogonal projection onto the span of $v$. 
Clearly, not every vector $v\in\CHx$ needs to be supported on a bounded set. However, the following simple consequence of the non-degeneracy condition of coarse geometric modules shows that we may always find bounded sets that ``quasi-contain'' the support of $v$.
\begin{lemma}[cf. \rfLemQuasisupportOfVectors]\label{lem:supports of vectors almost contained in bounded}
  Let $\CHx$ be an $\crse X$-module and let $\gauge$ be a non-degeneracy gauge. For every $v\in \CHx$ and $\varepsilon > 0$, there is an $A\in\fkA$ that is a finite union of disjoint $\gauge$-bounded sets such that
    \(
      \norm{v - \chf{A}(v)}<\varepsilon.
    \)
\end{lemma}

Observe that for every measurable $A\subseteq X$ we have $\chfcX{A}=1-\chf{A}$, so the condition in \cref{lem:supports of vectors almost contained in bounded} can be rewritten as $\norm{\chfcX{A}(v)}\leq\varepsilon$. We will use both notations interchangeably.
\begin{corollary}\label{cor:supports of vectors almost contained in bounded}
  Let $\crse X=\bigsqcup_{j\in J}\crse X_j$ be the decomposition into coarsely connected components. For all $v\in \CHx$ and $\varepsilon > 0$, there is an $A\in\fkA$ with
  \begin{itemize}
    \item  $\norm{v - \chf{A}(v)}<\varepsilon$;
    \item $A$ is contained in finitely many components of $\crse X$;
    \item $A_j \coloneqq A\cap X_j$ is bounded for every $j\in J$.
  \end{itemize}
\end{corollary}

\begin{corollary} \label{cor:ql-finite-rank-locally supported}
  Let $\crse X$ be coarsely connected and $t\in\CB(\CHx)$ a finite rank operator. Then for every $\delta>0$ there is a measurable bounded set $C\subseteq X$ such that $\norm{\chf Ct-t}\leq \delta$ and $\norm{t\chf C-t}\leq \delta$.
\end{corollary}
\begin{proof}
 A finite rank operator is a finite sum of rank-1 operators, and for those the claim follows easily from \cref{cor:supports of vectors almost contained in bounded}. Since $\crse X$ is coarsely connected, the union of the finitely many resulting bounded sets is still bounded.
\end{proof}

We will now move to discussing supports of operators between coarse geometric modules. Once again, in the topological setting one may use the topology to define the support of such an operator, while in the general coarse setting only containment makes sense.\footnote{
  \, However, we will see below that there is a well-defined \emph{coarse support} (see \cref{def:coarse support}). This notion plays a key role in our approach to rigidity of Roe-like \cstar{}algebras.
  }

Given two subsets $A\subseteq X$, $B\subseteq Y$ and a relation $R\subseteq Y\times X$, we say that \emph{$B$ is $R$-separated from $A$} if $B \cap R(A)= \emptyset$.
\begin{definition} \label{def:supp-operator}
  Let $t\colon \CHx\to\CHy$ be a bounded operator, and let $S\subseteq Y\times X$. We write \emph{$\supp(t)\subseteq S$} if $\chf B t\chf A = 0$ for any pair of measurable subsets $A\subseteq X$, and $B\subseteq Y$ such that $B$ is $S$-separated from $A$.
\end{definition}
If $\supp(t) \subseteq S$ in the sense of \cref{def:supp-operator} then we say that \emph{the support of $t$ is contained in $S$}. 
The following is rather straightforward.
\begin{lemma} \label{lemma:supp-vector-a-times-a}
 If $A',A\subseteq X$ are measurable and $v',v\in\CHx$ are vectors, then $\supp(\matrixunit{v'}{v})\subseteq A'\times A$ if and only if $\supp(v')\subseteq A'$ and $\supp(v)\subseteq A$. In particular, $\supp(p_v)\subseteq A\times A$ if and only if $\supp(v)\subseteq A$.
\end{lemma}
\begin{proof}
 Given $B',B\subseteq X$, observe that $B'$ is $(A'\times A)$-separated from $B$ if and only if at least one between $B\cap A$ and $B'\cap A'$ is empty.
 Now, the ``only if'' implication is readily deduced from the identity
 \[
  \matrixunit{v'}{v} = (\chf{A'}+\chfcX{A'})\matrixunit{v'}{v}(\chf A+\chfcX A) = \chf{A'}\matrixunit{v'}{v}\chf A.
 \]
 The ``if'' implication follows from $\chf{B'}\matrixunit{v'}{v}\chf{B}=\chf{B'}\chf{A'} \matrixunit{v'}{v}\chf A\chf{B}=0$.
\end{proof}

It is not hard to show that containments of supports of operators are fairly well behaved. Specifically, if $r,s\colon\CHx\to\CHy$ have $\supp(r)\subseteq S_r$ and $\supp(s)\subseteq S_s$, then
\begin{align*}
 \supp(r+s)\subseteq S_r\cup S_s
 \;\;\text{ and }\;\;
 \supp(r^*)\subseteq \op{S_r}.
\end{align*}
In fact, infinite sums pose no difficulties either:

\begin{lemma}\label{lemma:supports-operators-sum}
  Let $t_i\colon \CHx\to\CHy$, with $i \in I$, be a family of bounded operators such that $\sum_{i \in I} t_i$ converges strongly to a bounded operator $t$. If $\supp{(t_i)} \subseteq S_i$, then $\supp\paren{t} \subseteq \bigcup_{i \in I} S_i$.
\end{lemma}
\begin{proof}
  Suppose $\chf{B} t \chf{A}$ is not zero and fix a vector $v \in \CHx$ with $\norm{\chf{B} t \chf{A}(v)} > 0$. Then there is a finite set of indices $J\subseteq I$ so that
  \[
   0< \norm{\chf{B} \Bigparen{\sum_{j\in J}t_j} \chf{A}(v)} \leq \sum_{j\in J}\norm{\chf{B}t_j\chf{A}(v)}.
  \]
  It follows that $B\cap S_j(A)\neq\emptyset$ for at least one such $j \in J$. A fortiori, $B$ intersects $\paren{\bigcup_{i\in I}S_i}(A)$.
\end{proof}

\begin{remark}
  The converse of \cref{lemma:supports-operators-sum} is by no means true. For instance, the identity operator may be expressed as \(1 = a + (1-a)\) for some \(a \in \CB(\CHx)\) of unbounded propagation. 
\end{remark}

Composition of operators is only marginally more complicated. Indeed, given $t\colon\CHy\to \CHz$ and $s \colon \CHx \to \CHy$ with $\supp(t)\subseteq S_t$ and $\supp(s) \subseteq S_s$, then
\begin{equation}\label{eq:support of composition}
 \supp(ts)\subseteq S_t\circ\gauge_{\crse Y}\circ S_s
\end{equation}
where $\gauge_{\crse Y}$ is any non-degeneracy gauge for $\CHy$ (cf. \rfLemSupportComposition). We will see in \cref{subsec: block-entourages} that the situation is even simpler for discrete modules.

We conclude this preliminary discussion of containments of supports by remarking that it is also well behaved under taking joins of projections.

\begin{lemma}\label{lem:join of controlled projection}
  If $\paren{p_i}_{i\in I}$ is a family of projections with supports contained in $S\subseteq X\times X$, then the support of their join $p = \bigvee_{i\in I} p_i$ is contained in $S$ as well.
 \end{lemma}
 \begin{proof}
   Observe that the image of $p\chf{A}$ is contained in the closed span
   \[
    \overline{\angles{p_i(\CHx)\mid p_i\chf A\neq 0}} \leq \CHx.
   \]
   Therefore, if $\chf{A'}p_i\chf A=0$ for every $i\in I$, then  $\chf{A'}p\chf A=0$ too.
 \end{proof}

\section{Almost and quasi containment of support}\label{subsec: almost_n_quasi support}
Sometimes the support of an operator is just too large to be of any use. However, it can still be extremely useful to find relations that ``contain most of it''. There are at least two ways to make this idea precise, as shown by the following.

\begin{definition} \label{def:app_n_quasi-supp}
  Given a relation $S\subseteq Y\times X$ and $\varepsilon>0$, an operator $t \in \CHx\to\CHy$ is said to have
  \begin{enumerate} [label=(\roman*)]
    \item \emph{support $\varepsilon$-quasi-contained in $S$} if $\norm{\chf {B} t\chf A} \leq \varepsilon$ for every choice of $S$-separated measurable subsets $A\subseteq X$ $B\subseteq Y$.\footnote{\, In the case when $\crse{X} = \crse{Y}$ this is similar to the notion of having \emph{$\varepsilon$\=/propagation at most $R$} used \emph{e.g.}\ in~\cite{ozawa-2023}.}
    \item \emph{support $\varepsilon$-approximately contained in $S$} if there is some $t \in \CHx\to\CHy$ such that $\supp(s)\subseteq S$ and $\norm{s - t} \leq \varepsilon$.
  \end{enumerate}
\end{definition}

We warn the reader that the phrasing in \cref{def:app_n_quasi-supp} will keep appearing throughtout the rest of the memoir. Indeed, we will be interested in both ``approximate'' and ``quasi-local'' notions.

\begin{example}
  See $L^2(\RR)$ as an $(\RR,\abs{\variable})$-module, and consider the operator \(m\) given by pointwise multiplication by $e^{-x^2}$. It is clear that for every $\varepsilon>0$ we can choose $r>0$ large enough such that \(m\) has support $\varepsilon$-approximately contained in $[-r,r]\times[-r,r]$. This is valuable information to have, as \(m\) is qualitatively very different from, say, the identity operator.
\end{example}

It is immediate to observe that if $t$ has support $\varepsilon$-approximately contained in $S$, then it also has support $\varepsilon$-quasi-contained in $S$. The converse implication is not true in general (see \cref{rmk: approximable vs quasilocal} below). In the subsequent sections we will work with both notions, because the theory we develop applies equally well in either setting.
To start with, we remark that these notions are well behaved under composition.

\begin{lemma}[cf.\ \rfLemOperationsOnQuasiLocal]\label{lem:supports and operations quasi-local}
  Given $t\colon\CHy\to \CHz$ and $s \colon \CHx \to \CHy$ with support respectively $\varepsilon_t$ and $\varepsilon_s$-approximately contained in $S_t$ and $S_s$, then $ts$ has support $(\varepsilon_t\norm{s}+\varepsilon_s\norm{t})$-approximately contained in
  \(
  S_t\circ\gauge \circ S_s,
  \)
  where $\gauge$ is a non-degeneracy gauge. If $\CHy$ is admissible and $\gauge$ is an admissibility gauge, the same holds true for quasi-containment.
\end{lemma}
\begin{proof}
The statement for approximate containment follows directly from \cref{eq:support of composition}. For the quasi-containment, let $A$ and $B$ be measurable and $(S_s\circ\gauge\circ S_t)$\=/separated, where $\gauge$ is an admissibility gauge, and let $C$ be a measurable \gauge-controlled thickening of $S_t(B)$. Since $A$ and $C$ are still $S_s$\=/separated, and $X\smallsetminus C$ and $B$ are $S_t$\=/separated, we have
  \begin{equation*}
    \norm{\chf{A}st\chf{B}} 
    \leq \norm{\chf{A}s(\chfcX{C}t\chf{B})} + \norm{(\chf{A}s\chf{C})t\chf{B}}
    \leq \varepsilon_t\norm{s}+\varepsilon_s\norm{t}.  \qedhere
  \end{equation*}
\end{proof}

The most crucial fact that we need to record in this section is that all the above notions of containments of supports are closed conditions in the weak operator topology. More precisely, the following holds.

\begin{lemma}\label{lemma:qloc-sot-closed}
  Given $\varepsilon > 0$ and $S \subseteq Y\times X$, the sets
  \begin{enumerate}[label=(\roman*)]
    \item \label{lemma:qloc-sot-closed:1} $\braces{t\in\CB(\CHx)\mid t \;\, \text{has support contained in $S$}}$;
    \item \label{lemma:qloc-sot-closed:2} $\braces{t\in\CB(\CHx)\mid t \;\, \text{has support $\varepsilon$-quasi-contained in $S$}}$;
    \item \label{lemma:qloc-sot-closed:3} $\braces{t\in\CB(\CHx)\mid t \;\, \text{has support $\varepsilon$-approximately contained in $S$}}$
  \end{enumerate}
  are all \wot-closed in $\CB(\CHx)$.
\end{lemma}
\begin{proof}
  The proof of \cref{lemma:qloc-sot-closed:1} follows from that of \cref{lemma:qloc-sot-closed:2} by letting $\varepsilon=0$ throughout, so we stick to \cref{lemma:qloc-sot-closed:2,lemma:qloc-sot-closed:3}.
  Suppose that $\{s_\lambda\}_{\lambda \in \Lambda}$ is a net of operators that have support $\varepsilon$-quasi-contained in $S$ and weakly converge to some $s\in\CB(\CHx)$.
  Given $S$-separated measurable subsets $A\subseteq X$ and $B\subseteq Y$, and arbitrary unit vectors $w'\in\CH_{B}$ and $w\in\CH_A $, we then have
  \[
    \abs{\scal{w'}{s(w)} }=\lim_{\lambda \in \Lambda} \abs{\scal{w'}{s_\lambda(w)}} \leq \varepsilon.
  \]
  This implies that $\norm{\chf{B}s\chf A}\leq\varepsilon$, hence $s$ has support $\varepsilon$\=/quasi-contained in $S$.

  The proof of \cref{lemma:qloc-sot-closed:3} is given in~\cite{braga_gelfand_duality_2022}*{Proposition~3.7} (see also \cite{braga_farah_rig_2021}), but given its importance we include it below for the convenience of the reader. Let $\braces{t_\lambda}_{\lambda \in \Lambda}$ be a net of operators that have support $\varepsilon$-approximately contained in $S$ and strongly converge to some bounded operator $t\in\CB(\CHx)$.
  For each $\lambda \in \Lambda$ there is some $s_\lambda \in \CB(\CHx)$ of with $\supp(s_\lambda)\subseteq S$ such that $\norm{t_\lambda - s_\lambda} \leq \varepsilon$.
  By the Banach--Steinhaus Theorem (also known as the uniform boundedness principle), the family $\braces{s_\lambda}_{\lambda \in \Lambda}$ is uniformly bounded. It follows from the compactness of the unit ball in the weak operator topology that there is a subnet $\braces{s_{\mu}}_{\mu}$ converging in the weak operator topology to some $s \in \CB(\CHx)$. The operator $s$ has support contained in $S$ by \cref{lemma:qloc-sot-closed:1} and satisfies $\norm{t - s}\leq\varepsilon$. This proves that the set in \cref{lemma:qloc-sot-closed:3} is, in fact, weakly closed.
\end{proof}

\section{Block-entourages on discrete modules}\label{subsec: block-entourages}
We conclude this introduction to coarse geometric modules observing that working with discrete modules (as in \cref{def:discrete module}) makes it generally easier to control the supports of operators, for there are entourages that have particularly nice properties.
\begin{definition}\label{def:subordinate to partition}
  Let $\CHx$ be a discrete $\crse X$-module and $X=\bigsqcup_{i\in I}A_i$ a discrete partition. A \emph{block-relation subordinate to $(A_i)_{i\in I}$} is a relation $R\subseteq X\times X$ that is a union of blocks of the form $A_{j} \times A_i$ with $i,j\in I$. Quite naturally, a \emph{block-entourage} is an entourage $E\in \CE$ that is a block relation.
\end{definition}

If $R$ is a block-relation subordinate to $(A_i)_{i \in I}$, we will generally leave the partition $(A_i)_{i \in I}$ implicit and simply say ``block-relation''.

\begin{remark}
  Given a discrete partition $X=\bigsqcup_{i\in I}A_i$, the set of block-entourages is cofinal in $\CE$. This is seen by observing that for every $E\in\CE$ the thickening $\bigsqcup\braces{A_i\times A_j\mid A_i\times A_j\cap E\neq\emptyset}$ is a controlled block-entourage.
\end{remark}

An extremely convenient feature of block-relations is that if $R$ is a block-relation then $R(A)$ is measurable for every $A\subseteq X$ (this is due to condition \cref{def:discrete module:conv} in \cref{def:discrete module}). It follows that if $S_t$ and $S_s$ are block-relations and $t,s\in\CB(\CHx)$ are operators with $\sup(t)\subseteq S_t$ and $\supp(s)\subseteq S_s$, then
\begin{equation}
\supp(ts)\subseteq S_t\circ S_s
\end{equation}
(compare with \cref{eq:support of composition}). This is seen directly by writing
\[
  ts\chf A = \chf{S_t(S_s(A))}t\chf{S_s(A)}s\chf A.\footnote{\, Observe this computation does not make sense if $S$ is not a block-relation, as $S_s(A)$ may well be a non-measurable set.}
\]

An equally useful property is that, once a discrete partition $X=\bigsqcup_{i\in I}A_i$ has been chosen, if we denote the \emph{block-diagonal} by
\[
  \diag(A_i \mid i\in I)\coloneqq \Bigparen{\bigsqcup_{i\in I}A_i\times A_i}\in\CE,
\]
then $\diag(A_i \mid i\in I)$-controlled operators behave as ``0-propagation'' operators. Specifically, if $S$ is a block-relation subordinate to $(A_i)_{i\in I}$ then 
\[
  \diag(A_i \mid i\in I)\circ S= S = S\circ \diag(A_i \mid i\in I).
\] 
Hence, the following holds.

\begin{lemma} \label{lemma:diagonal-ais-preserve-supp}
  In the previous setting, if $\supp(t) \subseteq \diag(A_i \mid i\in I)$ and $\supp(s) \subseteq S$, then $\supp(ts), \supp(st) \subseteq S$.
\end{lemma}

\begin{example}
  If $(X,d)$ is a discrete metric space and $\CH_{u,\crse X}^{\kappa} = \ell^2(X;\CH)$ is its uniform geometric module of rank $\kappa$, then the singletons $\{\{x\}\mid x\in X\}$ form a discrete partition. Thus, every relation $R\subseteq X\times X$ is a block-relation. The block-diagonal is then, of course, just the diagonal $\Delta\subseteq X\times X$, and the space of operators with support contained in $\Delta$ is then $\ell^\infty(X;\CB(\CH))\leq \CB(\ell^2(X;\CH))$.
  In other words, these are the operators such that $\supp(t(v))\subseteq \supp(v)$ for every $v\in \CH_{u,\crse X}^{\kappa}$, whence the name ``0-propagation'' (an operator has propagation bounded by $r\geq 0$ if $\supp(t(v))$ is always contained in the $r$-neighborhood of $\supp(v)$).

  If $(X,d)$ is not discrete and $\CH=L^2(X,\mu)$ for some non-zero non-atomic measure $\mu$, then the singletons do not form a discrete partition as $\chf{x}(\CH)=\{0\}$ for every $x\in X$ and hence $\sum_{x\in X}\chf{x}=0\neq 1$. In particular, $\Delta$ is not a block-relation.
  Geometrically, it still makes sense to say that $L^\infty(X)\leq\CB(L^2(X,\mu))$ is the algebra of 0-propagation operators. However, this algebra is generally not very natural from a \emph{coarse} geometric perspective, and it is also often too small to be of use (note that in the discrete example $L^\infty(X;\CB(\CH))$ is not abelian when $\kappa>1$). In \rfThmCartanSubalg some \emph{non-commutative} Cartan subalgebras of \(\roecstar{L^2(X, \mu)}\) are constructed that play the role of $L^\infty(X;\CB(\CH))$. In particular the are \emph{not} \(L^\infty(X,\mu)\).

  Our replacement for \(L^\infty(X,\mu)\) is to choose a discrete partition of $(X,d)$ and consider the associated block-diagonal operators. Strictly-speaking, these operators do not have 0-propagation, but they do retain most of their useful properties.
\end{example}

We say that a family of vectors $(v_l)_{l\in L}$ is \emph{subordinate to $(A_i)_{i\in I}$} if for every $l\in L$ there is an $i\in I$ so that $v_l$ is supported on $A_i$.
Suppose that this is the case, and let $p$ denote the projection onto their closed span. Observe that $\supp(p)\subseteq \diag(A_i \mid i\in I)$ (see \cref{lemma:supp-vector-a-times-a,lem:join of controlled projection}). We may then apply \cref{lemma:diagonal-ais-preserve-supp} to deduce that composition with $p$ does not increase the support of operators supported on block-relations. Specifically, we have proved the following.

\begin{corollary}\label{cor: sub-basis-has-prop-zero}
  Let $\CHx$ be a discrete $\crse X$-module and $X=\bigsqcup_{i\in I}A_i$ a discrete partition.
  Let $(v_l)_{l\in L}$ be subordinate to $(A_i)_{i\in I}$ and let $p \in \CB(\CHx)$ be the orthogonal projection onto $\spnclosed{v_l \mid l \in L}$. If $S$ is a block-relation and $t\in\CB(\CH)$ is $S$-controlled, then $tp$ and $pt$ are $S$-controlled as well.
\end{corollary}

\begin{definition}\label{def: subordinate projection}
  If $\supp(p)\subseteq \diag(A_i \mid i\in I)$, we also say that the projection $p$ is \emph{subordinate to $(A_i )_{i\in I}$}. 
\end{definition}

The fact that projections subordinate to discrete partitions do not influence the block-support of operators will be very useful in some future technical steps. The fist application for it we already find in the last result of this section, namely the Baire property for the space of contractions of fixed support.

Recall that a topological space is \emph{Baire} if every countable intersection of dense open subsets is dense.
Let 
\begin{equation}\label{eq: S-Supp}
  \cscontr S \coloneqq \{t\in \CB(\CHx)_{\leq 1} \mid \supp(t)\subseteq S\}
\end{equation}
denote the set of contractions with support contained in $S$.
The end of this section is dedicated to showing that if $S$ is a block-relation, then the space $\cscontr S$
, when equipped with the strong operator topology, is Baire. This is a technical point, but it does play a key role in the theory of \cstar{}rigidity (see \cref{prop:prop-e-baire}).

For separable modules this follows from the usual Baire Category Theorem because the space of contractions $(\CB(\CH)_1,\sot)$ of a separable Hilbert space is completely metrizable (and the same is true for $\cscontr S$, as it is closed in $\CB(\CH)_1$ by \cref{lemma:qloc-sot-closed}). This observation already covers most useful cases, and it is the strategy followed in previous literature about rigidity, such as~\cite{braga_gelfand_duality_2022}*{Lemma~3.9} and~\cite{braga_farah_rig_2021}*{Lemma~4.9}.

There are, however, instances where it may be necessary to use non-separable modules as well (see \cref{def: roe-like of spaces}). In such cases, more care is needed because $(\CB(\CH)_1,\sot)$ is no longer metrizable and it is not true that a closed subspace of a Baire space is Baire in general. For this reason, the Baire property for $(\cscontr S,\sot)$ has to be verified ``by hand''.
This proof is quite direct and similar to a proof of the original Baire Category Theorem. The reader may safely skip it should they so please.

\begin{proposition} \label{prop:prop-e-baire}
  Given a discrete $\crse X$-module $\CHx$ and a block-relation $S\subseteq X\times X$, the space $(\cscontr S, \sot)$ is Baire.
\end{proposition}
\begin{proof}
  Let $\{\CU_n\}_{n \in \NN}$ be a countable family of open strongly dense sets $\CU_n \subseteq \cscontr S$, and let $\CU \coloneqq \bigcap_{n \in \NN} \CU_n$.
  We have to show that $\CU \cap \CV \neq \emptyset$, where $\CV \subseteq \cscontr S$ is any arbitrary open set.

  Let $X = \bigsqcup_{i \in I} A_i$ be the discrete partition to which $S$ is subordinate, and let also $(e_l)_{l\in L}$ be a Hilbert basis subordinate to $(A_i)_{i\in I}$.
  Since $\CU_1 \subseteq \cscontr S$ is dense, it follows from \cref{lemma:sot-basis-top}~\cref{lemma:sot-basis-top:nghbd-orth-sys} that there are $\delta_1 > 0$, a finite $V_1 \subseteq (e_l)_{l\in L}$ and $t_1 \in \cscontr S$ such that
  \[
    t_1 \in \left(t_1 + \sotnbhd{\delta_1}{V_1}\right) \cap \cscontr S \subseteq \CU_1 \cap \CV,
  \]
  where $\sotnbhd{\delta}{V}$ is an open neighborhood as in \cref{notation:sot-open-basis}.

  Likewise, $\CU_2 \cap (t_1 + \sotnbhd{\delta_1}{V_1} \cap \cscontr S) \neq \emptyset$, for $\CU_2 \subseteq \cscontr S$ is dense.
  Thus, there are $\delta_2 > 0$, a finite $V_2 \subseteq (e_l)_{l\in L}$ and $t_2 \in \cscontr S$ such that $t_2 \in t_2 + \sotnbhd{\delta_2}{V_2} \cap \cscontr S \subseteq \CU_2 \cap \CV$.
  Moreover, we may further assume that $\delta_2 \leq \delta_1/2$ and that $V_1 \subseteq V_2$, again by \cref{lemma:sot-basis-top}.
  Iterating this process yields countable sequences $\paren{t_n}_{n \in \NN} \subseteq \cscontr S$, $\paren{\delta_n}_{n \in \NN} \subseteq (0, 1)$, and finite sets $V_n \subseteq (e_l)_{l\in L}$ such that for every $n \in \NN$
  \begin{itemize}
    \item $\delta_{n+1} \leq \delta_n/2$;
    \item $V_n \subseteq V_{n+1}$ are finite orthonormal families subordinate to $(A_i)_{i \in I}$;
    \item $(t_{n+1} + \sotnbhd{\delta_{n+1}}{V_{n+1}}) \cap \cscontr S \subseteq \CU_{n+1} \cap (t_n + \sotnbhd{\delta_n}{V_n}) \cap \cscontr S$.
  \end{itemize}

  Let $V \coloneqq \bigcup_{n \in \NN} V_n$, and let $\CH_0 \coloneqq \spnclosed{v \mid v \in V} \leq \CHx$. Decompose $\CHx = \CH_0 \oplus \CH_0^\perp$ and let $p_0$ be the projection onto $\CH_0$.
  Since $S$ is a block-relation, it follows from \cref{cor: sub-basis-has-prop-zero} that $\supp(t_mp_0)\subseteq S$ for every $m\in \NN$.
  \begin{claim}\label{im-w-is-baire-claim}
    The family $\{t_mp_0\}_{m \in \NN}$ strongly converges to an operator $t\in\CB(\CH)_{\leq 1}$ that belongs to $t_n+\sotnbhd{V_n}{\delta_n}$ for every $n\in\NN$.
  \end{claim}
  \begin{proof}[Proof of Claim]
    Observe that for every $v\in V_n$ and $m'\geq m\geq n$ we have
    \begin{align}
      \begin{split}\label{eqn:im-w-is-baire-sot}
          \norm{t_mp_0 (v) - t_{m'}p_0 (v)} &=\norm{t_m (v) - t_{m'} (v)} \\
          &\leq \sum_{i=m}^{m'-1}\norm{t_i(v)-t_{i+1}(v)} \\
          &< \sum_{i = m+1}^{m'} \delta_i \leq \sum_{i = m+1}^{m'} \frac{\delta_m}{2^{i-m}} \leq \delta_m\leq \delta_n.
      \end{split}
    \end{align}
    This shows that $\paren{t_mp_0(v)}_{m \in \NN}$ is a Cauchy sequence for every $v \in V$, so it converges to some vector in $\CH$. Since $t_mp_0$ is always $0$ on $\CH_0^\perp$ and $V$ is a Hilbert basis of $\CH_0$, this implies that the $t_mp_0$ do strongly converge to some linear operator $t$. Moreover, $\norm{t} \leq 1$ since $\CB(\CH)_{\leq 1}$ is strongly closed.

    Furthermore, \cref{eqn:im-w-is-baire-sot} also shows that for every $n\in\NN$
    \[
     \norm{t(v)-t_n(v)}=\norm{t(v)-t_np_0(v)}<\delta_n
    \]
    for every $v\in V_n$, so that $t\in t_n +\sotnbhd{V_n}{\delta_n}$, as claimed.
  \end{proof}
  
  To conclude, observe that the $t$ constructed in \cref{im-w-is-baire-claim} has support contained in $S$ because $\cscontr S$ is strongly closed by \cref{lemma:qloc-sot-closed}. It follows that $t \in \CU \cap \CV$, as desired.
\end{proof}

\begin{remark}
  All the material of this section extends trivially to operators among different discrete modules $\CHx\to\CHy$, as long as two discrete partitions $X\coloneqq \bigsqcup_{i\in I}A_i$ and $Y\coloneqq \bigsqcup_{j\in J}B_i$ are chosen and used to define block-relations.
\end{remark}

\chapter{From coarse geometry to C*-algebras}\label{sec: roe algebras}
In this chapter we complete our quick introduction to the theory of coarse geometry and modules by defining several ``Roe-like'' \cstar{}algebras and their properties. Moreover, we will also add a discussion of submodules, which is a new tool that we introduce here and will be necessary to prove the more general version of the \cstar{}rigidity theorem.

\section{Coarse supports and propagation}\label{subsec: coarse supp and propagation}
We begin by discussing operators of controlled propagation. To do so, it is however better to say more in general what controlled operators are, as they will play a key role in the sequel. In turn, the most coarse-geometric way to introduce them starts by giving a truly coarse geometric perspective to the containments of supports as introduced in \cref{subsec: supports}.

We say that a coarse subspace $\crse{R \subseteq Y\times X}$ \emph{contains the support of an operator $t\colon \CHx\to\CHy$} if $\crse{R}$ has a representative $R\subseteq Y\times X$ such that $\supp(t)\subseteq R$

\begin{definition}[cf. \rfDefCoarseSupport]\label{def:coarse support}
  The \emph{coarse support} of an operator $t\colon \CHx\to \CHy$ is the smallest coarse subspace $\csupp(t)\crse{\subseteq Y\times X}$ containing the support of $t$.
\end{definition}

Once again, `smallest' signifies that it is a least element in the ordering by coarse containment, just as for the definition of coarse composition (cf.\ \cref{def: coarse composition}). However, quite unlike the coarse composition, it is not hard to show that operators between admissible modules always have a well-defined coarse support. Specifically, it is shown in  \rfPropCoarseSupport that if $\gaugex$ and $\gaugey$ are admissibility gauges of $\crse X$ and $\crse Y$ respectively, then 
\begin{equation}\label{eq: explicit coarse support}
  S\coloneqq \bigcup\left\{B\times A \;\middle|
  B\times A\text{ meas.}\ (\gaugey\otimes\gaugex)\text{-bounded},\ \chf{B}t\chf{A}\neq 0
  \right\}
\end{equation}
is a representative for the coarse support of $t\colon \CHx\to\CHy$ such that $\supp(t)\subseteq S$.
In the above, $B\times A$ is defined to be measurable if both $A$ and $B$ are.
In the following, we will not need the explicit description \eqref{eq: explicit coarse support}: all we need is that a coarse support as in \cref{def:coarse support} always exists.

\begin{example}
  If $X$ and $Y$ are proper metric spaces equipped with Borel measures, one may define the support $\supp(t)$ of an operator $t\colon L^2(X)\to L^2(Y)$ as the smallest closed subset of $Y\times X$ such that $\chf{B}t\chf{A}=0$ whenever $B\subseteq Y$ and $A\subseteq X$ are open and $\supp(t)$-separated (equivalently, $(y,x)$ does \emph{not} belong to $\supp(t)$ if and only if there are open neighborhoods $x\in A$, $y\in B$ such that $\chf{B}t\chf{A}=0$).
  It is not hard to show to $\supp(t)$ is a representative for the coarse support: $\csupp(t)=[\supp(t)]$.
  
  However, this definition of $\supp(t)$ relies on the topology of $X$ and $Y$, which is not part of their coarse geometric data. The main point of \cref{def:coarse support} is that the coarse support can be defined and used in general without relying on the existence of an underlying topology.
\end{example}

The fact that both (partial) coarse maps and coarse supports are defined as coarse subspaces of $\crse Y\times \crse X$ can be leveraged to confound between coarse maps and operators in the following way.
\begin{definition}\label{def:controlled and proper operator}
  A bounded operator $t\colon\CHx\to\CHy$ is
  \begin{enumerate}[label=(\roman*)]
    \item \emph{controlled} if its coarse support is a partial coarse map $\csupp(t)\crse{\colon X\to Y}$ (in which case we say that $t$ \emph{covers} $\csupp(t)$);
    \item \emph{proper} if so is $\csupp(t)$ (see \cref{def: proper}).
  \end{enumerate}
  If $t$ is a controlled operator, and $R\subseteq Y\times X$ is a controlled relation from $X$ to $Y$ containing the support of $t$, we may also say that $t$ is \emph{$R$-controlled}. In particular, a controlled operator is always controlled by its support.
\end{definition}

Importantly, if $t\colon \CHx\to \CHy$ and $s\colon\CHy\to\CHz$ are controlled operators such that $\cim(\csupp(t))\crse\subseteq \cdom(\csupp(s))$, then the coarse composition $\csupp(s)\crse\circ\csupp(t)$ is well-defined (cf.\ \cref{lem: composition exists}). Moreover, \eqref{eq:support of composition} shows 
\[
  \csupp(st)\crse\subseteq\csupp(s)\crse\circ\csupp(t).
\]
This will be used extensively in the sequel.
\begin{remark}\label{rmk: proper-iff-proper}
  It is shown in \rfLemProperIffProper that if $\CHx$ and $\CHy$ are locally admissible then $t\colon \CHx\to\CHy$ is proper if and only if for every bounded measurable $B\subseteq Y$ there is a bounded measurable $A\subseteq X$ such that $\chf Bt = \chf B t\chf A$. In other words, for every vector $w\in\CHy$ supported on $B$ the ``preimage'' $t^*(w)$ is supported on $A$.
\end{remark}

Of special interest are the operators in $\CB(\CHx)$ that perturb $\CHx$ ``locally'', \emph{i.e.}\ $\csupp(t)\crse\subseteq \cid_{\crse X}$. In addition, it is often convenient to have a quantitative control of ``how small of a neighborhood'' of the diagonal contains the support of $t$.
This is the content of the following.

\begin{definition} \label{def:controlled-prop operator}
  Given $E\in\CE$, an operator $t \in \CB(\CHx)$ has \emph{$E$-controlled propagation} if $\supp(t)\subseteq E$. We say that $t$ has \emph{controlled propagation} if its propagation is controlled by $E$ for some $E \in \CE$ (equivalently, $\csupp(t) \crse{\subseteq} \cid_{\crse X}$).
\end{definition}

\begin{rmk}
  Our terminology is somewhat redundant, as
  \[
  \text{$E$-controlled propagation}\implies \text{$E$-controlled}\implies \text{support contained in $E$}.
  \]
  This is because we prefer to reserve the word ``propagation'' for operators coarsely supported on the diagonal, so that the notion of ``controlled propagation'' is a direct analogue of what is commonly called ``finite propagation'' \cite{willett_higher_2020}. At the same time, we only use the term ``$R$-controlled'' if $R$ is a controlled relation to avoid confusion (every operator is ``$Y\times X$- controlled'', but not every operator is controlled). The remaining cases are covered by the generic ``having support contained in'' wording.
\end{rmk}

Note that \cref{lemma:supp-vector-a-times-a} shows that if a vector $v\in\CHx$ has support contained in a measurable bounded set then $p_v$ has controlled propagation. The converse is a little more complicated, and follows from the non-degeneracy condition. More generally, the following holds.

\begin{lemma}\label{lem: matrix unit controlled propagation}
  Given $v,w\in\CHx$, the matrix unit $\matrixunit{w}{v}$ has controlled propagation if and only if there is a measurable bounded set $C\subseteq X$ such that $\supp(v),\supp(w)\subseteq C$.
\end{lemma}
\begin{proof}
  Observe that for a matrix unit the condition $\chf{B}\matrixunit{w}{v}\chf{A}\neq 0$ is equivalent to asking that both $\chf{B}\matrixunit{w}{v}\neq 0$ and $\matrixunit{w}{v}\chf{A}\neq 0$.
  In particular, the explicit formula for the coarse support given in \cref{eq: explicit coarse support} becomes just the product $\widetilde B\times \widetilde A$ with 
  \begin{align*}
  \widetilde B & \coloneqq \bigcup\left\{B\;\middle| \text{measurable, \gaugex-bounded},\ \chf{B}\matrixunit{w}{v}\neq 0\right\}; \\
    \widetilde A & \coloneqq\bigcup\left\{A \;\middle| \text{ measurable, \gaugex-bounded},\ \matrixunit{w}{v}\chf{A}\neq 0\right\}.
  \end{align*}
  It is then trivial that $\widetilde B\times \widetilde A$ is a controlled entourage if and only if $\widetilde B\cup \widetilde A$ is a measurable bounded subset.
\end{proof}

In the core arguments of this text we need to work rather heavily with the weakenings of the notion of controlled propagation analogous to the weakenings of containment of supports defined in \cref{subsec: almost_n_quasi support}. In the following definitions, $E \in \CE$ denotes a controlled entourage.

\begin{definition} \label{def:cp-ope}
  An operator $t \in \CB(\CHx)$ is
  \begin{enumerate} [label=(\roman*)]
    \item \emph{$\varepsilon$-$E$-approximable} if there is some $s \in \CB(\CHx)$ whose propagation is controlled by $E$ and such that $\norm{s - t} \leq \varepsilon$.
    \item \emph{$\varepsilon$-approximable} if there is $E \in \CE$ such that $t$ is $\varepsilon$-$E$-approximable.
    \item \emph{approximable} if $t$ is $\varepsilon$-approximable for all $\varepsilon > 0$.
  \end{enumerate}
\end{definition}

\begin{definition} \label{def:ql-ope}
  An operator $t \in \CB(\CHx)$ is
  \begin{enumerate} [label=(\roman*)]
    \item \emph{$\varepsilon$-$E$-quasi-local} if $\norm{\chf {A'} t\chf A} \leq \varepsilon$ for every choice of measurable $A, A' \subseteq X$ where $A'$ is $E$\=/separated from $A$.
    \item \emph{$\varepsilon$-quasi-local} if there is $E \in \CE$ such that $t$ is $\varepsilon$-$E$-quasi-local.
    \item \emph{quasi-local} if $t$ is $\varepsilon$-quasi-local for all $\varepsilon > 0$.
  \end{enumerate}
\end{definition}

\begin{remark}\label{rmk: approximable vs quasilocal}
  It is immediate to verify that an approximable operator is quasi-local. On the other hand, Ozawa recently showed~\cite{ozawa-2023} that there are coarse spaces and modules which admit quasi-local operators that are \emph{not} approximable. In the language of Roe-like \cstar{}algebras (see \cref{def:roe-like-cstar} below) this means that the containment $\cpcstar{\CHx}\subseteq \qlcstar{\CHx}$ can be strict.
\end{remark}

It is a fairly simple but important observation that one can characterize whether an operator $t\colon\CHx\to \CHy$ is controlled by examining how it interacts with the operators of controlled propagation.

\begin{proposition}[cf. \rfPropControlledIffAdControlled]\label{prop:coarse functions preserve controlled propagation}
  Let $t\colon\CHx\to\CHy$ be bounded. Consider the assertions.
  \begin{enumerate} [label=(\roman*)]
    \item \label{prop:coarse functions preserve controlled propagation:1} $t$ is controlled.
    \item \label{prop:coarse functions preserve controlled propagation:2} For every $E\in\CE$ there is $F\in\CF$ such that $\Ad(t)$ sends operators of $E$-controlled propagation to operators of $F$-controlled propagation.
    \item \label{prop:coarse functions preserve controlled propagation:3} For every $E\in\CE$ there is $F\in\CF$ such that $\Ad(t)$ sends $\varepsilon$-$E$-approximable operators to $\varepsilon\norm{t}^2$-$F$-approximable operators.
    \item \label{prop:coarse functions preserve controlled propagation:4} For every $E \in \CE$ there is $F \in \CF$ such that $\Ad(t)$ sends $\varepsilon$-$E$-quasi-local operators to $\varepsilon\norm{t}^2$-$F$-quasi-local operators.
  \end{enumerate}
  Then \cref{prop:coarse functions preserve controlled propagation:1} $\Leftrightarrow$ \cref{prop:coarse functions preserve controlled propagation:2} $\Rightarrow$ \cref{prop:coarse functions preserve controlled propagation:3}. Moreover, \cref{prop:coarse functions preserve controlled propagation:1} $\Rightarrow$ \cref{prop:coarse functions preserve controlled propagation:4} when $\CHx$ is admissible.
\end{proposition}

\section{Several Roe-like algebras of coarse modules}\label{subsec: roe algs of modules}
In this section we finally define several algebras of operators on coarse geometric modules of special geometric significance. To begin, an operator $t\in \CB(\CHx)$ is \emph{locally compact} if $\chf{A} t$ and $t \chf{A}$ are compact for every bounded measurable $A\subseteq X$. It is rather straightforward to verify that
\[
\lccstar{\CHx}\coloneqq\{t\in \CB(\CHx)\mid t\text{ is locally compact}\}
\]
is \cstar{}sub-algebra of $\CB(\CHx)$ (cf. \rfLemQlIsAlgebra).

The properties of the supports of operators (see \cref{lemma:supports-operators-sum,eq:support of composition}) imply the following are both \Star{}algebras.
\begin{definition}\label{def:roe-like-star}
  Let $\CHx$ be a $\crse X$-module.
  \begin{enumerate} [label=(\roman*)]
    \item \label{def:roe-like-star:cp} the \emph{\Star{}algebra of operators of controlled propagation} is
    \[
    \cpstar{\CHx}\coloneqq\braces{t\in \CB(\CHx)\mid t\text{ has controlled propagation}};
    \]
    \item \label{def:roe-like-star:roe} the \emph{Roe \Star{}algebra} is
    \(
    \roestar{\CHx}\coloneqq \cpstar{\CHx}\cap\lccstar \CHx.
    \)
  \end{enumerate}
\end{definition}
Taking completions, we then obtain \cstar{}algebras. We collectively call the following \emph{Roe-like \cstar{}algebras}.
\begin{definition}\label{def:roe-like-cstar}
  Let $\CHx$ be a $\crse X$-module.
  \begin{enumerate} [label=(\roman*)]
    \item \label{def:roe-like-cstar:cp} the \emph{\cstar{}algebra of operators of controlled propagation} is
    \[
    \cpcstar{\CHx}\coloneqq\overline{\cpstar{\CHx}};
    \]
    \item \label{def:roe-like-cstar:roe} the \emph{Roe \cstar{}algebra} (or, simply \emph{Roe algebra}) is
    \[
    \roecstar{\CHx}\coloneqq \overline{\roestar{\CHx}};
    \]
    \item \label{def:roe-like-cstar:ql} (if $\CHx$ is admissible) the \emph{\cstar{}algebra of quasi-local operators} is
    \[
    \qlcstar{\CHx}\coloneqq\braces{t\in \CB(\CHx)\mid t\text{ is quasi local}}.
    \]
  \end{enumerate}
\end{definition}
\cref{lem:supports and operations quasi-local} implies that if $\CHx$ is admissible then $\qlcstar{\CHx}$ is indeed a \cstar{}algebra (cf. \rfLemLcIsAlgebra).
In particular, this holds when $\CHx$ is discrete, as will often be the case.

\begin{remark}
  Our ``definition'' of Roe-like \cstar{}algebras by enumeration of useful \cstar{}algebras is by no means a satisfying one, and we simply see it as a convenient way to state and prove results for all those algebras simultaneously.

  Some abstract definitions for Roe-type or Roe-like algebras are already present in literature \cites{spakula_rigidity_2013,spakula2019relative}, but they do not suit our purposes are they do not include \emph{e.g.}\ $\cpcstar{\variable}$ and $\qlcstar{\variable}$. At the same time, we are unsure \emph{how much} more general our ideal definition would be. One may wish to single out a family of \cstar{}algebras for which the techniques of \cstar{}rigidity apply and it is broad enough to cover all interesting cases. The latter is however a moving target (for instance, very recently a new class of ``strongly quasi-local'' operators has been introduced \cite{bao2023strongly}), and finding an optimal abstract condition encompassing all present and future use cases seems a hopeless endeavor.
\end{remark}

We have obvious inclusions
\[
 \roecstar\CHx\subseteq\cpcstar\CHx\subseteq\qlcstar\CHx.
\]
Both inclusions are strict in general. It is not hard to see that $\roecstar\CHx=\cpcstar\CHx$ if and only if $\CHx$ has \emph{locally finite rank}, \emph{i.e.}\ $\CH_A$ has finite rank for every measurable bounded subset $A\subseteq X$.
In contrast, understanding when $\cpcstar\CHx =\qlcstar\CHx$ is still an open problem, and it was proved only very recently that the inclusion is indeed strict if $X$ contains a family of expander graphs \cite{ozawa-2023}.

Observe that $\roecstar\CHx$ is clearly contained in $\cpcstar\CHx\cap\lccstar\CHx$. In the most meaningful cases, the converse containment also holds. Namely, the following is proven in \cite{roe-algs} (extending a previous result of \cite{braga_gelfand_duality_2022}).
\begin{theorem}[cf.\ \rfRoeAlgCap]\label{thm: roe as intersection}
  Let $\crse X$ be a coarsely locally finite coarse space, and let $\CHx$ be a discrete $\crse{X}$\=/module. Then $\cpcstar{\CHx} \cap \lccstar{\CHx} = \roecstar{\CHx}$.
\end{theorem}

\cref{thm: roe as intersection} can be very useful to prove that certain operators belong to the Roe algebra (\emph{e.g.}\ \cref{cor: aut-cp sends roe-to-roe}).

\begin{remark}\label{rmk: unital Roe algebras}
  The identity operator $1\in\CB(\CHx)$ has $\Delta_X$-controlled propagation, and hence belongs to $\cpcstar{\CHx}$ and $\qlcstar\CHx$. These two \cstar{}algebras are hence unital.
  On the other hand,  $1$ is locally compact if an only if $\CHx$ has locally finite rank. This shows that $\roecstar{\CHx}$ is generally not unital, and it is only unital when $\roecstar{\CHx}=\cpcstar{\CHx}$.
\end{remark}

In order to stress the unifying nature of our approach to the rigidity of Roe-like \cstar{}algebras, we use the following.

\begin{notation}
  Throughout the memoir, we will denote Roe-like \cstar{}algebras by $\roeclike{\variable}$.
\end{notation}

In some key technical arguments we will need to differentiate depending on whether $\roeclike\variable$ is unital or not (the latter case often requiring extra hypotheses concerning \emph{e.g.}\ coarse local finiteness). This is the main source of differences between $\roecstar\variable$ and the other two Roe-like \cstar{}algebras---at least insofar rigidity is concerned.

\begin{remark}\label{rmk: roelike of tensor with finite rank}
  Given an $\crse X$-module $\CHx$ and an arbitrary Hilbert space $\CH$, the tensor product $\CHx\otimes\CH$ is naturally an $\crse X$-module with the representation given by the simple tensors $\chf{A}\otimes 1_\CH$, and therefore $\roeclike{\CHx\otimes\CH}$ is defined.
  Since operators of the form $1\otimes t$ always have controlled propagation (they have ``0-propagation''), we trivially have containments
  \[
    \cpcstar\CHx\otimes\CB(\CH)\subseteq \cpcstar{\CHx\otimes\CH} \; \text{ and } \;
    \qlcstar\CHx\otimes\CB(\CH)\subseteq \qlcstar{\CHx\otimes\CH},
  \]
  where on the left-hand side we take the minimal (or, spatial) tensor product.
  If $\CH$ has infinite rank, the analogous containment fails for the Roe algebras because of the local compactness condition. However, it is still true that
  \[
    \roecstar\CHx\otimes\CK(\CH)\subseteq \roecstar{\CHx\otimes\CH}.
  \] 
  All these inclusions may be strict. In contrast, it is important to observe that
  \[
    \roeclike\CHx\otimes\CB(\CH) = \roeclike\CHx\otimes\CK(\CH) = \roeclike{\CHx\otimes\CH}
  \]
  if $\CH$ has finite rank.
\end{remark}

In the sequel, we will want to understand homomorphisms between Roe-like algebras. The first step to do so will be to use \cref{prop: non-degenerate hom is spatially implemented} to deduce that these homomorphisms are spatially implemented. In turn, the reason why \cref{prop: non-degenerate hom is spatially implemented} can be applied to homomorphisms of Roe-like algebras is the following, fairly straightforward, structural result.

\begin{proposition}[cf.\ \rfRoelikeDisconnected]\label{prop: roelike cap compacts}
  Let $\crse X=\bigsqcup_{i\in I} \crse X_i$ be the decomposition in coarsely connected components of $\crse X$, and suppose $X_i\subseteq X$ is measurable for all $i \in I$ (this is always the case if $\CHx$ is discrete).
  Then
  \[
    \roeclike{\CHx}\leq \prod_{i\in I}\CB(\CH_{X_i})
  \]
  and
  \[
    \roeclike{\CHx}\cap \CK(\CHx)= \bigoplus_{i\in I}\CK(\CH_i).
  \]
\end{proposition}

\begin{remark}
  Note that \cref{cor:ql-finite-rank-locally supported} already proves that if $\crse X$ is coarsely connected then $\roecstar\CHx$ contains the compact operators. Proving \cref{prop: roelike cap compacts} for disconnected extended metric spaces requires just a little extra care.
\end{remark}

\section{Coarse maps vs.\ homomorphisms}
In this section we begin exploring the relations between coarse maps between spaces, operators between modules, and homomorphisms between Roe-like algebras. In a nutshell, the directions
\[
\text{coarse maps}\implies \text{operators}\implies\text{homomorphisms}
\]
are fairly simple and classical. Their converses are hard, and are the main subject of this memoir. We now begin by recalling the classical direction. The following material is not necessary to solve the \cstar{}rigidity problem, but it does provide intuition and motivation for it.

To improve legibility, from this point on we will adopt the following convention.
\begin{notation}\label{notation: capital letters}
  We denote operators between coarse geometric modules with capital letters (\emph{e.g.}\ $T,U,V,W$\ldots). Lower-case letters ($s,t,u$\ldots) are used for operators within the same module, especially when considered as elements of Roe-like \cstar{}algebras.
\end{notation}

\smallskip

The natural way to go from linear operators to mappings of Roe-like algebras is by taking conjugation. Of course, not every operator will induce a map of Roe-like algebras, but the language we introduced in the previous sections lets us easily identify classes of operators that do so. Namely,
\cref{prop:coarse functions preserve controlled propagation} shows that if $T\colon\CHx\to\CHy$ is a controlled operator (cf.\ \cref{def:controlled and proper operator}) then $\Ad(T)\colon\CB(\CHx)\to\CB(\CHy)$ restricts to mappings
\[
  \cpstar{\CHx}\to\cpstar{\CHy} \;\; \text{ and } \;\; \cpcstar\CHx\to\cpcstar\CHy.
\]
Likewise, if $\CHx$ is admissible the same holds for $\qlcstar{\CHx}$. Moreover, if $\CHx$ is locally admissible and $T$ is also proper, then $\Ad(T)$ preserves local compactness as well (cf.\ \rfCorAdProperLocalCpt or \cref{item:thm:qp:qp}$ \Rightarrow $\cref{item:thm:qp:lc-to-lc} of \cref{thm: quasi-proper} below). Thus, it restricts to mappings at the level of $\roestar{\CHx}$ and $\roecstar\CHx$ as well.
In summary, we showed the following.

\begin{cor}\label{cor: Ad restricts to Roe-like}
  If $T\colon \CHx\to \CHy$ is a proper controlled operator 
  and $\CHx$ is admissible then conjugation induces mappings of Roe-like algebras
  \[\Ad(T)\colon \roeclike{\CHx}\to\roeclike{\CHy}.\]
\end{cor}

Note that if $T$ as above is an isometry, then $\Ad(T)$ is an injective \Star{}homomorphism.

\smallskip

In view of the above, we are now interested in associating isometries with coarse maps. The idea for doing so is straightforward: if we are given an injective map $f\colon X\to Y$ it is all but natural to construct an isometry $\ell^2(X)\to \ell^2(Y)$ by sending $\delta_x$ to $\delta_{f(x)}$.
The situation becomes more delicate if one has non-injective maps and/or wishes to investigate more general modules. However, so long as the modules are ample enough (cf.\ \cref{def: ample and faithful}), these complications are purely technical and can be overcome without too much effort.

Specifically, recall that an operator $T\colon \CHx \to \CHy$ covers a (partial) coarse map $\crse{f\colon X\to Y}$ if $\csupp(T) = \crse{f}$. Then the following holds.

\begin{proposition}[cf. \rfCorExistCoveringIsom]\label{prop: existence of covering iso}
  Let $\CHx$ and $\CHy$ be both discrete, $\kappa$-ample and of rank $\kappa$, and $\crse{f\colon X\to Y}$ any coarse map. Then:
  \begin{enumerate}[label=(\roman*)]
    \item $\crse{f}$ is covered by an isometry $T\colon \CHx\to\CHy$.
    \item If $\crse f$ is also proper, the assumption that $\CHx$ and $\CHy$ have rank $\kappa$ can be weakened to \emph{local} rank at most $\kappa$ (\emph{i.e.}\ $\chf\bullet$ has rank at most $\kappa$ on every bounded measurable subset).
    \item If $\crse f$ is a coarse equivalence, $T$ may be taken to be unitary.
  \end{enumerate}
\end{proposition}

If $\crse f$ is a coarse equivalence and $U$ is a unitary covering it, then $U^*$ is a unitary covering $\crse{f^{-1}}$ and $\Ad(U^*)$ induces the inverse \Star{}homomorphism on the level of Roe-like \cstar{}algebras. Putting the pieces together, we obtain:
\begin{cor}
  Let $\CHx$, $\CHy$, $\crse{f\colon X\to Y}$ be as above. Then
  \begin{enumerate}[label=(\roman*)]
    \item $\crse{f}$ is covered by an isometry $T$ which induces a \Star{}embedding at the level of $\cpcstar{\variable}$ and $\qlcstar{\variable}$.
    \item If $\crse f$ is also proper, $\Ad(T)$ defines a \Star{}embedding of Roe (\cstar{})algebras.
    \item If $\crse f$ is a coarse equivalence, $\Ad(T)$ may be taken to define a \Star{}isomorphism of Roe-like \cstar{}algebras.
  \end{enumerate}
\end{cor}

\begin{remark}
  With regard to the rank and ampleness assumption, note that in the vast majority of the cases of interest the coarse geometric modules will be separable, and the relevant ampleness condition ($\aleph_0$-ampleness) coincides with the classical notion of ample modules used to construct (non-uniform) Roe algebras.
  
  The main reason to be interested in non-separable modules would be to work with coarse spaces that are not coarsely equivalent to any countable set. In particular, every proper metric space can be comfortably studied using separable modules.
\end{remark}

Observe that assigning to a coarse map a covering isometry is a very non-canonical operation: a number of choices are involved in its construction.
However, it is not hard to show that the effects of making different choices largely disappear up to conjugating by controlled unitaries (cf.\ \cite{braga_gelfand_duality_2022} and/or \rfLemCoveringUniIsUnique). We will eventually leverage this fact to show that there is a natural homomorphism between the group of coarse equivalences and the groups of outer automorphisms of Roe-like \cstar{}algebras (cf.\ \cref{subsec: rig outer aut 2}).

\section{Submodules}\label{subsec: submodules}
We conclude this chapter with some fairly technical results which are important to provide a proof of \cstar{}rigidity in its most general form. The motivating reason is that there are instances where an operator $T\colon\CHx\to\CHy$ fails to have some desirable property, but it does have it when restricted to appropriate subspaces of $\CHx$.
For instance, if $\CHx$ does not have locally finite rank then $1$ is not locally compact, but one may obtain a well-behaved locally compact operator by restricting it to a locally finite rank subspace of $\CHx$.
To deal with such instances, we now introduce the notion of \emph{submodule}.

\begin{definition}\label{def: submodule}
  A \emph{submodule} of a coarse geometric module $(\fkA,\CHx,\chf\bullet)$ is a coarse geometric module $(\fkA',\CH'_{\crse X},\chf\bullet')$ where $\fkA'\subseteq \fkA$, $\CH'_{\crse X}\leq\CHx$ is closed and $\chf A'(v)=\chf A(v)$ for every $v\in\CH'_{\crse X}$ and $A \in \fkA'$.
  We say that such a submodule is \emph{subordinate to $\fkA'$} (and omit $\fkA'$ from the notation if clear from the context).
\end{definition}

\begin{example}
  The discrete setting gives most the obvious examples of submodules. Suppose that $\CHx$ is a discrete module over $\crse X$ where the singletons form a discrete partition $\CHx= \bigoplus_{x\in X}\CH_x$. In particular, $\fkA = \CP(X)$ and `$X$ is discrete'.
  Choose for every $x\in X$ a closed subspace $\CH'_x\leq \CH_x$, and let $\CH'_{\crse X}\coloneqq \bigoplus_{x\in X}\CH_x' \leq\CHx$. For every $A\subseteq X$ let $\chf{A}'$ be the restriction of $\chf A$ to $\CH_A'= \bigoplus_{x\in A}\CH_x'$. Then $(\fkA,\CH'_{\crse X},\chf{\bullet}')$ is a submodule of $(\fkA,\CHx,\chf{\bullet})$.

  One example that appears often in practice is when $\CHx = \ell^2(X)\otimes\CH$ is a uniform module, $\CH'<\CH$ is some (finite dimensional) subspace and $\CH'_{\crse X}=\ell^2(X)\otimes\CH'$. For instance, in the proof of \cref{thm: stable rigidity} we will start by choosing some vector $\xi\in \CH_1$ and letting $\CH'=\angles{\xi}$.
\end{example}

\begin{example}\label{exmp: submodule smaller fkA}
  To understand the role of $\fkA'$ it is sufficient to leave the reassuring setting of discrete modules. Let $X$ be a connected Riemannian manifold, $\CHx=L^2(M)$, $\fkA=\{\text{Borel subsets}\}$ and $\chf\bullet$ be the multiplication by the indicator function. If we wished to work with some locally finite rank subspace of $\CHx$ we would be in trouble, because once we take some function $\xi \in L^2(X)\smallsetminus\{0\}$ the images $\chf{A}(\xi)$ generate an infinite dimensional subspace (the whole of $L^2(\supp(\xi))$).
  This shows that to obtain something finite dimensional we are forced to drastically restrict the family of measurable sets.

  Taking inspiration from the discrete example, one simple way to select a faithful locally finite rank submodule of $L^2(X)$ is to choose a locally finite countable discrete partition $X=\bigsqcup_{i\in I}A_i$ and to pick for every $i\in I$ some function $\xi_i\in L^2(A_i)\smallsetminus\{0\}$. Then $\overline{\angles{f_i\mid i\in I}}<L^2(X)$ provides us with the required module if we let $\fkA'$ be the set of arbitrary unions of $A_i$.
  For instance, we may let $\CH'_{\crse X}<L^2(X)$ be the space of functions that are constant on $A_i$ for every $i\in I$.
\end{example}

Since $\chf A'$ coincides with $\chf A$ when defined, we usually drop the prime from the notation and simply write $\chf A$.
The following gives an alternative description of a submodule.

\begin{lemma}\label{lem: submodule iff commutes}
  Let $(\fkA,\CHx,\chf\bullet)$ be an $\crse X$-module, and let $\fkA'\subseteq \fkA$ be unital and $\CH'_{\crse X}\leq\CHx$ be closed. Denote by $p$ the projection onto $\CH'_{\crse X}$. Then $(\fkA',\CH'_{\crse X},\chf\bullet|_{\CH'_{\crse X}})$ is a submodule of $\CHx$ if and only if the following hold:
  \begin{enumerate}[label=(\roman*)]
    \item\label{item:lem: submodule iff commutes: comm} $p$ commutes with $\chf{A}$ for every $A\in\fkA'$;
    \item\label{item:lem: submodule iff commutes: nd} there is a gauge $\gauge$ such that $\angles{p\chf A\paren{\CHx}\mid A\in\fkA',\ \text{\gauge-bounded}}$ is dense in $\CH'_{\crse X}$.
  \end{enumerate}
\end{lemma}
\begin{proof}
  Suppose first that \cref{item:lem: submodule iff commutes: comm,item:lem: submodule iff commutes: nd} hold.
  It follows from \cref{item:lem: submodule iff commutes: comm} that for every $A\in\fkA'$ the projection $\chf A$ preserves $\CH'_{\crse X}$. Therefore, letting $\chf A'\coloneqq \chf{A}|_{\CH'_{\crse X}}$ defines a unital representation of $\fkA'$ in $\CB(\CH'_{\crse X})$. Condition \cref{item:lem: submodule iff commutes: nd} implies that $\chf\bullet'$ is non-degenerate.

  Vice versa, if $\CH'_{\crse X}$ is a submodule the non-degeneracy condition immediately implies \cref{item:lem: submodule iff commutes: nd}. For \cref{item:lem: submodule iff commutes: comm}, fix $A\in\fkA'$ and let $q\coloneqq\chf A$ for convenience. Observe that $pqp = qp$, whence
  \begin{equation*}
    \norm{qp - pq}^2 = \norm{(qp-pq)(qp-pq)}=\norm{q(pqp)-(pqp)-qpq+(pqp)q} = 0. \qedhere
  \end{equation*}
\end{proof}

Because of \cref{lem: submodule iff commutes}, we may (and often will) denote a submodule via the projection $p$ onto it. Observe that we are also being somewhat ambiguous with respect to the meaning of $p$, as we sometimes see it as an element in $\CB(\CHx)$, or as an operator $\CHx\to p(\CHx)$, or even $p(\CHx)\to\CHx$ (which is the inclusion map). This ambiguity allows us to write $\chf A'=p\chf A p$, as we will often do in the sequel.

\smallskip

Since we use (coarse) supports of operators to bridge between geometry and functional analysis, it is of paramount importance to make sure that these notions are well behaved under taking submodules. The following lemma does precisely this.

\begin{lemma}\label{lem:support vs submodules}
  Let $p\leq \CHx$ be a submodule and $\gauge$ be a gauge as by \cref{lem: submodule iff commutes}~\cref{item:lem: submodule iff commutes: nd}. Consider $T\colon p(\CHx)\to\CHy$ and $Tp\colon \CHx\to\CHy$. Then:
  \begin{enumerate}[label=(\roman*)]
    \item\label{item:lem:support vs submodules: large to small} $\supp(Tp)\subseteq R \Rightarrow \supp(T)\subseteq R$;
    \item\label{item:lem:support vs submodules:small to large} $\supp(T)\subseteq R \Rightarrow \supp(Tp)\subseteq \gauge \circ R$.
  \end{enumerate}

  Analogous statements hold true for a submodule $q\leq \CHy$ and operators $S\colon \CHx\to q(\CHy)$ and $q^*S = qS\colon \CHx\to \CHy$.
\end{lemma}
\begin{proof}
  The first implication is obvious, because the containment $\fkA'\subseteq \fkA$ implies that containment of support in a submodule is a weaker condition than in the larger module.

  For the second implication, by our assumption we see that 
  \[
    p = \bigvee\braces{p\chf A\paren{\CHx}\mid A\in\fkA',\ \text{\gauge-bounded}}.
  \]
  If $\chf{C}Tp\chf{B}\neq 0$ for some $B\in \fkA$ and $C\in\fkA'$, there must then be an $\gauge$-bounded $A\in \fkA'$ such that $\chf CTp\chf A\chf B\neq 0$.
  Since $p$ commutes with $\chf A$ for every $A\in\fkA'$, we also have
  \[
  0\neq \chf CTp\chf A\chf B = \chf CT\chf A p\chf A\chf B.
  \]
  Since both $C$ and $A$ belong to $\fkA'$, by definition of support we see that they cannot be $R$-separated. It follows that $C$ and $B$ are not $(R\circ\gauge)$-separated.

  The statements for $q$ are obtained by taking adjoints.
\end{proof}

\begin{cor}\label{cor: support and submodules}
  Let $p\leq \CHx$ and $q\leq \CHy$ be submodules. Then the coarse support of $T\colon p(\CHx)\to q(\CHy)$ coincides with the coarse support of the induced operator $T=qTp\colon\CHx\to\CHy$.
  In particular, $T$ is controlled if and only if it is controlled as an operator $T\colon\CHx\to\CHy$.
\end{cor}

\begin{cor}\label{cor: submodules have controlled propagation}
  If $p\leq \CHx$ is a submodule then $p\in\CB(\CHx)$ has finite propagation.
\end{cor}

Since $\csupp(p)\crse{\subseteq \cid_{X}}$, \cref{cor: Ad restricts to Roe-like} trivially implies the following.

\begin{cor}\label{cor: inclusion of submodule roe-like}
 If $p\leq \CHx$ is an admissible submodule then $\Ad(p)$ induce mappings
 \[\Ad(p)\colon\roeclike\CHx\to\roeclike p\]
 and \Star{}embeddings
 \[\Ad(p)\colon\roeclike p\to\roeclike \CHx.\]
\end{cor}
In view of the above, we may directly write $\roeclike p\subseteq\roeclike \CHx$.
When we do, we are implicitly using $\Ad(p)$ to embed $\roeclike p$ into $\roeclike \CHx$.

\begin{remark}
  Considering the identity on $L^2(X)$ in \cref{exmp: submodule smaller fkA}, we see that in \cref{lem:support vs submodules}~\cref{item:lem:support vs submodules:small to large} it is indeed necessary to enlarge the support when passing to a submodule: $\supp(1)\subseteq \Delta_X$, while the same is not true for $p$.
\end{remark}

\begin{example}\label{eg: cp as corner in Roe}
  Let $X$ be locally finite and $\CHx=\ell^2(X)\otimes\CH$ a uniform module. Let $p$ be the projection onto $\ell^2(X)\otimes \CH'$, where $\CH'<\CH$ is some finite rank subspace (\emph{e.g.}\ $\CH'=\angles\xi$). Then $\Ad(p)$ induces a \Star{}embedding of $\roecstar{p}=\cpcstar{p}$ into $\roecstar{\CHx}$.
  In particular, this yields numerous embeddings of the usual uniform Roe algebra into the Roe algebra of the coarse space, see \cref{subsec:rig roe algs spaces}.
\end{example}

If $\CHx$ is a discrete module, we say that a submodule $p$ is \emph{subordinate to the discrete partition $X =\bigsqcup_{i\in I}A_i$}, if $p$ is subordinate to the algebra consisting of arbitrary unions of the $A_i$. 
Concretely, a projection $p\in\CB(\CHx)$ defines a submodule of $\CHx$ subordinate to $(A_i )_{i\in I}$  (when equipped with the relevant $\fkA'\subseteq \fkA$)  if and only if $p$ commutes with $\chf{A_i}$ for every $i\in I$.

\begin{remark}
  Observe that if $p$ defines a submodule of $\CHx$ subordinate to $(A_i )_{i\in I}$ then it is also subordinate to $(A_i )_{i\in I}$ as a projection (cf.\ \cref{def: subordinate projection}).
\end{remark}

As shown in \cref{eg: cp as corner in Roe}, one key feature of submodules is that by restricting to locally finite rank submodules we can construct embedding of the (unital) \cstar{}algebras $\cpcstar{p}$ in the (usually not unital) Roe algebra $\roecstar\CHx$.
Once again, when we write $\cpcstar{p}\subseteq\roecstar\CHx$ we mean that the canonical embedding $\Ad(p)$ sends as $\cpcstar{p}$ into $\roecstar\CHx$.
This fact will play an important role later on, and we record this in the following.

\begin{lemma}\label{lem: exists submodule approximating operator}
  Let $\crse X$ be coarsely locally finite and $\CHx$ be discrete with discrete partition $X =\bigsqcup_{i\in I}A_i$.
  Then for every operator $t\in \roecstar{\CHx}$ and $\mu>0$ there is a submodule $p$ subordinate to $(A_i)_{i\in I}$ such that $pt\approx_\mu t$ and $\cpcstar{p}\subseteq\roecstar\CHx$.
\end{lemma}
\begin{proof}
  By \cref{cor: inclusion of submodule roe-like}, it follows that we always have the containment $\cpstar p\subseteq \cpstar\CHx$, so it only remains to verify the local compactness condition. It is then clear that $\cpstar p\subseteq \roestar\CHx$ if and only if $p$ has locally finite rank.
  The proof of the claim is now an easy consequence of the relatively well\=/known fact that Roe algebras over discrete modules on locally finite coarse spaces admit approximate units $(p_\lambda)_{\lambda \in \Lambda}\subseteq\roestar\CHx$ consisting of projections of locally finite rank that commute with the projections $\chf{A_i}$ given by an arbitrarily fixed discrete partition $X=\bigsqcup_{i\in I}A_i$ (see \cite{braga_gelfand_duality_2022}*{Proposition 2.1} or \rfThmApproxUnit for a complete proof in the current level of generality).  
  In fact, we may then choose $\lambda$ large enough such that $p_\lambda t\approx_\mu t$. Moreover, $\cpcstar{p_\lambda}\subseteq \roecstar{\CHx}$ for every $\lambda\in\Lambda$ because $p_\lambda$ has locally finite rank, so for any operator $s$ of finite propagation $p_\lambda sp_\lambda$ is locally compact.
\end{proof}

\begin{remark}\label{rmk: join of submodules}
  Let $p,q\leq \CHx$ be subordinate to the same algebra $\fkA'$.
  \begin{enumerate}[label=(\roman*)]
    \item The join $p\vee q$ is also a submodule of $\CHx$ subordinate to $\fkA'$ (it is important that $p$ and $q$ be subordinate to the same subalgebra to ensure that the join remains non-degenerate).
    \item
    Observe that $\cpcstar{p},\cpcstar{q}\subseteq\roecstar\CHx$ if and only if both $p$ and $q$ have locally finite rank, in which case the same holds for $p\vee q$ and hence $\cpcstar{p\vee q}\subseteq\roecstar\CHx$.
    Since the situation for the unital Roe-like algebras is even easier, this observation can be abbreviated as saying that if
    $\cpcstar{p},\cpcstar{q}\subseteq\roeclike\CHx$ then also $\cpcstar{p\vee q}\subseteq\roeclike\CHx$.
  \end{enumerate}
\end{remark}

\chapter{Uniformization theorems} \label{sec: uniformization}
The goal of this section is to prove~\cref{thm:intro: uniformization}, which is one of the main technical points in the rigidity theorems in the introduction.
The proof strategy is heavily inspired from the arguments of Braga--Farah--Vignati, especially \cites{braga_gelfand_duality_2022,braga_farah_rig_2021}.

\section{Definitions and setup}
The main objects of interest in this section are homomorphisms $\phi\colon\roeclike{\CHx}\to\roeclike\CHy$ among Roe-like \cstar{}algebras. If $\phi=\Ad(T)$ is spatially implemented, one may confound between $\phi$ and $T$ and use the latter to define geometric properties of the former. 
However, an arbitrary \Star{}homomorphism needs not be spatially implemented. As a first step, we thus need to extend the notion of ``geometrically meaningful'' (\emph{i.e.}\ controlled) from spatially implemented homomorphisms to arbitrary ones.
The inspiration in doing so comes from the observation that $T\colon \CHx\to\CHy$ is controlled if and only if $\Ad(T)$ preserves uniform control of propagation (see \cref{prop:coarse functions preserve controlled propagation}). 
\begin{definition}\label{def:controlled mapping}
  Let $\CD\subseteq \CB(\CHx)$ be a set operators and $\phi\colon\CD\to \CB(\CHy)$ any mapping.
  We say $\phi$ is \emph{controlled} if for every $E\in\CE$ there is some $F\in\CF$ such that for every $t\in\CD$ of $E$-controlled propagation the image $\phi(t)$ has $F$-controlled propagation.
\end{definition}

That is, a controlled map $\phi$ sends operators of \emph{equi} controlled propagation to operators of \emph{equi} controlled propagation. 
Observe that this definition makes sense for any mapping, not only \Star{}homomorphisms. For instance, if $T\colon\CHx\to\CHy$ is not an isometry then $\Ad(T)\colon\CB(\CHx)\to\CB(\CHy)$ is not a \Star{}homomorphism, but it still follows from \cref{prop:coarse functions preserve controlled propagation} that $\Ad(T)$ is controlled if and only if $T$ is a controlled operator (compare with \cref{def:controlled-prop operator}).
In the following, $\CD$ will often be $\cpstar\CHx$ (or its subset of contractions).

\medskip

With some thought, one quickly realizes that \cref{def:controlled mapping} is too restrictive to study arbitrary \Star{}homomorphisms of Roe-like \cstar{}algebras. In fact, not even isomorphisms of such \cstar{}algebras need to be controlled.
Instead, one can give \emph{approximate} and \emph{quasi} versions of \cref{def:controlled mapping} that are much better suited for this purpose. The following definitions should be compared with \cref{def:cp-ope,def:ql-ope} respectively.
\begin{definition}\label{def:quasi-controlled mapping}
  Let $\CD\subseteq \CB(\CHx)$ be a set operators and $\phi\colon\CD\to \CB(\CHy)$ any mapping.
  We say:
  \begin{enumerate}[label=(\roman*)]
    \item \label{def:quasi-controlled mapping: approx-control} $\phi$ is \emph{approximately-controlled} if for every $E \in \CE$ and $\varepsilon > 0$ there is some $F \in \CF$ such that for every $t\in\CD$ of $E$-controlled propagation the image $\phi(t)$ is $(\varepsilon\norm{t})$\=/$F$\=/approximable.
    \item \label{def:quasi-controlled mapping: quasi-control} $\phi$ is \emph{quasi-controlled} if for every $E \in \CE$ and $\varepsilon>0$ there is some $F\in\CF$ such that for every $t\in\CD$ of $E$-controlled propagation the image $\phi(t)$ is $(\varepsilon\norm{t})$\=/$F$\=/quasi-local.
  \end{enumerate}
\end{definition}

Note that approximately-controlled mappings are always quasi-controlled, while the converse needs not be true in general (recall \cref{rmk: approximable vs quasilocal}).
In the sequel we shall work with both concepts in \cref{def:quasi-controlled mapping}.
In fact, in the later sections we shall see that the weaker notion of ``quasi-control'' already suffices to prove rigidity for Roe-like \cstar{}algebras (namely \cref{thm:intro: rigidity})---and it is actually the setting were our techniques apply most naturally. This can be seen as the heart of the unifying nature of our techniques.
On the other hand approximate control is needed when dealing with the refined rigidity results (\emph{e.g.}\ \cref{cor:intro: all groups are iso}).

\begin{remark}
  The notion of approximately controlled mapping appears in~\cite{braga2021uniform}*{Definition 3.2} and \cite{braga_gelfand_duality_2022}*{Definition~3.1~(2)} under the name of \emph{coarse-like}.
  We have chosen to change the nomenclature to \emph{approximately-} and \emph{quasi-}controlled to highlight that the former deals with approximable operators while the latter with quasi-local ones.
\end{remark}

In the sequel we will often need to work with contractions having propagation controlled by a specific entourage. It is hence useful to introduce the following.

\begin{notation} \label{not:ql-cp-op}
  Given $E \in \CE$, we let
  \begin{align*}
    \cpop{E} &\coloneqq \{t \in \CB(\CHx)  \mid \supp(t) \subseteq E\}, \\
    \cpcontr{E} &\coloneqq \cpop{E}\cap\CB(\CHx)_{\leq 1}.
  \end{align*}
\end{notation}

In other words, $\cpcontr{E}$ is nothing but $\cscontr E$ as introduced in \eqref{eq: S-Supp}, except that here we insist that $E\in\CE$ and hence we prefer using the term ``propagation'' instead of ``containment of support''.
Observe that $\cpcontr{E}$ is a \wot-closed subset of $\CB(\CH)$ by \cref{lemma:qloc-sot-closed}.

From this point on, $\CD$ will always be $\cpstar\CHx$. More specifically, we will work with a mapping $\phi\colon\cpstar\CHx\to\roeclike\CHy$ which will, moreover, be linear.
The following theorem is the main goal of this section, and constitutes a prototypical example of ``uniformization result''.

\begin{theorem}\label{thm: uniformization}
  Let $\crse X$ and $\crse Y$ be coarse spaces, equipped with modules $\CHx$ and $\CHy$.
  Let $\phi \colon \cpstar{\CHx} \rightarrow \roeclike{\CHy}$ be a \Star{}homomorphism.
  Suppose that:
  \begin{enumerate}[label=(\roman*)]
    \item \label{thm: uniformization-assumption-dis} $\CHx$ is discrete;
    \item \label{thm: uniformization-assumption-adm} $\CHy$ is admissible;
    \item \label{thm: uniformization-assumption-cnt} $\CF$ is countably generated;
    \item \label{thm: uniformization-assumption-sot} $\phi$ is strongly continuous.
  \end{enumerate}
  If $\roeclike{\CHy}$ is $\cpcstar{\CHy}$, then $\phi$ is approximately-controlled. Likewise, if $\roeclike{\CHy}$ is $\qlcstar{\CHy}$, then $\phi$ is quasi-controlled.
\end{theorem}

It is remarkable that \cref{thm: uniformization} even holds on coarse spaces that are not coarsely locally finite.
Before discussing its other hypotheses, we introduce some notation. Let $\phi$ be fixed. For $\varepsilon > 0$ and $F \in \CF$, we let
\begin{itemize}
  \item $\phiinvql{\varepsilon}{F} \coloneqq \phi^{-1}\paren{\braces{t\in\CB\left(\CHy\right)\mid t\text{ is $\varepsilon$-$F$-quasi-local}}}$.
  \item $\phiinvcp{\varepsilon}{F} \coloneqq \phi^{-1}\paren{\braces{t\in\CB\left(\CHy\right)\mid t\text{ is $\varepsilon$-$F$-approximable}}}$.
\end{itemize}
The proof of \cref{thm: uniformization} is a consequence of a more abstract uniformization phenomenon (see \cref{thm: uniformization phenomenon} below), and the argument is almost identical in the quasi-local and approximable cases.
To make this last point more apparent, we further define:
\begin{itemize}
  \item $\phiinv{\varepsilon}{F}$ is $\phiinvcp{\varepsilon}{F}$ if the image of $\phi$ is in $\cpcstar\CHy$, and $\phiinvql{\varepsilon}{F}$ otherwise.
\end{itemize}

Since $\phi$ is additive, $\phiinv{\varepsilon}{F}$ is closed under taking minus sign. Moreover,
\begin{equation}\label{eq:additive CY}
  \phiinv{\varepsilon_1}{F_1} + \phiinv{\varepsilon_2}{F_2} \subseteq \phiinv{\varepsilon_1+\varepsilon_2}{F_1 \cup F_2}.
\end{equation}
The above containment will be used numerous times in the following arguments, and it will often take the form of ``given $F'\subseteq F$, if $t'\in \phiinv{\varepsilon/2}{F'}$ but $t\notin\phiinv{\varepsilon}{F}$, then $t+t'\notin\phiinv{\varepsilon/2}{F}$''. 

Since $\phi$ is linear, \cref{thm: uniformization} can be restated as saying that for every $E \in\CE$ and $\varepsilon>0$ there is $F \in \CF$ such that $\cpcontr{E} \subseteq \phiinv{\varepsilon}{F}$.
\begin{remark}\label{rmk: on uniformization hypotheses}
  Regarding the hypotheses of \cref{thm: uniformization}:
  \begin{enumerate}[label=(\alph*)]
    \item Conditions \cref{thm: uniformization-assumption-dis,thm: uniformization-assumption-cnt,thm: uniformization-assumption-sot} are crucial for the strategy of proof.
    \item Condition \cref{thm: uniformization-assumption-adm} is needed to guarantee that $\qlcstar{\CHy}$ is a \cstar{}algebra.
    \item The containment \eqref{eq:additive CY} is also crucial, it is therefore important that $\phi$ be additive. Fully fledged linearity of $\phi$ is used in \cref{lemma:quasi-controlled:qtp}.
    \item The assumption that $\phi$ be a \Star{}homomorphism is more convenient than necessary. It is only used in the proof of \cref{lem:e_vnvm non-uniformly-approximable}. Later on, we shall encounter another instance of uniformization result that does not need it (see \cref{thm: uniformization phenomenon}).
    \item\label{rmk: on uniformization hypotheses-auto cts} Note that in \cref{thm: uniformization} the map $\phi$ is not required to be norm-continuous. However, the \Star{}homomorphism condition automatically implies that it is contractive on every \cstar{}sub-algebra of $\cpstar{\CHx}$. In particular, this applies to $\CB(\CH_A)\subseteq\cpstar\CHx$ when $A\subseteq X$ is measurable and bounded (or, more generally, to the non-commutative Cartan subalgebras constructed in \rfThmCartanSubalg). Importantly, this also applies to all matrix units $\matrixunit{v'}{v}$ of controlled propagation, as \cref{lem: matrix unit controlled propagation} shows that they belong to one such $\CB(\CH_A)$. This last fact is used in \cref{lem:e_vnvm non-uniformly-approximable}.
    \item \cref{rmk:effective uniformization} below sharpens the statement of \cref{thm: uniformization}, making it ``effective''. This show that \cref{thm: uniformization} can be adapted to include other sorts of domains for $\phi$, such as $\CD=\cpop{E}$. To limit the already significant level of technicality, we will not take this route here and shall be content to let $\CD = \cpstar\CHx$.
  \end{enumerate}
\end{remark}

\begin{remark}
  A very particular class of \Star{}homomorphisms between Roe-like algebras are the homomorphisms $\phi\colon\cpcstar\CHx\to\cpcstar \CHy$ that send a \emph{Cartan subalgebra} into another one.
  Note that the Cartan subalgebras we consider here are, in general, non-commutative. We refer the reader to \rfThmCartanSubalg for details about this notion.

  Concretely, this means that there are discrete partitions $\crse X= \bigsqcup_{i\in I} A_i$ and $\crse Y= \bigsqcup_{j\in J} B_j$ such that $\phi$ maps \(\ell^\infty(I, \CB(\CH_i)) \subseteq \cpcstar \CHx\) into \(\ell^\infty(J, \CB(\CH_j)) \subseteq \cpcstar \CHx\), where \(\CH_i \coloneqq \chf{A_i}(\CHx)\) and \(\CH_j \coloneqq \chf{B_j}(\CHx)\) as usual. In other words $\phi$ sends operators supported on the block-diagonal of $\crse X$ to operators supported on the block-diagonal of $\crse Y$.

  If $\crse X$ has bounded geometry, it is then easy to directly show that such a $\phi$ is approximately controlled without relying on \cref{thm: uniformization}. Indeed,
  it follows from bounded geometry of \(\crse X\) that for any controlled entourage \(E \in \CE\) there are finitely many elements\footnote{\, The $n_i$ can be chosen to be \emph{normalizers} of $\ell^\infty(I, \CB(\CH_i))$ in  $\cpcstar \CHx$ (recall that \(n \in A\) \emph{normalizes} \(B \subseteq A\) if \(aBa^*, a^*Ba \subseteq B\)). These are induced by injective partial translation of the space.} \(n_1, \dots, n_k\in\cpcstar\CHx\) such that
  \[
    \cpop{E} \subseteq n_1 \ell^\infty(I, \CB(\CH_i)) + \dots + n_k \ell^\infty(I, \CB(\CH_i)).
  \] 
  Choosing some \(F_j \in \CF\) approximately containing the support of \(\phi(n_j)\) yields that each term \(n_j \ell^\infty(I, \CB(\CH_i))\) gets approximately mapped to \(\cpop{F_j}\) (cf.\ \cref{lemma:diagonal-ais-preserve-supp}). The controlled entourage \(F \coloneqq F_1 \cup \dots \cup F_k\) then witnesses the desired approximate control of \(\phi\).

  This simple technique seems, however, hard to utilize for coarsely locally finite $\crse X$ that are not of bounded geometry or for other Roe-like \cstar{}algebras. Knowing that $\cpcstar \variable$ is the multiplier algebra of $\roecstar\variable$, one could adapt this approach for homomorphisms of Roe \cstar{}algebras. However, the known proof of that $\cpcstar \variable= \CM(\roecstar\variable)$ relies on the rigidity theorem (cf.\ \cite{braga_gelfand_duality_2022} and \cref{cor: cpc is multiplier of roec}).
\end{remark}

The proof of \cref{thm: uniformization} is long and technical, and it is hence split in several sections.

\section{Uniformization phenomenon: the Baire strategy}
In this section we outline the key points of the proof of \cref{thm: uniformization}, postponing the most technical part to the next section.
We begin by proving the following intermediate result, necessary for the proof of \cref{thm: uniformization}.
In the quasi-local setting \cref{lem:spakula-willett easy} is essentially a corollary of \cite{spakula_rigidity_2013}*{Lemma~3.2}, while a proof in the approximable setting can be extracted from \cite{braga_farah_rig_2021}*{Lemma 4.9} or, more precisely, from \cite{braga_gelfand_duality_2022}*{Lemma 3.9}.
In all of those references the result is only stated for bounded geometry metric spaces, but with some work their proofs can be adapted to our more general setting as well.
We provide below a short proof for \cref{lem:spakula-willett easy} that works in both cases simultaneously. This proof relies on a first application of the Baire Category Theorem, and also serves as a blueprint for the strategy of proof of \cref{thm: uniformization}.
\begin{lemma}\label{lem:spakula-willett easy}
  Let $\paren{F_n}_{n\in \NN}$ be a cofinal nested set in $\CF$. Suppose there is $\varepsilon > 0$ and a family $\paren{t_n}_{n\in\NN}$ of orthogonal operators in $\CB(\CHy)_{\leq 1}$ so that $t_n$ is not $\varepsilon$\=/$F_n$\=/approximable (resp.\ $\varepsilon$\=/$F_n$\=/quasi-local). Then there is some $K \subseteq \NN$ such that $\sum_{k\in K}t_{k}$ is not approximable (resp.\ quasi-local).
\end{lemma}
\begin{proof}
  By \cref{lem: sums of orthogonal is cts}, the mapping
  \[
    \varphi \colon \CP(\NN) \to \CB\left(\CH\right), \;\; I \mapsto t_I\coloneqq \sum_{i \in I} t_i
  \]
  is continuous with respect to the \sot.
  We let $\CX_n$ be the preimage of the $\varepsilon/3$-$F_n$-approximable (resp.\ $\varepsilon/3$-$F_n$-quasi-local) operators.
  The sets $\paren{\CX_n}_{n \in \NN}$ form an increasing sequence, and are closed by continuity (cf.\ \cref{lemma:qloc-sot-closed}).

  Assume by contradiction that for every $I\subseteq \NN$ the operator $t_I$ is approximable (resp.\ quasi-local). This precisely means that $\CP(\NN) = \bigcup_{n\in\NN}\CX_n$.
  By the Baire property, there is some $n_0 \in \NN$ such that $\CX_{n_0}$ has non-empty interior. Let $I_0\subseteq \NN$ be an element in the interior of $\CX_{n_0}$: this means that there is some finite $J_0\subseteq \NN$ such that
  \begin{equation*}
    (I_0\cap J_0) + \CP(\NN\smallsetminus J_0) \subseteq \CX_{n_0}.
  \end{equation*}

  On the other hand, there is an $n \in \NN$ large enough such that for each of the finitely many $J \subseteq \CP(J_0)$ the operator $t_J$ is in $\CX_n$. We may assume $n= n_0$.
  Given any $I\subseteq \NN$ we may then write
  \begin{align*}
    t_I &= t_{I\cap J_0} + t_{(I\smallsetminus J_0)} \\
     & =t_{(I\smallsetminus I_0)\cap J_0} - t_{(I_0\smallsetminus I)\cap J_0} + (t_{I_0\cap J_0} + t_{I\smallsetminus J_0}) \in \CX_{n_0} +\CX_{n_0}+\CX_{n_0}.
  \end{align*}
  It follows that every $t_I$ is $\varepsilon$-$F_{n_0}$-approximable (resp.\ $\varepsilon$-$F_{n_0}$-quasi-local), against the hypothesis.
\end{proof}

Returning to \cref{thm: uniformization}, we need to show that if $\cpstar{\CHx}$ is contained in the union $\bigcup_{F\in\CF}\phiinv{\varepsilon}{F}$, then for every $E \in\CE$ and $\varepsilon>0$ there is some $F \in \CF$ such that $\cpcontr{E} \subseteq \phiinv{\varepsilon}{F}$. 
Observe that this is analogous to (and generalizes) the uniformization phenomenon that we just proved in \cref{lem:spakula-willett easy}.
It is hence natural to try to adapt the proof of \cref{lem:spakula-willett easy} in order to prove it.

There are two key points to the proof of \cref{lem:spakula-willett easy}. Firstly, one needs to work in a Baire space (in that case $\CP(\NN)$). Secondly, once an open subset $\Omega$ has been chosen (in that case $(I_0\cap J_0) + \CP(\NN\smallsetminus J_0)$), one needs to be able to decompose any operator as a finite sum of operators that belong to $\Omega$ or $\CX_{n_0}$ for some fixed $n_0$.

In the setting of \cref{thm: uniformization}, there are natural Baire spaces to consider, namely $\cpcontr{E}$ equipped with the \sot{}: we showed in \cref{prop:prop-e-baire} that ($\cpcontr{E}$,\sot) is a Baire space when $E$ is a block-entourage.
\footnote{\, On the contrary, $(\cpop E,\sot)$ is \emph{not} a Baire space.}
Finding the appropriate decompositions is more delicate, and will require making a judicious use of finite rank projections.
For later reference, it is important to further break down \cref{thm: uniformization} into a more fundamental `uniformization phenomenon'. To formulate it, it is convenient to introduce the following technical definition.

\begin{definition}\label{def: one-vector quasi-controlled mapping}
  With reference to \cref{def:quasi-controlled mapping}, a mapping $\phi\colon\CD\to \CB(\CHy)$ is \emph{one-vector approximately-controlled} if for every vector $v\in (\CHx)_{1}$ of bounded support, $E \in \CE$ and $\varepsilon > 0$ there is some $F \in \CF$ such that for every $t\in\cpop{E}$ with $tp_v\in\CD$, the image $\phi(tp_v)$ is $(\varepsilon\norm{t})$\=/$F$\=/approximable.

  We say that $\phi$ is \emph{one-vector quasi-controlled} if the above holds with `approximable' replaced by `quasi-local' everywhere.
\end{definition}

We need the following simple lemma.
\begin{lemma} \label{lem: uniformization finite rank}
  Let $\phi\colon\roestar\CHx\to \roeclike{\CHy}$ be additive, \Star{}preserving and one-vector approximately (resp.\ quasi)-controlled.
  Let also $p\in\CB(\CHx)$ be the projection onto a finite rank subspace of $\CHx$ generated by vectors of bounded support. 
  For every $E\in\CE$ and $\varepsilon>0$ there exists $F\in\CF$ such that for every $t\in\cpcontr E$ both $pt$ and $tp$ belong to $\phiinv{\varepsilon}{F}$.
\end{lemma}
\begin{proof}
  It is enough to prove the statement for $tp$, as the statement for $pt$ then follows by taking adjoints.

  Let $w_1,\ldots,w_n$ be vectors of bounded support that generate $p(\CHx)$. The Gramm--Schmidt orthogonalization procedure yields an orthonormal basis $v_1\ldots,v_k$ for $p(\CHx)$ consisting of vectors that still have bounded support.\footnote{
    \, In view of the effective version of \cref{thm: uniformization}, it is worthwhile noting that if the vectors $w_j$ are subordinate to a discrete partition $\paren{A_i}_{i\in I}$, then the same is true for the $v_j$. In particular, the projections $p_{v_j}$ all have propagation controlled by $\diag(A_i\mid i\in I)$.
  }
  We then have $p=p_{v_1}+\cdots +p_{v_k} \in \CB(\CHx)$.
  By the one-vector approximately (resp.\ quasi)-controlled condition, there are $F_1, \dots, F_k \in\CF$ such that $tp_{v_i}\in\phiinv{\varepsilon/k}{F_i}$ for every $t\in \cpcontr E$ and all $i = 1, \dots, k$.
  Taking $F\coloneqq F_1\cup\cdots\cup F_k$ immediately yields that
  \[
    tp=tp_{v_1}+\cdots +tp_{v_k} \in \phiinv{\varepsilon}{F}
  \]
  for every $t\in \cpcontr E$, as claimed.
\end{proof}

We may now state and prove the following.
\begin{theorem}[The uniformization phenomenon]\label{thm: uniformization phenomenon}
  Let $\phi\colon\roestar\CHx\to \roeclike{\CHy}$ be linear, \Star{}preserving, one-vector approximately (resp.\ quasi)-controlled and \sot-continuous.
  Suppose that 
  $\CHx$ is discrete, and
  $\crse Y$ is countably generated.
  Then $\phi$ is approximately (resp.\ quasi)-controlled.
\end{theorem}
\begin{proof}
  Arbitrarily fix a discrete partition $X=\bigsqcup_{i\in I}A_i$, some positive $\varepsilon>0$ and a block-entourage $E\in \CE$. Choose a nested cofinal sequence $\paren{F_n}_{n\in\NN}$ in $\CF$, and let $\CX_n \coloneqq \phiinv{\varepsilon}{F_n}\cap \cpcontr{E}$. Since $\phi$ is \sot-continuous, \cref{lemma:qloc-sot-closed} shows that $(\CX_n)_{n \in \NN}$ is a sequence of closed subsets of $\cpcontr{E}$.
  Since $(F_n)_{n \in \NN}$ is nested, so is $(\CX_n)_{n \in \NN}$. Moreover, $\cpcontr{E}=\bigcup_{n\in\NN}\CX_n$ by cofinality of $(F_n)_{n \in \NN}$.

  As $\cpcontr{E}$ is a Baire space (cf.\ \cref{prop:prop-e-baire}), there is some $n_0 \in \NN$ such that $\CX_{n_0}$ has non-empty interior.
  Explicitly, by \cref{lemma:sot-basis-top} this means that there is some $\delta>0$ small enough, a finite $V\in(\CHx)_{\leq 1}$ and some $t_0 \in\CX_{n_0}$ such that
  \begin{equation}\label{eq: thm: uniformization interior}
    \emptyset \neq \cpcontr{E} \cap \bigparen{t_0 + \sotnbhd{2\delta}{V}} \subseteq \CX_{n_0}.
  \end{equation}
  Furthermore, we may assume that $V$ is subordinate to the discrete partition $\paren{A_i}_{i\in I}$.

  By discreteness,
  for each $v\in V$ there is a finite $I_v\subseteq I$ such that $t_0(v)\approx_\delta \sum_{i\in I_v}w_{i,v}$, where $w_{i,v}\coloneqq\chf{A_i} t_0(v)$.
  Let $p$ denote the projection onto the finite dimensional vector subspace of $\CHx$ spanned by all the vectors $\paren{w_{i,v}}_{i\in I_v,v\in V}$ and $\paren{v}_{v \in V}$.
  Observe that $p$ satisfies the hypotheses of \cref{lem: uniformization finite rank} and it has support contained in $\diag(A_i\mid i\in I)$. Moreover,
  \(
    pt_0 p \in t_0 + \sotnbhd{\delta}{V}
  \).
  It then follows from \eqref{eq: thm: uniformization interior} that, in fact,
  \begin{equation}\label{eq: thm: uniformization interior 2}
    \cpcontr{E}\cap\bigparen{p t_0 p + \sotnbhd{\delta}{V}}\subseteq \CX_{n_0}.
  \end{equation}

  Let now $t\in \cpcontr{E}$ be arbitrary. By construction, the operators $pt$, $(1-p)t$, $tp$ and $t(1-p)$ also have $E$-controlled propagation (cf.\ \cref{lemma:diagonal-ais-preserve-supp}). We may write $t$ as the sum
  \[
    t= pt + (1-p)tp +(1-p)t(1-p).
  \]
  Crucially, \cref{lem: uniformization finite rank} implies that the first two terms belong to $\CX_{n}$ for some large enough $n$ (depending on $p$ but not on $t$). Choosing the largest of the two, we may as well assume that $n_0 = n$.
  On the other hand, since the image of $p$ contains $V$ it is clear that $(1-p)t(1-p)\in\sotnbhd{\delta}{V}\cap \cpcontr{E}$.
  Writing
  \[
    t= pt + (1-p)tp +\bigparen{(1-p)t(1-p) + p t_0 p} - p t_0 p,
  \]
  it follows from \eqref{eq: thm: uniformization interior 2} that $\cpcontr{E}\subseteq \CX_{n} +\CX_{n}+ \CX_{n}+ \CX_{n}\subseteq \phiinv{4\varepsilon}{F_{n}}$.
  The statement then follows by linearity.
\end{proof}

In light of \cref{thm: uniformization phenomenon}, in order to prove \cref{thm: uniformization}, it is sufficient to prove the one-vector approximable/quasi-control. 
We briefly mention that in order to do that we will use \cref{lem:spakula-willett easy}. Thus, \cref{thm: uniformization} uses a Baire-type argument \emph{twice}.
The first time it proves a uniformization theorem for \emph{orthogonal} operators, the second it extends it to operators that may fail to be orthogonal.

\section{The one-vector control}\label{ssec: uniformization rank-one}
In this section we complete the proof of \cref{thm: uniformization}. For this reason, we work under the following.
\begin{convention*}
  For the rest of this section, $\crse X$, $\crse Y$, $\CHx$, $\CHy$ and $\phi$ are fixed as in \cref{thm: uniformization}. All the statements below are proved under these assumptions (\emph{e.g.}\ regarding discreteness and countability of coarse geometric modules).
\end{convention*}

It remains to prove the following, which is surprisingly technical.
\begin{proposition} \label{prop:equi-controlled rank-1}
  $\phi$ is one-vector approximately (resp.\ quasi)-controlled. That is, for every fixed unit vector $v\in(\CHx)_{1}$ of bounded support, $E \in \CE$ and $\varepsilon > 0$ there is some $F \in \CF$ such that $tp_v \in \phiinv{\varepsilon}{F}$ for all $t \in \cpcontr{E}$.
\end{proposition}

We first need a few preliminary results.
The following is useful to cut operators to mutually orthogonal ones without losing control on the propagation.
\begin{lemma}\label{lemma:quasi-controlled:q-1-q}
  Let $E\in\CE$ be a block-entourage containing the diagonal, and let $\supp(t) \subseteq E$. Moreover, let $q \in \CB(\CHx)$ be the projection onto the closure of the image of $t$. Then both $q$ and $1-q$ have support contained in $E\circ\op E$.
\end{lemma}
\begin{proof}
  First observe that if $\supp(q) \subseteq E\circ\op E$ then $\supp(1-q) \subseteq \supp(1) \cup \supp(q) \subseteq  E\circ\op E$, so it suffices to prove that $\supp(q) \subseteq E\circ\op E$.

  The image of $t$ is spanned by the vectors $t(v)$ with $v$ supported in the block-diagonal. Since $\supp(tp_v)\subseteq \supp(t)$ for every such $v$, by \cref{lem:join of controlled projection}, it is enough to prove the lemma for $t$ of rank-1.
  We may write $t=\matrixunit{w}{v}$ for some $v,w\in\CHx\smallsetminus \{0\}$, in which case $q=p_w$ is a scalar multiple of $tt^*$ (see \cref{lem: matrixunit identities}). Then $q$ has propagation controlled by $E\circ\op E$.
\end{proof}

\begin{remark}
  \begin{enumerate}[label=(\roman*)]
    \item Observe that the bound $E\circ\op E$ in \cref{lemma:quasi-controlled:q-1-q} is sharp.
        Indeed, if $X = \{-1, 0, +1\}\subseteq \RR$, equipped with the usual metric, and $\CHx = \ell^2(X)$, we may consider $t = e_{-1,0} + e_{+1,0}$. This $t$ has propagation $1$, while the projection onto $\angles{\delta_{-1}+\delta_{+1}}$ has propagation $2$.
    \item The assumption that $E$ is a block-entourage in \cref{lemma:quasi-controlled:q-1-q} is mostly aesthetic. At the cost of further composing with certain gauges, a similar statement can be proved for any locally admissible coarse geometric module.
  \end{enumerate}
\end{remark}

The following is our replacement for \cite{braga_farah_rig_2021}*{Lemma 4.8}. The slightly awkward phrasing ``such that $\phi(qtp)$ is defined'' makes it more apparent that this result is effective, and one does not really need that $\phi$ is defined on the whole of $\cpstar\CHx$ (cf.\ \cref{rmk:effective uniformization}).
\begin{lemma} \label{lemma:quasi-controlled:qtp}
  Let $p, q \in \CB(\CHx)$ be finite rank projections. For every $\varepsilon > 0$ there is some $F \in \CF$ such that $qtp \in \phiinv{\varepsilon}{F}$ whenever $t \in \CB(\CHx)_{\leq 1}$ is such that $\phi(qtp)$ is defined.
\end{lemma}
\begin{proof}
  Consider the restriction of $\phi$ to the set $\CB_{q,p}$ of operators that can be written as $qtp$ for some $t\in \CB(\CHx)_{\leq 1}$ and belong to the domain of $\phi$ (\emph{i.e.}\ have controlled propagation).
  Note that $\CB_{q,p}$ is a convex subset of a finite dimensional vector space (since both $p$ and $q$ are of finite rank). Since $\phi$ is linear, so is its image $\phi(\CB_{p,q})$.
  Moreover, $\CB_{q,p}$ is compact. By \sot-continuity $\phi(\CB_{p,q})$ is also compact in the \sot{} and therefore---by finite dimensionality---also in the norm topology.
  It follows that there are finitely many operators $s_1,\ldots,s_k\in\phi(\CB_{p,q})$ such that $\phi(\CB_{p,q})$ is covered by the balls $\paren{s_i + B_{\varepsilon/2}}_{i = 1}^k$ of radius $\varepsilon/2$.
  For each $i = 1, \dots, k$ we may choose $F_i \in \CF$ such that $s_i$ is $\varepsilon/2$\=/$F_i$\=/approximable (resp.\ $\varepsilon/2$\=/$F_i$\=/quasi-local if $\phi$ lands into $\qlcstar{\CHy}$). Then the entourage $F \coloneqq F_1 \cup \dots \cup F_k \in \CF$ meets the requirements of the statement.
\end{proof}

\begin{remark}
  In our setting, if $\crse X$ is coarsely connected and $p$ and $q$ have controlled propagation, then $\phi(qtp)$ is defined for every $t\in\CB(\CHx)_{\leq 1}$ (this is using that $p$ and $q$ have finite rank and therefore have bounded support, \emph{e.g.}\ by \cref{lem: matrix unit controlled propagation}). Moreover, if we also assumed that $\phi$ is norm-continuous, then \cref{lemma:quasi-controlled:qtp} would also hold for $p,q\in\cpcstar{\CHx}$ (as opposed to $\cpstar{\CHx}$).
  To see this, it is enough to choose finite rank controlled propagation projections $p'$, $q'$ that approximate $p$ and $q$ up to $\varepsilon'\ll \varepsilon$ (this can be done by non-degeneracy because $p$ and $q$ have finite rank) and apply \cref{lemma:quasi-controlled:qtp} to $p'$ and $q'$ with constant $\varepsilon/2$: the norm continuity assumption shows that we can choose $\varepsilon'$ small enough so that $qtp \in \phiinv{\varepsilon}{F}$.
\end{remark}

\begin{remark}
  For uniform Roe algebras of locally finite coarse spaces (\emph{i.e.}\ for $X$ locally finite and $\CHx \coloneqq \ell^2(X)$), \cref{prop:equi-controlled rank-1} follows easily from \cref{lemma:quasi-controlled:qtp}. This fact (in the setting where $\phi$ lands in $\cpcstar{\CHy}$) is leveraged in \cite{braga_farah_rig_2021} and makes the proof of \cref{thm: uniformization} pleasantly shorter. The proof for modules of locally infinite rank is significantly harder, as can also be evinced from \cite{braga_gelfand_duality_2022}.
\end{remark}

\begin{lemma}\label{lem:e_vnvm non-uniformly-approximable}
  Given $E\in\CE$, $\varepsilon>0$ and $v\in(\CHx)_1$ of bounded support, suppose there is a family $(v_i)_{i\in I} \subseteq (\CHx)_{\leq 1}$ such that
  \begin{itemize}
    \item $\matrixunit{v_i}{v} \in \cpcontr E$ for every $i\in I$;
    \item for every $F\in\CF$ there is some $i\in I$ with $\matrixunit{v_i}{v}\notin \phiinv{\varepsilon}{F}$.
  \end{itemize}
  Then, for any given $\varepsilon' < \varepsilon^3$ and sequence $(\overline F_n)_{n\in\NN}$ in $\CF$, there is a sequence $\paren{i_n}_{n\in\NN}$ such that $\matrixunit{v_{i_{n+1}}}{v_{i_n}} \notin \phiinv{\varepsilon'}{\overline F_n}$.
\end{lemma}
\begin{proof}
  Let $t_i\coloneqq \phi(\matrixunit{v_i}{v})$. Since $\phi$ is a \Star{}homomorphism, 
  \[
    \phi(\matrixunit{v_{j}}{v_i})=\norm{v}^2\phi(\matrixunit{v_{j}}{v}\matrixunit{v}{v_i})
    =\phi(\matrixunit{v_{j}}{v}\matrixunit{v_i}{v}^*)=t_{j}t_i^*
  \]
  (see \cref{lem: matrixunit identities}). 
  Furthermore, observe that for all $i,j \in I$
  \begin{equation} \label{eq:lem:e_vnvm non-uniformly-approximable}
    t_j t_i^*t_i = \phi\left(\matrixunit{v_j}{v} \matrixunit{v}{v_i} \matrixunit{v_i}{v}\right) = \norm{v_i}^2 \phi\left(\matrixunit{v_j}{v} \matrixunit{v}{v}\right) = \norm{v_i}^2 t_j.
  \end{equation}
  As explained in \cref{rmk: on uniformization hypotheses}~\cref{rmk: on uniformization hypotheses-auto cts}, we also have $\norm{t_i}\leq\norm{\matrixunit{v_i}{v}}\leq 1$.
  
  \smallskip 

  We now do the case where $\phi$ lands in $\qlcstar{\CHy}$. Since $\norm{t_i}\leq\norm{\matrixunit{v_i}{v}}=\norm{v_i}$, and operators of norm less that $\varepsilon$ are obviously $\varepsilon$-$F$-quasi-local for every $F\in\CF$, we may without loss of generality assume that $\norm{v_i} \geq \varepsilon$ for all $i \in I$.
  For every $F\in\CF$ we fix an index $i(F)$ such that $t_{i(F)}$ is not $\varepsilon$\=/$F$\=/quasi-local. We also fix a $\delta > 0$ small enough (to be determined during the proof). The construction of the required sequence is performed inductively.

  Let $i_0\in I$ be arbitrary.
  Suppose that $i_n$ has already been chosen.  
  Then $t_{i_n}$ is $\delta$\=/$F_n'$-quasi-local for some $F_n' \in \CF$ large enough.
  Let 
  \[
    F_{n+1} \coloneqq \overline F_n\circ\gaugey \circ F_n',
  \]
  where $\gaugey$ is an admissibility gauge for $\CHy$. We claim that $i_{n+1} \coloneqq i(F_{n+1})$ meets the requirements of the statement.
  In fact, \eqref{eq:lem:e_vnvm non-uniformly-approximable} shows that
  \[
    t_{i_{n+1}} = \frac{t_{i_{n+1}} t_{i_n}^* t_{i_n}}{\norm{v_{i_n}}^2},
  \]
  and if $t_{i_{n+1}} t_{i_n}^*$ was $\varepsilon'$-$\overline F_n$-quasi-local it would follow from \cref{lem:supports and operations quasi-local} that $t_{i_{n+1}}$ is $(\varepsilon'+\delta)/\norm{v_{i_n}}^2$-$F_{n+1}$-quasi-local. If we had taken $\delta$ small enough, \emph{e.g.}\
  \[
    \frac{(\varepsilon'+\delta)}{\varepsilon^2}\leq\varepsilon,
  \]
  this would then lead to a contradiction.

  The case where $\phi$ lands in $\cpcstar\CHy$ is completely analogous. The index
  $i_0$ is chosen arbitrarily and 
  $F_n' \in \CF$ is taken such that $t_{i_n}$ is closer than $\delta$ to a contraction $s_n\in\CB(\CHy)$ of $F_n'$-controlled propagation (one can always approximate an approximable contraction by controlled contractions).
  We also let $F_{n+1} = \overline F_n\circ\gaugey \circ F_n'$, where this time $\gaugey$ is a non-degeneracy gauge. To show that $i_{n+1} \coloneqq i(F_{n+1})$ meets the requirements of the statement we observe that if there was some $s \in \CB(\CHy)_1$ with propagation controlled by $\overline F_n$ such that $t_{i_{n+1}}t_{i_n}^* \approx_{\varepsilon'} s$, then we would have
  \[
    t_{i_{n+1}} = \frac{t_{i_{n+1}} t_{i_n}^* t_{i_n}}{\norm{v_{i_n}}^2}
    \approx_{\delta/\norm{v_{i_n}}^2} \frac{t_{i_{n+1}} t_{i_n}^*s_n}{\norm{v_{i_n}}^2} 
    \approx_{\varepsilon'/\norm{v_{i_n}}^2} \frac{ss_n}{\norm{v_{i_n}}^2}.
  \]
  As the composition $ss_n$ has $F_{n+1}$-controlled propagation (cf.\ \cref{eq:support of composition}), this would lead to a contraction just as before.
\end{proof}

We may now prove \cref{prop:equi-controlled rank-1}.
\begin{proof}[Proof of \cref{prop:equi-controlled rank-1}]
  We argue by contradiction. Let $p \coloneqq p_v$ and assume that for each $F\in\CF$ there is some $t\in\cpcontr E$ so that $tp\notin\phiinv{\varepsilon}{F}$.
  Enlarging $E$ if necessary, we may further assume that $E$ is a block-entourage containing the diagonal (since $\CHx$ is discrete).
  Also fix a countable cofinal family $\paren{F_n}_{n \in \NN} \subseteq \CF$.

  Let $t_1 \in \cpcontr{E}$ be such that $t_1p \not\in \phiinv{\varepsilon}{F_1}$.
  Let $q_1 \in \CB(\CHx)_1$ be the orthogonal projection onto the image of $t_1 p$.
  In particular, $q_1$ is a rank-1 projection such that $t_1 p = q_1 t_1 p$.
  Observe that $p$ and $q_1$ are projecting onto vectors that have bounded support and are supported in the same coarse connected component of $\crse X$ (since $\supp(t_1) \subseteq E$). In particular, for every operator $t$ the composition $q_1tp$ has finite propagation and, hence, belongs to the domain of $\phi$. Applying \cref{lemma:quasi-controlled:qtp} yields some $F_1' \in \CF$ such that $q_1 \cpcontr{E} p \subseteq \phiinv{\varepsilon/2}{F_1'}$.

  Enlarging $F_2$ (and the successive $\paren{F_n}_{n \geq 2}$) if necessary, we may further assume that $F_1'\circ \op{(F_1')}\subseteq F_2$.
  By the assumption, there is some $t_2 \in \cpcontr{E}$ such that $t_2p \not\in \phiinv{\varepsilon}{F_2}$. Writing
  \[
    t_2 p = \left(1-q_1\right) t_2 p + q_1 t_2 p,
  \]
  and noting that $q_1 t_2 p \in \phiinv{\varepsilon/2}{F_1'} \subseteq \phiinv{\varepsilon/2}{F_2}$,
  we conclude that $(1-q_1) t_2 p \not\in \phiinv{\varepsilon/2}{F_2}$.
  Let $E'$ be a block-entourage containing the diagonal and so that $\supp(p)\subseteq E'$. It follows from \cref{lemma:quasi-controlled:q-1-q} that
  \[
    \supp(q_1) \subseteq (E\circ E') \circ \op{(E\circ E')}\eqqcolon E''.
  \]
  Since $E''$ is itself a block-entourage containing the diagonal, then
  \[
    \supp( (1-q_1) t_2 p )\subseteq E''\circ E\circ E' \eqqcolon \overline{E}.
  \]
  We now let $q_2$ be the orthogonal projection onto the subspace generated by the images of $t_1 p$ and $t_2 p$, and repeat the process. Note that $q_2$ (and hence also $1-q_2$) has propagation controlled by $E''$, as it is a join of projections of propagation controlled by $E''$ (see \cref{lem:join of controlled projection}).

  This iterative construction yields a sequence $\paren{q_n, t_n}_{n \in \NN}$ of finite rank projections $q_n$ and contractions $t_n \in \cpcontr{E}$ such that
  \begin{enumerate}[label=(\roman*)]
    \item $q_n \leq q_{n+1}$;
    \item $q_n t_n p = t_n p$;
    \item $(1-q_{n-1})t_n p\in \cpcontr{\overline{E}}$;
    \item $(1-q_{n-1})t_n p \not\in \phiinv{\varepsilon/2}{F_n}$ for all $n \in \NN$.
  \end{enumerate}
  Letting $v_n\coloneqq(1-q_{n-1})t_n p(v) = (1-q_{n-1})t_n (v)$, we may also rewrite
  \[
    (1-q_{n-1})t_n p = \matrixunit{v_n}{v}.
  \]
  We are now in position to apply \cref{lem:e_vnvm non-uniformly-approximable} with respect to $\paren{F_k}_{k\in \NN}$ and an arbitrarily chosen $\varepsilon'\in (0, (\varepsilon/2)^{3})$. This yields a subsequence $\paren{n_k}_{k\in \NN}$ such that
  \[
    s_k\coloneqq \matrixunit{v_{n_{k+1}}}{v_{n_k}} \not\in \phiinv{\varepsilon'}{F_k}
  \]
  for any $k \in \NN$.
  By construction, the vectors $\paren{v_n}_{n\in\NN}$ are mutually orthogonal and, hence, so are the operators $\paren{s_{k}}_{k \in \NN}$.
  Applying \cref{lem:spakula-willett easy}, we may find some $K \subseteq \braces{n_k}_{k \in \NN}$ such that the $\sot$-sum $\sum_{k\in K}\phi(s_{k})$ is not approximable (resp.\ quasi-local).
  However, $s \coloneqq \sum_{k \in K}s_{k}$ is a bounded operator and---since $\supp(s_k)\subseteq \overline{E}\circ\op{\overline{E}}$ for every $k\in\NN$---it follows from \cref{lemma:supports-operators-sum} that $\supp(s)\subseteq \overline{E}\circ\op{\overline{E}}$ as well.
  By strong continuity of $\phi$, it then follows that
  \[
    \sum_{k\in K}\phi(s_{k}) = \phi\paren{s} \in \cpcstar{\CHy}, \;\;\; \left(\text{resp.\ }\qlcstar{\CHy}\right),
  \]
  which is a contradiction.
\end{proof}

\begin{proof}[Proof of \cref{thm: uniformization}]
  Apply \cref{prop:equi-controlled rank-1} and \cref{thm: uniformization phenomenon}.
\end{proof}

We end the section with the following remark regarding \cref{thm: uniformization}.
\begin{remark} \label{rmk:effective uniformization}
  Observe that, even though we were interested in showing that for a given $E\in\CE$ and $\varepsilon>0$ there is an $F\in\CF$ with $\cpcontr E\subseteq\phiinv{\varepsilon}{F}$, during the proof that such an $F$ exists we have been forced to consider some larger entourages as well (and hence used the assumption that $\phi$ is defined on $\cpcontr{E'}$ for other $E'\in\CE$). This is why \cref{thm: uniformization} is stated for $\phi$ defined on the whole $\cpstar\CHx$.
  However, we remark that the required enlargement can always be explicitly controlled. That is, going carefully through the proof we see that it implies the following:
  \begin{center}
    \begin{minipage}{0.9\textwidth}
        \textit{\noindent
        Let $E\in\CE$ be a block-entourage containing the diagonal. There is some $\overline{E} \in \CE$ containing $E$ such that,
        if $\phi\colon\cpcontr{\overline{E}}\to\roeclike\CHy$ is a map satisfying all the other hypotheses of \cref{thm: uniformization}, then for every $\varepsilon>0$ there is some $F\in\CF$ such that $\cpcontr E\subseteq \phiinv{\varepsilon}{F}$.
        }
    \end{minipage}
  \end{center}
  The most tedious proof to make effective is that of \cref{prop:equi-controlled rank-1}, as one needs to control the propagation of the projections $\braces{q_n}_{n \in\NN}$. Matters are greatly simplified by noting that \cref{prop:equi-controlled rank-1} is only applied to projections $p$ that have support contained in $\diag(A_i\mid i\in I)$, and therefore do not increase the propagation.
\end{remark}

\chapter{Rigidity for weakly quasi-controlled operators} \label{sec: rigidity spatially implemented}
The goal of this chapter is to provide a rigidity framework for mappings between Roe-like \cstar{}algebras that are spatially implemented via \emph{weakly quasi-controlled} operators (see \cref{def: quasi-controlled operator} below).
Together with the uniformization phenomena proved in \cref{sec: uniformization}, this will be used in \cref{sec: rigidity phenomena} to prove rigidity of Roe-like \cstar{}algebras in full generality.
The results in this section are in many aspects orthogonal to those in \cref{sec: uniformization}, and these two sections can be read independently from one another.

As explained when introducing \cref{def:controlled mapping}, an operator $T\colon \CHx\to\CHy$ is controlled if and only if $\Ad(T)$ preserves equi controlled propagation.
\cref{def:quasi-controlled mapping} gives useful weakenings of this latter condition, which we now use in the following.

\begin{definition}\label{def: quasi-controlled operator}
  We say $T\colon\CHx\to\CHy$ is \emph{weakly quasi-controlled} if $\Ad(T)\colon$ $\CB(\CHx)\to\CB(\CHy)$ is quasi-controlled. Similarly, $T$ is \emph{weakly approximately-controlled} if $\Ad(T)$ is approximately-controlled.
\end{definition}

\begin{remark}
  The naming ``weakly'' is to help differentiating them from their ``stronger'' counterparts, defined in \cref{def: strong approx-quasi R-ctrl,def: strong approx-quasi ctrl}.
\end{remark}

In the following section we show that if $T\colon\CHx\to\CHy$ is weakly quasi-controlled then certain natural approximating relations define partial coarse maps $\cappmap[T]{\delta}{F}{E} \colon \crse X \to \crse Y$. These approximations play a crucial role in bridging between analytic properties of operators and geometric properties of the spaces. This interplay will culminate in the announced rigidity result (cf.\ \cref{thm: rigidity quasi-proper operators}).

\section{Construction of controlled approximations}\label{subsec: approximations}
As usual, let $\CHx$ and $\CHy$ be coarse geometric modules for the coarse spaces $\crse X$ and $\crse Y$ respectively.
\begin{definition} \label{def:approximating crse map}
  Let $T \colon {\CHx} \to {\CHy}$ be a bounded operator. Given $\delta \geq 0, E \in \CE$ and $F \in \CF$, we define an \emph{approximating relation} as
  \[
    \appmap[T]{\delta}{F}{E} \coloneqq 
    \bigcup\left\{B\times A 
    \text{ meas.\ $(F\otimes E)$-bounded}
    \;\middle|\; 
    \norm{\chf{B}T\chf{A}} > \delta \right\}.
  \]
\end{definition}

\begin{remark}
  Informally, $\appmap[T]{\delta}{F}{E}$ is the relation from $X$ to $Y$ given by the pairs of points $(y, x)$ such that $T$ sends a norm-one vector supported on a $E$-bounded neighborhood of $x$ to a vector whose restriction to an $F$-bounded neighborhood of $y$ has norm at least $\delta$.
  That is, $\appmap[T]{\delta}{F}{E}$ tries to approximate on the level of the spaces what $T$ does on the level of the modules by keeping track of those vectors that are somewhat concentrated on equi bounded subsets.
\end{remark}

In general, the approximating relations in \cref{def:approximating crse map} depend greatly on the choice of parameters. One important example to keep in mind is the following (cf.\ \cite{braga_farah_rig_2021}*{Example 6.3}).
\begin{example}\label{ex:expanders poor approximation}
  Let $\crse X$ be $(\NN,\{\Delta_{\NN}\})$ (\emph{i.e.}\ a coarsely disjoint union of countably many points) and let $\crse Y$ be $(\bigsqcup_{n\in \NN}\CG_n,\CE_d)$ be a disjoint union of a family of transitive expander graphs of increasing cardinality, where $d$ is the extended metric obtained by using the graph metric on each $\CG_n$ and setting different graphs to be at distance $+\infty$. For instance, we may take $\CG_n$ to be Cayley graphs of a residual sequence of finite quotients of a group with Kazhdan's property (T).
  
  Let $\CHx$ and $\CHy$ be the usual $\ell^2$-spaces, and let $T \colon \ell^2(\NN)\to \ell^2(\bigsqcup_{n\in\NN}\CG_n)$ be the operator sending $\delta_n$ to the function that is constantly equal to $\abs{\CG_n}^{-1/2}$ on $\CG_n$ and zero elsewhere. This $T$ is a rather nice isometry: the assumption that $\CG_n$ are expanders implies that $\Ad(T)$ is approximately controlled (cf.\ \cref{def:quasi-controlled mapping}) and defines a \Star{}embedding of $\cpcstar{\CHx}$ into $\cpcstar{\CHy}$.

  Let now $E \coloneqq \Delta_X$, and let $F$ be the entourage consisting of points of $Y$ at distance at most $r$ from one another. It follows from transitivity of $\CG_n$ that the approximation $\appmap[T]{\delta}{F}{E}$ in \cref{def:approximating crse map} will be of the form
  \[
    \appmap[T]{\delta}{F}{E} = \bigsqcup_{n=1}^N \CG_n\times \{n\}
  \]
  for some $N \in \NN$. Crucially, the value of $N$ depends on $r$ and $\delta$ (and it is $+\infty$ for $\delta=0$).
\end{example}

Observe that if $\delta =0$ and $\gaugex$, $\gaugey$ are non-degeneracy gauges then $\appmap[T]{0}{\gaugey}{\gaugex}$ coincides with the representative for the coarse support $\csupp(T)$ given in \cref{eq: explicit coarse support}.
In particular, $T$ is controlled if and only if $\appmap[T]{\delta}{F}{E}$ is controlled for every $\delta\geq 0$.
On the other hand, the following shows that weak quasi-control is already enough to ensure that all the approximations with $\delta>0$ are controlled.

\begin{lemma}\label{lem:approximation is controlled}
  If $T\colon\CHx\to\CHy$ is a weakly quasi-controlled operator, then the approximations $\appmap[T]{\delta}{F}{E}$ are controlled for every choice of $E$, $F$, and $\delta>0$.
\end{lemma}
\begin{proof}
  According to \cref{def:controlled_function}, we need to verify that for every $\bar E\in \CE$ there is some $\bar F\in\CF$ such that if $B'\times A', B\times A\subseteq Y\times X$ is a pair of measurable $(F\otimes E)$-bounded products 
  that moreover satisfy $ A'\cap\bar E(A)\neq\emptyset$, then $B'\times B\subseteq \bar F$. 
  
  Let then $A',A,B',B$ as above be fixed. By definition, this means that both $\norm{\chf{B}T\chf{A}}$ and $\norm{\chf{B'}T\chf{A'}}$ are larger than $\delta$. We may thus find unit vectors $v\in\CH_A$ and $v'\in\CH_{A'}$ whose images $w\coloneqq T(v)$ and $w'\coloneqq T(v')$ satisfy
  \[
    \norm{\chf{B}(w)}> \delta
    \;\; \text{and} \;\;
    \norm{\chf{B'}(w')}> \delta.
  \]
  Note that, in particular, $T$ cannot be $0$.

  Observe that the rank-1 operator $\matrixunit{v'}{v}$ has propagation controlled by $E\circ\bar E\circ E$ (cf.\ \cref{lemma:supp-vector-a-times-a}). By the assumption that $\Ad(T)$ is quasi-controlled, we deduce that for every $\varepsilon>0$ there is some $F_\varepsilon\in\CF$ depending only on $\varepsilon$, $\bar E$ (and $E$) such that $\matrixunit{w'}{w} = \Ad(T)(\matrixunit{v'}{v})$ is $\varepsilon$-$F_\varepsilon$-quasi-local. 
  By \cref{lem: matrixunit identities},
  \[
    \norm{\chf{B'}\matrixunit{w'}{w}\chf{B}} 
    = \norm{\chf{B'}(w')}\norm{\chf{B}(w)}
    > \delta^2.
  \]
  Letting $\varepsilon \coloneqq \delta^2 > 0$, we deduce that $B'$ and $B$ cannot be $F_\varepsilon$-separated. This implies that $B'\times B\subseteq F\circ F_\varepsilon\circ F$, so letting $\bar F\coloneqq F\circ F_\varepsilon\circ F$ concludes the proof. 
\end{proof}

\begin{corollary}\label{cor:approximation quasi-controlled}
  If $T$ is weakly quasi-controlled, then $\appmap[T]{\delta}{F}{E}$ defines a partial coarse map $\cappmap[T]{\delta}{F}{E} \colon \crse{X} \to \crse{Y}$ for every $E$, $F$, and $\delta>0$.
  Moreover, $\cappmap[T]{\delta}{F}{E}\crse \subseteq\csupp(T)$.
\end{corollary}
\begin{proof}
  The first part is immediate from \cref{lem:approximation is controlled}.
  The ``moreover'' part follows from \cref{eq: explicit coarse support} observing that if $\gaugey,\gaugex$ are non-degeneracy gauges containing $F$ and $E$ respectively then $\appmap[T]{\delta}{F}{E} \subseteq \appmap[T]{0}{\gaugey}{\gaugex}$.
\end{proof}

\begin{remark} \label{rem:app-maps not functorial}
  Albeit natural, associating approximating partial maps with linear operators is \emph{not} a functorial operation. Namely, if we are given $T_1\colon \CHx\to \CHy$ and $T_2\colon\CHy\to \CHz$ the composition $\appmap[T_2]{\delta_2}{D}{F}\circ\appmap[T_1]{\delta_1}{F}{E}$ need not have anything to do with $\appmap[T_2T_1]{\delta'}{D'}{E'}$ for any choice of the parameters $\delta' > 0$, $D' \in \CD$ and $E' \in \CE$.
  On one hand, the approximation $\appmap[T_2T_1]{\delta'}{D'}{E'}$ can be too \emph{large}, because it completely ignores any constraint coming from the middle space $\crse Y$. On the other hand, it can also be too \emph{small}, as it can happen that both $\norm{\chf{B}T_1\chf{A}}$ and $\norm{\chf{C}T_2\chf{B}}$ are large while $\norm{\chf{C}T_2T_1\chf{A}}=0$.
\end{remark}

Since $\norm{\chf BT\chf A}=\norm{\chf A T^*\chf B}$, we observe that the transposition $\op{\paren{\appmap[T]\delta FE}}$ coincides with the approximation of the adjoint $\appmap[T^*]\delta EF$. In particular, if $T^*$ is also weakly quasi-controlled then $\cappmap[T^*]\delta EF$ is a partially defined coarse map as well.
By definition, this means that $\cappmap[T]\delta FE$ is a partial coarse embedding.
The following is now a formal consequence of the properties of partial coarse maps (cf.\ \cref{prop:transpose is coarse inverse}).

\begin{proposition}\label{prop:adjoint is coarse inverse}
  Let $T\colon\CHx\to\CHy$ be an operator such that both $T$ and $T^*$ are weakly quasi-controlled. Arbitrarily fix $E\in\CE$, $F\in\CF$ and $\delta>0$.
  Then the compositions $\cappmap[T^*]\delta EF\circ \cappmap[T]\delta FE$ and $\cappmap[T]\delta FE\circ\cappmap[T^*]\delta EF$ are well-defined and are contained in $\cid_{\crse X}$ and $\cid_{\crse Y}$ respectively.
  Moreover, $\cappmap[T]\delta FE$ is coarsely everywhere defined if and only if $\cappmap[T^*]\delta EF$ is coarsely surjective (and vice versa).

  Lastly, the following are equivalent:
  \begin{enumerate}[label=(\roman*)]
    \item\label{item:prop:adjoint is coarse inverse-ce} $\cappmap[T]\delta FE$ is a coarse equivalence;
    \item\label{item:prop:adjoint is coarse inverse-*ce} $\cappmap[T^*]\delta EF$ is a coarse equivalence;
    \item\label{item:prop:adjoint is coarse inverse-cinv} $\cappmap[T]\delta FE$ and $\cappmap[T^*]\delta EF$ are coarse inverses of one another;
    \item\label{item:prop:adjoint is coarse inverse-csur} $\cappmap[T]\delta FE$ and $\cappmap[T^*]\delta EF$ are coarsely surjective;
    \item\label{item:prop:adjoint is coarse inverse-cdef} $\cappmap[T]\delta FE$ and $\cappmap[T^*]\delta EF$ are coarsely everywhere defined.
  \end{enumerate}  
\end{proposition}

\section{Intermezzo: a Concentration Inequality} \label{subsec:concen-ineq}
In this section we prove an analytic inequality that plays a crucial role in our main rigidity result, as it will let us describe the coarse image of the approximating maps $\cappmap[T]{\delta}{F}{E}$.

Informally, this inequality will be used as follows: suppose that a family $(w_i)_{i\in I}$ of orthonormal vectors in $\CHy$ is such that the projections onto the span of $(w_i)_{i\in J}$, where $J\subseteq I$, are uniformly quasi-local. For instance, such vectors may arise as images under a weakly quasilocal operator $T$ of an orthonormal family of zero-propagation vectors $v_i$. Further suppose that the span of the $w_i$ contains a norm one vector supported on a bounded set $B$. Then the Concentration Inequality will imply that there is some $i\in I$ such that ``a sizeable proportion'' of $w_i$ is concentrated on a controlled neighborhood of $B$. If $w_i=T(v_i)$ as above, this means that a suitable approximation of $T$ will map the support of $v_i$ to $B$.
This is how the Concentration Inequality is used in \cite{rigidIHES}, and it is akin to how the Shapley-Folkman Theorem is used in \cite{braga_rigid_unif_roe_2022}*{Theorem~1.2}.

The statement we actually prove (cf.\ \cref{prop: concentration-ineq}) is more refined (and technical), as this greater level of generality will be useful later. We first record two simple facts.

\begin{lemma}\label{lem: norm-realising is eigenvector}
  Let $T \colon \CH\to \CH'$ be a bounded operator, $c\coloneqq\norm{T}$ and $v_n\in\CH$ vectors of norm $1$ such that $\norm{T(v_n)}\to c$. Then $\norm{T^*T(v_n) - c^2v_n}\to 0$.
\end{lemma}
\begin{proof}
  Observe that $\norm{T^*T(v_n)}\to c^2$ (one inequality is clear, the other one is obtained by looking at $\scal{T^*T(v_n)}{v_n}$). We then have 
  \begin{align*} 
    \norm{T^*T(v_n) - c^2v_n}^2 & = \scal{T^*T(v_n) - c^2v_n}{T^*T(v_n) - c^2v_n} \\
    & = \norm{T^*T(v_n)}^2 + c^4 - 2c^2 \, \text{Re}\left(\scal{T^*T(v_n)}{v_n}\right) \\
    & = \norm{T^*T(v_n)}^2 + c^4 - 2c^2 \, \norm{T(v_n)}^2 \to 0. \qedhere
  \end{align*}
\end{proof}

The second result is a simple consequence of the parallelogram law.
By induction, the parallelogram law on Hilbert spaces implies that for any choice of $n$ vectors $v_1,\ldots,v_n\in \CH$ the average of the norms $\norm{\sum_{i=1}^n\varepsilon_i v_i}$ over the $2^n$-choices of sign $\varepsilon_i \in \{\pm 1\}$ is equal to $\sum_{i=1}^n{\norm{v_i}}$.
We then observe the following.
\begin{lemma} \label{parallelogram-ineq}
  Let $\paren{v_i}_{i \in I} \subseteq \CH$ be a family of vectors. Then
  \[ 
    \sup_{\varepsilon_i = \pm 1} \norm{\sum_{i \in I} \varepsilon_i v_i}^2 \geq \sum_{i \in I} \norm{v_i}^2, 
  \]
  where the $\sup$ is taken among all possible choices of $\varepsilon_i = \pm 1$ and it is $\infty$ if these sum do not converge unconditionally to vectors in $\CH$.
\end{lemma}
\begin{proof}
  If $I$ is finite this is just saying that the supremum is at least as large as the average.
  Let then $I$ be infinite. If some series $\sum_{i \in I} \varepsilon_i v_i$ is not unconditionally convergent, there is nothing to prove. If that is not the case (as will happen when we actually apply this lemma), then there must exist for any $\delta>0$ a finite subset $J\subseteq I$ such that $\norm{\sum_{i \in I\smallsetminus J} \varepsilon_i v_i}< \delta$ regardless of the choice of $\varepsilon_i$. Therefore
  \begin{align*}
    \sup_{\varepsilon_i = \pm 1} \norm{\sum_{i \in I} \varepsilon_i v_i}^2
    &\geq \left(\sup_{\varepsilon_i = \pm 1}\norm{\sum_{i \in J} \varepsilon_i v_i} 
        - \norm{\sum_{i \in I\smallsetminus J} \varepsilon_i v_i} \right)^2\\
    &= \sup_{\varepsilon_i = \pm 1}\norm{\sum_{i \in J} \varepsilon_i v_i}^2 + O(\delta) \geq \sum_{i \in J} \norm{v_i}^2 + O(\delta) = \sum_{i \in I} \norm{v_i}^2  + O(\delta),
  \end{align*}
  where we used finiteness of $J$ in the second inequality. Letting $\delta\to 0$ completes the proof.
\end{proof}

We may now state and prove the following key result. When first parsing the statement, the reader may assume that $\eta$ and $\kappa$ are equal to $1$ and $\delta>0$ is very small.

\begin{proposition}[Concentration Inequality] \label{prop: concentration-ineq}
  Let $\CHx$ be discrete with discreteness gauge \gauge, and $T\colon\CHx\to \CHy$ an operator $\eta$-bounded from below for some $\eta>0$. Fix $\delta,\kappa>0$, $F\in\CF$ and $B,C \subseteq Y$ measurable with $F(B)\subseteq C$ and $\norm{\chf BT}\geq \kappa$.
  Suppose that $\norm{\chf{C} T\chf A} \leq \delta$ for every measurable $\gauge$-controlled $A\subseteq X$. Then for every $\varepsilon > 0$ such that
  \[
    \varepsilon< \frac{\kappa^2 \, (\eta^2-\delta^2)^{\frac 12}}{2 \, \norm{T}}
  \]
  there is an $\gauge$-controlled projection whose image under $\Ad(T)$ is \emph{not} $\varepsilon$-$F$-quasi-local.
\end{proposition}
\begin{proof}
  Let $X=\bigsqcup_{i\in I} A_i$ be a discrete \gauge-controlled partition of $\crse{X}$. In particular, we have that $\norm{\chf{C}T\chf{A_i}}\leq\delta$ for every $i\in I$. For every $J\subseteq I$, the SOT-sum $\chf{J} \coloneqq \sum_{i\in J} \chf{A_i}$ exists and it is $\gauge$-controlled (see \cref{lemma:supports-operators-sum}). Let $q_i=T\chf{A_i}T^*$ and $q_J\coloneqq \Ad(T)(\chf{J})=\sum_{i\in J}q_i$ (these positive operators need not be projections, nor do they need to be orthogonal).
  Fix $\varepsilon$ as in the statement. We will show that there is some $J \subseteq I$ such that $q_J$ is not $\varepsilon$-$F$-quasi-local.

  Fix any sequence $\paren{v_n}_{n \in \NN} \subseteq \CHx$ of unit vectors such that 
  \[
    \lim_{n \to \infty} \norm{\chf{B}T (v_n)} =\norm{\chf{B}T}\geq \kappa
  \]
  and let $w_n\coloneqq T(v_n)$.
  For every $i \in I$, let $v_{n,i} \coloneqq \chf{A_i} (v_n)$ and $w_{n,i}\coloneqq T(v_{n,i})$. 
  Note that, by orthogonality of $\paren{\chf{A_i}}_{i \in I}$, we have that $1 = \norm{v_n}^2 = \sum_{i\in I} \norm{v_{n,i}}^2$. By the assumptions that $T$ is bounded below and the inequality $\norm{\chf{C}T\chf{A_i}}\leq\delta$, we deduce that
  \begin{align}\label{eq:meat in the complement}
    \begin{split}
      \norm{\chfcY C (w_{n,i})}^2 
      & = \norm{T(v_{n,i})}^2 - \norm{\chf{C}T(v_{n,i})}^2 
       \geq \left(\eta^2 - \delta^2\right) \, \norm{v_{n,i}}^2. 
    \end{split}
  \end{align}

  For each $n \in \NN$ we may combine \cref{parallelogram-ineq,eq:meat in the complement} on the family of vectors $\paren{\chfcY{C} (w_{n,i})}_{i \in I}$, to deduce that, for every $n \in \NN$,
  \begin{align*}
    \sup_{\varepsilon_i = \pm 1} \norm{\sum_{i\in I} \varepsilon_i \chfcY C (w_{n,i}) }^2 & \geq \sum_{i \in I} \norm{\chfcY{C} (w_{n,i})}^2  \\
    & \geq \sum_{i \in I} \left(\eta^2 - \delta^2\right) \, \norm{v_{n,i}}^2  \\
    & = \eta^2 - \delta^2.
  \end{align*}
  In particular, for each $n \in \NN$ we may fix a choice $\paren{\varepsilon_{n,i}}_{i \in I} \subseteq \{\pm 1\}$ so that 
  \[
    \inf_{n\in\NN} \, \norm{\sum_{i \in I} \varepsilon_{n,i} \chfcY{C} (w_{n,i})}^2 \geq \eta^2-\delta^2. 
  \]
  Let $J_n \coloneqq \{i \in I \mid \varepsilon_{n,i} = 1 \} \subseteq I$ and observe that
  \begin{align*}
    \norm{\sum_{i \in I} \varepsilon_{n,i} \chfcY C (w_{n,i})} 
    & \leq \norm{\sum_{i \in J_n} \chfcY C (w_{n,i})} + \norm{\sum_{i \in I\smallsetminus J_n} \chfcY C (w_{n,i})}.
  \end{align*}
  Thus, by substituting $J_n$ with $I\smallsetminus J_n$ if necessary, we may further assume that 
  \begin{equation}\label{eq:sum over subset}
    \inf_{n\in\NN} \, \norm{\sum_{i \in J_n} \chfcY C (w_{n,i})} \geq \frac{\left(\eta^2-\delta^2\right)^{1/2}}2.
  \end{equation}

  We claim that for $n \in \NN$ large enough $q_{J_n}$ is not $\varepsilon$-$F$-quasi-local. 
  We seek for a lower bound on $\norm{\chfcY{C} q_{J_n}\chf B}$ by testing it on the vectors $w_n$. Observe that
  \[
    q_{J_n}\chf B(w_n) = Tp_{J_n}(T^*\chf B)(\chf B T) (v_n).
  \]
  Since $\norm{\chf{B}T(v_n)}$ converges to $\norm{\chf{B}T}$ by assumption, applying \cref{lem: norm-realising is eigenvector} we deduce that 
  \[
    q_{J_n}\chf B(w_n) \approx \norm{\chf B T}^2Tp_{J_n}(v_n)
  \]
  for $n$ large enough.\footnote{\, Here, by $\approx$ we mean that $\norm{q_{J_n}\chf B(w_n) - \norm{\chf B T}^2Tp_{J_n}(v_n)} \to 0$ when $n \to \infty$.}
  Also note that $Tp_{J_n}(v_n)=\sum_{i\in J_n}T(v_{n,i})=\sum_{i\in J_n}w_{n,i}$.
  This implies that
  \begin{align*} 
    \norm{\chfcY C q_{J_n} \chf{B} \Bigparen{\frac{w_n}{\norm{w_n}} }}
    &\approx \frac{\norm{\chf B T}^2}{\norm{w_n}}\norm{\chfcY C Tp_{J_n} (v_n)} \\
    &\geq \frac{\kappa^2}{\norm{T}}\norm{\sum_{i\in J_n}\chfcY C (w_{n,i})}.
  \end{align*}
  When combined with \cref{eq:sum over subset}, we deduce that 
  \[ 
    \liminf_{n\to\infty} \, \norm{\chfcY{C} q_{J_n} \chf{B}} \geq \frac{\kappa^2}{\norm{T}} \cdot \frac{\left(\eta^2 - \delta^2\right)^{\frac 12}}{2} > \varepsilon.
  \]
  Since $F(B)\subseteq C$, this shows that for some large enough $n$ the operator $q_{J_n}$ is not $\varepsilon$-$F$-quasi-local, as desired.
\end{proof}

\begin{remark}
  In the sequel we will only apply \cref{prop: concentration-ineq} to isometries (and therefore $1$-bounded from below, \emph{i.e.}\ $\eta =1$). However, we find it useful to prove it for operators $\eta$-bounded from below, as this highlights the point where this assumption is needed. Besides, this extra generality is useful if one is interested in rigidity results for mappings of Roe-like \cstar{}algebras that fall short of being \Star{}homomorphisms (\emph{e.g.}\ certain almost homomorphisms). These ideas will, however, not be pursued further here.

  On the contrary, the freedom to choose $\kappa>0$ is directly important for the present memoir, as it is used in the proof of rigidity up to stable isomorphism (cf.\ \cref{thm: stable rigidity}).
\end{remark}

\section{Estimating coarse images of approximating maps} \label{subsec:maps-from-rk-1-to-rk-1}
The key missing ingredient in the proof of rigidity for mappings implemented by weakly quasi-controlled Lipschitz isomorphisms (cf.\ \cref{thm: rigidity quasi-proper operators} below) is an unconditional estimate for the image of the approximating relations defined in \cref{subsec: approximations}. This will be provided by the Concentration Inequality (cf.\ \cref{prop: concentration-ineq}), but we first need to introduce some notation.

Given a subspace $V$ of $\CHx$, we need to consider the subset of $X$ obtained as the union of subsets that ``contain a sizeable proportion of the support'' of a vector in $V$. Specifically, given $E\in\CE$ and $\kappa>0$ we let
\[ 
  \concvec{\kappa}{E}{V} \coloneqq 
  \bigcup \; \braces{A\subseteq X\text{ meas.\ $E$-bounded} \mid \norm{\chf{A}|_{V}} \geq \kappa}.
\]
Informally, $\concvec{\kappa}{E}{V}$ consists of the points $x\in X$ so that there is a vector in $V$ that is at least ``$\kappa$-concentrated'' in an $E$-neighborhood of $x$.

For any bounded $T \colon \CHx \to \CHy$ we may consider the image $T(\CHx)\subseteq \CHy$, and the associated subset $\concvec{\kappa}{F}{T(\CH_X)} \subseteq Y$.
Observe that this subset does \emph{not} depend in any way on the $\crse X$-module structure of $\CHx$.
The following simple consequence of \cref{prop: concentration-ineq} can then be seen as an unconditional estimate for the image of approximating maps introduced in \cref{def:approximating crse map}.
\begin{proposition}\label{prop: unconditional image estimate}
  Let $\CHx$ be  discrete with discreteness gauge $\gauge$, and let $\CHy$ be locally admissible. Fix $F_0\in\CF$, and $\kappa,\eta>0$.

  Then, for any $T\colon\CHx\to\CHy$ weakly quasi-controlled and $\eta$-bounded from below, and any $0<\delta<\eta$, there exists an $F\in\CF$ such that
  \[
    \concvec{\kappa}{F_0}{T(\CH_X)}\subseteq \image \Bigparen{\appmap[T]{\delta}{F}{\gauge}}.
  \]
\end{proposition}
\begin{proof}
  Arbitrarily fix some $\varepsilon > 0$ such that
  \[
    \varepsilon < \frac{(\eta\kappa)^2(\eta^2-\delta^2)^{\frac 12}}{2\norm{T}}. 
  \]
  Since $T$ is weakly quasi-controlled, there is some $F_1\in\CF$ such that $\Ad(T)$ sends $\gauge$-controlled contractions to $\varepsilon$-$F_1$-quasi-local operators. Enlarging it if necessary, we may assume that $F_1$ contains the diagonal of $\crse{Y}$.

  Let $\gaugey$ be a local admissibility gauge for $\CHy$.
  Fix an $F_0$-bounded measurable $B\subseteq Y$ with $\norm{\chf B|_{T(\CHx)}}\geq \kappa$.
  If $(w_n)_{n\in\NN}\subseteq T(\CHx)$ is a sequence of norm one vectors with $\norm{\chf{B}(w_n)}\to\kappa$, and we let $v_n\in\CHx$ be the preimage of $w_n$, then $\norm{v_n}\leq\eta^{-1}$ (since $T$ is $\eta$-bounded below). Therefore 
  \begin{equation}\label{eq:norm chf_BT}
    \norm{\chf BT}\geq \norm{\chf BT(v_n/\norm{v_n})}\geq \eta\norm{\chf BT(v_n)} \to \eta\kappa.
  \end{equation}
  We may also fix a measurable $C\subseteq Y$ such that $F_1(B)\subseteq C\subseteq \gaugey\circ F_1(B)$.
  By construction, we know that $T$ sends $\gauge$-controlled operators to $\varepsilon$-$F_1$-quasi-local operators. Since \cref{eq:norm chf_BT} holds, we apply the contrapositive of \cref{prop: concentration-ineq} with $\kappa' \coloneqq \eta\kappa$ to deduce that there is a \gauge-controlled measurable $A\subseteq X$ such that $\norm{\chf CT\chf A}\geq\delta$.

  Let $F\coloneqq \gaugey\circ F_1\circ F_0\circ\op{F_1}\circ\gaugey$ and observe that $C\times C$ is contained in $F$. By construction, $C\times A$ is contained in $\appmap[T]{\delta}{F}{\gauge}$. Since $B\subseteq C$, this completes the proof.
\end{proof}

\begin{remark}
  \cref{prop: unconditional image estimate} is not true without the assumption that $T$ be bounded from below. To see this, let $X_n=[n]$ seen as a coarsely disjoint union of $n$ points, and let $Y_n$ be a single point. Consider the partial isometry $T\colon\ell^2(X_n)\to\ell^2(Y_n) = \CCC$ sending the locally constant function $1/\sqrt n$ to $1$ and the orthogonal complement to $0$. Of course, $T$ is surjective, but in order to obtain a non-empty approximating map one needs to take $\delta$ as small as $n^{-1/2}$. The absence of a uniform bound implies that there is no hope to obtain unconditional surjectivity: an actual example is built from the above by considering a disjoint union of $X_n$ and $Y_n$, where $n \in \NN$.
\end{remark}

\section{Rigidity for weakly quasi-controlled operators}
In this final section we will prove the main result of the whole chapter (cf.\ \cref{thm: rigidity quasi-proper operators}). Using the theory we developed thus far, it is rather simple to leverage \cref{prop: unconditional image estimate} to deduce that a Lipschitz isomorphism $T\colon\CHx\to\CHy$ between discrete faithful modules such that both $T$ and $T^*$ are weakly quasi-controlled must give rise to a coarse equivalence between $\crse X$ and $\crse Y$ (see \cref{cor: rigidity controlled unitaries} below). 
However, rather than proving this statement directly we shall instead prove a more general theorem, which will be very useful when tackling rigidity under \emph{stable} isomorphisms of Roe-like \cstar{}algebras.
In that context it is not generally possible to find an operator $T$ that is weakly quasi-controlled on the whole $\CHx$. On the other hand, it is relatively easy to show that one can find operators whose restrictions to certain \emph{submodules} (cf.\ \cref{subsec: submodules}) are weakly quasi-controlled. We will then be able to apply the following general result.
\begin{theorem}\label{thm: rigidity quasi-proper operators}
  Fix locally admissible modules $\CHx$ and $\CHy$, an operator $T\colon \CHx\to \CHy$ and constants $\mu,\eta$ with $0\leq\mu < \eta/2$.
  Suppose there are chains of submodules $p_X^0\leq p_X'\leq p_X''$ of $\CHx$ and $p_Y^0\leq p_Y'\leq p_Y''$ of $\CHy$ with
  \[
    \renewcommand{\arraystretch}{1.5}\setlength{\arraycolsep}{2em}\begin{array}{cc}
      p_Y' Tp_X^0 \approx_\mu  Tp_X^0, &
      p_Y'' Tp_X' \approx_\mu  Tp_X', \\
      p_X' T^*p_Y^0 \approx_\mu  T^*p_Y^0, &
      p_X'' T^*p_Y' \approx_\mu  T^*p_Y'.
    \end{array}
  \]
  Further suppose that the following are satisfied.
  \begin{enumerate}[label=(\roman*)]
    \item $p_X^0$ and $p_Y^0$ are faithful;
    \item  $p_X'$ and $p_Y'$ are discrete;
    \item the restrictions of $T$ and $T^*$ to $p_X'$ and $p_Y'$ respectively are $\eta$-bounded from below;
    \item $Tp_X''\colon p_X''(\CHx)\to\CHy$ and $T^*p_Y''\colon p_Y''(\CHy)\to\CHx$ are weakly quasi-controlled.
  \end{enumerate}
  Then for every $\mu<\delta<\eta$ there are $E\in\CE$ and $F\in\CF$ large enough such that $\cappmap[p_Y''Tp_X'']{\delta-\mu}{F}E\colon\crse X\to\crse Y$ is a coarse equivalence with coarse inverse $\cappmap[p_X''T^*p_Y'']{\delta-\mu}EF$.
\end{theorem}

\noindent\textit{Proof.}
For notational convenience, let $S\coloneqq p_Y''Tp_X''$. To be completely clear, in the following we will consider $S$ as an operator $\CHx\to \CHy$. We could have as well used $p_Y''Tp_X''\colon p_X''\to p_Y''$, but $S$ has the notational advantage of not cluttering the superscripts. These pedantic details are motivated by the fact that the notation of the approximating relation is slightly abusive, in that they actually depend on the Boolean algebra (in the rest of this proof one should take extra care to the meaning of measurability). This is perhaps a little confusing, but certainly not a serious concern, see also \cref{rmk: for approximations the boolean alg does not matter}.

We start by observing that, since $\Ad(S)=\Ad(p_Y'')\Ad(T)\Ad(p_X'')$ and $p_Y'',p_X''$ have controlled propagation, $S$ is a weakly quasi-controlled operator. In fact, for every $E\in \CE$ we may apply \cref{prop:coarse functions preserve controlled propagation} to $p_X''$ and obtain some $E'\in\CE$ such that $\Ad(p_X'')(\cpop{E})\subseteq \cpop{E'}$. Given $\varepsilon>0$, weak quasi-control on $Tp_X''\colon p_X''\to\CHy$ yields an $F'$ such that operators of $E'$-controlled propagation are mapped via $\Ad(Tp_X'')$ to $\varepsilon$-$F'$-quasi-local operators. In turn, we apply \cref{prop:coarse functions preserve controlled propagation} again to obtain some $F\in \CF$ such that $\Ad(p_Y'')$ maps those operators to $\varepsilon$-$F$-quasi-local operators.
The same argument applies to the adjoint $S^*\colon\CHy\to\CHx$.

By \cref{lem:approximation is controlled}, it follows that for every choice of $\delta>0$, $E\in \CE$ and $F\in\CF$ the approximating relations
$\appmap[S]\delta F E$ and $\appmap[S^*]\delta E F$ are controlled.
We now aim to use these controlled relations to construct coarse equivalences.
By \cref{prop:adjoint is coarse inverse}, it will be enough to show that $\cappmap[S]\delta F E$ and $\cappmap[S^*]\delta E F$ are coarsely surjective for some appropriate choice of parameters. We prove it only for $S$, as the argument for $S^*$ is symmetric.
Later on we will also show that the approximations of $p_Y''Tp_X''$ and $p_X''T^*p_Y''$ (seen as operators among the submodules) give rise to coarse equivalences as well.

One difficulty at this point is that $S$ is generally not bounded from below, hence we are not yet in the position to use the Concentration Inequality. In the next steps we fix this by reducing the surjectivity estimates to properties of the operators among appropriate submodules.

\begin{claim}\label{claim:stable rigidity containment}
  For every $\delta\geq\mu$, $E\in\CE$, and $F\in\CF$ we have
  \(
    \appmap[Tp_X']\delta F E \subseteq \appmap[S]{\delta-\mu} F E.
  \)
\end{claim}
\begin{proof}[Proof of Claim]
  Let $B\times A\subseteq Y\times X$ be one of the defining blocks of the approximation of $Tp_X'$. That is, $B\times A$ is $F\otimes E$-bounded with $\CHy$-measurable $B$, $p_X'$-measurable $A$, and
  \[
    \delta 
    <\norm{{\chf{B}}T{p_X'}{\chf{A}}}.
  \]
  By hypothesis, we have that \(p_Y'' Tp_X' \approx_\mu  Tp_X'\), which implies that
  \begin{align*}
    \delta-\mu 
    &\leq \norm{{\chf{B}}{p_Y''}T{p_X'}{\chf{A}}},
  \end{align*}
  and the latter is a lower bound for $\norm{{\chf{B}}S{\chf{A}}}$ because ${p_X'}\leq {p_X''}$ and they both commute with $\chf{A}$ (see \cref{lem: submodule iff commutes}).
\end{proof}
We are now in a very good position, because $Tp_X'$ is a weakly quasi-controlled operator (as $p_X'$ is a submodule of $p_X''$) that is moreover $\eta$-bounded from below (by hypothesis). We shall then be able to leverage \cref{prop: unconditional image estimate} to prove coarse surjectivity of the approximating relations. Let 
\[
  \kappa \coloneqq (\eta-2\mu)/\norm{T} > 0.
\]

\begin{claim}\label{claim:stable rigidity surjective}
  There is a gauge $F_0\in\CF$ such that $\concvec{\kappa}{F_0}{Tp_X'(\CHx)}$
  is coarsely dense in $\crse Y$.
\end{claim}
\begin{proof}[Proof of Claim]
  Since $p_Y^0$ is faithful, we may choose a gauge $F_0$ such that there is a family $(B_j)_{j\in J}$ of $F_0$-bounded $p_Y^0$-measurable sets with  non-zero $\chf{B_j}p_Y^0$ and such that the union $\bigcup_{j\in J}B_j$ is coarsely dense in $\crse Y$.
  In particular, $\chf{B_j}$ commutes with $p_Y^0$ and $p_Y''$ for every $j\in J$ ($B_j$ is also $p_Y''$-measurable, because $p_Y^0$ is a submodule thereof) and $\norm{\chf{B_j}p_Y^0}=1$.
  
  Observe that
  \begin{align*}
    \norm{{\chf{B_j}}Tp_X'}
    \approx_\mu \norm{{\chf{B_j}}{p_Y''}T{p_X'}}
    &= \norm{{p_Y''}{\chf{B_j}}T{p_X'}} \\
    &\geq \norm{{p_Y^0}{\chf{B_j}}T{p_X'}}
    = \norm{{\chf{B_j}}{p_Y^0}T{p_X'}}
  \end{align*}
  and 
  \begin{align*}
    \norm{\chf{B_j} p_Y^0Tp_X'}
    \approx_\mu \norm{\chf{B_j} p_Y^0T}
    \geq \eta,
  \end{align*}
  where the last inequality holds because $T^*$ is $\eta$-bounded from below on $p_Y'$ (and hence on $p_Y^0$ as well).
  This shows that $\norm{{\chf{B_j}}Tp_X'}\geq \eta-2\mu$ and hence 
  \[
    \norm{\chf{B_j}|_{Tp_X'(\CHx)}}\geq (\eta-2\mu)/\norm{Tp_X'}=\kappa.
  \]
  That is, $B_j\subseteq\concvec{\kappa}{F_0}{Tp_X'(\CHx)}$.
  The claim follows.
\end{proof}

It is now simple to complete the proof.
Let $\gaugex$ be a discreteness gauge for $p_X'$ and fix a $\delta>0$ with $\mu<\delta<\eta$.
Since $Tp_X'\colon p_X'(\CHx)\to\CHy$ is a weakly quasi-controlled operator that is $\eta$-bounded from below from a discrete to a locally admissible module, and $0<\delta<\eta$ by assumption, \cref{prop: unconditional image estimate} yields a controlled entourage $F\in \CF$ such that
\[
  \concvec{\kappa}{F_0}{Tp_X'(\CHx)} \subseteq \image\Bigparen{\appmap[Tp_X']{\delta}{F}{\gaugex}}.
\]
\cref{claim:stable rigidity containment,claim:stable rigidity surjective} then imply that $\appmap[S]{\delta-\mu} F \gaugex$ is coarsely surjective.

By a symmetric argument, once a discreteness gauge $\gaugey$ for $p'_Y$ has been fixed there is an $E\in\CE$ such that $\appmap[S^*]{\delta-\mu} E \gaugey$ is coarsely surjective as well.
Since taking larger controlled entourages yields larger approximating relations, we may also assume that $\gaugex\subseteq E$ and $\gaugey\subseteq F$. We then have
\[
  \appmap[S]{\delta -\mu}{F}{\gaugex} \subseteq \appmap[S]{\delta -\mu}{F}{E}
  \quad\text{and}\quad
  \appmap[S^*]{\delta -\mu}{E}{\gaugey} \subseteq \appmap[S^*]{\delta -\mu}{E}{F}.
\]
Since $\delta-\mu >0$, this shows that $\cappmap[S]{\delta -\mu} F E$ and $\cappmap[S^*]{\delta -\mu} E F$ are coarsely surjective (partial) coarse maps, which are hence coarse inverses to one another by \cref{prop:adjoint is coarse inverse}.
\qed

\medskip

Letting $p_X^0=p_X'=p_X''=1_{\CHx}$ and $p_Y^0=p_Y'=p_Y''=1_{\CHy}$ in \cref{thm: rigidity quasi-proper operators} yields the following immediate consequence.

\begin{corollary}\label{cor: rigidity controlled unitaries}
  If $\CHx$, $\CHy$ are faithful discrete modules and $U\colon\CHx\to \CHy$ is a unitary such that both $U$ and $U^*$ are weakly quasi-controlled, then for every $0<\delta<1$ there are $E\in\CE$ and $F\in\CF$ large enough such that $\cappmap[U]{\delta}{F}E\colon\crse X\to\crse Y$ is a coarse equivalence with coarse inverse $\cappmap[U^*]{\delta}EF$.
\end{corollary}

\begin{remark}\label{rmk: for approximations the boolean alg does not matter}
  In the statement of \cref{thm: rigidity quasi-proper operators} there is some ambiguity when writing $\cappmap[p_Y''Tp_X'']{\delta-\mu}{F}E$, because the definition of the approximating relations depends on the boolean algebra defining the coarse geometric module, so one should specify if $p_Y''Tp_X''$ is seen as the extended operator $S\colon\CHx\to\CHy$ or the operator among submodules $p_Y''Tp_X''\colon p_X''\to\ p_Y''$.
  
  Up to passing to slightly larger entourages, this does not cause any issues. In fact, if $p_X''$ and $p_Y''$ are discrete, we may then assume that $F$ and $E$ be block entourages. In this case it is easy to see that every $E$-bounded set is contained in an $E$-bounded set that is a union of regions of the discrete partition, and is hence  $p_X''$-measurable. Of course, the same holds on the $Y$-side, hence
  \[
    \appmap[p_Y''Tp_X'']{\delta -\mu}{F'}{E'} = \appmap[S]{\delta -\mu}{F'}{E'}.
  \]
  This will most likely suffice to satisfy the reader. In fact, since $p_X'$ and $p_Y'$ are discrete, they are admissible. 
  A fortiori, $p_X''$ and $p_Y''$ are admissible as well, and we are not aware of meaningful examples of admissible modules that are not discrete.

  If, taken by an excess of zeal, one was to wonder what may happen in case that the modules be admissible but not discrete, one could argue as follow.
  Say that $\gaugex$ and $\gaugey$ are admissibility gauges of $p_X''$ and $p_Y''$ and let $E' = \gaugex\circ E\circ \gaugex$ and $F' = \gaugey\circ F\circ \gaugey$. Then one checks that
  \[
    \appmap[S]{\delta -\mu}{F}{E} \subseteq \appmap[p_Y''Tp_X'']{\delta -\mu}{F'}{E'}
    \quad\text{and}\quad
    \appmap[p_Y''Tp_X'']{\delta -\mu}{F'}{E'}\subseteq \appmap[S]{\delta -\mu}{F'}{E'}
  \]
  (the first containment uses admissibility, and the latter is trivial).
  In general these containments could provide limited information, as $\appmap[S]{\delta -\mu}{F}{E}$ as $\appmap[S]{\delta -\mu}{F'}{E'}$ might in principle be very different. However, since we are in a situation where $F$ and $E$ are large enough to ensure that $\cappmap[S]{\delta -\mu}{F}{E}$ is coarsely everywhere defined, it follows from \cref{lem: functions with same domain coincide} that they actually define the same coarse map.
\end{remark}

\chapter{Rigidity Phenomena} \label{sec: rigidity phenomena}
In this chapter we combine the results of \cref{sec: uniformization,sec: rigidity spatially implemented} in order to finally prove the most general rigidity result (cf.\ \cref{thm: stable rigidity}), and then record some consequences.
Our techniques allow us to work in the following level of generality.

\begin{notation} \label{notation: stable rigidity}
  In the following, $\CH_1$ and $\CH_2$ are (possibly different) non-zero Hilbert spaces, and $\roeclikeone{\CHx}$ and $\roecliketwo{\CHy}$ denote any Roe-like \cstar{}algebra associated to $\CHx$ and $\CHy$, possibly of ``different type''. For instance, $\roeclikeone{\CHx}$ may be $\roecstar{\CHx}$, whereas $\roecliketwo{\CHy}$ may be $\qlcstar{\CHy}$.
\end{notation}

\section{Spatial implementation of stable isomorphisms}
The first step towards proving that Roe-like \cstar{}algebras are rigid is to prove that an isomorphism between Roe-like \cstar{}algebras is always spatially implemented (by a unitary operator). We directly prove it for stabilized algebras, as the non-stable version follows by letting $\CH_1 \coloneqq \CCC \eqqcolon \CH_2$.
\begin{proposition}\label{prop: isos are spatially implemented}
  Let $\Phi\colon\roeclikeone{\CHx}\otimes\CK(\CH_1) \to\roecliketwo{\CHy}\otimes\CK(\CH_2)$ be an isomorphism. Suppose that every coarsely connected component of $\crse X$ and $\crse Y$ is measurable. Then $\Phi$ is spatially implemented by a unitary operator $U\colon\CHx\otimes\CH_1\to\CHy\otimes\CH_2$, \emph{i.e.}\ $\Phi = \Ad(U)$.
\end{proposition}
\begin{proof}
  Let $\crse X = \bigsqcup_{i\in I} \crse X_{i}$ be the decomposition in coarsely connected components of $\crse X$. By \cref{prop: roelike cap compacts} we have 
  \[
    \roeclikeone{\CHx} \leq \prod_{i\in I}\roeclikeone{\CH_{X_i}}
    \quad\text{and}\quad
    \CK(\CHx) \cap \roeclikeone{\CHx} = \bigoplus_{i \in I} \CK(\CH_{X_i}).
  \]
  Tensoring with $\CK(\CH_1)$ yields
  \begin{align}
    \roeclikeone{\CHx} \otimes \CK\left(\CH_1\right)
    & \leq \prod_{i\in I} \bigparen{\roeclikeone{\CH_{X_i}} \otimes \CK\left(\CH_1\right)}
    \leq \prod_{i\in I}\CB\left(\CH_{X_i}\otimes \CH_1\right), \label{eq:roeclike in product} \\
    \CK\left(\CHx\otimes \CH_1\right) & \cap \bigparen{\roeclikeone{\CHx}\otimes\CK\left(\CH_1\right)}
    \cong \bigoplus_{i \in I} \CK\left(\CH_{X_i}\otimes \CH_1\right). \label{eq:compacts are sum}
  \end{align}
  As a consequence, it follows that for every $i\in I$ the compacts $\CK(\CH_{X_i}\otimes\CH_1)$ are a minimal ideal in $\roeclikeone{\CHx} \otimes \CK(\CH_1)$. Likewise, every minimal ideal takes this form.
  
  The same considerations hold for the decomposition in coarsely connected components $\crse Y =\bigsqcup_{j\in J}\crse Y_j$. Since an isomorphism sends minimal ideals to minimal ideals, we deduce that there is a bijection $I\to J, i \mapsto j(i),$ such that $\Phi$ restricts to isomorphisms 
  \[
    \Phi_i\colon \CK\left(\CH_{X_i}\otimes\CH_1\right)\xrightarrow{\cong}\CK\left(\CH_{Y_{j(i)}}\otimes\CH_2\right).
  \]
  It follows that each $\Phi_i$ is \rlone (recall \cref{lem: rank-1 to rank-1 if hereditary}). Moreover, by \eqref{eq:compacts are sum} every rank-1 operator in $\roeclikeone{\CHx}\otimes\CK(\CH_1)$ must belong to some $\CK(\CH_{X_i}\otimes\CH_1)$. Thus, $\Phi$ itself is \rlone.

  By \cref{prop: non-degenerate hom is spatially implemented}, it only remains to show that the restriction of $\Phi$ to $\bigoplus_{i \in I} \CK(\CH_{X_i}\otimes \CH_1)$ defines a non-degenerate \Star{}representation into $\CB(\CHy\otimes\CH_2)$. This is readily done: consider a simple tensor $w\otimes \xi\in \CHy\otimes\CH_2$. For every $\varepsilon > 0$, by \cref{lem:supports of vectors almost contained in bounded}, we may find finitely many bounded measurable subsets $B_k\subseteq Y$ such that $w\approx_\varepsilon \sum_k\chf{B_k}(w)$. Let $q_k\in \roecliketwo{\CHy}\otimes \CK(\CH_2)$ denote the rank-1 projection onto the span of $\chf{B_k}(w)\otimes\xi$. 
  Since $B_k$ is bounded, $q_k\in \CK(\CH_{Y_j}\otimes \CH_2)$ for some $j\in J$, and therefore $p_k \coloneqq \Phi^{-1}(q_k)$ is a rank-1 projection in $\CK(\CH_{X_i}\otimes \CH_1)$ for some $i\in I$.
  This implies the required non-degeneracy.
\end{proof}

\begin{remark}
  The analogous statement of \cref{prop: isos are spatially implemented} need not be true for \Star{}homomorphisms $\roeclikeone{\CHx}\otimes\CK(\CH_1) \to\roecliketwo{\CHy}\otimes\CK(\CH_2)$. 
  This can be seen for instance by letting $\crse Y\coloneqq \crse X\sqcup\crse X$ and considering a diagonal embedding $\roeclike{\CHx}\to \roeclike{\CHy} \cong \roeclike{\CHx}\oplus\roeclike{\CHx}$.

  One can fix this issue by adding extra assumptions on the homomorphism, such as the requirement that the image be a hereditary subalgebra \cite{braga2020embeddings}. In this memoir we will not pursue this line of thoughts.
\end{remark}

\section{Stable rigidity of Roe-like algebras of modules}
In this section we prove \cref{thm:intro: rigidity}, for which we only need one last preliminary result, the proof of which varies a little depending on whether the \cstar{}algebra under consideration is unital or not (cf.\ \cref{rmk: unital Roe algebras}).

\begin{lemma}\label{lem: approximate image under stable by finite rank}
  Let $\CHy$ be discrete with discrete partition $(B_j)_{j\in J}$. Given an operator $T\colon \CHx\otimes\CH_1\to \CHy\otimes\CH_2$ and a submodule $p\leq \CHx\otimes\CH_1$, suppose that:
  \begin{enumerate}[label=(\roman*)]
    \item $\Ad(T)$ maps $\roeclikeone\CHx\otimes\CK(\CH_1)$ into $\roecliketwo\CHy\otimes\CK(\CH_2)$;
    \item $p\in\roeclikeone\CHx\otimes\CK(\CH_1)$;
    \item either $\roecliketwo\CHy$ is unital or $\crse Y$ is coarsely locally finite.
  \end{enumerate}
  Then, for every $\mu>0$ there is a submodule $r\leq\CHy\otimes\CH_2$ subordinate to $(B_j)_{j\in J}$ such that $\cpcstar{r}\subseteq\roecliketwo\CHy\otimes\CK(\CH_2)$ and 
  \[
    \norm{Tp- rTp} \leq \mu.
  \]
\end{lemma}
\begin{proof}
  By assumption, the operator $TpT^*$ is in $\roecliketwo{\CHy}\otimes\CK(\CH_2)$. Since finite rank projections form an approximate unit for $\CK(\CH_2)$, we may find a (finite rank) projection $q\in\CK(\CH_2)$ such that
  \[
  \norm{\paren{1 -1_{\CHy}\otimes q}TpT^*}\leq (\mu/2  )^2,
  \]
  where $1=1_{\CHy\otimes\CH_2}$. It follows that 
  \begin{align*}
    \norm{\paren{1 -1_{\CHy}\otimes q}Tp}^2
  &=\norm{\paren{1 -1_{\CHy}\otimes q}TpT^*\paren{1 -1_{\CHy}\otimes q}} \\
  &\leq (\mu/2)^2\norm{\paren{1 -1_{\CHy}\otimes q}} \leq (\mu/2)^2.
  \end{align*}
  This proves that $\norm{Tp- (1_{\CHy}\otimes q)Tp}\leq\mu/2$.

  If $\roecliketwo\CHy$ is unital we may already let $r \coloneqq 1_{\CHy}\otimes q$ and we are done.
  Otherwise, we are in the setting where $\roecliketwo\CHy=\roecstar\CHy$ and $\crse Y$ is coarsely locally finite. Observe that the operator $(1_{\CHy}\otimes q)TpT^*(1_{\CHy}\otimes q)$ is an element of $\roecstar{\CHy}\otimes\CB(q(\CH_2))$, and the latter is nothing but $\roecstar{\CHy\otimes q(\CH_2)}$, since $q$ has finite rank (cf.\ \cref{rmk: roelike of tensor with finite rank}).
  We may then apply \cref{lem: exists submodule approximating operator} to the discrete module $1_{\CHy}\otimes q$ to deduce that there is a submodule $r\leq 1_{\CHy}\otimes q$ subordinate to $(B_j)_{j\in J}$ such that $\cpcstar{r}\subseteq\roecstar{\CHy\otimes q(\CH_2)} \subseteq \roecstar\CHy\otimes\CK(\CH_2)$ and 
  \[
  \norm{(1-r)(1_{\CHy}\otimes q)TpT^*(1_{\CHy}\otimes q)}\leq (\mu/2  )^2.
  \]
  Arguing as above, we deduce $\norm{(1_{\CHy}\otimes q)Tp- r(1_{\CHy}\otimes q)Tp}\leq\mu/2$. Thus
  \[
    \norm{Tp-rTp}
    \leq\norm{Tp- (1_{\CHy}\otimes q)Tp} + \norm{(1_{\CHy}\otimes q)Tp- r(1_{\CHy}\otimes q)Tt}
    \leq \mu,
  \]
  as desired.
\end{proof}

We now have all the necessary ingredients to prove \cref{thm:intro: rigidity}. The precise statement we prove is the following (recall that $\roeclikeone \CHx$ and $\roecliketwo \CHy$ may denote different kinds of Roe-like \cstar{}algebras, see \cref{notation: stable rigidity}).

\begin{theorem}[cf.\ \cref{thm:intro: rigidity}]\label{thm: stable rigidity}
  Let $\crse X$ and $\crse Y$ be countably generated coarse spaces; $\CHx$ and $\CHy$ faithful discrete modules; and $\CH_1$ and $\CH_2$ Hilbert spaces. Suppose that $\roeclikeone{\CHx} \otimes \CK(\CH_1) \cong \roecliketwo{\CHy} \otimes \CK(\CH_2)$ and the following hold:
  \begin{enumerate}[label=(\roman*)]
    \item either $\roeclikeone{\CHx}$ is unital or $\crse X$ is coarsely locally finite,
    \item either $\roecliketwo{\CHy}$ is unital or $\crse Y$ is coarsely locally finite.
  \end{enumerate}
  Then $\crse X$ and $\crse Y$ are coarsely equivalent.
\end{theorem}
\begin{proof}
  Let $\Phi \colon \roeclikeone{\CHx}\otimes\CK(\CH_1) \to \roecliketwo{\CHy}\otimes\CK(\CH_2)$ be an isomorphism. 
  For discrete modules, connected components are always measurable. We may thus apply \cref{prop: isos are spatially implemented} to deduce that $\Phi=\Ad(U)$ for some unitary $U\colon \CHx\otimes\CH_1 \to \CHy\otimes\CH_2$.
  We shall now consider the tensor products $\CHx\otimes \CH_1$ and $\CHy\otimes \CH_2$ as discrete coarse geometric modules in their own right. Observe that the operator $U$ needs \emph{not} be weakly quasi-controlled, as shown by \cref{ex:tensor-k-screws-quasi-control}.
  Crucially, we will see that $U$ is weakly quasi-controlled when restricted to some carefully chosen submodules, and this will suffice to apply \cref{thm: rigidity quasi-proper operators}.

  Fix once and for all sufficiently large gauges $\gaugex$, $\gaugey$ and locally finite
  discrete partitions $X=\bigsqcup_{i\in I}A_i$ and $Y=\bigsqcup_{j\in J}B_j$ subordinate to these gauges and such that $\chf{A_i}$ and $\chf{B_j}$ are non-zero for every $i\in I$ and $j\in J$. This can be arranged thanks to the faithfulness assumptions on $\CHx$ and $\CHy$.
  If $\crse X$ (resp.\ $\crse Y$) is coarsely locally finite, we also assume that the respective partition is locally finite as well.

  Arbitrarily fix some $0<\mu< 1/2$ and a vector $\xi\in \CH_1$ of norm one, and let $q_\xi\in\CK(\CH_1)$ denote the projection onto its span. We choose a faithful submodule $r_X\leq \CHx$ subordinate to the partition $(A_i)_{i\in I}$ such that $p_X^0 \coloneqq r_X\otimes q_\xi$ is a submodule of $\CHx\otimes\CH_1$ with $\cpcstar {p_X^0} \subseteq \roeclikeone\CHx\otimes\CK(\CH_1)$. When $\roeclikeone\CHx$ is unital (\emph{e.g.}\ $\cpcstar\CHx$ or $\qlcstar\CHx$) this is easily accomplished letting $r_X=1_{\CHx}$.
  In the non-unital case, $(A_i)_{i\in I}$ is locally finite by assumption, and we may hence construct $r_X\in\roecstar{\CHx}$ choosing a locally finite rank projection commuting with $\chf{A_i}$ and such that $\chf{A_i}r_X$ is non-zero for every $i\in I$ (that is, choose non-empty finite rank vector subspaces of each $\CH_{A_i} = \chf{A_i}(\CHx)$).
  
  Similarly choose a submodule $p_Y^0\coloneqq r_Y\otimes q_\zeta$ of $\CHy \otimes \CH_2$, where $\zeta \in \CH_2$ is any vector of norm $1$.
  By construction, $p_X^0$ and $p_Y^0$ are faithful discrete coarse geometric modules subordinate to the partitions $(A_i)_{i\in I}$ and $(B_j)_{j\in J}$.

  We may now repeatedly apply \cref{lem: approximate image under stable by finite rank} to both $U$ and $U^*$ to obtain discrete submodules $p_X',p_X''$ of $\CHx\otimes\CH_1$ subordinate to $(A_i)_{i\in I}$, and $p_Y',p_Y''$ of $\CHy\otimes \CH_2$ subordinate to $(B_j)_{j\in J}$ such that 
  \[
    \renewcommand{\arraystretch}{1.5}\setlength{\arraycolsep}{2em}\begin{array}{cc}
      p_Y' Tp_X^0 \approx_\mu  Tp_X^0, &
      p_Y'' Tp_X' \approx_\mu  Tp_X', \\
      p_X' T^*p_Y^0 \approx_\mu  T^*p_Y^0, &
      p_X'' T^*p_Y' \approx_\mu  T^*p_Y',
    \end{array}
  \]
  and, furthermore, $\cpcstar{p_X'}$,$\cpcstar {p_X''}$, $\cpcstar {p_Y'}$ and $\cpcstar {p_Y''}$ are contained in $\roeclikeone\CHx\otimes\CK(\CH_1)$ and $\roecliketwo\CHy\otimes\CK(\CH_2)$ respectively.
  Since finite joins of submodules subordinate to the same partitions are still submodules (cf.\ \cref{rmk: join of submodules}), we may further assume that they form chains
  \[
    p_X^0\leq p_X'\leq p_X''
    \;\; \text{ and } \;\;
    p_Y^0\leq p_Y'\leq p_Y''.
  \]
  
  All the above modules are discrete, and the restriction of $U$ and $U^*$ to them is an isometry (which is hence $1$-bounded from below). Crucially, $\Ad(U)$ defines a (automatically strongly continuous) homomorphism from $\cpcstar{p_X''}$ into $\qlcstar{\CHy}\otimes \CK(\CH_2) \subseteq \qlcstar{\CHy\otimes\CH_2}$.
  We may hence apply \cref{thm: uniformization} and deduce that $U\colon p_X''\to\CHy\otimes\CH_2$ is weakly quasi-controlled. Analogously, $U^*\colon p_Y''\to\CHx\otimes\CH_1$ is weakly quasi-controlled as well. The statement then follows from \cref{thm: rigidity quasi-proper operators}.
\end{proof}

\begin{remark}
  Note that the coarse local finiteness assumption on $\crse{X}$ (resp.\ $\crse{Y}$) in \cref{thm: stable rigidity} is only needed when working with non-unital Roe algebras.  
  On a technical level, this is needed in two separate (but related) places: once when choosing $r_X$ and/or $r_Y$, and once when applying \cref{lem: exists submodule approximating operator} which, in turn, uses \rfThmApproxUnit. This last theorem necessitates the coarse space to be coarsely locally finite.
\end{remark}

A key point in the proof of \cref{thm: stable rigidity} is the passing to appropriate submodules of $\CHx\otimes \CH_1$ and $\CHy\otimes \CH_2$. This step is crucial in order to obtain weakly quasi-controlled operators, which are necessary to construct partial coarse maps via approximations. The following easy example shows why this step is indeed necessary.
\begin{example} \label{ex:tensor-k-screws-quasi-control}
  Let $\crse X=\crse Y$ be the set of integers $\ZZ$ with the usual metric. Let $\CHx=\CHy=\ell^2(\ZZ)$ and $\CH_1 = \CH_2 = \ell^2(\NN)$, and consider the unitary operator $U\colon \ell^2(\ZZ)\otimes\ell^2(\NN)\to \ell^2(\ZZ)\otimes\ell^2(\NN)$ sending the basis element $\delta_{k,n}$ to $\delta_{k+n,n}$ for every $k\in \ZZ$, $n\in\NN$.
  This operator is decidedly not weakly quasi-controlled. Indeed, consider
  \[
    s \colon \ell^2\left(\ZZ\right) \otimes \ell^2\left(\NN\right) \to \ell^2\left(\ZZ\right) \otimes \ell^2\left(\NN\right), \;\; \delta_{k,n} \mapsto
      \left\{
        \begin{array}{rl}
          \delta_{k,n+k} & \text{if } \, k \geq 0, \\
          0 & \text{otherwise.}
        \end{array}
      \right.
  \]
  Then $s$ has $0$ propagation, but $UsU^*$ is not quasi-local at all.
  Nevertheless, $\Ad(U)$ does define an isomorphism $\cpcstar{\CHx}\otimes\CK(\CH)\cong\cpcstar{\CHy}\otimes\CK(\CH)$.
  In order to see this, arbitrarily fix $\munit ij$ and $t \otimes \matrixunit{i}{j} \in \cpstar{\CHx} \otimes \CK(\CH)$. That is, $t$ is a sum $\sum_{\ell,m\in \ZZ} t_{\ell,m}\matrixunit{\ell}{m}$ where the supremum of $\abs{\ell - m}$ with $t_{\ell,m}\neq 0$ is bounded.
  We may then compute that
  \[
    \left(U \left(t \otimes \matrixunit{i}{j}\right) U^*\right) \left(\delta_{p,q}\right) =
      \left\{
        \begin{array}{ll}
          \sum_{\ell \in \ZZ} t_{\ell,p-q} \delta_{\ell+i,i} & \text{if } \, j = q, \\
          0 & \text{otherwise.}
        \end{array}
      \right.
  \]
  As $t_{\ell,p-q}$ can only be non-zero when $\ell$ stays uniformly close to $p-q$, and both $i$ and $q=j$ are fixed, $\ell+i$ stays uniformly close to $p$. Thus $U(t \otimes \matrixunit{i}{j}) U^*\in \cpstar{\CHy} \otimes \CK(\CH)$.
  To conclude that $\Ad(U)$ is an isomorphism it suffices to observe that the same argument applies to the adjoint operator $U^*\colon \delta_{k,n}\mapsto\delta_{k-n,n}$ as well.
\end{example}

\section{Stable rigidity of Roe-like algebras of coarse spaces} \label{subsec:rig roe algs spaces}
One immediate consequence of \cref{thm: stable rigidity} above is that ``(coarsely locally finite) countably generated coarse spaces are coarsely equivalent if and only if they have isomorphic Roe-like \cstar{}algebras''.
The issue with this statement is that so far we have only studied Roe algebras of \emph{modules}, not of the spaces themselves. Making this consequence precise requires recalling a few extra facts, which we do now.

The main observation is that \cref{prop: existence of covering iso} implies that the isomorphism class of Roe-like algebras does not depend on the choice of coarse geometric module, so long as it is discrete and as ample as its rank. The main point that remains to sort out is whether such a module exists at all. The answer, it turns out, depends on the size of the space.

Following \cite{roe-algs}, we say that the \emph{coarse cardinality} of $\crse X$ is the minimal cardinality of a coarse space coarsely equivalent to $\crse X$. 

\begin{example}
  It is easy to show that if $(X,d)$ is a non-empty separable metric space then $\crse X = (X,\CE_d)$ has countable coarse cardinality: it may be $1$ (bounded case) or $\aleph_0$ (unbounded case). Separable \emph{extended} metric spaces may have any countable cardinal as coarse cardinality, because the other finite cardinalities are reached by taking disjoint unions. Removing the separability assumption, their coarse cardinality can be any cardinal at all.
\end{example}

It is not hard to show (and it is proved in detail in \rfLemExistsDiscreteIffCardinality) that for any cardinal $\kappa$ a coarse space $\crse X$ admits a $\kappa$-ample discrete module of rank $\kappa$ if and only if the coarse cardinality of $\crse X$ is at most $\kappa$.
The following is then a meaningful definition.

\begin{definition}[cf.\ \rfDefRoeAlgsSpace]\label{def: roe-like of spaces}
  Let $\kappa$ be a cardinal and $\crse X$ be a coarse space of coarse cardinality at most $\kappa$.
  The \emph{Roe-like \cstar{}algebras of $\crse X$ of rank $\kappa$} are the Roe-like \cstar{}algebras associated with any $\kappa$-ample discrete $\crse X$-module $\CHx$ of rank $\kappa$. 
  These are denoted by $\roecstar[\kappa]{\crse X}$, $\cpcstar[\kappa]{\crse X}$, and $\qlcstar[\kappa]{\crse X}$.
  If $\kappa=\aleph_0$, we omit it from the notation.
\end{definition}

\begin{remark}
  \begin{enumerate}[label=(\roman*)]
    \item The above definition is indeed well posed, since it follows from \cref{prop: existence of covering iso} applied to the coarse equivalence $\cid_{\crse X} \crse{\colon X \to X}$ that different choices of $\kappa$-ample modules of rank $\kappa$ yield isomorphic (\cstar{})algebras.
    Since $\cid_{\crse X}$ is proper, one may also obtain well-defined \cstar{}algebras by considering modules that are $\kappa$-ample and of \emph{local} rank $\kappa$.
    \item If $\crse X = (X,d)$ is a proper metric space, $\roecstar{\crse X}$ is isomorphic to the classically defined Roe algebra $C^*(X)$ of the metric space $(X,d)$.
  \end{enumerate}
\end{remark}

A rather different algebra of classical interest is the \emph{uniform Roe algebra} of a coarse space $\crse{X}$~\cites{braga2020embeddings,braga_gelfand_duality_2022,braga_farah_rig_2021,braga_rigid_unif_roe_2022,bbfvw_2023_embeddings_vna}. 
Recall that we also call $\ell^2(X)$ the uniform module $\CH_{u, \crse{X}}$ of $\crse X$ (cf.\ \cref{ex:module-uniform}). We may define its Roe \cstar{}algebra to be the uniform Roe algebra of $\crse X$ (see \rfUnifRoeAlg for a more detailed discussion). Namely, we use the following.

\begin{definition} \label{ex:uniform-roe-algebra}
  Given any coarse space $\crse X$, the \emph{uniform Roe algebra} of $\crse X$, denoted by $\uroecstar{\crse X}$, is defined to be $\cpcstar{\ell^2(X)}$.
\end{definition}

As is well known, the uniform Roe algebra is \emph{not} an invariant of coarse equivalence: the easiest way to see this is by observing that all finite metric spaces are bounded, but give rise uniform Roe algebras of different dimensions. Thus, denoting it $\uroecstar{\crse X}$ is an abuse of our notational conventions regarding bold symbols, which we will only use in this chapter. Moreover, we shall presently see that uniform Roe algebras are only well behaved on coarse spaces that are uniformly locally finite (see \cref{def:loc-fin,rmk: uniform local finiteness}).

\medskip

Recall that two \cstar{}algebras $A,B$ are \emph{stably isomorphic} if $A\otimes\CK(\CH)\cong B\otimes \CK(\CH)$, where $\CH$ is a separable Hilbert space.
Complementing the existing literature, we may finally state and prove the following immediate consequence of \cref{thm: stable rigidity,prop: existence of covering iso}.

\begin{corollary}[cf.\ \cref{cor:intro: roe-rigidity}]\label{cor:roe-rigidity}
 Let $\crse{X}$ and $\crse{Y}$ be coarsely locally finite, countably generated coarse spaces of coarse cardinality $\kappa$. Consider:
  \begin{enumerate} [label=(\roman*)]
    \item \label{cor:roe-rigidity:ceq} $\crse X$ and $\crse Y$ are coarsely equivalent.
    \item \label{cor:roe-rigidity:roe} $\roecstar[\kappa]{\crse{X}}$ and $\roecstar[\kappa]{\crse{Y}}$ are \Star{}isomorphic.
    \item \label{cor:roe-rigidity:cp} $\cpcstar[\kappa]{\crse{X}}$ and $\cpcstar[\kappa]{\crse{Y}}$ are \Star{}isomorphic.
    \item \label{cor:roe-rigidity:ql} $\qlcstar[\kappa]{\crse{X}}$ and $\qlcstar[\kappa]{\crse{Y}}$ are \Star{}isomorphic.
  \end{enumerate}
  and
  \begin{enumerate}[label=(\roman*),resume]
    \item \label{cor:roe-rigidity:u-k-iso} $\uroecstar{\crse X}$ and $\uroecstar{\crse Y}$ are stably \Star{}isomorphic.
    \item \label{cor:roe-rigidity:morita} $\uroecstar{\crse X}$ and $\uroecstar{\crse Y}$ are Morita equivalent.
  \end{enumerate}
  Then \cref{cor:roe-rigidity:ceq} $\Leftrightarrow$ \cref{cor:roe-rigidity:roe} $\Leftrightarrow$ \cref{cor:roe-rigidity:cp} $\Leftrightarrow$ \cref{cor:roe-rigidity:ql} and \cref{cor:roe-rigidity:u-k-iso} $\Leftrightarrow$ \cref{cor:roe-rigidity:morita}. Moreover, \cref{cor:roe-rigidity:u-k-iso} $\Rightarrow$ \cref{cor:roe-rigidity:ceq} always holds, and \cref{cor:roe-rigidity:ceq} $\Rightarrow$ \cref{cor:roe-rigidity:u-k-iso} holds if $\crse{X}$ and $\crse{Y}$ are uniformly locally finite.
\end{corollary}
\begin{proof}
  The equivalence \cref{cor:roe-rigidity:u-k-iso} $\Leftrightarrow$ \cref{cor:roe-rigidity:morita} is classical. In fact, $\uroecstar{\variable}$ is always unital, and hence always contains a strictly positive element.\footnote{\, Recall that an element $a \in A$ is \emph{strictly positive} if $\rho(a) > 0$ for all states $\rho$ on $A$. In particular, if $A$ is unital $\rho(1_A) = 1 > 0$ for all states, and hence $1_A$ is strictly positive.
  }
  The equivalence \cref{cor:roe-rigidity:u-k-iso} $\Leftrightarrow$ \cref{cor:roe-rigidity:morita} is then an immediate consequence of \cite{brown_green_rieffel_book_1977}*{Theorem~1.2}.

  The fact that if $\crse{X}$ and $\crse{Y}$ are uniformly locally finite then \cref{cor:roe-rigidity:ceq} implies \cref{cor:roe-rigidity:u-k-iso} is also classical: the uniform local finiteness allows one to use a coarse equivalence to construct a well-behaved bijection $X\times\NN\to Y\times \NN$ that induces an isomorphism of the stabilized uniform Roe algebras. For the details, we refer to \cite{brodzki_property_2007}*{Theorem~4} (this theorem is stated for metric spaces, but the proof holds in general).

  At this level of generality, the implications \cref{cor:roe-rigidity:ceq} $\Rightarrow$ \cref{cor:roe-rigidity:roe,cor:roe-rigidity:cp,cor:roe-rigidity:ql} are all particular cases of \cref{prop: existence of covering iso} (and are once again classical for proper metric spaces). 
  
  Finally, our main contribution is the hardest implications. Namely, \cref{thm: stable rigidity} shows that any of \cref{cor:roe-rigidity:roe}--\cref{cor:roe-rigidity:u-k-iso} implies \cref{cor:roe-rigidity:ceq}, thus finishing the proof of this corollary (the implication \cref{cor:roe-rigidity:u-k-iso} $\Rightarrow$ \cref{cor:roe-rigidity:ceq} was already known for uniformly locally finite metric spaces: it is the main theorem of \cite{braga_rigid_unif_roe_2022}).
\end{proof}

\begin{remark}
  Note that \cref{thm: stable rigidity} shows that the ``upwards'' implications of \cref{cor:roe-rigidity} in the unital case do not even require coarse local finiteness (this assumption is only added to include the Roe algebras as well).  
  In particular, our proof of the implication \cref{cor:roe-rigidity:u-k-iso} $\Rightarrow$ \cref{cor:roe-rigidity:ceq} completely drops the bounded geometry assumption in \cite{braga_rigid_unif_roe_2022}*{Theorem 1.2}, not even local finiteness is required.
\end{remark}

The following example shows that the hypothesis that $\crse{X}$ and $\crse{Y}$ be uniformly locally finite for the implication \cref{cor:roe-rigidity:ceq} $\Rightarrow$ \cref{cor:roe-rigidity:u-k-iso} in \cref{cor:roe-rigidity} is sharp (fact that is most likely known to experts).

\begin{example} \label{ex:v-i-sharp}
  Let $X$ be the metric space obtained from $\NN$ by replacing each $n\in\NN$ with a cluster of $n$ points at distance $1$ from one another, while the distance of points between different clusters $n$ and $m$ is just $\abs{n-m}$.
  The associated coarse space $\crse X$ is clearly locally finite and coarsely equivalent to $\crse Y \coloneqq (\NN,\abs\variable)$.
  Observe that the clusters give rise to an obvious embedding of $\prod_{n\in \NN}M_n(\CCC)\hookrightarrow\uroecstar{\crse X}$, as operators permuting points within the clusters have propagation $1$.
  On the other hand, the uniform Roe algebra of $\crse Y$ is nuclear, as $Y$ has property A (cf.\ \cite{yu_coarse_2000} and \cite{brown_c*-algebras_2008}*{Theorem 5.5.7}), and hence so is $\uroecstar{\crse Y}\otimes\CK(\CH)$.
  A nuclear \cstar{}algebra cannot contain $\prod_{n\in \NN}M_n(\CCC)$, because nuclear \cstar{}algebras are exact (cf.\ \cite{brown_c*-algebras_2008}*{Section 2.3})---while $\prod_{n\in \NN}M_n(\CCC)$ is not (cf.\ \cite{brown_c*-algebras_2008}*{Exercise 2.3.6})---and exactness passes to subalgebras (this follows from the \emph{nuclearly embeddable} definition of exactness, see \cite{brown_c*-algebras_2008}*{Definition 2.3.2}).
  This shows that $\uroecstar{\crse X}$ and $\uroecstar{\crse Y}$ are not stably isomorphic.
\end{example}

\section{Rigidity of groups and semigroups}\label{subsec:inv-sem}
Any countable group $\Gamma$ admits a proper (right) invariant metric, which is moreover unique up to coarse equivalence.
In particular, $\crse \Gamma$ is a well defined uniformly locally finite countably generated coarse space. \cref{cor:roe-rigidity} shows that for any $\Gamma$ and $\Lambda$ countable groups, $\crse{\Gamma}$ and $\crse{\Lambda}$ are coarsely equivalent if and only if $\roecstar{\crse\Gamma}\cong\roecstar{\crse\Lambda}$, and if and only if  $\uroecstar{\crse\Gamma}$ and $\uroecstar{\crse\Lambda}$ are Morita equivalent/stably isomorphic.
It is well known that $\uroecstar{\crse\Gamma}$ is canonically isomorphic to $\ell^\infty(\Gamma)\rtimes_{\rm red}\Gamma$ (cf.\ \cite{brown_c*-algebras_2008}*{Proposition 5.1.3}), therefore $\crse{\Gamma}$ and $\crse{\Lambda}$ are coarsely equivalent if and only if $\ell^\infty(\Gamma)\rtimes_{\rm red}\Gamma$ and $\ell^\infty(\Lambda)\rtimes_{\rm red}\Lambda$ are Morita equivalent/stably isomorphic.

Having proved our main results for \emph{extended} metric spaces allows us to effortlessly extend these observations to the setting of \emph{inverse semigroups}.
We first need to introduce some notions (we also refer the reader to~\cite{chyuan_chung_inv_sem_2022} for a more comprehensive approach).
An \emph{inverse semigroup} is a semigroup $S$ such that for all $s \in S$ there is a unique $s^* \in S$ such that $ss^*s = s$ and $s^*ss^* = s^*$.\footnote{\, All inverse semigroups here considered shall be discrete.}
As usual, $e \in S$ is an \emph{idempotent} if $e^2 = e$. Note that in such case $e^* = e$, so $e$ is \emph{self-adjoint}. We say that $S$ is \emph{quasi-countable} if there is some countable $K \subseteq S$ such that $S = \angles{K \cup E}$, where $E \subseteq S$ is defined to be the set of idempotents of $S$. The semigroup $S$ has a \emph{unit} if there is some (necessarily unique) element $1 \in S$ such that $1 \cdot s = s \cdot 1 = s$ for all $s \in S$. If a unit exists, $S$ is called an \emph{inverse monoid}.

An extended metric $d \colon S \times S \to [0, \infty]$ is \emph{(right) sub-invariant} if it satisfies $d(s_1t,s_2t)\leq d(s_1,s_2)$ for every $s_1,s_2,t\in S$. It is said to be \emph{proper} if for every $r > 0$ there is a finite $F \subseteq S$ such that $y \in Fx$ whenever $0 < d(x, y) \leq r$.
Just as is done with discrete countable groups, every quasi-countable inverse semigroup $S$ can be equipped with a proper and right sub-invariant extended metric whose coarsely connected components are the \emph{$\CL$-classes} of $S$, where Green's equivalence relation $\CL$ is defined as $x \CL y \Leftrightarrow (x^*x = y^*y)$ (see \cite{chyuan_chung_inv_sem_2022}*{Definition~3.1}).
Moreover such an extended metric is unique up to coarse equivalence (cf.\ \cite{chyuan_chung_inv_sem_2022}*{Theorem~3.22}). 

Using the above metric, one may canonically see the quasi-countable inverse semigroup $S$ as a uniformly locally finite coarse space $\crse S$ and then consider its uniform Roe-algebra $\uroecstar{\crse S}$. Just as in the group case, if $S$ is also an inverse monoid this \cstar{}algebra can be canonically identified with the crossed product $\ell^{\infty}(S) \rtimes_{\rm red} S$ (cf.\ \cite{chyuan_chung_inv_sem_2022}*{Theorem 4.3}).
With the above definitions, the following is an immediate consequence of \cref{cor:roe-rigidity}.

\begin{corollary} \label{cor:inv-sem-roe-rigid}
  Let $S$ and $T$ be two quasi-countable inverse monoids, and let $\crse{S}$ and $\crse{T}$ be the coarse spaces obtained from $S$ and $T$ when equipped with any proper and right sub-invariant metric.
  The following are equivalent:
  \begin{enumerate} [label=(\roman*)]
    \item \label{cor:inv-sem-roe-rigid:ceq} $\crse{S}$ and $\crse{T}$ are coarsely equivalent.
    \item \label{cor:inv-sem-roe-rigid:roe} $\roecstar[|\crse{S}|]{\crse{S}}$ and $\roecstar[|\crse{T}|]{\crse{T}}$ are \Star{}isomorphic.
    \item \label{cor:inv-sem-roe-rigid:cp} $\cpcstar[|\crse{S}|]{\crse{S}}$ and $\cpcstar[|\crse{T}|]{\crse{T}}$ are \Star{}isomorphic.
    \item \label{cor:inv-sem-roe-rigid:ql} $\qlcstar[|\crse{S}|]{\crse{S}}$ and $\qlcstar[|\crse{T}|]{\crse{T}}$ are \Star{}isomorphic.
    \item \label{cor:inv-sem-roe-rigid:morita} $\ell^{\infty}(S) \rtimes_{\rm red} S$ and $\ell^{\infty}(T) \rtimes_{\rm red} T$ are Morita equivalent.
    \item \label{cor:inv-sem-roe-rigid:u-k-iso} $\ell^{\infty}(S) \rtimes_{\rm red} S$ and $\ell^{\infty}(T) \rtimes_{\rm red} T$ are stably isomorphic.
  \end{enumerate}
\end{corollary}

\begin{remark}
  As explained in \cite{braga_rigid_unif_roe_2022}, in the group-setting one may obtain sharper results. Namely, observe that $\uroecstar{\crse\Gamma}$ does not depend on the choice of proper metric: this is because uniform Roe algebras are preserved under \emph{bijective} coarse equivalences. For the same reason, it is also the case that if $\Gamma$ and $\Lambda$ are coarsely equivalent via a bijective coarse equivalence then $\uroecstar{\crse\Gamma}\cong \uroecstar{\crse\Lambda}$.
  In the opposite direction, if $\Gamma$ and $\Lambda$ are amenable then $\crse \Gamma$ and $\crse \Lambda$ have property A, and in this setting it is known that an isomorphism $\uroecstar{\crse\Gamma}\cong \uroecstar{\crse\Lambda}$ gives rise to a bijective coarse equivalence between $\Gamma$ and $\Lambda$. In the non-amenable case, \emph{every} coarse equivalence is close to a bijective coarse equivalence. Putting these two cases together shows that $\ell^\infty(\Gamma)\rtimes_{\rm red}\Gamma\cong \ell^\infty(\Lambda)\rtimes_{\rm red}\Lambda$ if and only if there is a bijective coarse equivalence between $\crse \Gamma$ and $\crse \Lambda$ (cf.\ \cite{braga_rigid_unif_roe_2022}*{Corollary 3.10}).

  These observations do not extend as easily to the semigroup setting. In fact, it is shown in \cite{chyuan_chung_inv_sem_2022}*{Theorem 3.23} that for every uniformly locally finite metric space $X$ there is a quasi-countable inverse monoid $S$ which has an $\CL$-class that is bijectively coarsely equivalent to $X$.
  This shows that inverse monoids can be rather wild, and should hence be handled with care.
  The problem to understand whether two uniformly locally finite metric space with isomorphic uniform Roe algebras must be bijectively coarsely equivalent is still open, and seems rather hard \cites{baudier2023coarse,braga_rigid_unif_roe_2022,white_cartan_2018}.
\end{remark}

\chapter{Quasi-proper operators vs.\ local compactness}\label{sec: quasi-proper}
We now start moving in the direction of a more refined rigidity result.
Recall that an operator is $T\colon \CHx\to\CHy$ is proper if so is $\csupp(T)$ (cf.\ \cref{def:controlled and proper operator}). As explained in \cref{sec: coarse geometric setup}, this condition is important because if $\CHx$ is locally admissible and $T$ is proper then $\Ad(T)$ maps locally compact operators to locally compact operators (see \rfCorAdProperLocalCpt for details). In particular, if $T$ is also assumed to be controlled then $\Ad(T)$ defines mappings of Roe algebras
\[
  \roestar\CHx\to\roestar\CHy
  \;\; \text{ and } \;\;
  \roecstar\CHx\to\roecstar\CHy.
\]
The main goal of this section is to prove that the converse holds as well, at least modulo \emph{quasification} (see \cref{thm: quasi-proper}). In turn, this will be useful to deduce the more refined rigidity results alluded to in the introduction.

\section{Quasi-properness}
We already recalled that if $\CHx$ and $\CHy$ are locally admissible then $T$ is proper if and only if for every bounded measurable $B\subseteq Y$ there is a bounded measurable $A\subseteq X$ such that $\chf BT=\chf BT\chf A$ (cf.\ \cref{rmk: proper-iff-proper}). The latter condition has the following natural \emph{quasification}.
\begin{definition} \label{def:quasi-proper}
  A bounded operator $T \colon \CHx \to \CHy$ is \emph{quasi-proper} if for all $B \subseteq Y$ bounded measurable and $\varepsilon > 0$ there is some bounded measurable $A \subseteq X$ such that
  \[
    \norm{\chfcX{A}T^*\chf{B}}= \norm{\chf{B}T\chfcX{A}}\leq \varepsilon.
  \]
\end{definition}

As a sanity check for \cref{def:quasi-proper}, it is worthwhile noting that an analogue of \cref{lem:approximation is controlled} holds true.
Namely, it is clear that $T$ is proper if and only if $\cappmap[T]{\delta}{F}{E}$ is proper for every $\delta \geq 0$. The following shows that if $T$ is quasi-proper then $\cappmap[T]{\delta}{F}{E}$ is still proper for every $\delta$ strictly greater than $0$.

\begin{lemma}
  Suppose that $\CHy$ is locally admissible, and let $T\colon\CHx\to\CHy$ be a quasi-proper operator. Then for every $\delta>0$ the approximating relation $\appmap[T]{\delta}{F}{E}$ is proper.
\end{lemma}
\begin{proof}
  Let $B_0\subseteq Y$ be bounded. By definition, 
  \[
    \paren{\appmap[T]{\delta}{F}{E}}^{-1}(B_0) = 
    \bigcup\left\{
      A\;\middle| \; \exists B,
      \begin{array}{c}
         \ B\times A\text{ meas.\ $(F\otimes E)$-bounded}, \\
         \norm{\chf{B}T\chf A}> \delta,\ B\cap B_0\neq\emptyset
      \end{array}
    \right\}.
  \]
  By local admissibility we may choose measurable bounded $B_1\subseteq Y$ containing $B_0$ and $F(B_0)$.
  By quasi-properness of $T$, we may then find a bounded measurable $A_1\subseteq X$ such that $\norm{\chf{B_1}T\chf{A_1} - \chf{B_1}T}\leq\delta$. Pick $B,A$ with $B\times A$ measurable $(F\otimes E)$-bounded, $\norm{\chf{B}T\chf A}> \delta$ and $B\cap B_0\neq\emptyset$.
  By construction, $B$ is contained in $B_1$. It hence follows
  \[
    \delta < \norm{\chf{B}T\chf A}\leq \norm{\chf{B_1}T\chf A},
  \]
  therefore $0 < \norm{\chf{B_1}T\chf{A_1}\chf A}$. This shows that $A_1\cap A\neq\emptyset$, so $A\subseteq E(A_1)$, which proves that $\paren{\appmap[T]{\delta}{F}{E}}^{-1}(B_0)\subseteq E(A_1)$ is bounded.
\end{proof}

We shall now turn to the main objective of this section (see \cref{thm: quasi-proper}), but we first record the following simple fact.
\begin{lemma}\label{lem:picking_orthogonal_vectors}
  Let $\crse X$ be coarsely connected, $B\subseteq Y$ measurable, and $H\leq\CH_B $ a finite dimensional subspace.
  Suppose $\varepsilon > 0$ is so that for every bounded measurable $A\subseteq X$ there is some unit vector $w \in \CH_B $ with $\norm{\chfcX{A}T^*(w)} > \varepsilon$.
  Then for every bounded measurable $A$ there exists $\bar w\in \CH_B  \cap {H}^\perp$ unit vector with $\norm{\chfcX{A}T^*(\bar w)} > \varepsilon/2$.
\end{lemma}
\begin{proof}
  Since $\crse X$ is coarsely connected, the projections $\chf{A}$ with $A$ measurable and bounded converge strongly to the identity as $A$ increases (cf.\ \cref{cor:supports of vectors almost contained in bounded}).
  Moreover, since $H$ is finite dimensional, it follows that the restrictions $\chf{A}T^*|_{H}$ converge in norm to $T^*|_{H}$.
  In particular, there is some large enough bounded measurable set $A_0 \subseteq X$ such that $\norm{\chfcX{A_0}T^*(u)}\leq \varepsilon/2$ for every unit vector $u\in {H}$. Let now $A \subseteq X$ be any given bounded measurable set. Enlarging it if necessary, we may assume that $A_0 \subseteq A$. Given $w\in \CH_B $ with $\norm{\chfcX{A}T^*(w)} > \varepsilon$ as in the hypothesis, let $w=w_{H} + w_\perp$ with $w_{H}\in {H}$ and $w_\perp\in {H}^\perp$. Observe that, in this situation,
  \[
    \norm{\chfcX{A} T^*(w_\perp)}
    \geq \norm{\chfcX{A} T^*(w)} - \norm{\chfcX{A} T^*(w_{H})}
    > \varepsilon - \frac{\varepsilon}{2}\,\norm{w_{H}} \geq \frac{\varepsilon}{2}.
  \]
  We may then let $\bar w\coloneqq w_\perp/\norm{w_\perp}$.
\end{proof}

Recall that $\cpop{E}$ denotes the set of operators of $E$-controlled propagation.
The following theorem is the main result of the section.

\begin{theorem} \label{thm: quasi-proper}
  Let $\crse X$ and $\crse Y$ be coarse spaces, with $\crse X$ coarsely connected, and let $\CHx$ and $\CHy$ be modules. Given an operator $T \colon \CHx \to \CHy$ and a gauge $\gauge\in \CE$, consider:
  \begin{enumerate}[label=(\roman*)]
    \item \label{item:thm:qp:qp} $T$ is quasi-proper.
    \item \label{item:thm:qp:lc-to-lc} $\Ad(T) (\lccstar{\CHx}) \subseteq \lccstar{\CHy}$.
    \item \label{item:thm:qp:roe to lc} $\Ad(T) (\cpop{\gauge}\cap\lccstar{\CHx}) \subseteq \lccstar{\CHy}$.
  \end{enumerate}
  Then~\cref{item:thm:qp:qp} implies~\cref{item:thm:qp:lc-to-lc} which (trivially) implies~\cref{item:thm:qp:roe to lc}. 
  If $\CHx$ is discrete, $\crse X$ is coarsely locally finite and $\gauge$ is a gauge witnessing both discreteness and local finiteness, then~\cref{item:thm:qp:roe to lc} implies~\cref{item:thm:qp:qp} as well, so the above are all equivalent.
\end{theorem}
\begin{proof}
  The fact that~\cref{item:thm:qp:qp} implies~\cref{item:thm:qp:lc-to-lc} is simple to show. Indeed, let $t \in \CB(\CHx)$ be any locally compact operator. For any bounded measurable $B \subseteq Y$ and any $\varepsilon > 0$, let $A = A(B, \varepsilon) \subseteq X$ be a bounded measurable set witnessing quasi-properness of $T$. In such case,
  \begin{align*}
    \norm{TtT^* \chf{B} - T t \chf{A} T^* \chf{B}} & \leq \norm{Tt}\norm{T^* \chf{B} - \chf{A} T^*\chf{B}}\leq \norm{Tt}\varepsilon.
  \end{align*}
  Observe that $t \chf{A}$ is compact, since $t$ is locally compact and $A$ is bounded. This shows that $TtT^* \chf{B}$ is approximated in norm by compact operators, and therefore is itself compact.
  An analogous argument shows that $\chf{B} TtT^*$ is also compact. Alternatively, we may also reduce to the previous case using the fact that the set of (locally) compact operators is closed under adjoints.

  \smallskip

  Rather that directly moving to \cref{item:thm:qp:roe to lc}$\Rightarrow$\cref{item:thm:qp:qp}, we shall first prove that if $\CHx$ is discrete and $\crse X$ is coarsely locally finite then \cref{item:thm:qp:lc-to-lc} already implies~\cref{item:thm:qp:qp}. This reduces the complexity of the arguments and is useful for \cref{rmk: lc implies proper when countable} below.
  Fix a locally finite discrete partition $X=\bigsqcup_{i\in I}C_i$.
  Suppose that $T$ is not quasi-proper, that is, there is some bounded set $B_0 \subseteq Y$ and $\varepsilon_0 > 0$ such that for all bounded $A \subseteq X$ there is some $w \in \CH_{B_0}$ of norm $1$ satisfying $\norm{\chf{A}T^*(w) - T^*(w)} > 2 \, \varepsilon_0$.
  We now make an inductive construction
  of a nested sequence of measurable bounded sets $A_n\subseteq X$, orthogonal vectors of norm one $w_n\in\CH_{B_0}$, and vectors $v_n\coloneqq T^*(w_n)$ such that
  \begin{itemize}
    \item $A_n$ is a (finite) union of $C_i$;
    \item $\norm{\chfcX{A_{n}}(v_n)} > {\varepsilon_0}$;
    \item $\norm{\chfcX{A_{n}}(v_i)}^2 \leq {\varepsilon_0}^2/2$ for every $i< n$.
  \end{itemize}

  \medskip\noindent\textit{Base Step}: fix an arbitrary bounded $A_1' \subseteq X$ and $w_1 \in \CH_{B_0}$ of norm $1$ with $\norm{\chf{A_1} (v_1) - v_1} > \varepsilon_0$, where $v_1 \coloneqq T^*(w_1)$. Let $A_1\coloneqq\bigsqcup\{C_i\mid C_i\cap A_1'\neq\emptyset\}$.
  
  \medskip\noindent\textit{Inductive Step}: let $A_1,\ldots, A_{n-1}$ and $w_1, \dots, w_{n-1} \in \CH_{B_0}$ be defined and let $v_i \coloneqq T^*(w_i)$ for $i = 1, \dots, n-1$. As $\crse X$ is connected, we may choose a bounded subset $A_n' \subseteq X$ containing $A_{n-1}$ and such that
  \[
    \norm{\chfcX{A_{n}'}(v_i)}^2 = \norm{\chf{A_{n}'}(v_i) - v_i}^2 \leq {\varepsilon_0^2}/{2}
  \]
  for every $i=1,\ldots, n-1$ (cf.\ \cref{cor:supports of vectors almost contained in bounded}). We then define $A_n\coloneqq\bigsqcup\{C_i\mid C_i\cap A_n'\neq\emptyset\}$, which is bounded by construction. Applying \cref{lem:picking_orthogonal_vectors} to ${H}_n \coloneqq \spn{w_1,\dots,w_{n-1}}$ we find some $w_n \in \CH_{B_0}\cap {H}_n^\perp$ of norm $1$ such that
  \begin{equation} \label{eq:thm:qp:vn-out-an}
    \norm{\chfcX{A_n}(v_n)} = \norm{v_n - \chf{A_n}(v_n)} > \varepsilon_0,
  \end{equation}
  where $v_n \coloneqq T^*(w_n)$.

  \smallskip

  Let $R_n\coloneqq A_{n+1}\smallsetminus A_n$. Since the $A_n \subseteq A_{n+1}$, it follows that the $\paren{R_n}_{n \in \NN}$ are pairwise disjoint.
  Observe that each $R_n$ is a finite union $R_n=\bigsqcup_{i\in I_n} C_i$. Since $(C_i)_{i\in I}$ is locally finite, the family $\paren{R_n}_{n \in \NN}$ is a fortiori locally finite.

  We may now consider the restrictions of the vectors $v_n$ to $R_n$, that is,
  \[
    \tilde{v}_n \coloneqq \chf{R_n} (v_n)
    = \chf{A_{n+1}} (v_n)- \chf{A_n} (v_n) = \chfcX{A_{n}} (v_n) - \chfcX{A_{n+1}} (v_n).
  \]
  It follows that
  \begin{equation}\label{eq:thm:qp:norm-vn-tilde}
   \norm{\tilde{v}_n}^2
    = \norm{\chfcX{A_{n}} (v_n)}^2 - \norm{\chfcX{A_{n+1}} (v_n)}^2
    > \varepsilon_0^2/2.
  \end{equation}
  Observe as well that $\tilde{v}_n$ and $\tilde{v}_m$ are orthogonal when $n \neq m$, since $R_n$ and $R_m$ are then disjoint.
  Let $\tilde p$ then be the projection onto the closed span of the vectors $\{\tilde v_n\}_{n \in \NN}$, \emph{i.e.}
  \[
   \tilde p\coloneqq \sum_{n\in\NN} p_{\tilde v_n}.
  \]
  Since $\paren{R_n}_{n\in\NN}$ is locally finite and each $p_{\tilde v_n}$ has rank one, the projection $\tilde p$ has locally finite rank, and in particular it is locally compact.
  \begin{claim}\label{claim:thm:qp:not-loc-cpt}
    The operator $\Ad(T)(\tilde p)$ is not locally compact.
  \end{claim}
  \begin{proof}
    Note that for every $n\in\NN$
    \[
    \scal{\Ad(T)(\tilde p)\chf{B_0} (w_n)}{w_n} = \scal{\tilde p T^*(w_n)}{\tilde p T^*(w_n)} = \norm{\tilde p (v_n)}^2.
    \]
    Since the vectors $\{\tilde v_n\}_{n \in \NN}$ are pairwise orthogonal, \cref{eq:thm:qp:norm-vn-tilde} implies
    \begin{equation}\label{eq:thm:norm pv_n}
      \norm{\tilde p (v_n)}^2 =
      \norm{p_{\tilde v_n}(v_n)}^2 + \sum_{m\neq n}\norm{p_{\tilde v_m}(v_n)}^2
      \geq \norm{\tilde v_n}^2>\frac{\varepsilon_0^2}{2}.
    \end{equation}
    Moreover, as $\paren{w_n}_{n \in \NN}$ are orthogonal vectors of norm one, it follows that $\Ad(T)(\tilde p)\chf{B_0}$ is not compact, and hence $\Ad(T)(\tilde p)$ is not locally compact.
  \end{proof}

  Since $\tilde p\in \lccstar{\CHx}$, \cref{claim:thm:qp:not-loc-cpt} proves the implication \cref{item:thm:qp:lc-to-lc} $\Rightarrow$ \cref{item:thm:qp:qp}.

  \smallskip

  We shall now further refine the above argument to show that \cref{item:thm:qp:roe to lc} $\Rightarrow$ \cref{item:thm:qp:qp}.
  We may further divide $\tilde v_n$ into its $C_i$-components, that is, consider
  \[ 
    u_{n,i} \coloneqq \chf{C_i} (\tilde v_n) = \chf{C_i} \chf{R_n} (v_n).
  \]
  We have that:
  \begin{itemize}
    \item $\scal{u_{n,i}}{u_{m,j}} = 0$ whenever $(n,i) \neq (m,j)$;
    \item $\sum_{i \in I_{n}} u_{n,i} = \chf{R_n} (v_n) = \tilde v_n$.
  \end{itemize}
  Consider now the projection onto the closed span of $\{u_{n,i}\}_{n \in \NN, i \in I_n}$, \emph{i.e.}\
  \[
    p \coloneqq \sum_{n \in \NN} \sum_{i \in I_n} p_{u_{n,i}}
  \]
  (some of the $u_{n,i}$ might be zero, in which case $p_{u_{n,i}}$ is just the zero operator).
  Since the family of $\paren{C_i}_{i\in I}$ is locally finite, we once again deduce that $p$ has locally finite rank and it is hence locally compact.
  Importantly, \cref{lem:join of controlled projection} shows that $p$ has propagation controlled by $\gauge$.
  
  We have now shown that $p$ belongs to $\cpop{\gauge}\cap\lccstar{\CHx}$.
  As in the proof of \cref{claim:thm:qp:not-loc-cpt}, observe that
  \[
  \scal{\Ad(T)(p)\chf{B_0} (w_n)}{w_n} = \scal{ p T^*(w_n)}{ p T^*(w_n)} = \norm{p (v_n)}^2.
  \]
  Moreover, note that $\tilde p = \tilde p p$, as $\tilde v_n = \sum_{i \in I_n} u_{n,i}$. It follows that
  \[
   \norm{p( v_n)}^2\geq \norm{\tilde p p( v_n)}^2
   =\norm{\tilde p( v_n)}^2
   >\frac{\varepsilon_0^2}{2},
  \]
  where the last inequality is just \cref{eq:thm:norm pv_n}. As in \cref{claim:thm:qp:not-loc-cpt}, since $\paren{w_n}_{n \in \NN}$ is an orthonormal family, it follows that $\Ad(T)(p)\chf{B_0}$ is not compact and hence $\Ad(T)(p)\notin \lccstar{\CHy}$.
\end{proof}

\begin{corollary}\label{cor: approximation is proper}
  Let $T\colon\CHx\to\CHy$ be weakly quasi-controlled, with $\CHx$ discrete.
  Suppose, moreover, that $\crse{X}$ is coarsely connected, locally finite and countably generated.
  If $\Ad(T)(\roestar{\CHx})\subseteq \lccstar{\CHy}$,
  then $\cappmap[T]{\delta}{F}{E}$ is a proper partial coarse map for any choice of $E\in\CE$, $F\in\CF$ and $\delta>0$.
\end{corollary}

\begin{remark}\label{rmk: lc implies proper when countable}
  The implication \cref{item:thm:qp:lc-to-lc} $\Rightarrow$ \cref{item:thm:qp:qp} in \cref{thm: quasi-proper} can be proved using different hypotheses. Specifically, instead of asking for discreteness and coarse local finiteness one may instead require $\crse X$ to be countably generated and $\CHx$ to be admissible. The proof is essentially the same, except that the $A_n$ are not required to be unions of $C_i$ (which are not defined anyway), but rather it is required that $E_n(A_n)\subseteq A_{n+1}$, where $(E_n)_{n \in \mathbb{N}}$ is a cofinal sequence in $\CE$. This requirement suffices to imply that $(R_n)_{n\in \NN}$ is locally finite, and the rest of the argument holds verbatim.

  Apart from this, the following examples show that the hypotheses of \cref{thm: quasi-proper} are sharp.
\end{remark}

\begin{example}
  It is straightforward to see that the coarse connectedness assumption of $\crse X$ in \cref{thm: quasi-proper} is necessary. Indeed, let $\crse X\coloneqq (\braces{-1,1},\{\Delta_X\})$ and let $\crse Y$ be a bounded coarse space. Let $\CHx= \ell^2(\{-1,1\}) = \CCC^2$ and $\CHy$ any Hilbert space of dimension at least $2$.
  Then no isometry $T\colon \CHx\to \CHy$ is quasi-proper. However, $\lccstar{\crse X} = \CB(\CHx) \cong \MM_2$, so $T$ trivially satisfies condition \cref{item:thm:qp:lc-to-lc} of \cref{thm: quasi-proper}.
\end{example}

\begin{example} \label{ex:non-cnt-gen-coarse-space-with-roe-compact}
  For the upwards implications of \cref{thm: quasi-proper} it is certainly necessary to assume coarse local finiteness of $\crse X$ (or countable generation, see \cref{rmk: lc implies proper when countable}). Indeed, let $\crse X = (X,\varcrs{\aleph_0})$ where $X$ is an uncountable set and $\varcrs{\aleph_0}$ is the coarse structure whose entourages are precisely the relations with at most countably many off-diagonal elements. Note that a subset of $X$ is bounded if and only if it is countable. Let $\CHx \coloneqq \ell^2(X)$. We claim that, in such case, $\lccstar{\CHx}=\CK(\CHx)$.

  It is enough to show that if $t\in\CB(\CHx)$ is \emph{not} compact then it is not locally compact either.
  Given such a $t$ there must be some $\delta>0$ and a sequence of unit vectors $\paren{v_n}_{n\in\NN}$ with $\norm{t(v_n)-t(v_m)}\geq \delta$ for every $n\neq m$.
  For every $n \in \NN$, there is some countable $A_n \subseteq X$ so that $\norm{\chf{A_n}t(v_n) - t(v_n)}\leq \delta/3$ (cf.\ \cref{cor:supports of vectors almost contained in bounded}).
  The union $A\coloneqq\bigcup_{n\in\NN}A_n$ is then countable, and hence bounded in $\crse X$. However, by the triangle inequality
  \[
   \norm{\chf{A} t(v_n)- \chf{A} t(v_m)} \geq \norm{t(v_n) - t(v_m)} - \frac{2}{3} \delta \geq \delta/3
  \]
  for every $n\neq m$. This shows that $\chf{A}t$ is not compact, hence $t$ is not locally compact, as desired.
\end{example}

We end the section with the following remarks.

\begin{remark} \label{rem:rigidity-fails-without-cnt-gen}
  Note that \cref{ex:non-cnt-gen-coarse-space-with-roe-compact} in fact shows that the main rigidity result \cref{thm: stable rigidity} fails if one of the spaces is allowed to be non-countably generated and non coarsely locally finite.
  Indeed, let $\crse{X}$ be as in \cref{ex:non-cnt-gen-coarse-space-with-roe-compact} and $\crse{Y} = \{{\rm pt}\}$ be trivial. We may then take the uniform module $\CHx\coloneqq \ell^2(X)$ and let $\CHy\coloneqq \ell^2(X)$ with the obvious representation $\chf{\rm pt} \coloneqq 1$. Then $\roecstar{\CHx} = \CK(\ell^2(X)) = \roecstar{\CHy}$, but $\crse{X}$ and $\crse{Y}$ are not coarsely equivalent.

  This issue will come up again in \cref{rem:multiplier-fails-without-cnt-gen}.
\end{remark}

\begin{remark}
  As explained before \cref{prop:adjoint is coarse inverse}, if $T$ and $T^*$ are both weakly quasi-controlled then $\cappmap[T]\delta FE$ is a partial coarse embedding. In particular, it is a proper partial coarse map.
  It is worthwhile observing that if $T^*$ is  weakly quasi-controlled then $T$ is in fact quasi-proper, therefore \cref{cor: approximation is proper} can be seen as an extension of this phenomenon.
\end{remark}

\begin{remark}
  If, in addition to all the other hypotheses in \cref{thm: quasi-proper}, $\crse X$ also has bounded geometry, the same methods as in \cref{thm: quasi-proper} also show that quasi-properness of $T \colon \CHx \to \CHy$
  is equivalent to $\Ad(T)$ sending the \Star{}algebra of ``uniformly finite rank operators of controlled propagation'' into $\lccstar\CHy$ (an operator $t$ has \emph{uniformly finite rank} if for every $F\in \CF$ the supremum of the ranks of the operators $t\chf B$ and $\chf B t$ with $B\subseteq Y$ measurable and $F$-controlled is finite). This stronger statement holds because, under the bounded geometry assumption, the family $(C_{i})_{i\in I}$ in the proof of \cref{thm: quasi-proper} is then \emph{uniformly} locally finite.
\end{remark}

\section{Rigidity vs.\ local compactness}\label{subsec: rigidity vs local compactness}
In this section we record an interesting application of \cref{thm: quasi-proper}. Namely, the proof of rigidity that we gave in \cref{thm: stable rigidity} has a shortcoming: given an isomorphism $\roecstar\CHx\cong\roecstar\CHy$, \cref{thm: stable rigidity} does not directly imply that the unitary $U$ inducing $\phi$ is weakly approximately controlled. This fact would become cumbersome in \cref{sec: strong approx and out,sec:consequences strong rigidity}.
Fortunately, this will not be an issue, because \cref{thm: quasi-proper} interacts particularly well with the uniformization phenomenon in \cref{thm: uniformization phenomenon}.

Recall the notion of one-vector approximately-control (cf.\ \cref{def: one-vector quasi-controlled mapping}). We then have the following.
\begin{proposition}\label{prop: quasi proper are one-vector app ctrl}
  Fix $T\colon \CHx \to \CHy$. Suppose that $\crse Y$ is coarsely locally finite, $\CHx$ is locally admissible, $\CHy$ is discrete, $\Ad(T)$ maps $\roestar\CHx$ into $\qlcstar\CHy$ and $\Ad(T^*)$ maps $\cpop{\gaugey}\cap\lccstar{\CHy}$ into $\lccstar{\CHx}$, where $\gaugey$ is a gauge witnessing discreteness and local finiteness.
  Then $\Ad(T)$ is one-vector approximately-controlled.
\end{proposition}
\begin{proof}
  Let $v\in(\CHx)_1$ be a vector of bounded support, say $\supp(v)\subseteq A_0$, and fix $E\in\CE$ and $\varepsilon>0$.
  Fix also $A\subseteq X$, a measurable and bounded neighborhood of $E(A_0)$.

  For every $t\in\cpop{E}$, $tp_v$ belongs to $\roestar\CHx$ and $tp_v= \chf Atp_v\chf A$. Since $\Ad(T)$ maps $\roestar\CHx$ into $\qlcstar\CHy$, we deduce that there is a coarsely connected component $Y_i\subseteq Y$ such that $\Ad(T)(tp_v)\in\CB(\CH_{Y_i})$ for every such $t$. In particular, if we let $T_i\coloneqq \chf{Y_i}T$ we then have $\Ad(T)(tp_v)=\Ad(T_i)(tp_v)$.

  Observe that $\Ad(T_i^*)$ a fortiori maps $\cpop{\gaugey}\cap\lccstar{\CHy}$ into $\lccstar{\CHx}$. Since $\crse Y_i$ is coarsely connected, we are now in the position of applying \cref{item:thm:qp:roe to lc} $\Rightarrow$ \cref{item:thm:qp:qp} of \cref{thm: quasi-proper} to deduce that $T_i^*$ is quasi-proper (inverting the role of $\crse X$ and $\crse Y$). This means that for any fixed $\delta>0$ there is a bounded measurable $B\subseteq Y$ such that
  \[
    \norm{(1-\chf B) T_i\chf A}\leq\delta.
  \]
  For any $t\in \cpop{E}$ we then have that the operator
  \[
    \Ad(T)(tp_v)=\Ad(T_i)(tp_v)= T_i\chf A tp_v\chf A T^*
  \]
  is within distance $2\delta\norm{T}\norm{t}$ of $\chf BT_i\chf A tp_v\chf A T^*\chf B$. Since the latter has support contained in $B\times B$, letting $\delta \coloneqq \varepsilon/(2\norm T)$ shows that $\Ad(T)(tp_v)$ is $(\varepsilon\norm t)$-$(B\times B)$-approximately controlled.
\end{proof}

Together with the uniformization phenomenon \cref{thm: uniformization phenomenon}, \cref{prop: quasi proper are one-vector app ctrl} has the following immediate consequence.

\begin{corollary}\label{cor: uniformization for roe algs}
  Fix $T\colon \CHx \to \CHy$. Suppose that $\crse Y$ is countably generated and coarsely locally finite, $\CHx$ and $\CHy$ are discrete, and we have
  \begin{align*}
    \Ad(T)(\roestar\CHx)& \subseteq \qlcstar\CHy, \;\; \text{and} \\
    \Ad(T^*)(\cpop{\gaugey}\cap\lccstar{\CHy})& \subseteq \lccstar{\CHx},
  \end{align*}
  where $\gaugey$ is a gauge witnessing discreteness and local finiteness.
  Then $T$ is weakly approximately controlled and quasi-proper.
\end{corollary}

\begin{corollary}\label{cor: iso of roe algs are weakly approximable}
  Let $\crse X,\crse Y$ be countably generated and coarsely locally finite, $\CHx$ and $\CHy$ discrete modules. If $U\colon \CHx \to \CHy$ is a unitary inducing an isomorphism $\Ad(U)\colon\roecstar{\CHx}\cong\roecstar\CHy$, then both $U$ and $U^*$ are weakly approximately controlled.
\end{corollary}

\begin{remark}
  The statements of \cref{prop: quasi proper are one-vector app ctrl,cor: uniformization for roe algs} are given in terms of $\cpop{\gaugey}\cap\lccstar{\CHy}$ to highlight the effective nature of the uniformization phenomenon (compare with \cref{rmk:effective uniformization}).
\end{remark}

\begin{remark}
  It is interesting to observe that \cref{cor: uniformization for roe algs} goes a long way towards the proof of \cstar{}rigidity assuming countable generation only on one of the spaces. Namely, the rest of the theory we developed shows that an isomorphism $\Phi\colon \roecstar\CHx\to\roecstar\CHy$ is implemented by a unitary $U$ and that its approximating relations give rise to coarsely surjective, proper, partial coarse maps.

  It one could show that such approximations are also coarsely everywhere defined and coarse embeddings without assuming countable generation of $\crse X$, this could be used to \emph{prove} that $\crse X$ must be countably generated.
  That is, this would show that isomorphisms of Roe-like \cstar{}algebras preserve the property of being countably generated, and would be a first rigidity statement for general coarse spaces.
  Note that this has indeed been proved if $\crse Y$ satisfies additional regularity assumptions, such as property A \cites{braga_farah_vignati_2022,braga2024operator}.
\end{remark}

\chapter{Refined rigidity: strong control notions}\label{sec: strong approx and out}
In this chapter we further push our techniques to generalize the ``Gelfand type duality'' studied in \cite{braga_gelfand_duality_2022}*{Theorem A}. In that paper, it is proven that, for a metric space $\crse{X} = (X, d)$ of bounded geometry with Yu's \emph{property A}~\cite{yu_coarse_2000}, there is a group isomorphism between:
\begin{itemize}
  \item the set of coarse equivalences of $\crse{X}$ (considered up to closeness);
  \item the set of outer automorphisms of the Roe algebra of $\crse{X}$. 
\end{itemize}
Using our rigidity techniques, we generalize this to arbitrary proper extended metric spaces. As usual, we will also work with $\cpcstar\variable$ and $\qlcstar\variable$. As it turns out, the outer automorphism groups of all these algebras are isomorphic to one another (cf.\ \cref{cor: all groups are iso}), which is perhaps surprising.

As a brief overview on this section, in \cref{subsec: strong approx-quasi control} we introduce a strengthening of the weak approximate and quasi versions of control for operators, and prove a stronger form of our rigidity result, which applies in a more restrictive---but still very general---setup. This is  the main technical result of the chapter (cf.\ \cref{thm: strong rigidity}). In the subsequent sections we illustrate various consequences of this more refined form of rigidity. In particular, in \cref{subsec: rig outer aut 1,subsec: rig outer aut 2}, we use it to establish a general form of duality between coarse equivalences and outer automorphisms of Roe-like \cstar{}algebras (cf.\ \cref{cor: isomorphic out aut,cor: all groups are iso}).

\section{Strong approximate/quasi control} \label{subsec: strong approx-quasi control}

Previously, we ``quasi-fied'' the notion of control for operators by generalizing the idea that $T$ is controlled if and only if $\Ad(T)$ maps operators of $E$-controlled propagation to operators of $F$-controlled propagation. This leads to the definition of weak approximate (resp.\ quasi-) control of operators.

On the other had, since we defined that an operator $T \colon \CHx \to \CHy$ is controlled if there is a controlled relation $R$ such that $\supp(T)\subseteq R$, there is a perhaps more straightforward strategy for ``quasi-fication''. Namely, by using approximate and quasi containment of supports as introduced in \cref{subsec: almost_n_quasi support}.
Recall that $T$ has support
\begin{enumerate} [label=(\roman*)]
  \item $\varepsilon$-approximately contained in $R$ if there is some $S \colon \CHx\to\CHy$ such that $\supp(S)\subseteq R$ and $\norm{S - T} \leq \varepsilon$;
  \item $\varepsilon$-quasi-contained in $R$ if $\norm{\chf {B} T\chf A} \leq \varepsilon$ for every choice of $R$-separated measurable subsets $A\subseteq X$ $B\subseteq Y$;  
\end{enumerate}
(see \cref{def:app_n_quasi-supp}). Inspired by these, we give the following definition.

\begin{definition} \label{def: strong approx-quasi R-ctrl}
  Let $R\subseteq Y\times X$ be a controlled relation. An operator $T\colon\CHx\to\CHy$ is:
  \begin{enumerate}[label=(\roman*)]
    \item \emph{$\varepsilon$-$R$-approximately controlled} if its has support $\varepsilon$-approximately contained in $R$.
    \item \emph{$\varepsilon$-$R$-quasi-controlled} if it has support $\varepsilon$-quasi-contained in $R$.
  \end{enumerate}
\end{definition}

\begin{definition} \label{def: strong approx-quasi ctrl}
  An operator $T\colon\CHx\to\CHy$ is:
  \begin{enumerate}[label=(\roman*)]
    \item \emph{strongly approximately controlled} if for every $\varepsilon>0$ there is a controlled relation $R \subseteq Y \times X$ such that $T$ is $\varepsilon$-$R$-approximately controlled.
    \item \emph{strongly quasi-controlled} if for every $\varepsilon>0$ there is a controlled relation $R \subseteq Y \times X$ such that $T$ is $\varepsilon$-$R$-quasi-controlled.
  \end{enumerate}
\end{definition}

As usual, $\varepsilon$-$R$-approximately controlled operators are always $\varepsilon$-$R$-quasi-controlled. In particular, strongly approximately controlled operators are always strongly quasi-controlled.

Strong approximate and quasi control are well behaved under finite sums and adjoint. Moreover, \cref{lem:supports and operations quasi-local} shows that if $T\colon\CHx\to\CHy$ is $\varepsilon_T$-$R_T$-approximately controlled, $S\colon\CHy\to\CHz$ is $\varepsilon_S$-$R_S$-approximately controlled and $\gaugey$ is a non-degeneracy gauge, then $ST$ is $(\varepsilon_S\norm T+\varepsilon_T\norm S)$-$(R_S\circ\gaugey\circ R_T)$-approximately controlled. If $\CHy$ is admissible and $\gaugey$ is an admissibility gauge, the same holds for the quasi-control.
In particular, compositions of strongly approximately/quasi-controlled operators on (admissible) modules remain strongly approximately/quasi-controlled.
The next lemma justifies using the adjective ``strong'' in \cref{def: strong approx-quasi R-ctrl,def: strong approx-quasi ctrl}.

\begin{lemma}\label{lem: str quasi-controlled is quasi-controlled}
  Any strongly approximately controlled operator $T\colon\CHx\to\CHy$ is weakly approximately-controlled.
  If $\CHx$ is admissible, any strongly quasi-controlled operator $T\colon\CHx\to\CHy$ is weakly quasi-controlled.
\end{lemma}
\begin{proof}
  Fix $\varepsilon>0$ and let $R \subseteq Y\times X$ be a controlled relation witnessing strong approximate (resp.\ quasi-) control for $T$ with respect to $\varepsilon$. Fix also a non-degeneracy (resp.\ admissibility) gauge $\gauge\in\CE$ for $\CHx$.

  Given an arbitrary controlled entourage $E \in \CE$, we have to show that there is some $F \in \CF$ such that, for any $s\in\cpstar\CHx$ of $E$-controlled propagation, the image $\Ad(T)(s)$ is $\varepsilon$\=/$F$\=/approximable (resp.\ $\varepsilon$\=/$F$\=/quasi-local).
  This is readily done: the composition estimates of \cref{lem:supports and operations quasi-local} directly imply that $TsT^*$ has support $(2\varepsilon\norm{T}\norm{s})$-approximately (resp.\ quasi-) contained in $F\coloneqq R\circ \gauge\circ E\circ\gauge\circ \op R$.
  Since $\gauge\circ E\circ\gauge\in\CE$ and $R$ is controlled, we see that $F$ is indeed a controlled entourage.
\end{proof}

\begin{remark}
  It seems unlikely that the converse of \cref{lem: str quasi-controlled is quasi-controlled} holds true in general. However, we shall see below (cf.\ \cref{prop: quasi-control vs str quasi-control}) that it does hold in certain cases that are important in applications, particularly with regards to \cref{cor: all groups are iso}.
\end{remark}

The coarse support of a controlled operator played a key role in our approach to the theory of rigidity we developed up to this point. To proceed further, the analogous constructs in the \emph{approximate} and \emph{quasi} settings that are defined below will be just as important.
\begin{definition}\label{def: approximate and quasi support}
  Given $T \colon \CHx \to \CHy$ and a coarse subspace $\crse {S\subseteq Y\times X}$, we define:
  \begin{enumerate}[label=(\roman*)]
    \item \label{def: approximate and quasi support:contains-heart} $\crse {S}$ \emph{contains the approximate support} of $T$ (denoted $\acsupp(T)\crse \subseteq \crse S$) if for every $\varepsilon>0$ there is an $R\subseteq Y\times X$ such that $T$ has support $\varepsilon$-approximately contained in $R$ and $\crse{R\subseteq S}$.
    \item \label{def: approximate and quasi support:a-supp} $\crse S$ is the \emph{approximate support of $T$} if it satisfies \cref{def: approximate and quasi support:contains-heart} and is the smallest with this property (\emph{i.e.}\ $\crse{S}\crse\subseteq\crse{S'}$ whenever $\crse{S'\subseteq Y\times X}$ is as in \cref{def: approximate and quasi support:contains-heart}). In such case, we write $\acsupp(T)\coloneqq \crse S$.
  \end{enumerate}
  The \emph{containment of the quasi support} (denoted $\qcsupp(T)\crse \subseteq \crse S$) and the \emph{quasi-support} (denoted $\qcsupp(T)$) are defined as above replacing `approximately' by `quasi' everywhere.
\end{definition}

\begin{remark}
  As was the case for coarse compositions, we are \emph{not} claiming that approximate/quasi supports exist in general. When they do not exist, writing
    \[
      \acsupp(T)\crse\subseteq\crse R 
      \;\;(\text{resp.\ }\qcsupp(T)\crse\subseteq\crse R).
    \]
    is a slight abuse of notation, which should be understood as an analogue of the notion of containment of supports of operators (cf.\ \cref{def:supp-operator}). 
    When they exist, approximate/quasi supports are clearly unique.
\end{remark}

\begin{remark}\label{rkm: about approximate and quasi support: contained in idx}
  Observe that $T\in\CB(\CHx)$ is approximable (\emph{i.e.}\ belongs to $\cpcstar\CHx$) if and only $\acsupp(T)\crse{\subseteq}\cid_{\crse X}$. Similarly, it is quasi-local (\emph{i.e.}\ belongs to $\qlcstar\CHx$) if and only $\qcsupp(T)\crse{\subseteq}\cid_{\crse X}$.
\end{remark}

If $\acsupp(T)\crse{\subseteq S}$, then $\acsupp(T^*)\crse{\subseteq} \op{\crse{S}}$. Moreover, if $\acsupp(T)$ does exist, then so does $\acsupp(T^*)$ (and $\acsupp(T^*)=\op{\acsupp(T)}$).
It is also clear that if $\acsupp(T_{1})\crse{\subseteq S_{1}}$ and $\acsupp(T_{2})\crse{\subseteq S_{2}}$ then $\acsupp(T_{1}+T_{2})\crse{\subseteq S_{1} \cup S_{2}}$. The same considerations hold for $\qcsupp(T)$ as well.

The composition of coarse supports is also well behaved, at least insofar as their coarse composition is well-defined (cf.\ \cref{def: coarse composition}).
  
\begin{lem}\label{lem: composition of approximate and quasi support}
  Suppose that $T_1\colon\CHx\to\CHy$, $T_2\colon\CHy\to\CHz$ are operators with
  \[
  \acsupp(T_1)\crse{\subseteq S}_1
  \quad\text{and}\quad 
  \acsupp(T_2)\crse{\subseteq S}_2.
  \]
  where $\crse{S}_1\crse{\subseteq Y\times X}$ and $\crse{S}_2\crse{\subseteq Z\times Y}$
  admit a coarse composition. Then 
  \[
    \acsupp(T_2T_1) \crse\subseteq\crse S_2\crse \circ\crse S_1.
  \]
  If $\CHy$ is admissible, the same holds for quasi-supports as well.
\end{lem}
\begin{proof}
  Fix $\varepsilon>0$, and choose relations $R_i$ such that $T_i$ has support $\delta$-approximately/quasi contained in $R_i$, for some $\delta>0$ much smaller than $\varepsilon$. Then \cref{lem:supports and operations quasi-local} 
   shows that $T_2\circ T_1$ has support $\delta(\norm{T_1}+\norm{T_2})$-approximately/quasi contained in $R_2\circ\gauge\circ R_1$, where $\gauge$ is any non-degeneracy/admissibility gauge for $\CHy$.
   Since $[R_2]\crse{\subseteq S_{2}}$ and $[\gauge\circ R_1] \crse{\subseteq S_{1}}$, this proves the lemma.
\end{proof}

In particular, \cref{lem: composition of approximate and quasi support} applies when $\crse S_1$, $\crse S_2$ are partial coarse maps and $\cim(\crse S_1)\crse \subseteq\cdom(\crse S_2)$, cf.\ \cref{lem: composition exists}.

\smallskip

Having approximate/quasi supports at hand, the characterization of a controlled operator as one $T$ with controlled coarse support $\csupp(T)$ suggests even more ways of ``quasi-fying'' the definition of control. Namely, we may define:

\begin{definition}\label{def:effective control}
  An operator $T\colon\CHx\to\CHy$ is:
  \begin{enumerate}[label=(\roman*)]
    \item \emph{effectively approximately controlled} if $\acsupp(T)$ is a partial coarse map $\crse{S\colon X\to Y}$;
    \item \emph{almost effectively approximately controlled} if $\acsupp(T)\crse{\subseteq S}$ for some partial coarse map $\crse{S\colon X\to Y}$.
  \end{enumerate}
  \emph{Effective quasi-control} and \emph{almost effective quasi-control} are defined analogously.
\end{definition}

\begin{remark}
  As usual, an important aspect of \cref{def:effective control} above is that \(\acsupp(T)\) needs not exist. When it does exist, the almost effective and effective versions are equivalent.
\end{remark}

There is an obvious hierarchy among all these definitions. Namely:
\begin{equation}
\text{effective}\implies
\text{almost effective}\implies
\text{strong}\implies
\text{weak}. \label{eq:hierarchy}
\end{equation}
Fortunately, we shall soon see that for operators inducing isomorphisms of Roe-like \cstar{}algebras all of the above are actually equivalent.
It is not clear to us which of the converse implications are true in general.

\begin{remark}
  The effectively approximately/quasi controlled operators are those with which it is the most natural to associate mappings at the level of spaces. Since the uniformization phenomenon yields weakly approximately/quasi controlled operators, the final part of the proof of \cstar{}rigidity can be framed as a journey to reverse all the implications in the hierarchy \eqref{eq:hierarchy}.
\end{remark}

Since we are dealing with coarse subspaces, which are defined as smallest elements and may or may not exist, it is convenient to slightly extend the meaning of the coarse containment symbol $\crse \subseteq$ by saying that
$\crse{A\subseteq B}$ if for every $\crse C$ with $\crse{B\subseteq C}$ we also have $\crse{A\subseteq C}$.
For example, using this convention we may rephrase \cref{lem: composition of approximate and quasi support} by simply saying
\[
\acsupp(T_2T_1)\crse \subseteq\acsupp(T_2)\crse \circ \acsupp(T_1),
\]
without worrying whether the approximate supports and their coarse composition are actually defined.
Similarly, since approximate containment implies quasi-containment we see that:

\begin{lemma}\label{lem: quasi-supp in approx-supp}
  For every operator $T\colon\CHx\to\CHy$
  \[
  \qcsupp(T)\crse \subseteq\acsupp(T).
  \]
\end{lemma}
\begin{proof}
  By definition, if $\acsupp(T)\crse{\subseteq S}$ then for every $\varepsilon>0$ there is $R\crse{\subseteq S}$ that $\varepsilon$-approximately contains the support of $T$. Hence $T$ has support $\varepsilon$-quasi-contained in $R$, which exactly shows that $\qcsupp\crse{\subseteq S}$.
\end{proof}

\begin{remark}
  In the following we will often encounter operators for which $\qcsupp(T)\crse = \acsupp(T)$. However, this is not a general phenomenon. In fact, if $T\in\qlcstar\CHx\smallsetminus\cpcstar\CHx$ is some quasi-local but not approximable operator (such operators exist by \cite{ozawa-2023}), then $\qcsupp(T)\crse \subseteq \cid_{\crse X}$ but $\acsupp(T)$ is not coarsely contained in $\cid_{\crse X}$ (cf.\ \cref{rkm: about approximate and quasi support: contained in idx}).
\end{remark}

The quasi support interacts very well with the approximating relations defined in \cref{subsec: approximations}.

\begin{lemma}\label{lem: approximating relation contained in quasi-support}
  Given $T\colon \CHx\to \CHy$, $E\in\CE$, $F\in\CF$, and $\delta>0$, we have
  \[
    \cappmap[T]{\delta}{F}{E} \crse{\subseteq}\qcsupp(T).
  \]
\end{lemma}
\begin{proof}
  Fix any $\crse{S \subseteq Y\times X}$ with $\qcsupp(T)\crse{\subseteq S}$. By definition, there must be some $R\subseteq Y\times X$ such that $T$ has support $\delta$-quasi-contained in $R$ and $R\crse{\subseteq S}$.
  Let $B\times A$ is one of the defining blocks of $\appmap[T]{\delta}{F}{E}$ (recall \cref{def:approximating crse map}), then $(B\times A)\cap R \neq \emptyset$ by definition of quasi-containment of the support. 
  It follows that $\appmap[T]{\delta}{F}{E}\subseteq F\circ R\circ E \crse{\subseteq S}$.
\end{proof}

Having established this, in the following section we will need to work to show that in the cases of interest in \cstar{}rigidity one can choose $F,E,\delta>0$ such that
\[
  \qcsupp(T)\crse\subseteq\cappmap[T]{\delta}{F}{E},
\]
thus proving effective $\, \Longleftrightarrow \,$ weak in \eqref{eq:hierarchy}.

\begin{example}
  It is easy to see that in general there need not be $F,E,\delta>0$ such that $\qcsupp(T)\crse = \cappmap[T]{\delta}{F}{E}$. For instance, let $\crse X = \mathbb N$ seen as a disjoint union of points (\emph{i.e.}\ equipped with the extended metric $d(n,m)=\infty$ for every $n\neq m$) and let $T\in\CB(\ell^2(\mathbb N))$ be multiplication by the function $1/n\in\ell^\infty(\mathbb N)$.
  In this case, $\qcsupp(T)=\acsupp(T)=\cid_{\crse X}$, while $\appmap[T]{\delta}{F}{E}$ is always finite. Indeed, $\appmap[T]{\delta}{F}{E}$ only contains pairs $(m,m)$ such that \(1/m \geq \delta\).

  It seems then tempting to take unions $\bigcup \appmap[T]{\delta}{F}{E}$ to attempt to construct approximate/quasi supports. This is however a delicate matter. Consider $\crse X = \NN\times \NN$, where the first copy of $\mathbb N$ is a disjoint union as before, while the second one is given the usual metric. Coarsely, this space looks like countably many disconnected copies of half-lines.
  Let $T\in\CB(\ell^2(\NN\times\NN))$ be the operator sending $\delta_{n,m}$ to $\delta_{n,m+n}/n$.
  It is then still true that $\qcsupp(T)=\acsupp(T)=\cid_{\crse X}$. However, taking a union of $\appmap[T]{\delta}{\Delta}{\Delta}$ with $\delta\to 0$ results in the relation $\{((n,m),(n,n+m))\mid n,m\in \NN\}$, which does not belong to $\CE$ and is hence not coarsely contained in $\cid_{\crse X}$.
\end{example}

\section{A refined rigidity theorem}
Our interest in approximate and quasi supports of operators stems from the following, which is the main theorem of the chapter.  Observe that it is a strong (but ample and non-stable) version of \cref{thm: stable rigidity}.

\begin{theorem}[Refined Rigidity]\label{thm: strong rigidity}
  Let $\crse X$ and $\crse Y$ be countably generated coarse spaces; $\CHx$ and $\CHy$ ample discrete modules.
  Suppose that $\phi \colon \roeclikeone{\CHx} \to \roecliketwo{\CHy}$ is a \Star{}isomorphism, and that (at least) one of the following holds:
  \begin{enumerate}[label=(\roman*)]
    \item \label{thm: strong rigidity:cp-ql} one of $\roeclikeone{\CHx}$ and $\roecliketwo{\CHy}$ is unital,
    \item \label{thm: strong rigidity:roe}  $\crse{X}$ and $\crse{Y}$ are coarsely locally finite.
  \end{enumerate}
  Then the following assertions hold.
  \begin{enumerate}[label=(\alph*)]
    \item\label{thm: strong rigidity:one} $\roeclikeone\variable=\roecliketwo\variable$ (unless $\cpcstar\CHx = \qlcstar\CHx$ and $\cpcstar\CHy = \qlcstar\CHy$, in which case they are of course interchangeable).
    \item\label{thm: strong rigidity:two} $\phi = \Ad(U)$ where $U\colon\CHx\to\CHy$ is a unitary such that both $U$ and $U^*$ are effectively quasi-controlled and the quasi supports
    \[
      \qcsupp\left(U\right) \colon \crse X \to \crse Y \; \text{and} \; \qcsupp\left(U^*\right) \colon \crse Y\to\crse X
    \]
    are coarse inverses of one another. In particular, they give a coarse equivalence.
    \item\label{thm: strong rigidity:three} If either $\roeclikeone{\CHx}$ or $\roecliketwo{\CHy}$ is not $\qlcstar{\variable}$ then $U$ and $U^*$ are also effectively approximately controlled and
    \[
    \acsupp(U)=\qcsupp(U)
    \quad\text{ and }\quad
    \acsupp(U^*)=\qcsupp(U^*).
    \]
  \end{enumerate}
\end{theorem}

\begin{remark}
  The statement of \cref{thm: strong rigidity} is using rather heavily that ``Roe-like'' is defined as a specific list of \cstar{}algebras. It is not clear to us how a more conceptual formulation of the theorem may look like.
\end{remark}

The heart of the proof of \cref{thm: strong rigidity} is the following criterion, which allows us to upgrade weak approximate (resp.\ quasi-) control to its strong counterpart.
\begin{proposition}\label{prop: quasi-control vs str quasi-control}
  Let $\CHx$ be ample and discrete, $\CHy$ admissible, $T\colon\CHx\to\CHy$ a weakly approximately (resp.\ quasi-) controlled isometry such that $\cappmap[T]{\delta}{F}{E}$ is coarsely everywhere defined for some choice of $\delta>0$, $E\in\CE$ and $F\in\CF$.
  Then $T$ is effectively approximately (resp.\ quasi-) controlled and 
  \[
  \cappmap[T]{\delta}{F}{E}=\acsupp(T)=\qcsupp(T)
  \quad (\text{resp.}\ \cappmap[T]{\delta}{F}{E}=\qcsupp(T)).
  \]
\end{proposition}

For now, we treat \cref{prop: quasi-control vs str quasi-control} as a black box, and proceed with the proof of \cref{thm: strong rigidity}.

\begin{proof}[Proof of \cref{thm: strong rigidity}]
  The first steps of the proof follow that of \cref{thm: stable rigidity}. By \cref{prop: isos are spatially implemented}, $\phi=\Ad(U)$ for some unitary $U\colon\CHx\to\CHy$. However, unlike what we did in \cref{thm: stable rigidity}, we now aim to show that the whole of $U$ and $U^*$ are is quasi-controlled, without the need to restrict to submodules.

  Of course, if one of $\roeclikeone{\CHx}$ and $\roecliketwo{\CHy}$ is unital, then the other one must be unital as well. In particular, they both contain $\cpcstar\variable$ and are contained in $\qlcstar\variable$. We thus deduce from \cref{thm: uniformization} that $U$ and $U^*$ are weakly quasi-controlled (and even weakly approximately controlled if they send $\cpcstar{\variable}$ into $\cpcstar{\variable}$).
  
  The non-unital case requires an alternative argument (this is the shortcoming mentioned in \cref{subsec: rigidity vs local compactness}).
  In this case both  $\roeclikeone\variable $ and $\roecliketwo\variable$ must be $\roecstar\variable$.
  Weak approximability of $U$ and $U^*$ is then given by \cref{cor: iso of roe algs are weakly approximable}.
  
  \smallskip
  Now, \cref{thm: rigidity quasi-proper operators}---or, rather, \cref{cor: rigidity controlled unitaries}---yields $\delta>0$, $E\in \CE$ and $F\in\CF$ such that $\cappmap[U]\delta F E$ and $\cappmap[U^*]\delta E F$ are coarse equivalences inverse to one another. Assertion \cref{thm: strong rigidity:two} directly follows from \cref{prop: quasi-control vs str quasi-control}.
  
  \smallskip
  For \cref{thm: strong rigidity:three}, observe that if $\roecliketwo{\variable}$ is not $\qlcstar\variable$ then---always by \cref{prop: quasi-control vs str quasi-control}---$U$ is effectively approximately controlled with $\acsupp(U) = \qcsupp(U) = \cappmap[U]\delta F E$.
  Since approximate/quasi supports are well behaved under taking adjoints and $\op{\cappmap[U]\delta F E}=\cappmap[U^*]\delta E F$, we obtain as a consequence that $U^*$ must also be effectively approximately controlled with 
  \[
    \acsupp(U^*) = \qcsupp(U^*) = \op{\acsupp(U)} = \cappmap[U^*]\delta E F.
  \]
  If it is the case that $\roeclikeone{\variable}$ is not $\qlcstar\variable$, the symmetric argument applies.

  \smallskip
  It only remains to complete the proof of \cref{thm: strong rigidity:one}. The case of $\roecstar\variable$ is clear, as it is the only non-unital option.
  If $\roeclikeone\CHx = \cpcstar\CHx$, then $\phi^{-1}$ maps $\cpcstar\CHy$ into $\cpcstar\CHx$. It then follows that $U^*$ is effectively approximately controlled, with $\acsupp\left(U^*\right) = \cappmap[U^*]\delta E F$. Since taking adjoints is highly compatible with approximate supports, it follows that
  \[
    \acsupp\left(U\right) = \cappmap[U]\delta F E.
  \]
  Since the latter is a coarse map, we deduce that $U$ itself is effectively approximately controlled.
  It follows that $\Ad(U)$ maps both $\cpcstar\CHx$ and $\qlcstar\CHx$ into $\qlcstar\CHy$.
  If $\cpcstar\CHx\neq \qlcstar\CHx$, it then follows that $\cpcstar\CHy\neq \qlcstar\CHy$ as well and that $\roecliketwo\CHy = \cpcstar\CHy$.
  A symmetric argument shows that if $\roeclikeone\CHx = \qlcstar{\CHx}$, then $\roecliketwo\CHy = \qlcstar\CHy$.
\end{proof}

The following consequence is immediate, and will be used later.

\begin{cor}\label{cor: unitaries weak control adjoint iff effective ce}
  Let $\crse X$ and $\crse Y$ be countably generated coarse spaces and $U\colon \CHx\to\CHy$ a weakly approximately (resp.\ quasi) controlled unitary among ample discrete modules. The following conditions are equivalent.
  \begin{enumerate}[label=(\alph*)]
    \item $U^*$ is weakly approximately (resp.\ quasi) controlled;
    \item $U$ is effectively approximately (resp.\ quasi) controlled and 
    \[
      \acsupp(U)\crse{\colon X\to Y} \quad (\text{resp.}\ \qcsupp(U)\crse{\colon X\to Y}) 
    \]
    is a coarse equivalence.
  \end{enumerate}
\end{cor}
\begin{proof}
  One implication is immediate, because if $\acsupp(U)\crse{\colon X\to Y}$ is a coarse equivalence then $\acsupp(U^*)=\op{\acsupp(U)}$ is a coarse map, showing that $U^*$ is effectively (whence weakly) approximately controlled.
  Vice versa, if both $U$ and $U^*$ are weakly approximately controlled, then $\Ad(U)$ induces an isomorphism $\cpcstar\CHx\to\cpcstar\CHy$, and the claim follows from \cref{thm: strong rigidity}~\cref{thm: strong rigidity:three}.
  The argument for the ``quasi'' statement is completely analogous.
\end{proof}

Unpacking the notation, the (proof of) the approximable case of \cref{thm: strong rigidity} implies the following.

\begin{corollary}[cf.\ \cite{rigidIHES}*{Theorem 4.5}]\label{cor:  aut is limit of controlled}
  Let $\crse X$ and $\crse Y$ be countably generated coarse spaces; $\CHx$ and $\CHy$ ample discrete modules and $U\colon \CHx \to \CHy$ is a unitary.
  Suppose that either:
  \begin{enumerate}[label=(\alph*)]
    \item $\Ad(U)\colon \cpcstar{\CHx} \to \cpcstar{\CHy}$ is a \Star{}isomorphism; or
    \item $\crse X$ and $\crse Y$ are coarsely locally finite and $\Ad(U)\colon \roecstar{\CHx} \to \roecstar{\CHy}$ is a \Star{}isomorphism.
  \end{enumerate}
  Then $U$ is the norm-limit of operators that are coarsely supported on a (fixed) coarse equivalence $\crse{f\colon X\to Y}$.
\end{corollary}

\begin{remark}\label{rmk: functoriality of qcsupp}
  Combining \cref{lem: composition of approximate and quasi support} with \cref{lem: functions with same domain coincide}, we see that \cref{thm: strong rigidity} implies that if $\CHx$ is an ample module then $\qcsupp$ defines a group homomorphism from $\aut(\roeclike\CHx)$ into the group of coarse equivalences $\coe{\crse{X}}$ (see \cref{def: CE and CUni} below).
  This fact will be important in the applications of \cref{thm: strong rigidity} explained in \cref{subsec: rig outer aut 1}.

  More generally, let $\roeclike\variable$ denote $\cpcstar{\variable}$ or $\qlcstar\variable$ and consider the category \Cat{IsoCMod$_{\CR}$} whose objects are ample discrete coarse geometric modules over countably generated coarse spaces and where morphisms from $\CHx$ to $\CHy$ are the isomorphisms $\roeclike\CHx\to\roeclike\CHy$. Then assigning $\crse X$ to $\CHx$ and $\qcsupp(U)$ to $\Ad(U)$ defines a functor
  \[
    \Cat{IsoCMod}_{\CR}\to\Cat{Coarse}
  \]
  to the category of coarse spaces (the image of this functor only consists of countably generated coarse spaces and coarse equivalences). The same holds for $\roeclike\variable = \roecstar\variable$ under the additional assumption that the coarse spaces be coarsely locally finite.
\end{remark}

\begin{remark}
  There is an interesting categorical point of view on the rigidity of spaces using \cstar{}categories \cite{Krutoy2025categorical}. Ignoring technicalities arising from choosing cardinals and only considering coarse spaces and modules of said coarse cardinality and ampleness, one may associate to any coarse space $\crse X$ the category $\CA(\crse X)$ of coarse geometric modules over $\crse X$, where the morphisms from $\CHx$ to $\CHx'$ are operators $T\colon \CHx\to\CHx'$ with $\acsupp(T)\crse\subseteq \cid_{\crse X}$.
  This category is actually a \cstar{}category.

  Following the definitions, we see that there is $\cpcstar\CHx$ is nothing but the set of endomorphisms of $\CHx$ as an object in $\CA(\crse X)$. In turn, the rigidity theorem can be interpreted as saying that one may associate to a \Star{}equivalence of categories $F\colon \CA(\crse X)\to \CA(\crse Y)$ a coarse equivalence $\crse{X\cong Y}$. We refer \cite{Krutoy2025categorical} for details and other related facts concerning e.g.\ full and faithful functors and coarse embeddings.
\end{remark}

\section{Upgrading weak to strong control}
We shall now proceed with the proof of \cref{prop: quasi-control vs str quasi-control}. In order to show it, we first need a few preliminary results.
Let $\CHx$ be discrete, and let $X = \bigsqcup_{i \in I} A_i$ be an $\gauge$\=/controlled discrete partition of $X$. 
Given a finite set of indices $J\subseteq I$, let $\Lambda^J$ be the set of functions $\lambda$ associating to each $j\in J$ a finite dimensional vector subspace of $\CH_{A_j}$.
Let $p_{\lambda(j)}$ be the projection onto $\lambda(j)$, and let
\[
  p_{J,\lambda}\coloneqq\sum_{j\in J} p_{\lambda(j)}
\]
be the projection onto the span of $\braces{\lambda(j)}_{j \in J}$. 
Ordering $\Lambda^J$ by pointwise inclusion makes each $(p_{J,\lambda})_{\lambda\in\Lambda^J}$ into a net.
Let $\Lambda_{\rm fin}\coloneqq\bigcup\Lambda^J$ be the set of all pairs $(J,\lambda)$ with $J$ finite and $\lambda$ defined on $J$.
We order $\Lambda_{\rm fin}$ by extension, so that $(p_{J,\lambda})_{(J,\lambda)\in\Lambda_{\rm fin}}$ becomes a net as well.\footnote{\,
The net $(p_{J,\lambda})_{(J,\lambda)\in\Lambda_{\rm fin}}$ is a close relative of the approximate unit constructed in \rfApproxUnit and used in \cref{lem: exists submodule approximating operator}. The difference is that here we only consider finitely supported functions $\lambda$.}
\begin{lemma} \label{lemma:discrete-module-admits-nice-projections}
  The net $\paren{p_{J,\lambda}}_{(J,\lambda)\in\Lambda_{\rm fin}}$ described above satisfies:
  \begin{enumerate}[label=(\roman*)]
    \item \label{lemma:discrete-module-admits-nice-projections:2} $p_{\lambda(j)} \leq \chf{A_j}$;
    \item \label{lemma:discrete-module-admits-nice-projections:1} $p_{J,\lambda}$ has finite rank for all $(J,\lambda)\in\Lambda_{\rm fin}$;
    \item \label{lemma:discrete-module-admits-nice-projections:3} $\supp(p_{J,\lambda})\subseteq \gauge$ for all $(J,\lambda) \in \Lambda_{\rm fin}$;
    \item \label{lemma:discrete-module-admits-nice-projections:5} $(p_{J,\lambda})_{(J,\lambda)\in\Lambda_{\rm fin}}$ strongly converges to $1$.
  \end{enumerate}
\end{lemma}
\begin{proof}
  Assertions~\cref{lemma:discrete-module-admits-nice-projections:2,lemma:discrete-module-admits-nice-projections:1} hold by construction, while~\cref{lemma:discrete-module-admits-nice-projections:3} follows from \cref{lemma:supp-vector-a-times-a,lem:join of controlled projection}.
  Lastly, \cref{lemma:discrete-module-admits-nice-projections:5} follows from discreteness.
\end{proof}

\begin{lemma}\label{lem: find unitary for strong approximate}
  Let $X=\bigsqcup_{i\in I} A_i$ be a discrete partition, $\gaugey$ an admissibility gauge for $\CHy$,  $R\subseteq Y\times X$ a relation, and $T\colon\CHx\to\CHy$ an operator.
  Suppose that:
  \begin{itemize}
    \item $\CH_{A_i}$ has infinite rank for every $i\in I$;
    \item $\norm{\chf{C_i} T\chf{A_i}}\leq \varepsilon$ for every $i\in I$ and measurable $C_i\subseteq Y$ that is $R$-separated from $A_i$.
  \end{itemize}
  Then, for every $(J,\lambda)\in\Lambda_{\rm fin}$ there is a $\diag(A_i\mid i\in I)$-controlled unitary $s_{J,\lambda}\in \cpstar\CHx$ such that $Ts_{J,\lambda}p_{J,\lambda}$ has support $\varepsilon$-approximately controlled in $\gaugey\circ R$.
\end{lemma}
\begin{proof}
  To begin with, choose for every $i\in I$ a measurable $\gaugey$-controlled thickening $B_i$ of $R(A_i)$ and let $C_i\coloneqq Y\smallsetminus B_i$. 
  In particular, $\norm{\chf{C_i} T\chf{A_i}}\leq \varepsilon$ for every $i\in I$.

  Now, enumerate the (finite) index set $J$. For each $i\in J$, we iteratively define a unitary operator $s_i \in \CB(\CH_{A_i})$ as follows. Since the span
  \[
    V_i\coloneqq \angles{T^*\chf{C_i}\chf{C_j}Ts_jp_{\lambda(j)}(\CHx) \mid j\in J,\ j < i}
  \]
  is finite dimensional, we may (arbitrarily) choose a unitary $s_i \in \CB(\CH_{A_i})$ such that $s_ip_{\lambda(i)}s_i^*$ is orthogonal to $V_i$. 
  That is,
  \[
    (s_ip_{\lambda(i)}s_i^*)(T^*\chf{C_i}\chf{C_j}Ts_jp_{\lambda(j)}) = 0
  \]
  for every $j<i$.
  Note that here is where we used that $\CH_{A_i}$ has infinite rank.

  This choice is made so that the operators $(\chf{C_i}Ts_ip_{\lambda(i)})_{i\in J}$ are pairwise orthogonal. In fact, given $j\neq i$ we obviously have
  \[
    \chf{C_j}Ts_jp_{\lambda(j)}(\chf{C_i}Ts_ip_{\lambda(i)})^*
    =\chf{C_j}Ts_jp_{\lambda(j)}p_{\lambda(i)}s_i^*T^*\chf{C_i} = 0
  \]
  and, since $s_i$ is a unitary on $\CH_{A_i}$, by construction we also have
  \[
    (\chf{C_i}Ts_ip_{\lambda(i)})^*\chf{C_j}Ts_jp_{\lambda(j)}
    =p_{\lambda(i)}s_i^*T^*\chf{C_i}\chf{C_j}Ts_jp_{\lambda(j)} = 0.
  \]

  We then define $s_{J, \lambda} \coloneqq \sum_{i\in J}s_i +\sum_{i\in I\smallsetminus J}1_{\CH_{A_i}}$, and observe that it is a unitary of $\CHx$ supported on $\diag(A_i\mid i\in I)$. We claim that such $s_{J, \lambda}$ satisfies our requirements.
  The operator $T_{J, \lambda} \coloneqq\sum_{i\in J} \chf{B_i}Ts_i p_{\lambda(i)}$ is $(\gaugey\circ R)$-controlled by construction. Moreover,
  \[
    \norm{Ts_{J, \lambda} p_{J,\lambda} - T_{J, \lambda}} = \norm{\sum_{i\in J}\bigparen{Ts_ip_{\lambda(i)} - \chf{B_i}Ts_ip_{\lambda(i)}}}
    = \norm{\sum_{i\in J}\chf{C_i}Ts_ip_{\lambda(i)}}.
  \]
  By orthogonality, the latter is simply $\max_{i\in J}\norm{\chf{C_i}Ts_ip_{\lambda(i)}}\leq \varepsilon$.
\end{proof}

\begin{lemma}\label{lem: quasi-control vs str quasi-control}
  Let $T\colon\CHx\to\CHy$ be weakly quasi-controlled and $\CHy$ admissible.
  Suppose that $\cappmap[T]{\delta_0}{F_0}{E_0}$ is coarsely everywhere defined for some $\delta_0>0$, $E_0\in\CE$ and $F_0\in\CF$.
  Then, for every $E\in \CE$ and $\delta>0$ there exists some $F\in\CF$ such that for every measurable $E$-controlled $A\subseteq X$ we have
    \[
      \norm{\chf{C}T\chf{A} - T\chf{A}}\leq{\delta}
    \]
    for every measurable $C\subseteq Y$ disjoint from $\appmap[T]{\delta_0}{F}{E_0}(A)$.
\end{lemma}
\begin{proof}
  Since the domain of $\appmap[T]{\delta_0}{F_0}{E_0}$ is coarsely dense, we may find a large enough gauge $E_1\in\CE$ such that for every $E$-bounded (measurable) $A\subseteq X$ there is a defining pair $B_0\times A_0\subseteq \appmap[T]{\delta_0}{F_0}{E_0}$ with $A\subseteq E_1(A_0)$. Here, by definition, $B_0\times A_0$ is measurable and $(F_0\otimes E_0)$-controlled, and we may fix some $v_0\in (\CH_{A_0})_1$ with $\norm{\chf{B_0}T(v_0)}>{\delta_0}$.

  Now, for any arbitrary $v\in(\CH_A)_1$, the operator $\matrixunit{v}{v_0}$ is $E_1$-controlled by \cref{lemma:supp-vector-a-times-a}. The weak quasi-control condition on $T$ implies that for every $\varepsilon>0$ there is an $F_1\in\CF$, depending only $E_1$ and $\varepsilon$, such that $T\matrixunit{v}{v_0}T^*$ is $\varepsilon$-$F_1$-quasi-local.
  Moreover, observe that 
  \begin{equation*}
    \scal{T^*\chf{B_0}T(v_0)}{v_0}
    =\scal{\chf{B_0}T(v_0)}{\chf{B_0}T(v_0)} > {\delta_0}^2.
  \end{equation*}
  Now, if $C\subseteq Y$ is measurable and $F_1$-separated from $B_0$, by quasi-locality of $T \matrixunit{v}{v_0} T^*$, we have:
  \begin{align*}
    \varepsilon &\geq \norm{\chf{C}T\matrixunit{v}{v_0}T^*\chf{B_0}}  \\
     & \geq\norm{\chf{C}T(v)}\scal{v_0}{T^*\chf{B_0}\Bigparen{\frac{T(v_0)}{\norm{T(v_0)}}} } > \norm{\chf{C}T(v)} \, \frac{{\delta_0}^2}{\norm{T(v_0)}}.
  \end{align*}
  Thus
  \[
    \norm{\chf{C}T(v)} < \varepsilon\norm{T}/{\delta_0}^2.
  \]
  Now we are almost done. Let $\varepsilon \coloneqq \delta_0^2\delta/\norm{T}$, and consider the resulting $F_1\in\CF$---we may assume that $F_1$ contains the diagonal. 
  Let $F\coloneqq \gaugey\circ F_1\circ F_0\circ \op{F_1}\circ \gaugey$, where $\gaugey$ is an admissibility gauge for $\CHy$. This choice is so that $F_1(B_0)$ is contained in some measurable $F$-controlled set. Since $B_0\subseteq F_1(B_0)$, the set $F_1(B_0)$ is contained in the image of $\appmap[T]{\delta_0}{F}{E_0}$, so any $C$ disjoint from  $\appmap[T]{\delta_0}{F}{E_0}(B_0)$ is $F_1$-separated form $B_0$. The claim then follows.
\end{proof}

We may finally prove \cref{prop: quasi-control vs str quasi-control}.

\begin{proof}[Proof of \cref{prop: quasi-control vs str quasi-control}]
  By \cref{lem: approximating relation contained in quasi-support,lem: quasi-supp in approx-supp}, we already know that we always have coarse containments
  \[
    \cappmap[T]{\delta}{F}{E} \crse\subseteq \qcsupp(T)\crse\subseteq\acsupp(T).
  \]
  What we need to show is that, under the assumptions of \cref{prop: quasi-control vs str quasi-control}, $\cappmap[T]{\delta}{F}{E}$ coarsely contains the approximate (resp.\ quasi-) coarse support of $T$. Namely, we need to verify the approximate (resp.\ quasi) version of \cref{def: approximate and quasi support}~\cref{def: approximate and quasi support:contains-heart}.

  Fix a large enough gauge $\gaugex$ such that there is a $\gaugex$-controlled discrete partition $X=\bigsqcup_{i\in I}A_i$ with $\CH_{A_i}$ of infinite rank, and let $\gaugey$ be a gauge for admissibility and non-degeneracy of $\CHy$.
  Fix also  $\varepsilon>0$. 
  
  For any $\varepsilon_1>0$, \cref{lem: quasi-control vs str quasi-control} yields an $F_1\in\CF$ such that $\norm{\chf CT \chf{A}}\leq\varepsilon_1$ whenever $A$ is measurable and $\gaugex$-controlled and $C\subseteq Y$ is measurable and $\appmap[T]{\delta}{F_1}{E}$-separated from $A$.
  We may then apply \cref{lem: find unitary for strong approximate} to deduce that for every $(J,\lambda)\in\Lambda_{\rm fin}$ there is a unitary $s_{J,\lambda}$ supported on $\diag(A_i\mid i\in I)$ such that $Ts_{J,\lambda}p_{J,\lambda}$ is $\varepsilon_1$-$R_1$-approximately controlled for $R_1\coloneqq \gaugey\circ\appmap[T]{\delta}{F_1}{E}$.
  Since $T$ and $s_{J, \lambda}$ are isometries, we observe that 
  \[
    Tp_{J,\lambda} = (Ts_{J,\lambda}^*T^*)(Ts_{J,\lambda}p_{J,\lambda}).
  \]
  By the assumption on $T$, for any $\varepsilon_2>0$ there is an $F_2\in \CF$ depending only on \gaugex{} and $\varepsilon_2$ such that $Ts_{J,\lambda}^*T^*$ is $\varepsilon_2$-$F_2$-approximable (resp.\ quasi-local).
  It follows that $Tp_{J,\lambda}$ is $\varepsilon_3$-$R_3$-approximately (resp.\ quasi-) controlled for 
  \[
    \varepsilon_3 = \varepsilon_2+\varepsilon_1
    \;\;\text{and}\;\;
    R_3 \coloneqq F_2\circ\gaugey\circ R_1
  \]
  (here we have used that the net $(Tp_{J,\lambda})_{(J,\lambda)\in\Lambda_{\rm fin}}$ strongly converges to $T$). On the other hand, \cref{lemma:qloc-sot-closed} shows that the set of $\varepsilon_3$-$R_3$-approximately (resp.\ quasi-) controlled operators is \sot closed. It then follows that $T$ itself $\varepsilon_3$-$R_3$-approximately (resp.\ quasi-) controlled.

  Thus, to finish, let $\varepsilon_1 \coloneqq \varepsilon/2 \eqqcolon \varepsilon_2$, and observe that
  \[
    R_3 = F_2\circ\gaugey\circ\gaugey\circ\appmap[T]{\delta}{F_1}{E}
  \]
  is a controlled thickening of the controlled relation $\appmap[T]{\delta}{F_1}{E}$. In particular, $\crse R_3$ is a partial coarse map that coarsely contains $\cappmap[T]{\delta}{F}{E}$. Since $\cappmap[T]{\delta}{F}{E}$ is coarsely everywhere defined, this implies that indeed $\crse R_3 = \cappmap[T]{\delta}{F}{E}$ (cf.\ \cref{lem: functions with same domain coincide}). This shows that $\cappmap[T]{\delta}{F}{E}$ satisfies \cref{def: approximate and quasi support}~\cref{def: approximate and quasi support:contains-heart}. 
\end{proof}

\chapter{Consequences of the refined rigidity} \label{sec:consequences strong rigidity}
In this last chapter we explore some consequences of \cref{thm: strong rigidity}. For the most part, these are generalizations of \cite{braga_gelfand_duality_2022}.

\section{Multipliers of the Roe algebra} \label{subsec: multiplier}
In this section we extend \cite{braga_gelfand_duality_2022}*{Theorem~4.1} to the setting of coarsely locally finite coarse spaces with countably generated coarse structure.
Namely, we show that in this case $\cpcstar{\CHx}$ coincides with the multiplier algebra of $\roecstar{\CHx}$.

As is well known, since $\roecstar\CHx\leq\CB(\CHx)$ is a non degenerate concretely represented \cstar{}algebra, the multiplier algebra $\multiplieralg{\roecstar\CHx}$ is equal to the idealizer of $\roecstar\CHx$ in $\CB(\CHx)$ 
\[
  \multiplieralg{\roecstar\CHx} = \{t\in\CB(\CHx)\mid ts,st\in \roecstar\CHx\ \forall s\in \roecstar\CHx\} 
\]  
(see, \emph{e.g.}\ \cite{blackadar2006operator}*{Theorem~II.7.3.9}). In particular, we may compare $\multiplieralg{\roecstar{\CHx}}$ and $\cpcstar{\CHx}$, as they are both subalgebras of $\CB(\CHx)$. 
The following is now a simple consequence of the rigidity theory we developed.
\begin{corollary}[cf.~\cite{braga_gelfand_duality_2022}*{Theorem 4.1}] \label{cor: cpc is multiplier of roec}
  Let $\crse X$ be countably generated and coarsely locally finite, and $\CHx$ be a discrete ample module. Then
  \[
    \multiplieralg{\roecstar{\CHx}}=\cpcstar\CHx.
  \]
\end{corollary}
\begin{proof}
  The inclusion $\cpcstar\CHx\subseteq\multiplieralg{\roecstar{\CHx}}$ is the easy implication as it suffices to observe that if $t$ has controlled propagation and $s$ is locally compact then $st$ and $ts$ are locally compact as well (details are given in \rfMultRoeinCcp).
  For the converse containment, it is enough to show that all the unitaries in $\multiplieralg{\roecstar{\CHx}}$ belong to $\cpcstar\CHx$. Let then $U$ be such a unitary, and note that $\Ad(U)$ defines a \Star{}isomorphism of $\roecstar{\CHx}$. By \cref{thm: strong rigidity}, $U$ is effectively approximately controlled and its approximate support $\acsupp(U)$ is a partial coarse function. It only remains to verify that $\acsupp(U)\crse\subseteq \cid_{\crse X}$.

  Let $t\in\roestar\CHx$ be any operator with $\csupp(t)=\cid_{\crse X}$, \emph{e.g.}\ a locally finite rank projection selecting a rank one vector on each $\CH_{A_i}$ for some locally finite discrete partition.
  By assumption, $tU^*$ is still in $\roecstar\CHx$, therefore
  \[
    \cid_{\crse X} = \csupp(tU^*U) \crse\subseteq \cid_{\crse X}\crse\circ \acsupp(U) 
    =\acsupp(U).
  \]
  Since $\cid_{\crse X}$ is coarsely everywhere defined, it follows that $\cid_{\crse X} = \acsupp(U)$ (see \cref{lem: functions with same domain coincide}), and hence $U$ belongs to $\cpcstar\CHx$.
\end{proof}

\begin{remark}\label{rmk: ample is not necessary}
  The \emph{ample} condition on $\CHx$ in \cref{cor: cpc is multiplier of roec} is actually \emph{not} necessary for the conclusion to hold. The proof we just presented does require it, though: ampleness is needed for the trick powering \cref{lem: find unitary for strong approximate}.
  In view of \cref{rmk: unital Roe algebras}, the only modules that we do not cover here are those modules that are neither locally finite nor ample. That is, there is a bounded $A\subseteq X$ with $\chf{A}$ of infinite rank and at the same time there are arbitrarily thick $A'\subseteq X$ such that $\chf{A'}$ has finite rank. Such a setup seems to be unlikely to appear in applications, so we feel content with our version of \cref{cor: cpc is multiplier of roec}.
  The reader wishing to prove \cref{cor: cpc is multiplier of roec} in full generality may adapt the proof given in \cite{braga_gelfand_duality_2022}*{Theorem~4.1}. In doing so, some effort is spared by observing that \cite{braga_gelfand_duality_2022}*{Claim~4.2} is a direct consequence of \cref{thm: quasi-proper}. The approximate units of \cref{lemma:discrete-module-admits-nice-projections} are also helpful.

  One may of course completely avoid using \cref{thm: quasi-proper}. In that case, it is interesting to observe that one can use \cref{cor: cpc is multiplier of roec} instead of \cref{cor: uniformization for roe algs} at the beginning of the proof of \cref{thm: strong rigidity} to show that $U$ and $U^*$ are weakly approximately controlled: an isomorphism sending $\roecstar\CHx$ to $\roecstar\CHy$ extends to an isomorphism of their multiplier algebras, so the usual uniformization theorem applies (cf.\ \cref{thm: uniformization}). However, we find the proof relying on \cref{thm: quasi-proper} more conceptual and easier to generalize.
\end{remark}

\begin{remark} \label{rem:multiplier-fails-without-cnt-gen}
  Observe that \cref{cor: cpc is multiplier of roec} does \emph{not} hold without the assumption that $\crse{X}$ be countably generated and/or coarsely locally finite. Indeed, let $\crse{X}$ be the coarse space described in \cref{ex:non-cnt-gen-coarse-space-with-roe-compact}, where it is also shown that $\lccstar{\ell^2(X)} = \CK(\ell^2(X))$. Since $\crse{X}$ is coarsely connected it follows that $\roecstar{\ell^2(X)} = \CK(\ell^2(X))$ as well (cf.\ \cref{prop: roelike cap compacts} and the subsequent remark).
  Thus, $\multiplieralg{\roecstar{\ell^2(X)}} = \CB(\ell^2(X))$, whereas one can easily construct bounded operators that are not approximable, \emph{i.e.}\ $\cpcstar{\ell^2(X)} \neq \CB(\ell^2(X))$.
\end{remark}

\section{(Outer) automorphisms of Roe algebras}\label{subsec: outer of Roe}
We keep exploring the consequences of \cref{thm: strong rigidity} examining what it implies for automorphisms of Roe-like \cstar{}algebras. For reasons that will become apparent soon, we are particularly interested in the \emph{outer} automorphisms of a Roe-like \cstar{}algebra $\roeclike{\CHx}$, and how these relate to the coarse equivalences of $\crse{X}$ (cf.\ \cref{cor: isomorphic out aut,cor: all groups are iso}).
This line of investigation was started in \cite{braga_gelfand_duality_2022}*{Theorem C}, where $\Out(\roecstar{\crse X})$ is studied under the assumption of property A.

First we prove that any automorphism of $\cpcstar{\CHx}$ leaves $\roecstar{\CHx}$ invariant.
\begin{corollary} \label{cor: aut-cp sends roe-to-roe}
  Let $\crse X$ be a coarsely locally finite and countably generated coarse space, and let $\CHx$ be ample and discrete. 
  Then, every automorphism of $\cpcstar\CHx$ sends $\roecstar{\CHx}$ into itself.
\end{corollary}
\begin{proof}
  Fix $\phi\in \aut(\cpcstar\CHx)$. By \cref{thm: roe as intersection} we known that
  \[
    \roecstar\CHx = \cpcstar\CHx\cap\lccstar{\CHx}.
  \]
  It is, hence, enough to show that $\phi$ preserves local compactness. By \cref{thm: strong rigidity}, we deduce that $\phi=\Ad(U)$ for some unitary $U$ with $U^*$ strongly approximately controlled. In particular, $U^*$ is also strongly quasi-controlled. This implies that for every measurable and bounded $A\subseteq X$ and $\varepsilon>0$ there is a measurable and bounded $B\subseteq X$ such that $U^*\chf A\approx_\varepsilon \chf BU^*\chf A$.
  Given any $t\in\lccstar{\CHx}$, observe that
  \[
    \phi(t)\chf A = UtU^*\chf A = \lim_{B\text{ bounded}} U (t\chf B) U^*\chf A,
  \]
  and the right-hand side is a limit of compact operators.
\end{proof}

Recall that the group of outer automorphisms of a \cstar{}algebra $A$ is defined as
\[
\out(A)\coloneqq \aut(A)/\U(\CM(A)),
\]
where unitaries in the multiplier algebras $\CM(A)$ act by conjugation on $A$. Generally speaking, automorphisms that are implemented by unitaries in \(\CM(A)\) are called \emph{inner}. It makes thus sense that elements of the quotient \(\out(A)\) above are called \emph{outer}.
Combining \cref{cor: cpc is multiplier of roec,cor: aut-cp sends roe-to-roe} yields the following.
\begin{corollary}\label{cor: aut of Roe}
    Let $\crse X$ be a coarsely locally finite and countably generated coarse space. Let $\CHx$ be ample and discrete. Then
    \[
      \aut(\roecstar\CHx) = \aut(\cpcstar\CHx) \; \text{and} \; \Out(\roecstar\CHx) = \Out(\cpcstar\CHx),
    \]
    where equality of $\aut(\variable)$ is meant as subsets of $\U(\CHx)$.
\end{corollary}

\section{Outer automorphisms vs.\ coarse equivalences I}\label{subsec: rig outer aut 1}
As it turns out, there is a very strong relation between (outer) automorphisms of Roe-like \cstar{}algebra and coarse equivalences, which we shall explore now. 
The arguments in this section were essentially observed by Krutoy in \cite{Krutoy2025categorical}, while our original approach is included in the next section.

\begin{definition}
  Let $\coe{\crse X}$ be the set of coarse equivalences $\crse{f \colon X \to X}$.
\end{definition}
True to our notational conventions, $\coe{\crse X}$ is a set of equivalence classes of controlled relations. That is, coarse equivalences are considered up to closeness.
Assume now that $\crse X$ be countably generated and let $\CHx$ be an discrete ample module.
Consider the following set of unitaries:
\begin{align*}
  \eacuni\CHx
  &\coloneqq\braces{U\in \U(\CHx)\mid U,\ U^* \text{ weakly approximately controlled}} \\
  &=\braces{U\in \U(\CHx)\mid \acsupp(U)\in\coe {\crse X}}  
\end{align*}
(observe that the equality of the latter two is given by \cref{cor: unitaries weak control adjoint iff effective ce}).
Associating with each $U\in \eacuni\CHx$ its approximate coarse support $\acsupp(U)$ gives a natural mapping
\[
\acsupp \colon \eacuni\CHx\to \coe{\crse X}.
\]

Observe that both $\eacuni\CHx$ and $\coe{\crse X}$ are groups under composition, and we already noticed that $\acsupp$ is actually a homomorphism (cf.\ \cref{rmk: functoriality of qcsupp}). In fact, for every $U,V\in \eacuni\CHx$ we know by \cref{lem: composition of approximate and quasi support} that
\[
\acsupp(UV)\crse \subseteq \acsupp(U)\crse \circ\acsupp(V).
\]
Since $\acsupp(UV)$ is itself a coarse equivalence (and is in particular coarsely everywhere defined), equality follows from \cref{lem: functions with same domain coincide}.

Recall however that if $\CHx$ is $\kappa$-ample and of local rank at most $\kappa$, then every coarse equivalence is covered by some controlled unitary (cf.\ \cref{prop: existence of covering iso}). This implies that $\acsupp$ is a surjective homomorphism.
Its kernel is also easy to describe. In fact, we already observed (cf.\ \cref{rkm: about approximate and quasi support: contained in idx}) that $\acsupp(U)=\cid_{\crse X}$ if and only if $U\in\cpcstar\CHx$.
That is, we have a canonical isomorphism
\[
\acsupp \colon \eacuni\CHx/ \U(\cpcstar\CHx)\xrightarrow{\ \cong\ } \coe{\crse X}.
\]
On the other hand, conjugation defines another natural homomorphism
\[
  \Ad\colon \eacuni\CHx\to\aut(\cpcstar\CHx),
\]
and \cref{thm: strong rigidity} shows that this homomorphism is also surjective.
By definition, the kernel of $\Ad$ is the group of unitaries in $\eacuni\CHx$ that centralize $\cpcstar\CHx$.

Let $\crse X= \bigsqcup_{i\in I}\crse X_i$ be the decomposition in coarsely connected components, and $\CH_{\crse X_i}$ the restriction of the module to each component. By \cref{prop: roelike cap compacts} we know that
\[
\bigoplus_{i\in I} \CK(\CH_{\crse X_i}) \leq \cpcstar\CHx.
\]
It then follows that the commutant of $\cpcstar\CHx$ in $\CB(\CHx)$ is equal to $\prod_{i\in I}\CCC\cdot \chf{\crse X_i}$.
As a consequence, $\Ad$ descends to a canonical isomorphism
\[
  \Ad \colon \eacuni\CHx /\braces{\sum_{i\in I}\lambda_i\chf{\crse X_i}\mid \lambda_i\in \CCC} \xrightarrow{\ \cong\ }\aut(\cpcstar\CHx).
\]
In turn, quotienting out $\U(\cpcstar\CHx)$ descends to another isomorphism
\[
  \eacuni\CHx/ \U(\cpcstar\CHx)\xrightarrow{\ \cong\ } \aut(\cpcstar\CHx)/\U(\cpcstar\CHx) = \out(\cpcstar\CHx).
\]

Of course, all these considerations hold verbatim for
\begin{align*}
  \eqcuni\CHx
  &\coloneqq\braces{U\in \U(\CHx)\mid U,\ U^* \text{ weakly quasi-controlled}} \\
  &=\braces{U\in \U(\CHx)\mid \qcsupp(U)\in\coe {\crse X}}  
\end{align*}
and $\qcsupp$. We thus obtain the following.

\begin{cor}\label{cor: isomorphic out aut}
  Let $\crse X$ be countably generated, $\CHx$ $\kappa$-ample discrete and of local rank at most $\kappa\geq \aleph_0$. Then there are canonical isomorphisms:
  \[
  \coe{\crse{X}}
  \cong \frac{\eacuni\CHx}{\U(\cpcstar\CHx)}
  \cong \frac{\eqcuni\CHx}{\U(\qlcstar\CHx)}
  \cong\out(\cpcstar\CHx)
  \cong\out(\qlcstar\CHx).
  \]
\end{cor}

\section{Outer automorphisms vs.\ coarse equivalences II}\label{subsec: rig outer aut 2}
In the previous section we proved \cref{cor: isomorphic out aut} by showing that $\acsupp$ and $\qcsupp$ give rise to surjective homomorphisms onto $\coe{\crse X}$. One alternative approach is to work directly with \cref{prop: existence of covering iso} to map $\coe{\crse X}$ into $\U(\CHx)$.
The issue here is that for a given coarse equivalence $\crse f\colon\crse X\to\crse Y$ the unitary $U\colon\CHx\to\CHy$ provided by \cref{prop: existence of covering iso} is very much not unique.
In turn, this lack of uniqueness results in difficulties in trying to define a homomorphism $\coe{\crse X}\to \U(\CHx)$.

In hindsight, it is rather natural to expect that to obtain a homomorphism one should take some quotient on the right hand side.
This strategy works well, because the unitaries provided by \cref{prop: existence of covering iso} are unique up to conjugation (cf.\ \cite{braga_gelfand_duality_2022}, and \rfLemCoveringUniIsUnique).
To make make this statement precise, we introduce the following.

\begin{definition}[cf.\ \rfDefCEandCUni] \label{def: CE and CUni}
  Let $\crse X$ be a coarse space, and let $\CHx$ be an $\crse{X}$-module. We denote by $\cuni{\CHx}$ the set of unitaries $U \colon \CHx \to \CHx$ such that both $U$ and $U^*$ are controlled.
\end{definition}

\begin{remark}
  There are obvious containments:
  \[
    \cuni{\CHx} \subseteq \eacuni{\CHx} \subseteq \eqcuni{\CHx}.
  \]
\end{remark}

Observe that the group of unitaries $\U(\cpstar{\CHx})$ is a normal subgroup of $\cuni{\CHx}$. Assigning with each coarse equivalence a covering unitary via \cref{prop: existence of covering iso} defines a (non-canonical) mapping $\coe{\crse X}\to \cuni{\CHx}$, and it is not hard to see that this mapping becomes a canonical homomorphism when quotienting out $\U(\cpstar{\CHx})$. With some extra care, one may even prove the following.

\begin{theorem}[cf.\ \rfIsoCEtoCtrUni]\label{thm: coarse equivalences embed in cuni}
  Let $\CHx$ be a discrete $\kappa$-ample module of local rank $\kappa$, where $\kappa>0$. There is a canonical isomorphism 
  \[
    \rho\colon\coe{\crse X}\xrightarrow{\ \cong\ } \cuni{\CHx}/\U(\cpstar{\CHx}).
  \]
\end{theorem}
\begin{remark}
  The map $\rho$ in \cref{thm: coarse equivalences embed in cuni} is defined using \cref{prop: existence of covering iso} to choose covering unitaries.
  The ``canonical'' adjective means that the resulting map $\rho$ does not depend on the choice made.
\end{remark}

This approach is inverse to that of \cref{subsec: rig outer aut 1}, which would consist of noting that sending a unitary in $\cuni{\CHx}$ to the coarse equivalence obtained from its approximating relations defines a homomorphism $\cuni{\CHx}\to\coe{\crse X}$ whose kernel is $\U(\cpstar{\CHx})$.

\smallskip

Once again, $\Ad$ induces a homomorphism onto the group of automorphisms of Roe-like \cstar{}algebras, whose kernels consist of central unitaries.
We consider the induced homomorphisms to the outer automorphisms groups:
\[
  \sigma_{\CR}\colon\cuni{\CHx}/\U(\cpstar{\CHx})\rightarrow \out(\roeclike{\CHx}),
\]
where $\CR$ may be either of `cp', `ql' or `Roe' (in the ql case it is necessary to assume that $\CHx$ be admissible to make sure that $\sigma_{\rm ql}$ is well-defined).
It is proved in \rfEmbeddingCtrUniintoOut (extending \cite{braga_gelfand_duality_2022}*{Section 2.2}) that these homomorphisms are usually injective.

\begin{theorem}[cf.\ \rfEmbeddingCtrUniintoOut]\label{thm: injection into outer}
  Let $\CHx$ be an admissible $\crse X$-module. Then $\Ad$ induces canonical homomorphisms
  \[
    \sigma_{\CR}\colon\cuni{\CHx}/\U(\cpstar{\CHx})\rightarrow \out(\roeclike{\CHx}).
  \]
  Moreover, if $\multiplieralg{\roeclike{\CHx}} \subseteq \qlcstar{\CHx}$ then $\sigma_{\CR}$ is injective.
  In particular, the latter is the case when $\roeclike\CHx$ is $\cpcstar\CHx$ or $\qlcstar{\CHx}$.
\end{theorem}

With \cref{cor: cpc is multiplier of roec} at hand, we also have:
\begin{corollary}
  If $\crse X$ is countably generated and coarsely locally finite and $\CHx$ is a discrete (and ample)\footnote{\, This is not necessary, see \cref{rmk: ample is not necessary}.}, then $\sigma_{\rm Roe}$ is also injective.
\end{corollary}

Applying \cref{thm: injection into outer}, we now obtain the following.

\begin{corollary}\label{cor: surjection to Out}
  Let $\crse X$ be a countably generated coarse space and $\CHx$ be discrete and $\kappa$-ample of local rank $\kappa\geq\aleph_0$. Then $\sigma_{\rm cp}$ and $\sigma_{\rm ql}$ are surjective.

  If $\crse{X}$ is also coarsely locally finite, then $\sigma_{\rm Roe}$ is surjective as well.
\end{corollary}
\begin{proof}
  Fix $\phi\in \aut(\roeclike\CHx)$.
  By \cref{thm: strong rigidity}, $\phi=\Ad(U)$, where the quasi support $\qcsupp(U)\colon\crse X\to\crse X$ is a coarse equivalence.
  Let $V\in\cuni{\CHx}$ be a unitary covering $\qcsupp(U)$, which exists by \cref{prop: existence of covering iso}.
  We wish to prove that $[\Ad(V)]=[\phi]$ in $\Out(\roeclike\CHx)$.
  That is, we must show that the unitary $V^*U$ belongs to $\multiplieralg{\roeclike\CHx}$.
  In the quasi-local case we have
  \begin{align*}
    \qcsupp\left(V^*U\right) & \crse\subseteq \qcsupp\left(V^*\right)\crse\circ\qcsupp\left(U\right) \\
    & \crse{=} \qcsupp\left(U\right)^{-1}\crse\circ \qcsupp\left(U\right) \crse{=} \cid_{\crse X}.
  \end{align*}
  The approximate case uses $\acsupp(U)$ and is analogous. The Roe algebra case follows from the approximate one and \cref{cor: aut of Roe}.
\end{proof}

\begin{remark}
  In the setting of \cref{subsec: rig outer aut 1}, one can of course deduce both injectivity and surjectivity of $\sigma_\CR$ by the isomorphisms of \cref{cor: isomorphic out aut}. In fact, the proof of \cref{cor: surjection to Out} is essentially repeating the argument that $\eqcuni{\CHx}\to\coe{\crse X}$ is surjective (hence has a section) and
  \[
    \eqcuni{\CHx}/\U(\qlcstar \CHx) \to \out(\qlcstar\CHx)
  \]
  is surjective and well-defined.
  
  For injectivity, this essentially amounts to observing that
  \[
  \U(\cpstar\CHx) = \cuni{\CHx}\cap\U(\cpcstar\CHx) = \cuni{\CHx}\cap\U(\qlcstar\CHx),
  \]
  which can be deduced by observing that for a unitary $U\in\cuni{\CHx}$ the supports $\csupp(U)$, $\acsupp(U)$ and $\qcsupp(U)$ coincide (and are contained in \(\cid_{\crse X}\) if \(U\) is moreover contained in \(\U(\cpstar \CHx)\)).
\end{remark}

We may then add $\cuni{\CHx}/\U(\cpstar{\CHx})$ to the list of isomorphic groups in \cref{cor: isomorphic out aut}. We do this in the following somewhat verbose statement.

\begin{corollary}\label{cor: all groups are iso}
  Let $\crse X$ be a coarse space, and let $\CHx$ be a discrete and $\kappa$-ample $\crse{X}$\=/module of local rank $\kappa\geq \aleph_0$. Then the following groups are isomorphic.
  \begin{enumerate}[label=(\roman*)]
    \item The group of coarse equivalences $\crse{f \colon X \to X}$.
    \item The group of controlled unitaries of $U \colon \CHx \to \CHx$, with $U^*$ controlled, up to unitary equivalence in $\cpcstar{\CHx}$.
  \end{enumerate}
  If $\crse{X}$ is countably generated, then the above are also isomorphic to: 
  \begin{enumerate}[label=(\roman*), resume]
    \item The group of outer automorphisms of $\cpcstar{\CHx}$.
    \item The group of outer automorphisms of $\qlcstar{\CHx}$.
    \item The group of unitaries $U \in \U(\CHx)$ such that \(\acsupp(U)\) defines a coarse equivalence up to unitary equivalence in \(\cpcstar \CHx\).
    \item The group of unitaries $U \in \U(\CHx)$ such that \(\qcsupp(U)\) defines a coarse equivalence up to unitary equivalence in \(\qlcstar \CHx\).
  \end{enumerate}
  Lastly, if $\crse{X}$ is coarsely locally finite as well, then they are isomorphic to:
  \begin{enumerate}[label=(\roman*), resume]
    \item The group of outer automorphisms of $\roecstar{\CHx}$.
  \end{enumerate}
\end{corollary}

As final words, it is worth noting that the isomorphism between $\Out(\cpcstar{\CHx})$ and $\Out(\qlcstar{\CHx})$ in \cref{cor: all groups are iso} is \emph{not} induced by the identity on $\aut(\variable)$ (as was instead the case for $\out(\roecstar\CHx)=\out(\cpcstar\CHx)$, cf.\ \cref{cor: aut of Roe}).
Namely, even though both $\aut(\cpcstar{\CHx})$ and $\aut(\qlcstar\CHx)$ can be seen as subgroups of $\U(\CHx)$, it is in general not true that $\aut(\cpcstar{\CHx}) = \aut(\qlcstar\CHx)$.
As a matter of fact, when $\cpcstar{\CHx}\neq \qlcstar\CHx$ the opposite is true. The following is inspired from \cite{ozawa-2023} and should be compared with \cref{cor: aut-cp sends roe-to-roe}.

\begin{corollary}
  Let $\crse X$ be a coarsely locally finite and countably generated coarse space. Let $\CHx$ be discrete and $\kappa$-ample of rank $\kappa\geq \aleph_0$. Suppose there is some unitary $U \in \qlcstar{\CHx}$ that is not approximable. Then $\Ad(U)$ does not send $\cpcstar{\CHx}$ into itself.
\end{corollary}
\begin{proof}
  Let $U$ be some unitary in $\qlcstar{\CHx}$. If $\Ad(U)$ sends $\cpcstar{\CHx}$ into itself then, by \cref{thm: strong rigidity} it follows that $U$ is strongly approximately controlled, and $\acsupp(U)=\qcsupp(U)$.
  Since $U$ is quasi-local by assumption, $\qcsupp(U)$ has to be $\cid_{\crse X}$. Then also $\acsupp(U)=\cid_{\crse X}$, so $U$ belongs to $\cpcstar{\CHx}$.
\end{proof}

\bibliography{BibRigidity}

\end{document}

%% file: AuxCommandsBecauseWindows.tex
\newcommand{\supp}{{\rm Supp}}

\newcommand{\CHx}{{\CH_{\crse X}}}
\newcommand{\CHy}{{\CH_{\crse Y}}}
\newcommand{\CHz}{{\CH_{\crse Z}}}

\newcommand{\C}{\mathbb{C}}

\makeatletter
\def\appmap{\@ifnextchar[{\@withappmap}{\@withoutappmap}}
  \def\@withappmap[#1]#2#3#4{{{f}^{#1}_{#2, #3, #4}}}
  \def\@withoutappmap#1#2#3{{{f}_{#1, #2, #3}}}
\def\cappmap{\@ifnextchar[{\@withcappmap}{\@withoutcappmap}}
  \def\@withcappmap[#1]#2#3#4{{\crse{f}^{#1}_{#2, #3, #4}}}
  \def\@withoutcappmap#1#2#3{{\crse{f}_{#1, #2, #3}}}
\def\crseimg{\@ifnextchar[{\@withcrseimg}{\@withoutcrseimg}}
  \def\@withcrseimg[#1]#2#3{{I^{#1}_{#2, #3}}}
  \def\@withoutcrseimg#1#2{{I_{#1, #2}}}
\makeatother

\newcommand*{\concvec}[3]{{\rm C\hspace{-.5em}V}_{#1,#2}{\paren{#3}}}

\newcommand*{\cstar}{\texorpdfstring{$C^*$\nobreakdash-\hspace{0pt}}{*-}}
\newcommand*{\Star}{\(^*\)\nobreakdash-}

\newcommand*{\multiplieralg}[1]{\CM{\left(#1\right)}}

\newcommand*{\uroecstar}[1]{C^*_{\rm u}{\left(#1\right)}}

\newcommand*{\lccstar}[1]{C^*_{{\rm lc}}{\left(#1\right)}}
\makeatletter
  \def\roeclike{\@ifnextchar[{\@withroeclike}{\@withoutroeclike}}
    \def\@withroeclike[#1]#2{\CR^*_{#1}{\left(#2\right)}}
    \def\@withoutroeclike#1{\CR^*{\left(#1\right)}}
  \def\roeclikeone{\@ifnextchar[{\@withroeclikeone}{\@withoutroeclikeone}}
    \def\@withroeclikeone[#1]#2{\CR^*_{#1}{\left(#2\right)}}
    \def\@withoutroeclikeone#1{\CR^*_1{\left(#1\right)}}
  \def\roecliketwo{\@ifnextchar[{\@withroecliketwo}{\@withoutroecliketwo}}
    \def\@withroecliketwo[#1]#2{\CR^*_{#1}{\left(#2\right)}}
    \def\@withoutroecliketwo#1{\CR^*_2{\left(#1\right)}}
  \def\roecstar{\@ifnextchar[{\@withroecstar}{\@withoutroecstar}}
    \def\@withroecstar[#1]#2{C^*_{#1,\rm Roe}{\left(#2\right)}}
    \def\@withoutroecstar#1{C^*_{\rm Roe}{\left(#1\right)}}
  \def\roestar{\@ifnextchar[{\@withroestar}{\@withoutroestar}}
    \def\@withroestar[#1]#2{\C_{#1,\rm Roe}{\left[#2\right]}}
    \def\@withoutroestar#1{\C_{\rm Roe}{\left[#1\right]}}
  \def\classicalroecstar{\@ifnextchar[{\@withclassicalroecstar}{\@withoutclassicalroecstar}}
    \def\@withclassicalroecstar[#1]#2{C^*{\left(#2\right)}}
    \def\@withoutclassicalroecstar#1{C^*{\left(#1\right)}}
  \def\cpcstar{\@ifnextchar[{\@withcpcstar}{\@withoutcpcstar}}
    \def\@withcpcstar[#1]#2{C^*_{#1,\rm cp}{\left(#2\right)}}
    \def\@withoutcpcstar#1{C^*_{\rm cp}{\left(#1\right)}}
  \def\cpstar{\@ifnextchar[{\@withcpstar}{\@withoutcpstar}}
    \def\@withcpstar[#1]#2{\C_{#1,\rm cp}{\left[#2\right]}}
    \def\@withoutcpstar#1{\C_{\rm cp}{\left[#1\right]}}
  \def\qlcstar{\@ifnextchar[{\@withqlcstar}{\@withoutqlcstar}}
    \def\@withqlcstar[#1]#2{C^*_{#1,\rm ql}{\left(#2\right)}}
    \def\@withoutqlcstar#1{C^*_{\rm ql}{\left(#1\right)}}
\makeatother

\newcommand*{\cpop}[1]{{\rm Prop}{(#1)}}
\newcommand*{\cpcontr}[1]{{\rm Prop}{\left(#1\right)}_{\leq 1}}
\newcommand*{\cscontr}[1]{#1\text{-Supp}_{\leq 1}}
\newcommand*{\piinvql}{\CY}
\newcommand*{\piinvcp}{\CZ}
\newcommand*{\phiinvql}[2]{\piinvql_{#1, #2}}
\newcommand*{\phiinvcp}[2]{\piinvcp_{#1, #2}}
\newcommand*{\phiinv}[2]{\CX_{#1, #2}}

\newcommand*{\sotnbhd}[2]{\CN_{#1, #2}}

\newcommand*{\sot}{\text{SOT}}
\newcommand*{\wot}{\text{WOT}}

\newcommand*{\coe}[1]{{\rm CE}{\left(#1\right)}}
\newcommand*{\eacuni}[1]{{\rm EACUni}{{\left(#1\right)}}}
\newcommand*{\eqcuni}[1]{{\rm EQCUni}{{\left(#1\right)}}}
\newcommand*{\cuni}[1]{{\rm CtrUni}{{\left(#1\right)}}}

\newcommand*{\spn}[1]{{\rm Span}\{#1\}}
\newcommand*{\spnclosed}[1]{\overline{{\rm Span}}\{#1\}}

\newcommand*{\charfunc}[1]{\mathbbm{1}_{#1}}                        \newcommand*{\chf}[1]{\charfunc{#1}}
\newcommand*{\charfunccomp}[2]{\charfunc{#1 \smallsetminus #2}}     
\newcommand*{\charfunccompX}[1]{\charfunccomp{X}{#1}}               \newcommand*{\chfcX}[1]{\charfunccompX{#1}}
\newcommand*{\charfunccompY}[1]{\charfunccomp{Y}{#1}}               \newcommand*{\chfcY}[1]{\charfunccompY{#1}}

\newcommand{\gauge}{\ensuremath{\widetilde{E}}{}}
\newcommand{\gaugex}{\ensuremath{\widetilde{E}_{\crse X}}{}}
\newcommand{\gaugey}{\ensuremath{\widetilde{F}_{\crse Y}}{}}

\newcommand*{\matrixunit}[2]{e_{#1,#2}}

\definecolor{darkgreen}{rgb}{0.0, 0.5, 0.0}
\newcounter{fcomment}

\definecolor{brown}{rgb}{0.55, 0.3, 0.25}
\newcounter{dcomment}

\definecolor{newpurple}{rgb}{0.8, 0, 0.9}

\newcommand{\CCC}{\mathbb{C}}

\newcommand{\MM}{\mathbb{M}}		\newcommand{\NN}{\mathbb{N}}
		
		\newcommand{\RR}{\mathbb{R}}

		\newcommand{\ZZ}{\mathbb{Z}}

\newcommand{\CA}{\mathcal{A}}		\newcommand{\CB}{\mathcal{B}}
		\newcommand{\CD}{\mathcal{D}}
\newcommand{\CE}{\mathcal{E}}		\newcommand{\CF}{\mathcal{F}}
\newcommand{\CG}{\mathcal{G}}		\newcommand{\CH}{\mathcal{H}}
		
\newcommand{\CK}{\mathcal{K}}		\newcommand{\CL}{\mathcal{L}}
\newcommand{\CM}{\mathcal{M}}		\newcommand{\CN}{\mathcal{N}}
		\newcommand{\CP}{\mathcal{P}}
		\newcommand{\CR}{\mathcal{R}}
		
\newcommand{\CU}{\mathcal{U}}		\newcommand{\CV}{\mathcal{V}}
		\newcommand{\CX}{\mathcal{X}}
\newcommand{\CY}{\mathcal{Y}}		\newcommand{\CZ}{\mathcal{Z}}

\newcommand{\fkA}{\mathfrak{A}}		\newcommand{\fkB}{\mathfrak{B}}

\DeclarePairedDelimiter\abs{\lvert}{\rvert}		
\DeclarePairedDelimiter\norm{\lVert}{\rVert}		
\DeclarePairedDelimiter\angles{\langle}{\rangle}	

\DeclarePairedDelimiter\paren{(}{)}			
\DeclarePairedDelimiter\braces{\{}{\}}			

	\newcommand{\bigparen}[1]{\paren[\big]{#1}}
	\newcommand{\Bigparen}[1]{\paren[\Big]{#1}}

\newcommand{\scal}[2]{\angles{#1,#2}}			

\DeclareMathOperator{\id}{id}				
\DeclareMathOperator{\rk}{rk}				

\DeclareMathOperator{\aut}{Aut}				
\DeclareMathOperator{\out}{Out}				

\DeclareMathOperator{\diag}{diag}
\DeclareMathOperator{\image}{im}

\DeclareMathOperator{\Out}{Out}

\DeclareMathOperator{\rank}{rank}

\DeclareMathOperator{\U}{U}

\theoremstyle{plain}
\newtheorem{thm}{Theorem}[section]		\newtheorem{theorem}[thm]{Theorem}
		\newtheorem{proposition}[thm]{Proposition}
\newtheorem{lem}[thm]{Lemma}			\newtheorem{lemma}[thm]{Lemma}
\newtheorem{cor}[thm]{Corollary}		\newtheorem{corollary}[thm]{Corollary}

\newtheorem{claim}[thm]{Claim}

\newtheorem*{thm*}{Theorem}			\newtheorem*{theorem*}{Theorem}
\newtheorem*{prop*}{Proposition}		\newtheorem*{proposition*}{Proposition}
\newtheorem*{lem*}{Lemma}			\newtheorem*{lemma*}{Lemma}
\newtheorem*{cor*}{Corollary}			\newtheorem*{corollary*}{Corollary}
\newtheorem*{qu*}{Question}			\newtheorem*{question*}{Question}
\newtheorem*{conj*}{Conjecture}			\newtheorem*{conjecture*}{Question}
\newtheorem*{fact*}{Fact}
\newtheorem*{claim*}{Claim}
\newtheorem*{case*}{Case}
\newtheorem*{problem*}{Problem}

\numberwithin{equation}{section}

\newtheorem{alphthm}{Theorem}			

\newtheorem{alphcor}[alphthm]{Corollary}           

\theoremstyle{definition}
		\newtheorem{definition}[thm]{Definition}
\newtheorem{notation}[thm]{Notation}
		\newtheorem{convention}[thm]{Convention}

\newtheorem*{de*}{Definition}			\newtheorem{definition*}{Definition}
\newtheorem*{notation*}{Notation}
\newtheorem*{conv*}{Convention}			\newtheorem*{convention*}{Convention}

\theoremstyle{remark}
\newtheorem{rmk}[thm]{Remark}			\newtheorem{remark}[thm]{Remark}
			\newtheorem{example}[thm]{Example}

\newtheorem*{remark*}{Remark}

\DeclareMathAlphabet{\mathbit}{OT1}{cmr}{bx}{it}

\crefname{thm}{Theorem}{Theorems}              \crefname{theorem}{Theorem}{Theorems}
\crefname{prop}{Proposition}{Propositions}     \crefname{proposition}{Proposition}{Propositions}
\crefname{lem}{Lemma}{Lemmas}                  \crefname{lemma}{Lemma}{Lemmas}
\crefname{rmk}{Remark}{Remarks}                \crefname{remark}{Remark}{Remarks}
\crefname{cor}{Corollary}{Corollaries}         \crefname{corollary}{Corollary}{Corollaries}
\crefname{qu}{Question}{Questions}             \crefname{question}{Question}{Questions}
\crefname{conj}{Conjecture}{Conjectures}       \crefname{conjecture}{Conjecture}{Conjectures}
\crefname{fact}{Fact}{Facts}
\crefname{claim}{Claim}{Claims}
\crefname{case}{Case}{Cases}
\crefname{alphthm}{Theorem}{Theorems}          \crefname{alphcor}{Corollary}{Corollaries}
\crefname{alphprop}{Proposition}{Propositions}

\ExplSyntaxOn
\cs_new_protected:Nn \egreg_embolden_command:N
 {\cs_set_eq:cN { __egreg_ \cs_to_str:N #1 : } #1
  \cs_set_protected:Npn #1 { \bm { \use:c { __egreg_ \cs_to_str:N #1 : } } }}
\cs_new_protected:Nn \egreg_embolden_char:n
 {\exp_args:Nc \mathchardef { __egreg_#1: } = \mathcode`#1 \scan_stop:
  \cs_set_protected:cn { __egreg_#1_bold: } { \bm { \use:c { __egreg_#1: } } }
  \char_set_active_eq:nc { `#1 } { __egreg_#1_bold: }
  \mathcode`#1 = "8000 \scan_stop:}
\cs_new_protected:Nn \egreg_embolden:
 {\clist_map_function:nN
   {
    A,B,C,D,E,F,G,H,I,J,K,L,M,N,O,P,Q,R,S,T,U,V,W,X,Y,Z,
    a,b,c,d,e,f,g,h,i,j,k,l,m,n,o,p,q,r,s,t,u,v,w,x,y,z,
    *,<,>,/,-,1,2,3,4,5,6,7,8,9,0
   }
   \egreg_embolden_char:n
  \clist_map_function:nN
   {
    \alpha,\beta,\gamma,\delta,\epsilon,\varepsilon,\zeta,\eta,\theta,\vartheta,\iota,\kappa,\lambda,\mu,\nu,\xi,\pi,\varpi,\rho,\varrho,\sigma,\varsigma,\tau,\upsilon,\phi,\varphi,\chi,\psi,\omega,
    \Gamma,\varGamma,\Delta,\varDelta,\Theta,\varTheta,\Lambda,\varLambda,\Xi,\varXi,\Pi,\varPi,\Sigma,\varSigma,\Upsilon,\varUpsilon,\Phi,\varPhi,\Psi,\varPsi,\Omega,\varOmega,
    \cup,\cap,
    \neq,\cong,\ncong,
    \in,\ell,
    \subset,\subseteq,\supset,\supseteq,\subsetneq,\supsetneq,\nsubseteq,\nsupseteq,
    \leq,\geq,\lneq,\gneq,\lhd,\rhd,\trianglelefteq,\trianglerighteq,\triangleright,\trianglerighteq,\circ,
   }
   \egreg_embolden_command:N}
\NewDocumentCommand{\crse}{m}
 {\group_begin:
  \egreg_embolden:
  #1
  \group_end:}
\ExplSyntaxOff

\newcommand{\cid}{\mathbf{id}}

\DeclareMathOperator{\cim}{\mathbf{im}}

\DeclareMathOperator{\cdom}{\mathbf{dom}}
\DeclareMathOperator{\csupp}{\mathbf{Supp}}
\DeclareMathOperator{\qcsupp}{\mathbf{q-Supp}}
\DeclareMathOperator{\acsupp}{\mathbf{a-Supp}}

\newcommand*{\op}[1]{#1^{\scriptscriptstyle\rm T}}

\makeatletter
\def\varcrs{\@ifnextchar[{\@withvarcrs}{\@withoutvarcrs}}
\def\@withvarcrs[#1]#2{\CE_{\rm #2}^{\rm #1}}
\def\@withoutvarcrs#1{\CE_{\rm #1}}
\def\varFcrs{\@ifnextchar[{\@withvarFcrs}{\@withoutvarFcrs}}
\def\@withvarFcrs[#1]#2{\CF_{\rm #2}^{\rm #1}}
\def\@withoutvarFcrs#1{\CF_{\rm #1}}
\makeatother

\newcommand{\csub}{
                    \preccurlyeq}

\newcommand{\Ad}{{\rm Ad}}

\newcommand{\Cat}[1]{\ifmmode \text{\normalfont \textbf{#1}} \else {\normalfont \textbf{#1}}\fi}
\newcommand{\mhyphen}{\textnormal{-}}
\newcommand{\variable}{\,\mhyphen\,}